\documentclass[letterpaper, 11pt,abstract=true]{scrartcl}
\usepackage[english]{babel}
\usepackage[T1]{fontenc}
\usepackage[utf8]{inputenc}
\usepackage{libertine}
\usepackage{color,amssymb,amsfonts,mathrsfs,graphicx,setspace,amsmath}
\usepackage{amsthm}
\usepackage{caption}
\usepackage[left=3cm,right=3cm,top=2.5cm,bottom=3cm]{geometry}
\usepackage[most]{tcolorbox}
\usepackage[framemethod=TikZ]{mdframed}
\usepackage{thm-restate}
\usepackage{fullpage}
\usepackage{xcolor}
\usepackage{tikz}
\usetikzlibrary{shapes.geometric,hobby,calc} 
\usetikzlibrary{decorations.pathmorphing}
\usetikzlibrary{decorations.text}
\usetikzlibrary{shapes.misc}
\usetikzlibrary{calc}
\usetikzlibrary{arrows.meta}
\usetikzlibrary{decorations,shapes}
\usepackage{hyperref}
\usepackage{nameref,cleveref}
\usepackage{macros}
\usepackage{tikzmacros}
\pgfdeclarelayer{dungeon}
\pgfdeclarelayer{background}
\pgfdeclarelayer{foreground}
\pgfsetlayers{dungeon,background,main,foreground}

\usepackage{refcount}

\usepackage{scalerel}
\usepackage{enumitem}
\setlist{nosep}

\usepackage{dsfont}
\usepackage{soul}
\usepackage{xspace}
\usepackage[prependcaption,colorinlistoftodos]{todonotes}
\usepackage{titling}
\usepackage{marginnote}
\usepackage{subcaption}

\usepackage{stmaryrd}

\usepackage{textpos}

\usepackage{comment}

\newcommand{\cC}{\mathcal{C}}

\newcommand{\cP}{\mathcal{P}}
\newcommand{\cQ}{\mathcal{Q}}
\newcommand{\cS}{\mathcal{S}}

\definecolor{CornflowerBlue}{rgb}{0.39, 0.58, 0.93}
\definecolor{LavenderMagenta}{rgb}{0.93, 0.51, 0.93}
\definecolor{PastelOrange}{rgb}{1.0, 0.7, 0.28}
\definecolor{AppleGreen}{rgb}{0.55, 0.71, 0.0}

\hypersetup{
	colorlinks=true,
	linkcolor=myBlue,
	citecolor=myBlue,
	urlcolor=myLightBlue,
	bookmarksopen=true,
	bookmarksnumbered,
	bookmarksopenlevel=2,
	bookmarksdepth=3,
	hypertexnames=false 
}

\title{\Huge{Potatoes and Onions}}

\predate{}
\date{}
\postdate{}

\DeclareRobustCommand{\authorthing}{
    Marcin Briański\thanks{Marcin Briański was partially supported by Polish National Science Centre grant number 2019/34/E/ST6/00443.}\\[0.05cm]
		\normalsize\aff{Jagiellonian University, Kraków}\\
		\normalsize\href{mailto:marcin.brianski@doctoral.uj.edu.pl}{marcin.brianski@doctoral.uj.edu.pl}\\
		\and
	    Meike Hatzel\thanks{Part of this work has been done while Meike Hatzel was working supported by the Institute for Basic Science (IBS-R029-C1).}\\[0.05cm] \normalsize\aff{Technical University of Darmstadt}\\
	    \normalsize\href{mailto:research@meikehatzel.com}{research@meikehatzel.com}\\
	    \and
	    Karolina Okrasa\thanks{Part of this work has been done while Karolina Okrasa was working at the University of Oxford, supported by the project UKRI EP/X024431/1.}\\[0.05cm]
        \normalsize\aff{Warsaw University of Technology}\\
        \normalsize\href{mailto:karolina.okrasa@pw.edu.pl}{karolina.okrasa@pw.edu.pl}
	    \and
	    Michał Pilipczuk\thanks{This work has been done while Micha\l{} Pilipczuk was working at the Institute of Informatics of the University of Warsaw. In particular, his work  on this manuscript was supported by the project BOBR
	    	that has received funding from the European Research Council (ERC) under the European Union’s Horizon
	    	2020 research and innovation programme (grant agreement No. 948057).}\\[0.05cm]
 		\normalsize\aff{Technical University of Darmstadt}\\
 		\normalsize\href{mailto:pilipczuk@mathematik.tu-darmstadt.de}{pilipczuk@mathematik.tu-darmstadt.de}
 
}
\author{\authorthing}

\newif\ifcomment
  \commenttrue

\begin{document}
\maketitle
\begin{textblock}{20}(-1.3,  5.65)
	\includegraphics[width=40px]{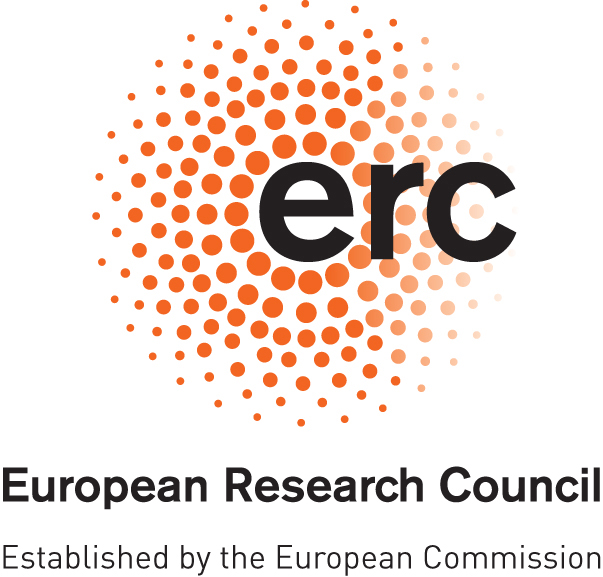}%
\end{textblock}
\begin{textblock}{20}(-1.8,  5.65)
	\includegraphics[width=80px]{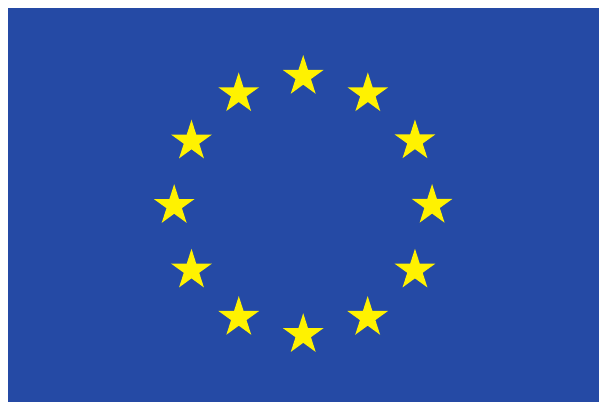}%
\end{textblock}

\begin{abstract}
	The first major step in the theory of graph minors of Robertson and Seymour was the \say{grid theorem}, a fundamental duality result describing that any graph of sufficiently high \treewidth must contain a large grid as a minor. Since this pioneering work, multiple grid theorems have been established in structural graph theory, linking various width parameters of combinatorial objects with the non-existence of grid-like obstructions. In particular, Kawarabayashi and Kreutzer~[STOC~'15] proved a variant of the grid theorem for directed graphs that links directed treewidth with the non-existence of topological minors of large cylindrical walls. On the other hand, to study immersions of undirected graphs --- another containment relation that focuses on edge-disjointness, rather than vertex-disjointness --- Wollan~[JCTB~'15] introduced a parameter \emph{tree-cutwidth} and proved that it is dual to immersions of large walls.
	
	
	In this work, we present a version of Wollan's parameter for directed graphs, the \emph{directed tree-cutwidth}.
	The main result is a grid theorem for directed tree-cutwidth, establishing that it is dual to directed immersions of large cylindrical walls.
\end{abstract}

\thispagestyle{empty}
\newpage
\setcounter{page}{1}

\section{Introduction}\label{sec:intro}

The fundamental Grid Minor Theorem of Robertson and Seymour~\cite{RobertsonS86} says that there is a function $f\colon \N\to \N$ such that every graph of treewidth more than $f(k)$ must contain the $k\times k$ grid as a minor.
In turn, a minor model of a $k\times k$ grid forces the treewidth of a graph to be at least $k$.
These two statements combined prove a \emph{duality} between the existence of a tree-decomposition of small width --- witnessing that the graph has a tree-like structure --- and the existence of a large grid minor --- a large bidimensional object witnessing that the graph is not tree-like.
Since the work of Robertson and Seymour, multiple duality results of this kind --- commonly called \emph{grid theorems} --- have been discovered in structural graph theory, linking parameters describing the tree-likeness of various combinatorial objects with the non-existence of grid-like obstructions.
For our work, two such grid theorems are of particular importance.

First, Kawarabayashi and Kreutzer~\cite{kawarabayashi2015directed} proved in 2015 a directed variant of the Grid Minor Theorem, which links the boundedness of the \emph{directed treewidth} of a directed graph (digraph) with the non-existence of topological minors of large cylindrical walls\footnote{The original formulation of Kawarabayashi and Kreutzer speaks of excluding cylindrical grids as \emph{butterfly minors}, which is another containment notion for digraphs. Their formulation is easily seen to be equivalent to the one stated here.}. Here, we say that a digraph $H$ is a topological minor of a digraph $D$ if one can map the vertices of $H$ to distinct vertices of $D$, and the edges of $H$ to directed paths in $D$ that connect the images of respective endpoints and are pairwise internally vertex-disjoint.
Also, a \emph{cylindrical grid} of order $k$ is a grid-like directed graph that consists of $k$ concentric cycles oriented the same way, and $2k$ radial paths linking the cycles via an alternating orientation pattern (see~\cref{fig:different_walls}).
A \emph{cylindrical wall} of order \(k\) is then obtained by splitting every vertex \(u\) of the cylindrical grid of order \(k\) into two vertices \(u^+\) and \(u^-\) connected by an edge $(u^+,u^-)$; $u^+$ inherits all incoming edges of \(u\) and $u^-$ inherits all outgoing edges.

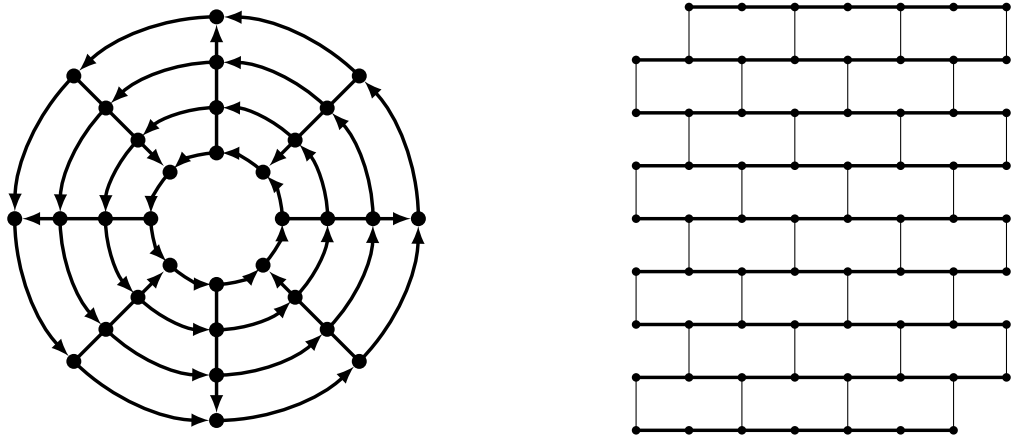
\begin{figure}[!ht]
	\begin{center}
	\begin{tikzpicture}
		\node (left) at (-4,0) {
			\begin{tikzpicture}[scale=0.9]
				\CylindricalGrid{4}{1}{black}{black}
			\end{tikzpicture}};
		
		\node (right) at (4,0) {
			\begin{tikzpicture}[scale=0.7]

				\foreach \x in {0, 2, 4, 6}{
					\foreach \y in {-1,1,3,5}{
						\draw (\x, \y+1)-- (\x, \y);						
					}
				}
				\foreach \x in {1, 3, 5, 7}{
					\foreach \y in {0,2,4,6}{
						\draw (\x, \y+1)-- (\x, \y);						
					}
				}
				
				\foreach \i in {0,...,6}
				{
					\draw[edge] (0,\i) -- (7,\i);
				}
				\draw[edge] (1,7) -- (7,7);
				\draw[edge] (0,-1) -- (6,-1);
				
				\foreach \i in {1,...,6}
				{
					\foreach \j in {-1,...,7}
					{
						\node[vertex,scale=0.5] at (\i,\j) {};
					}
				}
				\foreach \i in {0}
				{
					\foreach \j in {-1,...,6}
					{
						\node[vertex,scale=0.5] at (\i,\j) {};
					}
				}
				\foreach \i in {7}
				{
					\foreach \j in {0,...,7}
					{
						\node[vertex,scale=0.5] at (\i,\j) {};
					}
				}
				\end{tikzpicture}};
	\end{tikzpicture}
	\end{center}
	\caption{Left: A cylindrical grid of order \(4\). A cylindrical wall of order $4$ can be obtained from this digraph by splitting every vertex $u$ not lying on the two extremal cycles into two vertices $u^+,u^-$, connected by an edge $(u^+,u^-)$; $u^+$ inherits all incoming edges and $u^-$ inherits all outgoing edges. Right: An undirected~wall.}
	\label{fig:different_walls}
\end{figure}

Second, in 2015, Wollan~\cite{wollan2015structure} proved a grid theorem for \emph{immersions} of undirected graphs, which is a containment notion defined in the same way as topological minors, except that the paths modelling the edges are required to be edge-disjoint instead of internally vertex-disjoint\footnote{Technically, this notion is called a \emph{weak immersion}, as opposed to a strong immersion where we additionally forbid every path modelling an edge $e$ to pass through the images of vertices other than the endpoints of $e$. Throughout this paper we work exclusively with weak immersions and, consequently, we just call them \emph{immersions}.}. More precisely, Wollan introduced a graph parameter \emph{tree-cutwidth} and proved a duality result linking it to immersions of large walls (see~\cref{fig:different_walls}). Tree-cutwidth roughly resembles treewidth, but is fundamentally based on decomposing the graph using edge cuts instead of vertex cuts. The underlying notion of a graph decomposition is called a \emph{tree-cut decomposition}:
it is a tree $T$ where each node $x$ is assigned a \emph{bag} $\beta(x)$ so that the bags $\{\beta(x)\colon x\in V(T)\}$ form a \emph{near-partition} of the vertex set of the decomposed graph~$G$. (Every vertex of $G$ belongs to exactly one bag, but there may exist empty bags.) The \emph{width}\footnote{In~\cite{wollan2015structure}, Wollan defines the width slightly differently, but as proved in~\cite{GiannopoulouPRT21}, his definition is functionally equivalent to the one given below.} of a tree-cut decomposition $(T,\beta)$ is the least integer $k$ such that
\begin{itemize}
	\item every bag has size at most $k$;
	\item every edge of $T$ induces a \emph{cut} (bipartition of $V(G)$) of \emph{order} at most $k$ (so with at most $k$ edges crossing the cut); and
	\item every node of $T$ is incident to at most $k$ edges of $T$ that induce a cut of order larger than $2$.
\end{itemize}
Similarly to the definition of treewidth, the \emph{tree-cutwidth} of a graph $G$ is the minimum possible width of a tree-cut decomposition of $G$.

\paragraph*{Our contribution.} In this work we prove a directed counterpart of the result of Wollan~\cite{wollan2015structure}, or equivalently, an immersion counterpart of the result of Kawarabayashi and Kreutzer~\cite{kawarabayashi2015directed}. That is, we introduce a digraph width parameter called \emph{directed tree-cutwidth}, and we prove a grid theorem linking it to immersions of \emph{cylindrical walls}. The discovered duality is explained by the following two statements, which are our main results.

\begin{restatable}{theorem}{ThmImm}
	\label{thm:grid_or_small_dtcw}
	There is a function \(\thmFkt{thm:grid_or_small_dtcw}{f}{}\colon \N\to \N\) such that for any digraph \(D\), one of the following~holds:
	\begin{enumerate}[label=\textbf{(I\arabic*)},leftmargin=1.2cm]
		\item \(D\) contains a cylindrical wall of order $k$ as an immersion, or
		\item the directed tree-cutwidth of $D$ is at most \(\thmFkt{thm:grid_or_small_dtcw}{f}{k}\).
	\end{enumerate}
	Moreover, there is a fixed-parameter algorithm (with parameter $k$) to find either an immersion model of a cylindrical wall of order $k$ in $D$, or a directed tree-cut decomposition of $D$ of width at most~\(\thmFkt{thm:grid_or_small_dtcw}{f}{k}\).
\end{restatable}

\begin{restatable}{theorem}{ThmgridImmImplhighDTCW}
	\label{thm:imm_implies_large_dtcw}
	There is a function \(\thmFkt{thm:imm_implies_large_dtcw}{f}{}\colon \N\to \N\) such that for any positive integer \(k\),
	every digraph containing a cylindrical wall of order \(\thmFkt{thm:imm_implies_large_dtcw}{f}{k}\) as an immersion has \dtcw at least \(k\).
\end{restatable}

Due to its technicality, we postpone the formal definition of directed tree-cutwidth to \cref{sec:decomps}, but let us quickly discuss its main features. Similarly to (undirected) tree-cutwidth, the underlying notion of a decomposition, say of a digraph $D$, consists of a tree $T$ and a bag function $\beta$ mapping the nodes of $T$ to the corresponding bags that form a near-partition of the vertex set of $D$. To define the \emph{width} of such a decomposition, we take the smallest integer $k$ that simultaneously bounds the following measures:
\begin{itemize}
	\item \emph{Bag sizes:} Every bag must be of size at most $k$.
	\item \emph{Porosity:} Every edge of $T$ must induce a cut (bipartition of the vertex set) of $D$ such that there must be a set of at most $k$ edges that hits all the directed cycles crossing the cut.
	\item \emph{Onion-guarding:} For every node $x$ of $T$, there must be a set of at most $k$ edges that hits certain structures, called \emph{onions}, crossing $x$.  
\end{itemize}
These three conditions are in a direct analogy to the three conditions defining the width of an (undirected) tree-cut decomposition. We remark that, as expected, directed tree-cutwidth is monotone under taking immersions (see \cref{thm:immersion-monotone}).

At this point, we would like to highlight a major difference between our work and the preceding works on the structure theory for directed graphs. The traditional way to measure the complexity of a cut $(A,B)$ of a digraph $D$ is to consider its \emph{order}: the minimum of $\Abs{\Out{}{A}}$ --- the number of edges going from $A$ to $B$ --- and $\Abs{\Out{}{B}}$ --- the number of edges going from $B$ to $A$. While this approach makes certain submodularity arguments work naturally, it also introduces an unwieldy lack of symmetry: every cut has a ``bounded'' and an ``unbounded'' orientation, which makes uncrossing cuts in digraphs very cumbersome. Throughout most of this paper, we instead measure the complexity of a cut $(A,B)$ by its \emph{porosity}: the minimum size of an edge hitting set for directed cycles visiting both sides of the cut. Note that a cut of order $k$ also has porosity at most $k$, but the converse implication does not always hold. The definition of porosity is symmetric and allows for natural uncrossing of cuts: if $(A,B)$ and $(A',B')$ are cuts of porosity $p$ and $p'$ respectively, then all the four cuts $(A\cap A',B\cup B')$, $(A\cap B',B\cup A')$, $(A\cup A',B\cap B')$, $(A\cup B',B\cap A')$ have porosity at most $p+p'$ (see \cref{lemma:porosity-boolean}). The drawback is that porosity is not a submodular measure, so at certain places we resort to using the traditional order~instead.

Our usage of the concept of porosity is directly inspired by the digraph width measure \emph{cyclewidth}, due to Hatzel, Rabinovich, and Wiederrecht~\cite{HatzelRW19}, where this idea emerged first. Similarly to porosity, we also measure the degree of a vertex in a non-traditional way: for a vertex $u$, instead of focusing on the minimum or the maximum or the sum of the indegree and the outdegree of $u$, we consider the \emph{cycle-degree}\footnote{Note that this is a different concept than defined by Chartrand, Saba and Wormald~\cite{1984othercycledegree} under the same name. We allow this clash, as we are not aware of any use of this notion since.} of $u$, defined as the minimum size of an edge hitting set for cycles passing through~$u$. (By Menger's Theorem, this is equal to the maximum size of a packing of edge-disjoint directed cycles passing through $u$.)
In general, we believe that this paper showcases that these non-traditional measures of cuts or degrees in digraphs offer new possibilities, and may play an important role in future research on digraph structure theory.
Recently a vertex-version of cycle-degree has been used by Nisse, Picasarri-Arrieta, and Sau~\cite{2024vertexcycledegree} in the context of dicolourings.
We expand this topic into some concrete questions in~\cref{sec:conclusions}.

Let us exemplify the advantages of porosity using one of the major technical results that we prove along the way towards \cref{thm:grid_or_small_dtcw}; this result is of independent interest. In the following, by a \emph{tree-partition} of a digraph $D$ we just mean a pair $(T,\beta)$ where $T$ is a tree and $\beta$ is a bag function defined on the nodes of $T$ so that $\{\beta(x)\colon x\in V(T)\}$ is a near-partition of the vertex set of $D$. The \emph{porosity} of a tree-partition $(T,\beta)$ is the maximum porosity among the cuts induced by the edges of $T$, and $(T,\beta)$ is \emph{$k$-\linked} if for any pair of distinct vertices $u,v$ of $D$ that belong to the same bag of $\beta$, there are $k$ edge-disjoint paths from $u$ to $v$ and $k$ edge-disjoint paths from $v$ to $u$. In \cref{sec:potato} we prove the~following.

\begin{restatable}[Potato Theorem]{theorem}{PotatoLemma}
	\label{lemma:potato}
	There exists a function \(\thmFkt{lemma:potato}{f}{} \colon \N \to \N\) such that for every positive integer~$k$, every digraph \(D\) has a $k$-\linked tree-partition of \porosity at most \(\thmFkt{lemma:potato}{f}{k}\).
	
	Moreover, there is a fixed-parameter algorithm that finds, for any digraph $D$ and parameter $k$, a $k$-\linked tree-partition of porosity at most \(\thmFkt{lemma:potato}{f}{k}\).
\end{restatable}

Thus, \cref{lemma:potato} shows that \emph{any} digraph can be decomposed, in a tree-like manner, into $k$-edge-connected pieces using cuts of porosity bounded by a function of $k$. (Our proof yields a bound of $\thmFkt{lemma:potato}{f}{k}\leq 2^{\Oh(k^2)}$.) A result of a similar nature was proved by Giannopoulou, Kawarabayashi, Kreutzer, and Kwon~\cite{directedTangleArxiv,directedTangleJournal}: every digraph has a directed tree-decomposition of edge-width $\Oh(k^2)$ that distinguishes all directed tangles of order $k$. The precise formulation of these terms is highly technical, but the bottom line is that this result provides only a loose grasp on how the exposed separations of order $\Oh(k^2)$, which distinguish ``$k$-connected parts'', are organised in a tree, and in particular it does not provide a single laminar family of separations doing this job. Intuitively speaking, our \cref{lemma:potato} shows that this can be achieved in a much cleaner way in the context of tree-partitions and edge cuts, provided we allow measuring the complexity of cuts by their porosity instead of their order.

Finally, we also emphasize that \cref{thm:grid_or_small_dtcw} and \cref{lemma:potato} yield constructive algorithms that find the desired objects in every digraph $D$, allowing for potential algorithmic applications (see~\cref{sec:conclusions}).

\paragraph{Bounded degree.}
It is natural that in digraphs of bounded degree~\cref{thm:imm_implies_large_dtcw} follows from the Directed Grid Theorem~(\cref{thm:gridtheorem}) due to the qualitative equivalence of \emph{directed treewidth} to a parameter called \emph{directed branchwidth}~\cite{thesisWiederrecht,wiederrecht2019note}.
We discuss the bounded degree case in detail in \cref{sec:boundeddeg}, where we show that on digraphs of bounded degree directed branchwidth is equivalent to our parameter \emph{directed tree-cutwidth}.


\paragraph*{Other related work.} We remark that there has been a significant amount of work on immersions in digraphs that are \emph{Eulerian}: every vertex has outdegree equal to the indegree.
For instance, DeVos, McDonald, Mohar, and Scheide~\cite{DeVosMMS13} proved a result similar to \cref{lemma:potato} for Eulerian digraphs. Recently, Cavallaro, Kawarabayashi, Kreutzer gave a fixed-parameter algorithm for the \textsc{Edge Disjoint Paths} problem in Eulerian digraphs~\cite{CavallaroKK24}, as well as proved the Eulerian digraphs embeddable into a fixed surface are well-quasi-ordered under the strong immersion order~\cite{WGO-strong-immersions}. However, the setting of immersions in Eulerian digraphs is quite different from that in general digraphs, because thanks to the condition that the digraph is Eulerian, every cut of order $k$ actually has only $k$ edges going in both directions. This makes Eulerian digraphs behave much more like undirected graphs when it comes to width measures based on edge cuts, rendering the corresponding toolbox inapplicable in our setting.



\section{Overview}\label{sec:overview}


In this section we provide an informal overview
of the introduced definition of \dtcw (\cref{def:dtcw}), as well as
of the proofs of \cref{thm:grid_or_small_dtcw,thm:imm_implies_large_dtcw}, focusing on main conceptual insights.

\subsection{Motivating examples and the definition of \dtcw}

The definition of \dtcw is guided by the same general philosophy as Wollan's tree-cut width
in the undirected setting~\cite{wollan2015structure}, and the directed treewidth, which over time converged~\cite{kreutzer_width-measures_2014,kawarabayashi2015directed,directedTangleJournal,kawarabayashi2022directed} to a slightly different definition than the original one by Johnson, Robertson, Seymour, and Thomas~\cite{johnson2001directed} and Reed~\cite{reed_introducing_1999}; see~\cite{dtw_def_comparison} for a discussion of different definitions.
Our goal is to design a structural parameter that captures exclusion of a large cylindrical wall as an immersion.


A first na\"ive attempt would be the following.
As the underlying notion of a decomposition, take a tree \(T\) together with a near-partition \(\beta\) of \(V(D)\) indexed by
\(V(T)\); this is called a \emph{tree-partition} of $D$. For each edge \(e \in E(T)\), we may consider the cut \((X,Y)\) of $D$ \emph{induced} by $e$:  $X$ is the union of the bags residing in one component of \(T \setminus e\), and $Y$ is the union of the bags in the other component.
To define the directed tree-cutwidth of such a tree-partition $(T,\beta)$, one may try to imitate the definition of directed treewidth (see \cref{def:dtw}) and require that (i) each bag consists of a bounded number of vertices, and (ii) each cut induced by an edge $e$ of $T$ has small porosity (which, recall, means that there is a small edge set hitting all directed cycles that visit both sides of the corresponding cut).


Requiring only conditions (i) and (ii) above is not restrictive enough: observe that every digraph $D$ with bounded cycle-degree --- for instance, the cylindrical wall --- admits such a tree-partition, obtained by taking $T$ to be a star with as many leaves as the vertices of $D$, and assigning every leaf a singleton bag containing a different vertex of $D$ (the centre is assigned an empty bag).
The intuition is that we have to somehow bound the ``total interaction'' around every node $x$ of the tree-partition.
In the context of directed treewidth, this is done by postulating an upper bound on the size of a set called $\Gamma(x)$ which additionally to the vertices in the bag assigned to \(x\) also contains vertices from \textsl{vertex-guards} $\gamma(e)$ associated with the edges $e$ incident to $x$.
The idea of these vertex-guards is that they control the behaviour of vertex-disjoint paths between vertices lying in the bags of different subtrees of \(T-e\).
For us, it would be the easiest to just bound the maximum degree of the tree $T$ underlying the tree-partition.
This, however, would be too restrictive, as it disallows
decomposing a large bidirected star using a directed tree-cut decomposition of constant width.
Note here that a bidirected star does not contain a cylindrical wall of order 2 as an immersion, so it should have constant directed tree-cutwidth.



Here, we again take an inspiration from the work of Wollan~\cite{wollan2015structure}, who observed that in the context of immersions in undirected graphs, there is a qualitative difference between an $N$-leaf star with every edge replaced with two parallel edges (call it $S_N^{\times 2}$), and one where every edge is replaced with three parallel edges (call it $S_N^{\times 3}$):
\begin{itemize} 
	\item $S_N^{\times 2}$ does not contain an immersion of a $3\times 3$ wall, so it should have constant tree-cutwidth.
	\item $S_N^{\times 3}$, for $N$ large enough, contains any subcubic graph as an immersion; so in particular, a large wall. Hence, the graphs $S_N^{\times 3}$ for $N\in \N$ should have unbounded tree-cutwidth.
\end{itemize}
In the definition of tree-cutwidth, this distinction is reflected by explicitly controlling the number of edges incident to a node $x$ of $T$ that induce a cut of order $3$ or more, and disregarding in this calculation the cuts of order $2$ or less.

In the directed setting there is a very similar obstruction: the onion star.
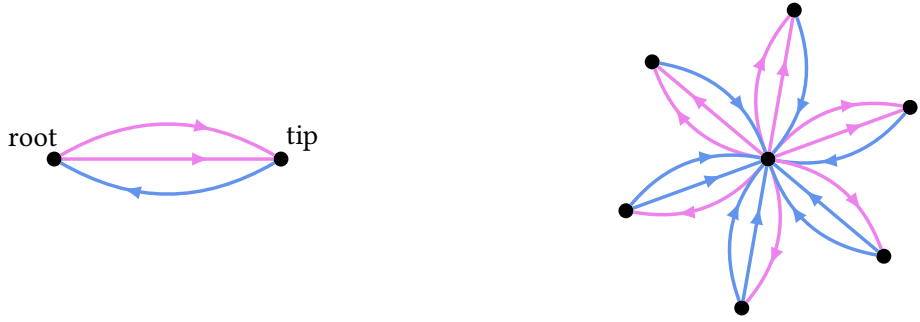
\begin{figure}[!ht]
	\begin{center}
		\begin{tikzpicture}
			\node (onion) at (-4,0) {\begin{tikzpicture}
					\node (centre) at (0,0) {};
					\node[vertex] (root) at ($(centre)+(180:1.5)$) {};
					\node[vertex] (stem) at ($(centre)+(0:1.5)$) {};
					\node (stem-label) at ($(stem)+(45:0.4)$) {$\text{tip}$};
					\node (root-label) at ($(root)+(135:0.4)$) {$\text{root}$};
					
					\draw (root) edge[directededgedecorated,LavenderMagenta] (stem);
					\draw (root) edge[directededgedecorated,LavenderMagenta,bend left] (stem);
					\draw (stem) edge[directededgedecorated,CornflowerBlue,bend left] (root);
			\end{tikzpicture}};
			
			\node (onion-star) at (4,0) {\begin{tikzpicture}
					\node[vertex] (centre) at (0,0) {};
					
					\foreach \i in {0,...,5}
					{
						\node[vertex] (s-\i) at ($(centre)+({360/6*\i+20}:2)$) {};
					}
					\foreach \i in {0,1,2}
					{
						\draw (centre) edge[directededgedecorated,LavenderMagenta] (s-\i);
						\draw (centre) edge[directededgedecorated,LavenderMagenta,bend left] (s-\i);
						\draw (s-\i) edge[directededgedecorated,CornflowerBlue,bend left] (centre);
					}
					\foreach \i in {3,4,5}
					{
						\draw (centre) edge[directededgedecorated,LavenderMagenta,bend left] (s-\i);
						\draw (s-\i) edge[directededgedecorated,CornflowerBlue,bend left] (centre);
						\draw (s-\i) edge[directededgedecorated,CornflowerBlue] (centre);
					}
			\end{tikzpicture}};
		\end{tikzpicture}		
	\end{center}
	\caption{Left: An onion. Right: An onion-star of order 3.}
	\label{fig:onion-star}  
	\label{fig:onion}
\end{figure}
An \emph{onion} is a digraph consisting of two vertices --- the \emph{root} and the \emph{tip} ---
two parallel edges from the root to the tip, and a single edge from the tip to the root, see~\cref{fig:onion}.
An onion star of order \(k\) is obtained from a star with \(2k\) leaves by replacing \(k\) of its edges with an onion with the root mapped to the centre vertex
and the tip to the leaf, and replacing the remaining \(k\) edges with an onion
with the tip mapped to the centre and the root to the leaf.
As already recognized by Bo\.zyk, Defrain, Pilipczuk, and Okrasa~\cite{bozyk2022digraphs}, this simple structure is rich enough to immerse all objects we care about:
onion-stars are immersion-universal for subcubic strongly connected digraphs.
In particular, every cylindrical wall can be embedded in a large enough onion star.
Therefore, when investigating an edge $e$ of a tree-partition $(T,\beta)$, we have to care what kind of onions can \emph{cross} it --- have the root and the tip on the opposite sides of the cut induced by $e$. Note that there are two potential ways in which an onion may cross $e$, depending which side of the associated cut contains the root and which one contains the tip.


This distinction is illustrated by the \((N,\ell)\)-almost onion-star.
Here we have a centre and \(N+\ell\) leaves, with exactly \(\ell\) edges replaced by onions of one
orientation, and the remaining \(N\) edges replaced by onions of the opposite orientation.
If we fixed \(\ell\) and let \(N\) be arbitrarily large, then such digraphs still do not
immerse a large cylindrical wall.
So a satisfactory decomposition should have bounded directed tree-cutwidth on all
\((N,\ell)\)-almost onion-stars and also on all \((\ell,N)\)-almost onion-stars,
with the bound depending only on~\(\ell\), but not on \(N\).

The way we resolve this issue is by introducing \emph{onion-guards}.
At each node \(x\) of the decomposition tree $T$ we allow a bounded-size set of edges
whose role is to hit immersion models of onions that cross cuts incident with \(x\)
in a prescribed orientation.
More precisely, we require that for each edge of $T$ incident with \(x\),
where the corresponding cut is \((X,Y)\) with \(Y\) being the side containing the bag of \(x\), one of the following two behaviours occurs (and for each edge incident with $x$, the same one):
\begin{itemize}
	\item the onion-guard at \(x\) hits every onion immersion model whose root lies in \(Y\) and tip lies in \(X\);~or 
	\item the onion-guard at \(x\) hits every onion immersion model whose tip lies in \(Y\)
and root lies in \(X\).
\end{itemize}
Thus, each node chooses one global orientation type --- out-hitting or in-hitting ---
and its onion-guard controls all cuts incident with that node in that orientation. Note that an onion immersion model whose root and tip lie in two different components of $T-x$ must be hit by the onion-guard regardless of the chosen orientation --- the model counts both as an ``in-onion'' and as an ``out-onion''.

This leads us to our final definition of the directed tree-cutwidth of a tree-partition $(T,\beta)$: it is the smallest $k$ such that
\begin{itemize}
	\item every bag has size at most $k$,
	\item every cut induced by an edge of $T$ has porosity at most $k$, and
	\item every node of $T$ can be assigned an onion-guard of size at most $k$ satisfying the property described in the previous paragraph.
\end{itemize} 
And \emph{directed tree-cutwidth} of a digraph $D$ is the minimum possible directed tree-cutwidth of a tree-partition of $D$. Formally, we define \emph{\dtcdecomps} as tree-partitions $(T,\beta)$ endowed with an \emph{edge-guard} function, mapping every edge of $T$ to a suitable hitting set witnessing the porosity, and an \emph{onion-guard} function, mapping every node $x$ of $T$ to its suitable onion-guard for $x$; and the \emph{width} of a \dtcdecomp is the maximum size of a bag, an edge-guard, or an onion-guard.
With the definition in place, we now sketch the proofs of \cref{thm:grid_or_small_dtcw,thm:imm_implies_large_dtcw}.



\subsection{Wall immersion forces large \dtcw}

Let us start with the easier implication, \cref{thm:imm_implies_large_dtcw}.
We have to argue that the existence of an immersion model of a large cylindrical wall forces the \dtcw to be large.
We first argue (\cref{thm:immersion-monotone}) that \dtcw is monotone under taking immersions; this follows by naturally replacing edges in edge-guards (hitting sets witnessing porosity) and onion-guards by edges containing them in the immersion-model.
Thanks to that, for \cref{thm:imm_implies_large_dtcw} it suffices to
prove that a large enough cylindrical wall itself has large \dtcw.
The proof is conceptually quite straightforward, although with some rough
edges hiding behind the details.
At a high level, the intuition is that a cylindrical wall is both
highly spread out with respect to any \dtcdecomp of bounded width,
and highly robust in the sense that many pairs of vertices can serve as
the root and tip of some onion immersion model.

Suppose towards a contradiction that a large cylindrical wall \(W\) admits a tree-partition $(T,\beta)$ of small directed tree-cutwidth.
By a standard argument, there exists a node \(x\) of $T$ whose removal breaks $T$ into components, each accommodating only at most half of the branch-vertices (that is vertices of degree 3) of~$W$.

Consider the onion-guard associated with \(x\).
As this guard is small, we may choose a large subwall \(W'\subseteq W\)
completely avoiding all edges of that onion-guard.
The key property of cylindrical walls is that within such a large subwall
one can realise an onion immersion model with root $u$ and tip $v$ for very many pairs \((u,v)\) of vertices, essentially
provided only that the degrees of \(u\) and \(v\) do not forbid
them from acting as the root and the tip.
In other words, onion immersions are abundant and robust inside the wall.

Now consider a cut of $W$ induced by some edge of $T$ incident to \(x\), and let
\((X,Y)\) be the corresponding cut of \(W'\), where \(Y\) is the side containing the bag of \(x\).
If \(x\) is out-hitting, then any onion immersion model with root in \(Y\) and tip in \(X\)
must meet the onion-guard of \(x\); if \(x\) is in-hitting, the analogous statement
holds with the roles of root and tip reversed.
But \(W'\) avoids the onion-guard entirely.
Hence, for many pairs of vertices in \(W'\), the two endpoints cannot lie on opposite
sides of such a cut in the forbidden orientation.
As onion immersions exist for a large proportion of admissible pairs, this forces a
large fraction of the vertices of \(W'\) to lie in the same component of
\(T \setminus x\), or else in the bag of \(x\) itself.
This contradicts the choice of \(x\) as we would see too many vertices
in a single component of $T-x$.

\subsection{No wall immersion implies bounded \dtcw}

We now proceed to sketching the involved implication, that is, the proof of \cref{thm:grid_or_small_dtcw}.
Let \(D\) be a digraph excluding a cylindrical wall of order $k$ as an immersion.
Our goal is to show that \(D\) admits a tree-partition of directed tree-cutwidth bounded by a function of $k$.

The first step in the proof is to obtain a ``coarse'' tree-partition of $D$ that, roughly speaking, separates every pair of vertices that could be in principle separated by a cut of bounded porosity.
This is formalised through the notion of a \emph{$k$-\linked tree-partition} (\cref{def:treept,def:linked}), which requires that for any two vertices $u,v$ residing in the same bag, there is an edge-disjoint flow of size $k$ from $u$ to $v$ and also an edge-disjoint flow of size $k$ from $v$ to $u$.
Here, we prove a somewhat surprisingly clean decomposition theorem: the \nameref{lemma:potato}~(\cref{lemma:potato}).
It states that one can always find a tree-partition \((T,\beta)\) of porosity bounded in terms of $k$ that is $k$-\linked.
We believe that the \nameref{lemma:potato} is a result of independent interest and, potentially, a wider scope of applicability; we postpone an overview of its proof to the~\cref{sec:overview-potato}.

The preliminary decomposition $(T,\beta)$ provided by the~\nameref{lemma:potato} has the following feature:
if some bag \(\beta(x)\) is large (in terms of $k$), then \(D\) already contains a large wall as an immersion.
Indeed, by construction, vertices inside a bag can be pairwise connected by large flows in both directions,
and a theorem of Bożyk, Defrain, Okrasa, and Pilipczuk~\cite{bozyk2022digraphs}
then gives a large onion star as an immersion
(see~\cref{lem:onion_star}), which contradicts our assumption.
Thus, we may henceforth assume that all bags of \((T,\beta)\) have bounded size.

The next step is to analyse each bag through its associated torso, see~\cref{fig:torso_simplified}.
For a node \(x \in V(T)\), its torso \(\torso{x}\) is obtained by keeping the vertices of \(\beta(x)\),
and compacting each component hanging below \(x\) into a single \emph{auxiliary} onion in which
the tip collects edges entering that component, while the root collects edges leaving it.
See~\cref{sec:torso} for a formal definition.

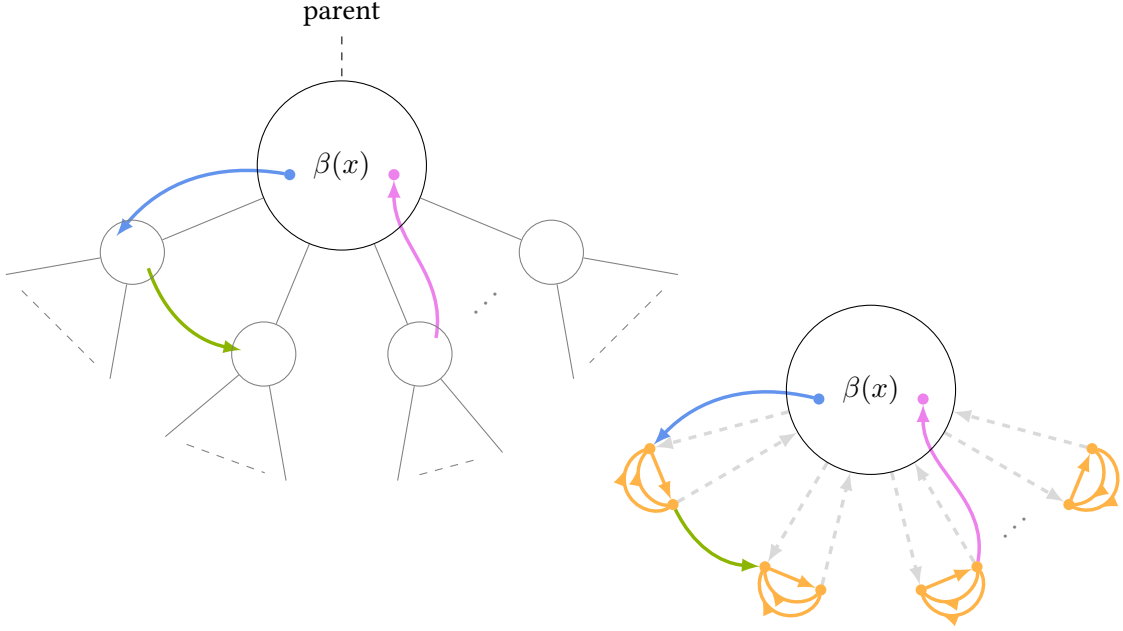
\begin{figure}[!ht]
	\begin{center}
		\begin{tikzpicture}
			\node (decomp) at (-3,0) {
				\begin{tikzpicture}
					\node (centre) at (0,0) {};
					\node[draw,circle,inner sep=0.5cm] (bag-v) at ($(centre)$) {\(\beta(x)\)};
					\node (parent) at ($(bag-v)+(90:2)$) {parent};
					
					\node[draw=gray,gray,circle,inner sep=0.3cm] (c-1) at ($(bag-v)+({360/16*9}:3)$) {};
					\node[draw=gray,gray,circle,inner sep=0.3cm] (c-2) at ($(bag-v)+({360/16*11}:2.7)$) {};
					\node[draw=gray,gray,circle,inner sep=0.3cm] (c-3) at ($(bag-v)+({360/16*13}:2.7)$) {};
					\node[draw=gray,gray,circle,inner sep=0.3cm] (c-4) at ($(bag-v)+({360/16*15}:3)$) {};
					
					\node[gray,rotate=43] at ($(c-3)!0.5!(c-4)$) {\(\dots\)};
					
					\node[vertex,scale=0.7,CornflowerBlue] (e-tail) at ($(bag-v)+(190:0.7)$) {};
					\node (e-head) at ($(c-1)+(135:0.3)$) {};
					\node[vertex,scale=0.7,LavenderMagenta] (g-head) at ($(bag-v)+(350:0.7)$) {};
					\node (g-tail) at ($(c-3)+(45:0.3)$) {};
					
					\node (f-tail) at ($(c-1)+(315:0.3)$) {};
					\node (f-head) at ($(c-2)+(170:0.3)$) {};
					
					
					\begin{pgfonlayer}{background}
						\draw[dashed] (parent) to (bag-v);
						
						\draw[gray] (bag-v) to (c-1);
						\draw[gray] (bag-v) to (c-2);
						\draw[gray] (bag-v) to (c-3);
						\draw[gray] (bag-v) to (c-4);
						
						\draw[gray] (c-1) to ($(c-1)+(190:1.7)$);
						\draw[gray] (c-1) to ($(c-1)+(260:1.7)$);
						\draw[gray,dashed,shorten <=0.3cm,shorten >=0.3cm] ($(c-1)+(190:1.7)$) to ($(c-1)+(260:1.7)$);
						
						\draw[gray] (c-2) to ($(c-2)+(220:1.7)$);
						\draw[gray] (c-2) to ($(c-2)+(280:1.7)$);
						\draw[gray,dashed,shorten <=0.3cm,shorten >=0.3cm] ($(c-2)+(220:1.7)$) to ($(c-2)+(280:1.7)$);
						
						\draw[gray] (c-3) to ($(c-3)+(260:1.7)$);
						\draw[gray] (c-3) to ($(c-3)+(310:1.7)$);
						\draw[gray,dashed,shorten <=0.3cm,shorten >=0.3cm] ($(c-3)+(260:1.7)$) to ($(c-3)+(310:1.7)$);
						
						\draw[gray] (c-4) to ($(c-4)+(350:1.7)$);
						\draw[gray] (c-4) to ($(c-4)+(280:1.7)$);
						\draw[gray,dashed,shorten <=0.3cm,shorten >=0.3cm] ($(c-4)+(350:1.7)$) to ($(c-4)+(280:1.7)$);

						\draw[directededge,CornflowerBlue,bend right] (e-tail) to (e-head.center);
						\draw[directededge,LavenderMagenta,out=80,in=270] (g-tail.center) to (g-head);
						\draw[directededge,AppleGreen,bend right] (f-tail.center) to (f-head.center);
					\end{pgfonlayer}	
			\end{tikzpicture}};
			
			\node (torso) at (4,-3) {
				\begin{tikzpicture}
					\node (centre) at (0,0) {};
					\node[draw,circle,inner sep=0.5cm] (bag-v-t) at ($(centre)$) {\(\beta(x)\)};
					
					\node (c-1-t) at ($(bag-v)+({360/16*9}:3)$) {};
					\node (c-2-t) at ($(bag-v)+({360/16*11}:2.7)$) {};
					\node (c-3-t) at ($(bag-v)+({360/16*13}:2.7)$) {};
					\node (c-4-t) at ($(bag-v)+({360/16*15}:3)$) {};
					
					\node[vertex,scale=0.7,PastelOrange] (a-1) at ($(c-1-t)+({360/16*9-90}:0.4)$) {};
					\node[vertex,scale=0.7,PastelOrange] (b-1) at ($(c-1-t)+({360/16*9+90}:0.4)$) {};
					
					\node[vertex,scale=0.7,PastelOrange] (a-2) at ($(c-2-t)+({360/16*11-90}:0.4)$) {};
					\node[vertex,scale=0.7,PastelOrange] (b-2) at ($(c-2-t)+({360/16*11+90}:0.4)$) {};
					
					\node[vertex,scale=0.7,PastelOrange] (a-3) at ($(c-3-t)+({360/16*13-90}:0.4)$) {};
					\node[vertex,scale=0.7,PastelOrange] (b-3) at ($(c-3-t)+({360/16*13+90}:0.4)$) {};
					
					\node[vertex,scale=0.7,PastelOrange] (a-4) at ($(c-4-t)+({360/16*15-90}:0.4)$) {};
					\node[vertex,scale=0.7,PastelOrange] (b-4) at ($(c-4-t)+({360/16*15+90}:0.4)$) {};
					
					\node[gray,rotate=43] at ($(c-3-t)!0.5!(c-4-t)$) {\(\dots\)};
					
					\node[vertex,scale=0.7,CornflowerBlue] (e-tail) at ($(bag-v-t)+(190:0.7)$) {};
					\node[vertex,scale=0.7,LavenderMagenta] (g-head) at ($(bag-v-t)+(350:0.7)$) {};
					
					
					\begin{pgfonlayer}{background}
						\foreach \i in {1,2,3,4}{
							\draw[directededge,PastelOrange] (a-\i) to (b-\i);
						}					
						\draw[directededge,CornflowerBlue,bend right] (e-tail) to (a-1);
						\draw[directededge,LavenderMagenta,out=80,in=270] (b-3) to (g-head);
						\draw[directededge,AppleGreen,bend right] (b-1) to (a-2);
						
						\foreach[evaluate={\j=int(7+2*\i);}] \i in {1,2,3,4}{
							\draw (b-\i) edge[directededgedecorated,PastelOrange,quick curve through={($(b-\i)!0.5!(a-\i)+({360/16*\j}:0.5)$)}] (a-\i);
							\draw (b-\i) edge[directededgedecorated,PastelOrange,quick curve through={($(b-\i)!0.5!(a-\i)+({360/16*\j}:0.3)$)}] (a-\i);
						}
						
						\draw[directededge,gray!30!white,dashed] (bag-v-t) to (a-1);
						\draw[directededge,gray!30!white,dashed] (bag-v-t) to (a-2);
						\draw[directededge,gray!30!white,dashed] (bag-v-t) to (a-3);
						\draw[directededge,gray!30!white,dashed] (bag-v-t) to (a-4);
						
						\draw[directededge,gray!30!white,dashed] (b-1) to (bag-v-t);
						\draw[directededge,gray!30!white,dashed] (b-2) to (bag-v-t);
						\draw[directededge,gray!30!white,dashed] (b-3) to (bag-v-t);
						\draw[directededge,gray!30!white,dashed] (b-4) to (bag-v-t);
						
					\end{pgfonlayer}
					
%
%
%
			\end{tikzpicture}};
		\end{tikzpicture}
	\end{center}
	\caption{Illustration of how the torso~\(\torso{x}\) for a node \(x \in \V{T}\) is obtained.
	Every subtree is replaced by an \textcolor{PastelOrange}{onion} depicted in \textcolor{PastelOrange}{orange} and the edges in and out of the subtrees are preserved.}
	\label{fig:torso_simplified}
\end{figure}

Note that even though the bag of $x$ has bounded size, $\torso{x}$ can be arbitrarily large, for $x$ may have an arbitrarily high degree in $T$.
However, one of the key properties of torsos is that in \(\torso{x}\), the only vertices that can still have large cycle-degree are the original vertices of \(\beta(x)\).
All the other vertices come from ``compressed hanging pieces'' and have cycle-degree \(1\).
Moreover, if some \(\torso{x}\) were to contain a large cylindrical wall as an immersion,
then this immersion could be lifted back to~\(D\); this is proved in~\cref{lem:grid_imm_in_torsos}.
Since \(D\) excludes a large wall immersion, so does every torso.

At this point we can invoke the Directed Grid Theorem of Kawarabayashi and Kreutzer~\cite{kawarabayashi2022directed}
(see~\cref{thm:gridtheorem}) on the bounded cycle-degree part of each torso.
This yields a directed tree-decomposition of these parts, and because the vertices have bounded cycle-degree,
this decomposition can be converted into a directed tree-cut decomposition of bounded width.
What remains is to understand how to reincorporate the vertices of \(\beta(x)\)
into this decomposition.

This is the role of the argument in~\cref{sec:councils}.
Here we use the fact that vertices in \(\beta(x)\) can be pairwise connected by large flows,
together with the recent \EP theorem for tripods by Briański, Hatzel, Okrasa and Pilipczuk~\cite{2024tripodEP}
(see \cref{thm:tripod-EP}). Here, roughly speaking, a \emph{tripod} is a rooted tree with two leaves, oriented towards the root, but it is better to think of a tripod as of an onion with subdivided edges, and the root split into three vertices: two sources and one sink. Following this interpretation, 
applying the theorem from~\cite{2024tripodEP} leads to three possible outcomes:
\begin{itemize}
  \item there is a small edge set hitting onion immersion models in  \(\torso{x}\) that correspond to "out-onions" (with respect to $x$) immersion models in $D$;
  \item there is a small edge set hitting onion immersion models in  \(\torso{x}\) that correspond to "in-onions" (with respect to $x$) immersion models in $D$;
  \item there is a large packing of both types of onion immersion models.
\end{itemize}

In the third case, the two large packings can be combined using the flows connecting the vertices of \(\beta(x)\)
to build a large onion star in \(D\), which in particular contains a large cylindrical wall immersion,
contrary to our assumptions.
Therefore, only the first two outcomes can occur in \(D\),
and we use them to successfully decompose $\torso{x}$.

The final step, carried out in~\cref{sec:hard},
is to glue the decompositions of torsos obtained for different nodes of $T$ in the tree-partition \((T,\beta)\) constructed using the~\nameref{lemma:potato}.
Here, to ensure that this stitching process preserves bounded width, we exploit two
properties:
\begin{itemize}
	\item that the tree-partition $(T,\beta)$ has bounded porosity; and
	\item that in the construction of $\torso{x}$, each subtree hanging below $x$ was replaced by a small onion.
\end{itemize}
In particular, the second property allows us to assume that in the decomposition of $\torso{x}$, each small onion modelling a subtree below $x$ has the root and the tip placed in the same bag, and this bag becomes a natural ``attachment point'' for the subtree. This
allows us to naturally combine the decomposition of the torsos into a single \dtcdecomps of the whole digraph $D$, making the stitching process work
out rather beautifully.

\subsection{Proof of the~\nameref{lemma:potato}}\label{sec:overview-potato}

Finally, let us comment on the proof of \nameref{lemma:potato} (\cref{lemma:potato}), which is presented in \cref{sec:potato}. The proof proceeds by induction on $k$: we consecutively construct a $1$-\linked tree-partition, a $2$-\linked tree-partition, a $3$-\linked tree-partition, and so on, each time refining the tree-partition obtained in the previous step. Therefore, the main step of the induction is captured by the following task: given a $k$-\linked tree-partition $(T,\beta)$, construct a $(k+1)$-\linked tree-partition $(T',\beta')$ whose porosity is not much larger than the porosity of $(T,\beta)$.

Consider any node $x$ of $T$ and denote $S\coloneqq \beta(x)$ for brevity.
By assumption, $S$ is $k$-flow-connected (i.e., for every pair of vertices $u,v$, there is an edge-disjoint flow of size $k$ from $u$ to $v$ and from $v$ to~$u$), but not necessarily $(k+1)$-flow connected. Therefore, we would like to refine $S$ by splitting it into a tree-partition $(T_x,\beta_x)$ with $(k+1)$-flow-connnected bags. Performing such a refinement for every node $x$ of $T$ will yield the transformation from $(T,\beta)$ to $(T',\beta')$.

\newcommand{\Cc}{\mathcal{C}}

The first, na\"ive approach to splitting $S$ is to consider an inclusion-wise maximal family $\Cc_S$ of cuts of $D$ of order $k$ that is \emph{laminar} on $S$ in the following sense:
\begin{itemize}
	\item if for each cut $(A,B)\in \Cc_S$ we consider the restriction $(A,B)|_S\coloneqq (A\cap S,B\cap S)$, then these restrictions are pairwise different and have both sides non-empty; and
	\item no two cuts $(A,B),(A',B')$ \emph{cross} on $S$, that is, have all the four sets $A\cap A'\cap S$, $A\cap B'\cap S$, $B\cap A'\cap S$, $B\cap B'\cap S$ non-empty.
\end{itemize}
It is easy to see that such a maximal family $\Cc_S$ splits $S$ into $(k+1)$-flow-connected pieces, and these pieces can be naturally organized into a tree whose edges are labelled by the cuts of $\Cc_S$. Indeed, if we fix an arbitrary vertex $s\in S$, order $\Cc_S$ by the inclusion of sides that do not contain $s$, and add an artificial root larger than all the elements of $\Cc_S$, then this defines on $\Cc_S$ a structure of a \emph{tree order} --- the ancestor-descendant order of a tree. The edges of this tree can be then labelled by the corresponding elements of $\Cc_S$, while the nodes can be associated with pieces ($k+1$-flow-connected subsets of $S$) defined as suitable differences between $s$-avoiding sides of the cuts of $\Cc_S$. This yields the first approximation of $(T_x,\beta_x)$.

Now, we would like to stitch the tree-partitions $(T_x,\beta_x)$, for $x$ ranging over the nodes of $T$, into a single tree-partition $(T,\beta)$ of $D$. The issue is that the cuts of $\Cc_{S}$, where $S=\beta(x)$, are required to be laminar only on $S$. Therefore, they may arbitrarily cross outside of $S$ and we have little control over how the cuts of $\Cc_{\beta(x)}$ interact with the cuts of $\Cc_{\beta(y)}$, for different nodes $x,y$ of $T$. To remedy this, we perform a very delicate uncrossing scheme on the families $\Cc_{\beta(x)}$ to combine them into a single, globally laminar family of cuts. The uncrossing crucially relies on the notion of \emph{important cuts} (more commonly known as \emph{important separators}), a technique borrowed from the field of parameterized algorithms (see e.g.~\cite[Chapter~8]{PAbook}). More precisely, we amend the cuts in each family $\Cc_{\beta(x)}$ by adding/subtracting some parts of $D$ lying outside of $\beta(x)$, and these amendments are defined as certain important cuts in $D$. We exploit the classic bounds on the number of distinct important cuts to show that the porosity of the cuts grows in a controlled way during the uncrossing process. The fact that we measure cuts using their porosity, instead of order, is crucial in this argument.

\section{Preliminaries}\label{sec:prelims}

\subsection{Basics}

For a positive integer $k$, we write $[k]$ for the set $\Set{1,\ldots,k}$. A \emph{near-partition} of a set $U$ is a collection $\FF$ of subsets of $U$ (possibly empty) such that every element of $U$ belongs to exactly one set in $\FF$.

\paragraph*{(Di)graphs.}
We use standard terminology and notations for graphs and directed graphs (\emph{digraphs}, for short). The vertex set and the edge set of a graph $G$ are denoted by $V(G)$ and $E(G)$ respectively, while the vertex set and the edge set of a digraph $D$ are denoted by $V(D)$ and $E(D)$, respectively. All graphs and digraphs considered in this paper are finite, so we always assume these sets to be finite.

Let $D$ be a digraph. An edge from $x$ to $y$ in $D$ is denoted by $(x,y)$, and we call $x$ and $y$ the \emph{tail} and the \emph{head} of $(x,y)$, respectively, and denote them by $\tail{e}$ and $\head{e}$. 
In all what follows we assume that for each edge $e$ of a digraph we have $\tail{e}\neq \head{e}$; so we disallow \emph{loops} --- edges with same head and tail. However, we allow \emph{parallel edges}: multiple edges with the same head and tail. Thus, formally, $E(D)$ is a multiset of edges, but we tacitly apply the standard set notation to it. On the other hand, all graphs considered in this paper are \emph{simple}: they do not contain loops or parallel edges (edges with same endpoints). An (undirected) edge with endpoints $u,v$ is denoted $uv$.


A \emph{subgraph} $D'$ of a digraph $D$ is a digraph with $V(D') \subseteq V(D)$ and $E(D') \subseteq E(D)$. For $X$ being a set of edges or a set of vertices, by $D-X$ we denote the digraph obtained from $D$ by removing all the edges, respectively vertices, of $X$. We say that a set of edges \(F\) \emph{hits} a subgraph $D'$ of $D$ if \(F\cap E(D')\neq \emptyset\).

For a set $S \subseteq V(D)$, we write \(\Out{D}{S}\) for the set of edges whose tail belong to \(S\).
Similarly, \(\In{D}{S}\) consists of edges whose head belong to $S$. 



A \emph{trail} $Q$ is a sequence $(e_1,\ldots,e_m)$ of pairwise distinct edges of $D$ such that for every $i \in [m-1]$ we have $\head{e_i}=\tail{e_{i+1}}$.
If $a=\tail{e_1}$ and $b=\head{e_m}$, we call such $Q$ a \emph{$a$-$b$-trail}.
A trail \emph{visits} a vertex $u$ if $u$ is the head or the tail of one of its edges. A \emph{path} is a trail $Q=(e_1,\ldots,e_m)$ that visits every vertex at most once, that is, for any $i,j\in [m]$, we have $\head{e_i}=\tail{e_j}$ if and only if $j=i+1$.
A \emph{cycle} is a trail for which $\tail{e_1}=\head{e_m}$. Note that we allow a cycle to visit a vertex multiple times, but every edge may appear on a trail at most once. To distinguish this, a \emph{simple cycle} is a cycle that visits every vertex at most once, that is, whose edges have pairwise different heads.

Let $P=(e_1,\ldots,e_m)$ be a path that visits vertices $a$ and $b$, in this order. 
We denote by ${a}\xrightarrow{P} {b}$ the subpath of $P$ that starts with $a$ and ends with $b$.
If $a = \tail{e_1}$ (resp.~$b = \head{e_m}$) we may simplify the notation and write $bP$ (resp.~$aP$).

\newcommand{\red}[1]{\Fkt{\variablestyle{red}}{#1}}

If now $Q=(e_{m+1},\ldots,e_n)$ is a path such that $\head{e_m}=\tail{e_{m+1}}$, by $P \cdot Q$ we denote a \emph{concatenation} of $P$ and $Q$, i.e., a sequence $(e_1,\ldots,e_m,e_{m+1},\ldots,e_n)$.
Note that in general $P \cdot Q$ does not need to be a path (not even a trail), but it contains a path from \(\Start{P}\) to \(\End{Q}\) that we call a \emph{reduction} of $P \cdot Q$, and denote by \(\red{P \cdot Q}\).
When combining the concatenation notation with the subpath notation, we will write $a\xrightarrow{P}  b \xrightarrow{Q} c$ instead of $a\xrightarrow{P}  b \cdot b \xrightarrow{Q} c$.

For a vertex $v$ of $D$, we define the \emph{cycle-degree} of \(v\), denoted \(\Cycledeg{D}{v}\), as the maximum number of edge-disjoint cycles that contain an edge adjacent to $v$.

\paragraph*{(Rooted) trees.} We often use (rooted) trees as structures underlying our decompositions.
A \emph{tree} is a connected undirected graph that does not contain any cycle.
A \emph{rooted tree} is a digraph obtained by orienting the edges of a tree so that there is exactly one vertex, called the \emph{root}, that can reach every other vertex by a directed path. In the context of a (rooted) tree $T$ that underlies a decomposition of a digraph $D$, we often call the vertices of $T$ \emph{nodes} in order to distinguish them from the vertices of $D$.

Let \(T\) be a rooted tree.
We say that a node $x$ is an \emph{ancestor} of a node $y$, and that $y$ is a \emph{descendant} of $x$, if $y$ is reachable from $x$ by a path in \(T\). Note that thus, every node is its own ancestor and descendant. Also, for every edge \((x,y)\) of $T$, we call \(x\) the \emph{parent} of $y$ and \(y\) a \emph{child} of \(x\). 
The subtree of $T$ rooted at a node \(x\) is the subtree induced by $x$ and all its descendants; it is denoted by \(\subtree{T}{x}\), while $\treeUpper{T}{x}$ denotes the subtree induced by all the nodes not contained in \(\subtree{T}{x}\).

Every edge \(e=xy\) in a tree \(T\) splits the \(T\) into two subtrees, which we shall denote by \(\treeSplit{T}{x}{e}\) (containing \(x\)) and \(\treeSplit{T}{y}{e}\) (containing \(y\)).
If the tree is rooted, then we refer to the subtree not containing the root by \(\treeLower{T}{e}\) (this subtree is equal to $\subtree{T}{\head{e}}$), and to the part containing the root by \(\treeUpper{T}{e}\) (this subtree is equal to $\treeUpper{T}{\head{e}}$).

It often is the case that a (rooted) tree $T$ is endowed with a \emph{bag function} $\beta\colon V(T)\to 2^U$, mapping every node $x$ of $T$ to its bag $\beta(x)$ so that $\{\beta(x)\colon x\in V(T)\}$ form a near-partition of some universe~$U$ (usually, the vertex set of a digraph).
In this context, for a subtree $T'$ of $T$ we denote $\beta(T')\coloneqq \bigcup_{x\in V(T')} \beta(x)$. Note that for every edge $e = xy$ of $T$, we have that $\Brace{\beta(\treeSplit{T}{x}{e}),\beta(\treeSplit{T}{y}{e})}$ is a partition of $U$; we call it the \emph{cut} induced by $e$.

We write \(\inc{e}{x}\) if the edge \(e\) of a tree $T$ is incident to the node \(x\) of \(T\).

\paragraph*{Flows.}
For  two sets of vertices \(A, B \subseteq \V{D}\),
an \emph{\(A\)-\(B\)-flow} of order \(k\) is a collection of \(k\) edge-disjoint paths, each starting in a vertex of \(A\) and ending in a vertex of \(B\).
If \(A = \Set{a}\) and/or \(B = \Set{b}\), then we write \(a\)-\(B\)-flow/\(A\)-\(b\)-flow/\(a\)-\(b\)-flow instead.
A set \(X \subseteq \V{D}\) is \emph{\(k\)-flow-connected} if for every two vertices \(u,v \in X\) there is a \(u\)-\(v\)-flow of order \(k\) in \(D\).

We use the classic Menger's Theorem.

\begin{theorem}[Menger's Theorem~\cite{menger1927}]
	\label{thm:menger}
	Let \(D\) be a digraph, \(s,t\) two vertices in \(D\) and \(k \in \N\).
	Then either \(D\) contains an \(s\)-\(t\)-flow of order \(k\) or a set \(F\) of less than \(k\) edges such that \(D-F\) does not contain an \(s\)-\(t\)-path.
\end{theorem}

Note the following straightforward corollaries.

\begin{corollary}
	\label{cor:setFlow}
	Let \(Z\) be a \(d\)-flow-connected set in a digraph \(D\).
	Then for every \(z \in Z\) and \(X \subseteq Z\) there is a \(z\)-\(X\)-flow of order \(d\). 
\end{corollary}

\begin{corollary}\label{cor:cycleDeg-hitting}
	Let \(v\) be a vertex of a digraph \(D\).
	Then the cycle-degree of \(v\) is equal to the minimum size of a set of edges \(F\) such that \(F\) hits every cycle containing an edge adjacent to \(v\).
\end{corollary}

\paragraph*{Ramsey's Theorem.}
A graph is \emph{complete} if every two vertices are adjacent.
Two disjoint sets \(A\) and \(B\) of vertices in a graph $G$ are \emph{complete} to each other if every vertex of \(A\) is adjacent to every vertex of \(B\), and \emph{anti-complete} to each other if no vertex of \(A\) is adjacent to any vertex of \(B\).

We shall use the following two versions of Ramsey's Theorem: the bipartite one and the colourful one.

\begin{theorem}\label{thm:BipRamsey}
	There is a function \(\bRamsey\colon \N\times\N\to \N\) such that the following holds:
	For every bipartite graph \(G\) with both sides, say \(A\) and \(B\), of size at least \(\bRamsey(a,b)\), there exist subsets $A' \subseteq A$ and $B' \subseteq B$ of sizes $a$, and $b$ respectively, that are either complete or anti-complete to each other.  
\end{theorem}

\begin{theorem}\label{thm:ColourRamsey}
	For every \(k \in \N\) there exists a function \(\colRamsey\colon \N^k \to \N\) such that the following holds:
	For every complete graph \(G\) on at least  \(\colRamsey(x_1,\dots,x_k)\) vertices, whose edges are coloured with \(k\) colours, say \(c_1,\dots,c_k\), there exists \(i \in [k]\) and a set \(K\) of \(x_i\) vertices of \(G\) such that all the edges with both endpoints in \(K\) have colour \(c_i\).
\end{theorem}

We use the shorthands $\bRamsey(t)\coloneqq \bRamsey(t,t)$ and $\colRamsey^k(t)\coloneqq \colRamsey(\underbrace{t,t,\ldots,t}_{k\text{ times}})$.


\subsection{Immersions and topological minors}
Let $D$ and $H$ be digraphs.
We say that $D$ admits $H$ as an \emph{immersion} if there exists an
\emph{immersion model} of $H$ in $D$, that is, a mapping $\pi$ defined on the vertices and edges of $H$ satisfying the following:
\begin{itemize}
	\item $\pi$ maps vertices of $H$ to pairwise different vertices of $D$;
	\item $\pi$ maps each edge $(x,y) \in E(H)$ to a $(\pi(x),\pi(y))$-trail in $D$; and
	\item the trails in $\{\pi(e) \colon e \in E(H)\}$ are pairwise edge-disjoint.
\end{itemize}
Note that in this definition, by shortcutting the trails if necessary, we may assume that each trail $\pi(e)$ does not visit any vertex more than once, i.e., it forms a path.

The notions of topological minors and of topological minor models are defined in the same way, except we require that the paths in $\{\pi(e) \colon e \in E(H)\}$ are pairwise vertex-disjoint apart from possibly sharing endpoints. Clearly, if $H$ is a digraph without parallel edges, then any digraph $D$ containing $H$ as a topological minor also contains $D$ as an immersion.



\subsection{Cylindrical grids and walls}

A \emph{cylindrical grid} of order \(k\) is a digraph consisting of \(k\) vertex-disjoint simple cycles \(C_1,\dots,C_k\) of length \(2k\) and \(2k\) vertex-disjoint paths \(P_1,\dots,P_k,Q_1,\dots,Q_k\) with $\bigcup_{i\in [k]} V(C_i)=\bigcup_{i\in [k]} V(P_i)\cup V(Q_i)$ such that every path \(P_i\) intersects the cycles in order \(C_1,\dots,C_k\) and every path \(Q_i\) intersects the cycles in order \(C_k,\dots,C_1\).
Additionally, every cycle $C_i$ intersects the paths in cyclic order \(P_1,Q_1,\dots,P_k,Q_k\).

We obtain the \emph{cylindrical wall} of order \(k\) from the cylindrical grid of order \(k\) by splitting every vertex of the cycles \(C_2,\dots,C_{k-1}\) into two vertices: one remains the head of both in-coming edges, the other remains the tail of both out-going edges, and we connect the first to the latter by an edge; see \cref{fig:wall}.
For $i \in [k]$, we denote the consecutive vertices of $P_i$ by $p^-_{i,1},p^+_{i,2},p^-_{i,2},p^+_{i,3}\ldots, p^-_{i,k-1}, p^+_{i,k},$ and the consecutive vertices of $Q_i$ by $q^-_{i,k},q^+_{i,k-1},q^-_{i,k-1},q^+_{i,k-2},\ldots, q^-_{i,2}, q^+_{i,1}.$
We denote by $P^+$ the set of all vertices of in-degree 2 that belong to the paths $P_i$, $i\in [k]$, and the definitions of the sets $P^-,Q^+$ and $Q^-$ follow analogously.
The cylindrical wall of order \(k\) is denoted by \(\gridImm{k}\).


\begin{figure}[!ht]
	\begin{center}
		\begin{tikzpicture}[scale=0.8]
			\def\dist{1.5}
			\node (centre) at (0,0) {};
			\node[vertex,label=315:{\(p^{-}_{1,1}\)}] (p-1-1-m) at ($(0,0)$) {};
			\node[vertex,label=315:{\(p^{-}_{2,1}\)}] (p-2-1-m) at ($(p-1-1-m)+(270:{2*\dist})$) {};
			\node[vertex,label=315:{\(p^{-}_{3,1}\)}] (p-3-1-m) at ($(p-2-1-m)+(270:{2*\dist})$) {};
			\node[vertex,label=315:{\(p^{-}_{4,1}\)}] (p-4-1-m) at ($(p-3-1-m)+(270:{2*\dist})$) {};
			
			\node[vertex,label=315:{\(q^{+}_{1,1}\)}] (q-1-1-p) at ($(p-1-1-m)+(270:{\dist})$) {};
			\node[vertex,label=315:{\(q^{+}_{2,1}\)}] (q-2-1-p) at ($(q-1-1-p)+(270:{2*\dist})$) {};
			\node[vertex,label=315:{\(q^{+}_{3,1}\)}] (q-3-1-p) at ($(q-2-1-p)+(270:{2*\dist})$) {};
			\node[vertex,label=315:{\(q^{+}_{4,1}\)}] (q-4-1-p) at ($(q-3-1-p)+(270:{2*\dist})$) {};
			
			\foreach \i in {1,2,3,4}
			{
				\node[vertex,label=315:{\(p^{+}_{\i,2}\)}] (p-\i-2-p) at ($(p-\i-1-m)+(0:{\dist})$) {};
				\node[vertex,label=315:{\(p^{-}_{\i,2}\)}] (p-\i-2-m) at ($(p-\i-2-p)+(0:{\dist})$) {};
				\node[vertex,label=315:{\(p^{+}_{\i,3}\)}] (p-\i-3-p) at ($(p-\i-2-m)+(0:{\dist})$) {};
				\node[vertex,label=315:{\(p^{-}_{\i,3}\)}] (p-\i-3-m) at ($(p-\i-3-p)+(0:{\dist})$) {};
				\node[vertex,label=315:{\(p^{+}_{\i,4}\)}] (p-\i-4-p) at ($(p-\i-3-m)+(0:{\dist})$) {};
			}
			\foreach \i in {1,2,3,4}
			{
				\node[vertex,label=315:{\(q^{-}_{\i,2}\)}] (q-\i-2-m) at ($(q-\i-1-p)+(0:{\dist})$) {};
				\node[vertex,label=315:{\(q^{+}_{\i,2}\)}] (q-\i-2-p) at ($(q-\i-2-m)+(0:{\dist})$) {};
				\node[vertex,label=315:{\(q^{-}_{\i,3}\)}] (q-\i-3-m) at ($(q-\i-2-p)+(0:{\dist})$) {};
				\node[vertex,label=315:{\(q^{+}_{\i,3}\)}] (q-\i-3-p) at ($(q-\i-3-m)+(0:{\dist})$) {};
				\node[vertex,label=315:{\(q^{-}_{\i,4}\)}] (q-\i-4-m) at ($(q-\i-3-p)+(0:{\dist})$) {};
			}
			
			\foreach \i in {1,2,3,4}{
				\node (P-\i) at ($(p-\i-1-m)+(-0.85,0)$) {\textcolor{CornflowerBlue}{\(P_{\i}\)}};
				\node (Q-\i) at ($(q-\i-1-p)+(-0.85,0)$) {\textcolor{CornflowerBlue}{\(Q_{\i}\)}};
				
				\begin{pgfonlayer}{background}
					\draw[directededge,CornflowerBlue] (p-\i-1-m) to (p-\i-2-p);
					\begin{scope}
						\clip (p-\i-2-p.center) -- (p-\i-2-m.center) -- ($(p-\i-2-m)+(0,-0.5)$) -- ($(p-\i-2-p)+(0,-0.5)$) -- cycle;
						\draw[directededge,CornflowerBlue] (p-\i-2-p) to (p-\i-2-m);
					\end{scope}
					\begin{scope}
						\clip (p-\i-2-p.center) -- (p-\i-2-m.center) -- ($(p-\i-2-m)+(0,0.5)$) -- ($(p-\i-2-p)+(0,0.5)$) -- cycle;
						\draw[directededge,LavenderMagenta] (p-\i-2-p) to (p-\i-2-m);
					\end{scope}
					\draw[directededge,CornflowerBlue] (p-\i-2-m) to (p-\i-3-p);
					\begin{scope}
						\clip (p-\i-3-p.center) -- (p-\i-3-m.center) -- ($(p-\i-3-m)+(0,-0.5)$) -- ($(p-\i-3-p)+(0,-0.5)$) -- cycle;
						\draw[directededge,CornflowerBlue] (p-\i-3-p) to (p-\i-3-m);
					\end{scope}
					\begin{scope}
						\clip (p-\i-3-p.center) -- (p-\i-3-m.center) -- ($(p-\i-3-m)+(0,0.5)$) -- ($(p-\i-3-p)+(0,0.5)$) -- cycle;
						\draw[directededge,LavenderMagenta] (p-\i-3-p) to (p-\i-3-m);
					\end{scope}
					\draw[directededge,CornflowerBlue] (p-\i-3-m) to (p-\i-4-p);
					
					\draw[directededge,CornflowerBlue] (q-\i-4-m) to (q-\i-3-p);
					\begin{scope}
						\clip (q-\i-3-p.center) -- (q-\i-3-m.center) -- ($(q-\i-3-m)+(0,-0.5)$) -- ($(q-\i-3-p)+(0,-0.5)$) -- cycle;
						\draw[directededge,CornflowerBlue] (q-\i-3-p) to (q-\i-3-m);
					\end{scope}
					\begin{scope}
						\clip (q-\i-3-p.center) -- (q-\i-3-m.center) -- ($(q-\i-3-m)+(0,0.5)$) -- ($(q-\i-3-p)+(0,0.5)$) -- cycle;
						\draw[directededge,LavenderMagenta] (q-\i-3-p) to (q-\i-3-m);
					\end{scope}
					\draw[directededge,CornflowerBlue] (q-\i-3-m) to (q-\i-2-p);
					\begin{scope}
						\clip (q-\i-2-p.center) -- (q-\i-2-m.center) -- ($(q-\i-2-m)+(0,-0.5)$) -- ($(q-\i-2-p)+(0,-0.5)$) -- cycle;
						\draw[directededge,CornflowerBlue] (q-\i-2-p) to (q-\i-2-m);
					\end{scope}
					\begin{scope}
						\clip (q-\i-2-p.center) -- (q-\i-2-m.center) -- ($(q-\i-2-m)+(0,0.5)$) -- ($(q-\i-2-p)+(0,0.5)$) -- cycle;
						\draw[directededge,LavenderMagenta] (q-\i-2-p) to (q-\i-2-m);
					\end{scope}
					\draw[directededge,CornflowerBlue] (q-\i-2-m) to (q-\i-1-p);
					
					\draw[directededge,LavenderMagenta] (p-\i-1-m) to (q-\i-1-p);
					\draw[directededge,LavenderMagenta] (p-\i-2-m) to (q-\i-2-p);
					\draw[directededge,LavenderMagenta] (p-\i-3-m) to (q-\i-3-p);
					\draw[directededge,LavenderMagenta] (p-\i-4-p) to (q-\i-4-m);
				\end{pgfonlayer}
			}
			
			\foreach[evaluate={\j=int(\i+1);}] \i in {1,2,3}{
				\draw[directededge,LavenderMagenta] (q-\i-1-p) to (p-\j-1-m);
				\draw[directededge,LavenderMagenta] (q-\i-2-m) to (p-\j-2-p);
				\draw[directededge,LavenderMagenta] (q-\i-3-m) to (p-\j-3-p);
				\draw[directededge,LavenderMagenta] (q-\i-4-m) to (p-\j-4-p);
			}
			
			\foreach \x/\y in {q-4-1-p/p-1-1-m,q-4-2-m/p-1-2-p,q-4-3-m/p-1-3-p,q-4-4-m/p-1-4-p}{
				\begin{scope}
					\clip ($(\x)+(-0.2,-1)$) -- ($(\y)+(-0.2,1)$) --($(\y)+(5,1)$) -- ($(\x)+(5,-1)$) -- cycle;
					\draw[directededge,LavenderMagenta,out=260,in=100] (\x) to (\y);
				\end{scope}
				\begin{scope}
					\clip ($(\x)+(-0.2,-1)$) -- ($(\y)+(-0.2,1)$) --($(\y)+(-5,1)$) -- ($(\x)+(-5,-1)$) -- cycle;
					\draw[directededge,LavenderMagenta,out=260,in=100,dashed] (\x) to (\y);
				\end{scope}
			}
			
			\node[LavenderMagenta] (C-1-label) at ($(p-1-1-m)+(-0.2,1.3)$) {\(C_1\)};
			\node[LavenderMagenta] (C-2-label) at ($(p-1-2-p)+(-0.2,1.3)$) {\(C_2\)};
			\node[LavenderMagenta] (C-3-label) at ($(p-1-3-p)+(-0.2,1.3)$) {\(C_3\)};
			\node[LavenderMagenta] (C-4-label) at ($(p-1-4-p)+(-0.2,1.3)$) {\(C_4\)};
			
			
		\end{tikzpicture}
	\end{center}
	\caption{The cylindrical wall of order 4.}
	\label{fig:wall}
\end{figure}
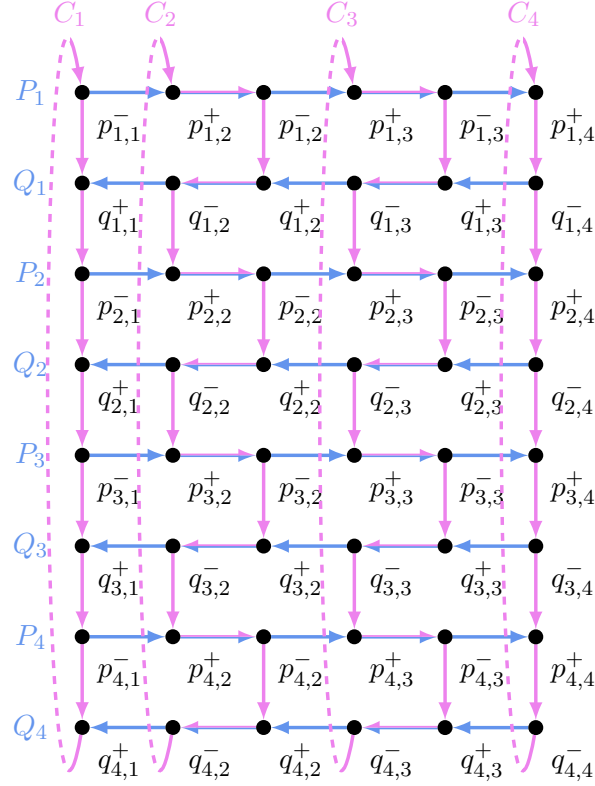

\subsection{Onions and tripods}

\paragraph*{Onions.}
An \emph{onion} is a digraph $Z$ which consists of two distinct vertices $x$ and $y$, and three edges: two from $x$ to $y$ and one from $y$ to $x$.
We call $x$ the \emph{root} and $y$ the \emph{tip} of $Z$. (See~\cref{fig:onion}.) 

An \emph{in-onion-star} of order $k$, denoted $S^\variablestyle{in}_{k}$, is a digraph consisting of $k$ onions that all share the same tip (called the \emph{centre} of $S^\variablestyle{in}_{k}$) but their roots are pairwise different.
Similarly, an \emph{out-onion-star} of order $k$, denoted $S^\variablestyle{out}_k$, is a digraph consisting of $k$ onions that all share the same root (called the \emph{centre} of $S^\variablestyle{out}_{k}$) but their tips are pairwise different.
An \emph{onion-star} of order $k$ is obtained by identifying the centres of an in-onion-star and an out-onion-star both of order $k$, see~\cref{fig:onion-star}.

As already noted in \cite{bozyk2022digraphs}, to embed any strongly connected subcubic (i.e., with maximum total degree at most $3$) digraph as an immersion, it suffices to find an immersion model of an onion-star of a sufficiently large order. This observation in particular applies to cylindrical walls.

\begin{lemma}
	\label{lem:gridImm_from_onion_star}
	Let \(H\) be a digraph with \(t\) vertices in which every vertex has total degree at most \(3\) and both in-degree and out-degree at most \(2\).
	Then every digraph containing the onion-star of order \(t\) as an immersion also contains \(H\) as an immersion.
	In particular, if a digraph \(D\) contains an onion-star of order $2k^2$ as an immersion, then \(D\) also contains \(\gridImm{k}\) as an immersion.
\end{lemma}

When speaking about an immersion model of an onion, we often implicitly identify it with its image, that is, the subgraph of the host graph formed by all the vertices and edges participating in the model. Also, the root and the tip of an onion immersion model is the image of the root and the tip, respectively.

\paragraph*{Tripods.} A \emph{migration digraph} is a triple \(\Brace{D,S,T}\) consisting of a digraph \(D\) and two vertex sets \(S,T\subseteq V(D)\).
An \emph{(\(S\)-\(T\)-)out-tripod} in \(\Brace{D,S,T}\) is a subgraph of \(D\) consisting of
\begin{itemize}
	\item two distinct vertices \(s_1,s_2 \in S\),
	\item a vertex \(t \in T\),
	\item a vertex \(c \in \V{D}\) called the \emph{centre}, and
	\item an \((s_1,c)\)-path, an \((s_2,c)\)-path, and a \((c,t)\)-path, all vertex-disjoint except for sharing~\(c\).
\end{itemize}
Similarly, an \emph{(\(S\)-\(T\)-)in-tripod} is a subgraph of \(D\) consisting of
\begin{itemize}
	\item two distinct vertices \(t_1,t_2 \in T\),
	\item a vertex \(s \in T\),
	\item a vertex \(c \in \V{D}\) called the \emph{centre}, and
	\item an \((s,c)\)-path, a \((c,t_1)\)-path, and a \((c,t_2)\)-path, all vertex-disjoint except for~sharing~\(c\).
\end{itemize}
Note that in these definitions, we allow the distinguished vertices to coincide, except for requiring $s_1\neq s_2$ in an out-tripod and $t_1\neq t_2$ in an in-tripod.

We say \emph{tripod} for a tripod of either type.
Briański, Hatzel, Okrasa, and Pilipczuk~\cite{2024tripodEP} proved that tripods have the \EP-property.

\begin{theorem}[{\cite{2024tripodEP}}]
	\label{thm:tripod-EP}
	There is a function \(\thmFkt{thm:tripod-EP}{f}{}\colon \N\to \N\) such that for every migration digraph \(D\) and \(k\in \N\)~either
	\begin{itemize}
		\item there is a family of \(k\) edge-disjoint \(S\)-\(T\)-in-tripods (resp.~out-tripods) in \(D\), or
		\item there is a set $X$ consisting of at most \(\thmFkt{thm:tripod-EP}{f}{k}\) edges such that \(D-X\) does not contain any \(S\)-\(T\)-in-tripod (resp.~out-tripod).
	\end{itemize}
	Moreover, there is a fixed-parameter algorithm to find, for any digraph \(D\) and parameter $k$, either of these objects in $D$.
\end{theorem}
Although the algorithmic counterpart of \cref{thm:tripod-EP} is not explicitely stated in \cite{2024tripodEP}, we remark that the proof can be adapted in a straightforward way to obtain one. Namely, the proof relies on basic constructions executable in fixed-parameter time and the following black-boxes: Menger's Theorem (\cref{thm:menger}), Matroid Intersection Theorem~\cite{edmonds1970}, and Onion Harvesting Lemma (\cref{lem:onion_harvesting}). All these black-boxes have efficient algorithmic counterparts.

\begin{figure}[!ht]
	\centering
	\begin{tikzpicture}
		\node (centre) at (0,0) {};
		\def\dist{0.5}
		\node[vertex,CornflowerBlue] (s-3) at ($(centre)+(180:2.5)$) {};
		\node[vertex,CornflowerBlue] (s-2) at ($(s-3)+(90:\dist)$) {};
		\node (s-1-label) at ($(s-2)+(130:0.25)$) {$s_1$};
		\node[vertex,CornflowerBlue] (s-1) at ($(s-2)+(90:\dist)$) {};
		\node[vertex,CornflowerBlue] (s-4) at ($(s-3)+(270:\dist)$) {};
		\node (s-2-label) at ($(s-4)+(130:0.25)$) {$s_2$};
		\node[vertex,CornflowerBlue] (s-5) at ($(s-4)+(270:\dist)$) {};
		
		\node[vertex,PastelOrange] (t-centre) at ($(centre)+(280:0.3)$) {};
		\node[PastelOrange!85!black] (centre-label) at ($(t-centre)+(60:0.25)$) {$c$};
		
		\node[vertex,LavenderMagenta] (t-3) at ($(centre)+(0:2.5)$) {};
		\node[vertex,LavenderMagenta] (t-2) at ($(t-3)+(90:\dist)$) {};
		\node[vertex,LavenderMagenta] (t-1) at ($(t-2)+(90:\dist)$) {};
		\node[vertex,LavenderMagenta] (t-4) at ($(t-3)+(270:\dist)$) {};
		\node[vertex,LavenderMagenta] (t-5) at ($(t-4)+(270:\dist)$) {};
		\node (t-label) at ($(t-5)+(60:0.25)$) {$t$};
		
		\draw[pathdecorated] (s-2) to (t-centre);
		\draw[pathdecorated] (s-4) to (t-centre);
		\draw[path] (t-centre) to (t-5);
		
		\begin{pgfonlayer}{background}
			\def\setdist{0.4}
			\draw[line width=1.5pt,CornflowerBlue] ($(s-1)+(180:\setdist)$) to[closed,curve through ={($(s-1)+(90:\setdist)$) ($(s-1)+(0:\setdist)$) ($(s-1)!0.5!(s-5)+(0:\setdist)$) ($(s-5)+(0:\setdist)$) ($(s-5)+(270:\setdist)$) ($(s-5)+(180:\setdist)$)}] ($(s-5)!0.5!(s-1)+(180:\setdist)$);
			
			\node[CornflowerBlue] (S-label) at ($(s-5)+(170:1)$) {$S$};
			
			\draw[line width=1.5pt,LavenderMagenta] ($(t-1)+(180:\setdist)$) to[closed,curve through ={($(t-1)+(90:\setdist)$) ($(t-1)+(0:\setdist)$) ($(t-1)!0.5!(t-5)+(0:\setdist)$) ($(t-5)+(0:\setdist)$) ($(t-5)+(270:\setdist)$) ($(t-5)+(180:\setdist)$)}] ($(t-5)!0.5!(t-1)+(180:\setdist)$);
			
			\node[LavenderMagenta] (T-label) at ($(t-5)+(10:1)$) {$T$};
		\end{pgfonlayer}
	\end{tikzpicture}
	\caption{An out-tripod with its \textcolor{PastelOrange!85!black}{centre $c$}, the two sources $s_1$ and $s_2$ and its sink $t$.
	This figure is reproduced and originally provided in~\cite{2024tripodEP}.}
	\label{fig:tripod}
\end{figure}
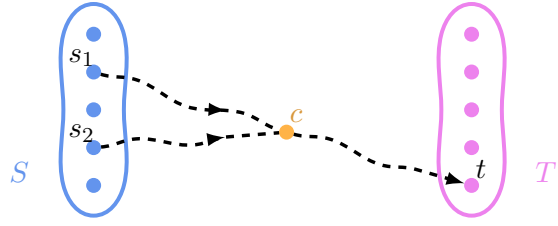

\subsection{Directed cuts}
A \emph{(directed) cut} of a digraph \(D\) is a bi-partition \(\Brace{A,B}\) of \(\V{D}\).
The cut \(\Brace{A,B}\) has \emph{order} \(k\) if there are at most \(k\) edges with tail in \(A\) and head in \(B\), i.e.~\(\Abs{\Out{}{A}} \leq k\), or there are at most \(k\) edges with tail in \(B\) and head in \(A\), i.e.~\(\Abs{\Out{}{B}} \leq k\). 
Thus, formally, $(A,B)$ and $(B,A)$ denote exactly the same object. 
We also write \(\Abs{\Brace{A,B}} \leq k\) to denote that \(\Brace{A,B}\) has order at most $k$.

A cycle $C$ in $D$ \emph{crosses} the cut \(\Brace{A,B}\) if it visits both a vertex belonging to $A$ and a vertex belonging to $B$.
The \emph{porosity} of \(\Brace{A,B}\) is the minimum size of a set of edges that intersects all cycles crossing \(\Brace{A,B}\).
The cut is called \emph{trivial} if \(A = \emptyset\) or \(B = \emptyset\).
For disjoint vertex subsets $X,Y\subseteq V(D)$, an \emph{$X$-$Y$-cut} is a cut $(A,B)$ with $X\subseteq A$ and $Y\subseteq B$.
For a non-empty set \(X \subseteq \V{D}\), we say that the cut $(A,B)$ \emph{splits} $X$ if both $A\cap X$ and $B\cap X$ are non-empty.


\section{Decompositions}\label{sec:decomps}

In this section we introduce a width measure for directed graphs that serves the role analogous to the role of tree-cutwidth in undirected graphs.
A large cylindrical wall as an immersion shall witness it to be large, while the absence of an immersion of a large cylindrical wall ensures it to be small.

In the context of immersion models, many paths of the model may run through a single vertex of the host digraph.
A standard way of limiting this behaviour is to restrict the degree of the tree underlying the decomposition, for instance to assume that it is subcubic.
This idea is implemented in the definition of multiple width measures, for example in branchwidth, rankwidth, and cyclewidth.
In our setting however, this restriction would be too strong, as not every vertex of high degree can be used to actually obtain an immersion model of a large cylindrical wall.
For example, a large bidirected star does not contain any interesting immersions (in particular, large cylindrical walls); yet intuitively, it would be unnatural to try to decompose it using a subcubic decomposition.
A large onion star, on the other hand, contains a large cylindrical wall as an immersion.
In order to capture this difference we do not restrict the degree of our decomposition tree, but instead introduce edge sets that we call \emph{onion guards}.
Intuitively, these prevent the nodes of high degree in a decomposition to harbour a lot of onions spread out over all of their neighbours.
We remark that a similar, yet different mechanism is used in the definition of (undirected) tree-cutwidth, where cuts induced by edges of the decomposition work differently depending on whether the number of edges crossing the cut is at most two, or three or~more.

\paragraph*{Tree-partitions.}
We proceed to defining the decompositions and the corresponding width parameters.
We start with a simple notion that underlies all our decompositions.

\begin{definition}[Tree-partition]
  \label{def:treept}
	Given a digraph $D$, a triple $(T,\beta,\gamma)$ is a {\em{tree-partition}} of $D$ if $T$ is a tree and we have the following:
	\begin{enumerate}[label=\textbf{(P\arabic*)},leftmargin=1.2cm]
		\item \(\beta \colon \V{T} \to 2^{\V{D}}\) maps every node $x$ of \(T\) to a subset of vertices \(\beta(x)\subseteq \V{D}\), called the \emph{bag} of $x$, so that $\{\beta(x)\colon x\in \V{T}\}$ is a near-partition of \(\V{D}\).\label{item:tp:bags}
		\item \(\gamma \colon \E{T} \to 2^{A(D)}\) maps every edge~$e=xy$ of \(T\) to a set of edges \(\gamma(e)\subseteq A(D)\), called the \emph{edge-guard} of~$e$, so that \(\gamma(e)\) hits all the cycles in $D$ that cross the cut induced by $e$; i.e.,  $(\beta(\treeSplit{T}{x}{e}),\beta(\treeSplit{T}{y}{e}))$.\label{item:tp:edgeguards}
	\end{enumerate}
	The \emph{porosity} of $(T,\beta,\gamma)$ is the maximum size of an edge-guard, that is, $\max_{e\in E(T)} |\gamma(e)|$.
\end{definition}

Note that in the definition above, we do \emph{not} bound the maximum degree of the tree $T$.

In the definition of a tree-partition, we may allow $T$ to be a rooted tree, and thus way we may speak about \emph{rooted tree-partitions}.
We sometimes need to work with rooted tree-partitions that have the following basic connectivity property: We say that a rooted tree-partition $(T,\beta,\gamma)$ of a digraph $D$ is \emph{\coherent} if for every non-root node $x$, the set $\bigcup \{\beta(y)\colon y\textrm{ is a descendant of }x\textrm{ in }T\}$ induces a strongly connected subgraph of $D$.
The following standard lemma shows that every rooted tree-partition can be adjusted to a coherent one without enlarging the bags or the edge-guards.

\begin{lemma}
	\label{lem:rooting_potdecomps}
	Let \(D\) be a digraph and let \((T, \beta, \gamma)\) be a rooted tree-partition of \(D\).
	Then, there exists a coherent tree-partition \((T', \beta', \gamma')\) of \(D\) such that:
	\begin{itemize}
		\item for each node \(x' \in V(T')\) there exists a node \(x \in V(T)\) such that \(\beta'(x') \subseteq \beta(x)\); and
		\item for each edge \(e' \in E(T')\) there exists an edge \(e \in E(T)\) such that \(\gamma'(e') = \gamma(e)\);
	\end{itemize}
	Moreover, $(T',\beta',\gamma')$ can be found in time polynomial in the size of $D$ and $T$.
\end{lemma}
\begin{proof}
	
	
	For every node $x \in V(T)$ denote by $h(x)$ the maximum number of nodes on a path in $T$ leading from $x$ to a leaf.
	Let $r$ be the root of $T$, and \(\TT_0 \coloneqq (T, \beta, \gamma)\).
	For every $i \in [h(r)-1]$ let \(T^{\text{rest}}_i\) be the subtree of $T$ induced by all vertices for which $h(x) >i$.
	We will define a sequence \(\TT_1,\TT_2,\TT_3,\ldots,\TT_{h(r)-1}\) of rooted tree-partitions of~\(D\), where \(\TT_i = (T_i, \beta_i, \gamma_i)\),
	satisfying the following conditions for each $i$:
	\begin{enumerate}
		\item $V(T^{\text{rest}}_i) \subseteq V(T_{i})$, in particular $r \in V(T_i)$; moreover, $r$ is a root of $T_i$; 
		\item for each node \(x' \in V(T_i)\) there exists a node \(x \in V(T)\) such that \(\beta_i(x') \subseteq \beta(x)\);
		\item for each edge \(e' \in E(T_i)\) there exists an edge \(e \in E(T)\) such that \(\gamma_i(e') = \gamma(e)\); and
		\item for each node \(x \in V(T_i) \setminus V(T^{\text{rest}}_i)\), the set
		\(\bigcup \Set{\beta_i(y) \colon y \text{ is a descendant of } x\text{ in }T_i}\)
		induces a strongly connected subgraph of \(D\);
	\end{enumerate}
	Suppose we have already defined \(\TT_i = (T_i, \beta_i, \gamma_i)\).
	Let \(X\subseteq V(T^{\text{rest}}_i)\) be the set of all nodes $x$ for which \(h(x)=i + 1\).
	If \(X = \Set{r}\) the construction terminates with \((T',\beta',\gamma')\coloneqq ( T_i, \beta_i, \gamma_i)\).
	Otherwise, every node in \(X\) has a parent.
	For \(x \in X\), let \(p(x)\) be the parent of \(x\) in \(T_i\).
	As \(h(p(x)) > h(x)\) by definition of $h$, the node \(p(x)\) does not belong to \(X\).
	Let \(S_x \coloneqq \bigcup \Set{\beta_i(y) \colon y \text{ is a descendant of } x \text{ in }T_i}\).
	Since the nodes of \(X\) are pairwise incomparable in the ancestor-descendant relation of $T_i$,
	the sets \(S_x\) for $x\in X$ are pairwise disjoint.
	
	We now proceed to define \(T_{i + 1}\).
	For each $x\in X$, let \(C^x_1, C^x_2, \dots, C^x_{\ell(x)}\) be the strongly connected components of \(D[S_x]\).
	Note that the tree $T^{\text{rest}}_{i+1}$ is induced by the vertices $V(T^{\text{rest}}_i) \setminus X$, and is not empty since it includes at least the root \(r\), as well as all the parents of the nodes from \(X\).
	The tree \(T_{i + 1}\) will have \(r\) as its root and as its nodes the set
	\(V(T^{\text{rest}}_{i+1}) \cup \bigcup _ {x \in X} V(\subtree{T}{x}) \times [\ell(x)]\).
	The edges of \(T_{i + 1}\) are defined as follows:
	\begin{itemize}
		\item \(V(T^{\text{rest}}_{i+1})\) induces a tree isomorphic to \(T^{\text{rest}}_{i+1}\);
		\item \(V(\subtree{T}{x}) \times \Set{j}\), for each $x\in X$ and \(j \in [\ell(x)]\), induces a subtree
		isomorphic to \(\subtree{T}{x}\); and
		\item the node \((x, j)\), for each $x\in X$ and \(j \in [\ell(x)]\), has \(p(x) \in V(T^{\text{rest}}_{i+1})\) as its parent.
	\end{itemize}
	This gives us the tree \(T_{i + 1}\).
	Next, we define the bag function \(\beta_{i + 1}\) as follows (see \cref{fig:potdecomp_rooting_step}):
	\begin{equation*}
		\beta_{i + 1}(z) = \begin{cases}
			\beta_i(z), & \text{ if } z \in V(T^{\text{rest}}_{i+1})\\
			\beta_i(z') \cap V(C^x_j), & \text{ if } z = (z', j) \text{ where } z' \in V(\subtree{T}{x}), x \in X, \text{ and } j \in [\ell(x)]\text{.}
		\end{cases}
	\end{equation*}
	Finally, we complete the construction with the definition of the edge-guard function \(\gamma_{i + 1}\):
	\begin{equation*}
		\gamma_{i + 1}(a,b) = \begin{cases}
			\gamma_i(a, b), & \text{ if } a,b \in V(T^{\text{rest}}_{i+1}),\\
			\gamma_i(p(x),x), & \text{ if } a = p(x), b=(x, j) \text{ where } x \in X, j \in [\ell(x)],\\
			\gamma_i(a', b'), & \text{ if } a = (a', j), b = (b', j) \text{ where } a', b' \in V(\subtree{T}{x}), x \in X, j \in [\ell(x)]\text{.}
		\end{cases}
	\end{equation*}
	\begin{figure}
		\centering
		\begin{tikzpicture}
			\path[use as bounding box] (-5,-5) rectangle (5,2);
			\tikzstyle{child} = [draw,thick,isosceles triangle,rotate=90,scale=0.8]
			\tikzstyle{help} = [draw,thick,rectangle]
			\def\distH{2}
			\node (left) at (-4,0) {
				\begin{tikzpicture}
					\path[use as bounding box] (-2,-2) rectangle (2,2);
					\node[draw,thick,circle] (vertex) at (0,0) {};
					\node (parent) at ($(vertex)+(110:1.2)$) {};
					\node (parent-p) at ($(parent)+(110:0.3)$) {};
					\node[child] (child-1) at ($(vertex)+(-2.2,-\distH)$) {\phantom{\(T_1\)}};
					\node[child] (child-2) at ($(vertex)+(0,-\distH)$) {\phantom{\(T_2\)}};
					\node[child] (child-3) at ($(vertex)+(2.2,-\distH)$) {\phantom{\(T_3\)}};
					
					\node (child-1-text) at ($(child-1)$) {\(T_1\)};
					\node (child-2-text) at ($(child-2)$) {\(T_2\)};
					\node (child-3-text) at ($(child-3)$) {\(T_3\)};
					
					\draw[edge] (vertex) to (child-1.east);
					\draw[edge] (vertex) to (child-2.east);
					\draw[edge] (vertex) to (child-3.east);
					
					\draw[edge] (vertex) to (parent);
					\draw[edge,dotted] (parent-p.center) to (parent.center);
					
					\node[blue!60!white,draw,rectangle,rounded corners] at ($(child-1)+(270:1.4)$) {
						\begin{tikzpicture}[scale=0.7]
							\node (centre) at (0,0) {};
							\node[scale=2,thick] (cyc-1) at ($(centre)+(180:1)$) {};
							\draw[directededge,PastelOrange] ($(cyc-1)+(0:0.4)$) arc (0:355:0.4);

							\node (cyc-2) at ($(centre)+(60:1)$) {};
							\draw[directededge,LavenderMagenta] ($(cyc-2)+(0:0.4)$) arc (0:355:0.4);
							\node (cyc-3) at ($(centre)+(-60:1)$) {};
							\draw[directededge,CornflowerBlue] ($(cyc-3)+(0:0.4)$) arc (0:355:0.4);
							
							\draw[directededge,shorten >=4] (cyc-1) to (cyc-2);
							\draw[directededge,shorten <=5,shorten >=5] (cyc-2) to (cyc-3);
							\draw[directededge,shorten <=3,shorten >=4] (cyc-1) to (cyc-3);
					\end{tikzpicture}};
					
					\node[blue!60!white,draw,rectangle,rounded corners] at ($(child-2)+(270:1.4)$) {
						\begin{tikzpicture}[scale=0.7]
							\node (centre) at (0,0) {};
							
							\node[scale=2,thick] (cyc-1) at ($(centre)+(180:1)$) {};
							\draw[directededge,PastelOrange] ($(cyc-1)+(0:0.4)$) arc (0:355:0.4);
							
							\node (cyc-2) at ($(centre)+(0:1)$) {};
							\draw[directededge,LavenderMagenta] ($(cyc-2)+(0:0.4)$) arc (0:355:0.4);
							
							\draw[directededge,shorten >=5,shorten <=1] (cyc-1) to (cyc-2);
					\end{tikzpicture}};
					
					\node[blue!60!white,draw,rectangle,rounded corners] at ($(child-3)+(270:1.4)$) {
						\begin{tikzpicture}[scale=0.7]
							\node (centre) at (0,0) {};
							
							\node[scale=2,thick] (cyc-1) at ($(centre)$) {};
							\draw[directededge,CornflowerBlue] ($(cyc-1)+(0:0.4)$) arc (0:355:0.4);
					\end{tikzpicture}};
			\end{tikzpicture}};
			
			\node (right) at (4,0) {\begin{tikzpicture}
					\path[use as bounding box] (-2,-2) rectangle (2,2);
					\node[draw,thick,circle] (vertex) at (0,0) {};
					\node (parent) at ($(vertex)+(110:1.2)$) {};
					\node (parent-p) at ($(parent)+(110:0.3)$) {};
					\node[child] (child-1) at ($(vertex)+(-3,-\distH)$) {\phantom{\(T_1\)}};
					\node[child] (child-11) at ($(vertex)+(-1.8,-\distH)$) {\phantom{\(T_1\)}};
					\node[child] (child-111) at ($(vertex)+(-0.6,-\distH)$) {\phantom{\(T_1\)}};
					\node[child] (child-2) at ($(vertex)+(0.6,-\distH)$) {\phantom{\(T_2\)}};
					\node[child] (child-22) at ($(vertex)+(1.8,-\distH)$) {\phantom{\(T_2\)}};
					\node[child] (child-3) at ($(vertex)+(3,-\distH)$) {\phantom{\(T_3\)}};
					
					\node (child-1-text) at ($(child-1)$) {\(T_1\)};
					\node (child-11-text) at ($(child-11)$) {\(T_1\)};
					\node (child-111-text) at ($(child-111)$) {\(T_1\)};
					\node (child-2-text) at ($(child-2)$) {\(T_2\)};
					\node (child-22-text) at ($(child-22)$) {\(T_2\)};
					\node (child-3-text) at ($(child-3)$) {\(T_3\)};
					
					\draw[edge] (vertex) to (child-1.east);
					\draw[edge] (vertex) to (child-11.east);
					\draw[edge] (vertex) to (child-111.east);
					\draw[edge] (vertex) to (child-2.east);
					\draw[edge] (vertex) to (child-22.east);
					\draw[edge] (vertex) to (child-3.east);
					
					\draw[edge] (vertex) to (parent);
					\draw[edge,dotted] (parent-p.center) to (parent.center);
					
					\node[blue!60!white,draw,rectangle,rounded corners] at ($(child-1)+(270:1)$) {
						\begin{tikzpicture}[scale=0.7]
							\node (centre) at (0,0) {};
							
							\node[scale=2,thick] (cyc-1) at ($(centre)$) {};
							\draw[directededge,PastelOrange] ($(cyc-1)+(0:0.4)$) arc (0:355:0.4);
					\end{tikzpicture}};
					\node[blue!60!white,draw,rectangle,rounded corners] at ($(child-11)+(270:1)$) {
						\begin{tikzpicture}[scale=0.7]
							\node (centre) at (0,0) {};
							
							\node[scale=2,thick] (cyc-1) at ($(centre)$) {};
							\draw[directededge,LavenderMagenta] ($(cyc-1)+(0:0.4)$) arc (0:355:0.4);
					\end{tikzpicture}};
					\node[blue!60!white,draw,rectangle,rounded corners] at ($(child-111)+(270:1)$) {
						\begin{tikzpicture}[scale=0.7]
							\node (centre) at (0,0) {};
							
							\node[scale=2,thick] (cyc-1) at ($(centre)$) {};
							\draw[directededge,CornflowerBlue] ($(cyc-1)+(0:0.4)$) arc (0:355:0.4);
					\end{tikzpicture}};
					
					\node[blue!60!white,draw,rectangle,rounded corners] at ($(child-2)+(270:1)$) {
						\begin{tikzpicture}[scale=0.7]
							\node (centre) at (0,0) {};
							
							\node[scale=2,thick] (cyc-1) at ($(centre)$) {};
							\draw[directededge,PastelOrange] ($(cyc-1)+(0:0.4)$) arc (0:355:0.4);
					\end{tikzpicture}};
					\node[blue!60!white,draw,rectangle,rounded corners] at ($(child-22)+(270:1)$) {
						\begin{tikzpicture}[scale=0.7]
							\node (centre) at (0,0) {};
							
							\node[scale=2,thick] (cyc-1) at ($(centre)$) {};
							\draw[directededge,LavenderMagenta] ($(cyc-1)+(0:0.4)$) arc (0:355:0.4);
					\end{tikzpicture}};
					
					\node[blue!60!white,draw,rectangle,rounded corners] at ($(child-3)+(270:1)$) {
						\begin{tikzpicture}[scale=0.7]
							\node (centre) at (0,0) {};
							
							\node[scale=2,thick] (cyc-1) at ($(centre)$) {};
							\draw[directededge,CornflowerBlue] ($(cyc-1)+(0:0.4)$) arc (0:355:0.4);
					\end{tikzpicture}};
			\end{tikzpicture}};
			
			\node[scale=2] at ($(left)!0.45!(right)$) {\(\Rightarrow\)};
		\end{tikzpicture}
		
		\caption{
			Illustration depicting local modification of the \potdecomp \(\TT_i\).
			Blue blocks below each subtree depict the partition of the subgraph decomposed by each subtree into strongly connected components.
		}
		\label{fig:potdecomp_rooting_step}
	\end{figure}
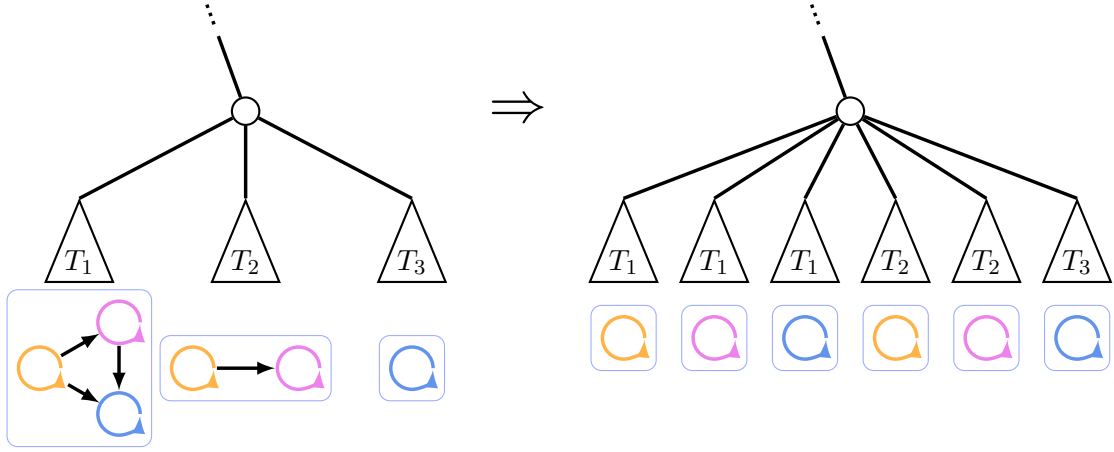
	
	We now verify the postulated properties of \(\TT_{i+1}\coloneqq (T_{i+1},\beta_{i+1},\gamma_{i+1})\).
	That properties 1--3 hold immediately from the construction and from the assumption that they hold for \(\TT_i\).
	Property~4 holds by construction for the nodes
	of height \(i + 1\), while for the nodes of height \(i\) or less, it follows by the assumption about \(\TT_i\)
	and the observation that the strongly connected components we split by at height \(i + 1\) cannot ever split
	a strongly connected component induced by the bags of a subtree deeper down.
	
	Finally, we need to verify that $\TT_{i+1}$ is indeed a rooted tree-partition of $D$.
	For~\cref{item:tp:bags}, this follows trivially from the construction and the assumption that
	\((T, \beta, \gamma)\) is a rooted tree-partition of~$D$.
	For~\cref{item:tp:edgeguards}, observe that every edge \(e\) of \(T_{i + 1}\) induces either a cut
	already present in \(T_i\) (with the same guard set), or a cut of the form
	\((A \cap C^x_j, B \cup (A \setminus C^x_j))\) where \((A,B)\) is the cut induced in $\TT_i$ by an edge \(e'\) incident to \(V(T_x)\), for some $x\in X$.
	But noting
	that \(C^x_j\) is a strongly connected component of a superset of \(A\), every cycle that contains
	a vertex in \(A \cap C^x_j\) and a vertex outside of \(A \cap C^x_j\) must necessarily contain
	a vertex in \(A\) and a vertex outside of \(A\).
	Hence, the cycle must be hit by the guard set \(\gamma_i(e') = \gamma_{i + 1}(e)\).
	This concludes the verification of the properties of $\TT_{i+1}$.
	
	As a last step, to ensure that we do not introduce too many empty nodes in $T_{i+1}$, if $x \in V(T_{i+1}) \setminus  V(T^{\text{rest}}_{i+1})$ is such that $\beta_{i+1}(x)=\emptyset$, we exhaustively do the following:
\begin{enumerate}[label=\textbf{(A\arabic*)},leftmargin=1.2cm]
	\item if $x$ is a leaf of $T_{i+1}$, we delete $x$;
	\item if $x$ has precisely two neighbours $y$ and $z$, we delete $x$, add edge $yz$ and set $\gamma_{i+1}(yz)=\gamma_{i+1}(xy)$;
\end{enumerate}	
	Note that these modifications preserve properties 1--3 and the fact that $\TT_{i+1}$ is a tree-partition rooted in~$r$.
	It is straightforward to verify that property 4 is also preserved, as since $x \neq r$, one of the neighbours of $x$ must be a parent of $x$.
	
    The construction finishes after at most $h(r)$ rounds, yielding a coherent tree-partition $\TT'$.
 	Moreover, constructing $\TT'$ takes time polynomial in $|E(D)|$ and $|V(T)|$; in particular, by (A1)-(A2), each of the nodes of $T'$ corresponding to empty bag is of degree at least 3, thus the size of $T'$ is polynomial in $|V(D)|$ (and thus in $|E(D)|$).
\end{proof}

Finally, as mentioned in \cref{sec:overview}, we make crucial use of tree-partitions whose bags are well flow-connected.
This property is formalised in the following definition.

\begin{definition}\label{def:linked}
	We say that a tree-partition $(T,\beta,\gamma)$ of a digraph $D$ is \emph{$k$-\linked} if for every node $x$ of $T$, the set $\beta(x)$ is $k$-flow-connected in $D$.
\end{definition}

\paragraph*{Directed tree-cutwidth.}
The notions of directed tree-cut decompositions and of directed tree-cutwidth arise by enriching tree-partitions with hitting sets for onions that \say{cross} the nodes.
However, the concept of a \say{crossing} onion has to be chosen carefully.

Let \(\Brace{T,\beta,\gamma}\) be a tree-partition of a digraph $D$.
Then, an \emph{out-onion} of $x \in \V{T}$ is an onion immersion model $Z$ in $D$ such that the nodes $y,z$ of $T$ whose bags contain the tip and the root of~$Z$, respectively, are such that $y\neq x$ and $y$ and $z$ do not lie in the same connected component of~$T-x$.
\emph{In-onions} of $x \in \V{T}$ are defined analogously, except we require $z\neq x$.
Note that if $y$ and $z$ belong to different components of $T-x$ then $Z$ is both an in- and an out-onion of $x$, and if $y$ and $z$ belong to the same component of $T-x$ then $Z$ is neither an in- nor an out-onion of~$x$.

\begin{definition}[\Dtcw]
  \label{def:dtcw}
	Given a digraph \(D\), a quadruple \(\Brace{T,\beta,\gamma,\theta}\) is a \emph{\dtcdecomp} of \(D\) if
	\begin{enumerate}[label=\textbf{(D\arabic*)},leftmargin=1.2cm]
		\item \(\Brace{T,\beta,\gamma}\) is a tree-partition of $D$, and \label{item:dtcw:tp}
		\item \(\theta \colon \V{T} \to 2^{A(D)}\) maps every node $x$ of \(T\) to a set of edges in \(\theta(x)\subseteq A(D)\), called the \emph{onion-guard} of \(x\), so that $\theta(x)$ hits all the out-onions of $x$ or all the in-onions of \(x\).
		\label{item:dtcw:onionguards}
	\end{enumerate}
	
	The \emph{width} of \(\Brace{T,\beta,\gamma,\theta}\) is defined as
	\begin{equation*}
		\max\left\{\max_{v \in \V{T}} \Abs{\beta(v)}, \max_{e \in \E{T}} \Abs{\gamma(e)}, \max_{v \in \V{T}} \Abs{\theta(v)}\right\},
	\end{equation*}
	while the \emph{\dtcw} of \(D\) is defined as the minimum width over all possible \dtcdecomps of \(D\).
\end{definition}

We remark that when working with tree-partitions or directed tree-cut decompositions, we will sometimes assume $T$ to be rooted.
Then these are called \emph{rooted tree-partitions} and \emph{rooted directed tree-cut decompositions}, respectively.


%

	
\paragraph*{Torsos.}
  \label{sec:torso}
  
	For a rooted tree-partition \(\Brace{T,\beta,\gamma}\) of a digraph $D$ and a node $x$ of $T$, we define the \emph{torso} \(\torso{x}\) as follows; see \cref{fig:torso} for an illustration.
	Let \(y_1,\dots,y_{\ell}\) be the children of \(x\) in \(T\).
	
	\begin{figure}[!ht]
		\begin{center}
			\begin{tikzpicture}
				\node (decomp) at (-3,0) {
					\begin{tikzpicture}
						\node (centre) at (0,0) {};
						\node[draw,circle,inner sep=0.5cm] (bag-v) at ($(centre)$) {\(\beta(x)\)};
						\node (parent) at ($(bag-v)+(90:2)$) {parent};
						
						\node[draw=gray,gray,circle,inner sep=0.3cm] (c-1) at ($(bag-v)+({360/16*9}:3)$) {\(c_1\)};
						\node[draw=gray,gray,circle,inner sep=0.3cm] (c-2) at ($(bag-v)+({360/16*11}:2.7)$) {\(c_2\)};
						\node[draw=gray,gray,circle,inner sep=0.3cm] (c-3) at ($(bag-v)+({360/16*13}:2.7)$) {\(c_3\)};
						\node[draw=gray,gray,circle,inner sep=0.3cm] (c-4) at ($(bag-v)+({360/16*15}:3)$) {\(c_{\ell}\)};
						
						\node[gray,rotate=43] at ($(c-3)!0.5!(c-4)$) {\(\dots\)};
						
						\node[vertex,scale=0.7,CornflowerBlue] (e-tail) at ($(bag-v)+(190:0.7)$) {};
						\node (e-head) at ($(c-1)+(135:0.3)$) {};
						\node[vertex,scale=0.7,LavenderMagenta] (g-head) at ($(bag-v)+(350:0.7)$) {};
						\node (g-tail) at ($(c-3)+(45:0.3)$) {};
						
						\node (f-tail) at ($(c-1)+(315:0.3)$) {};
						\node (f-head) at ($(c-2)+(170:0.3)$) {};
						
						\node (e-label) at ($(e-tail)!0.5!(e-head)+(100:0.6)$) {\textcolor{CornflowerBlue}{\(e\)}};
						\node (f-label) at ($(f-tail)!0.5!(f-head)+(225:0.6)$) {\textcolor{AppleGreen}{\(f\)}};
						\node (g-label) at ($(g-tail)!0.5!(g-head)+(10:0.3)$) {\textcolor{LavenderMagenta}{\(g\)}};
						
						\begin{pgfonlayer}{background}
							\draw[dashed] (parent) to (bag-v);
							
							\draw[gray] (bag-v) to (c-1);
							\draw[gray] (bag-v) to (c-2);
							\draw[gray] (bag-v) to (c-3);
							\draw[gray] (bag-v) to (c-4);
							
							\draw[gray] (c-1) to ($(c-1)+(190:1.7)$);
							\draw[gray] (c-1) to ($(c-1)+(260:1.7)$);
							\draw[gray,dashed,shorten <=0.3cm,shorten >=0.3cm] ($(c-1)+(190:1.7)$) to ($(c-1)+(260:1.7)$);
							
							\draw[gray] (c-2) to ($(c-2)+(220:1.7)$);
							\draw[gray] (c-2) to ($(c-2)+(280:1.7)$);
							\draw[gray,dashed,shorten <=0.3cm,shorten >=0.3cm] ($(c-2)+(220:1.7)$) to ($(c-2)+(280:1.7)$);
							
							\draw[gray] (c-3) to ($(c-3)+(260:1.7)$);
							\draw[gray] (c-3) to ($(c-3)+(310:1.7)$);
							\draw[gray,dashed,shorten <=0.3cm,shorten >=0.3cm] ($(c-3)+(260:1.7)$) to ($(c-3)+(310:1.7)$);
							
							\draw[gray] (c-4) to ($(c-4)+(350:1.7)$);
							\draw[gray] (c-4) to ($(c-4)+(280:1.7)$);
							\draw[gray,dashed,shorten <=0.3cm,shorten >=0.3cm] ($(c-4)+(350:1.7)$) to ($(c-4)+(280:1.7)$);

							\draw[directededge,CornflowerBlue,bend right] (e-tail) to (e-head.center);
							\draw[directededge,LavenderMagenta,out=80,in=270] (g-tail.center) to (g-head);
							\draw[directededge,AppleGreen,bend right] (f-tail.center) to (f-head.center);
						\end{pgfonlayer}	
				\end{tikzpicture}};
				
				\node (torso) at (4,-3) {
					\begin{tikzpicture}
						\node (centre) at (0,0) {};
						\node[draw,circle,inner sep=0.5cm] (bag-v-t) at ($(centre)$) {\(\beta(x)\)};
						
						\node (c-1-t) at ($(bag-v)+({360/16*9}:3)$) {};
						\node (c-2-t) at ($(bag-v)+({360/16*11}:2.7)$) {};
						\node (c-3-t) at ($(bag-v)+({360/16*13}:2.7)$) {};
						\node (c-4-t) at ($(bag-v)+({360/16*15}:3)$) {};
						
						\node[vertex,scale=0.7,PastelOrange] (a-1) at ($(c-1-t)+({360/16*9-90}:0.4)$) {};
						\node[vertex,scale=0.7,PastelOrange] (b-1) at ($(c-1-t)+({360/16*9+90}:0.4)$) {};
						
						\node[vertex,scale=0.7,PastelOrange] (a-2) at ($(c-2-t)+({360/16*11-90}:0.4)$) {};
						\node[vertex,scale=0.7,PastelOrange] (b-2) at ($(c-2-t)+({360/16*11+90}:0.4)$) {};
						
						\node[vertex,scale=0.7,PastelOrange] (a-3) at ($(c-3-t)+({360/16*13-90}:0.4)$) {};
						\node[vertex,scale=0.7,PastelOrange] (b-3) at ($(c-3-t)+({360/16*13+90}:0.4)$) {};
						
						\node[vertex,scale=0.7,PastelOrange] (a-4) at ($(c-4-t)+({360/16*15-90}:0.4)$) {};
						\node[vertex,scale=0.7,PastelOrange] (b-4) at ($(c-4-t)+({360/16*15+90}:0.4)$) {};
						
						\node[gray,rotate=43] at ($(c-3-t)!0.5!(c-4-t)$) {\(\dots\)};
						
						\node[vertex,scale=0.7,CornflowerBlue] (e-tail) at ($(bag-v-t)+(190:0.7)$) {};
						\node[vertex,scale=0.7,LavenderMagenta] (g-head) at ($(bag-v-t)+(350:0.7)$) {};
						
						\node (e-label) at ($(e-tail)!0.5!(a-1)+(100:0.6)$) {\textcolor{CornflowerBlue}{\(e'\)}};
						\node (f-label) at ($(b-1)!0.5!(a-2)+(225:0.6)$) {\textcolor{AppleGreen}{\(f'\)}};
						\node (g-label) at ($(b-3)!0.5!(g-head)+(10:0.35)$) {\textcolor{LavenderMagenta}{\(g'\)}};
						
						\begin{pgfonlayer}{background}
							\foreach \i in {1,2,3,4}{
								\draw[directededge,PastelOrange] (a-\i) to (b-\i);
							}					
							\draw[directededge,CornflowerBlue,bend right] (e-tail) to (a-1);
							\draw[directededge,LavenderMagenta,out=80,in=270] (b-3) to (g-head);
							\draw[directededge,AppleGreen,bend right] (b-1) to (a-2);
							
							\foreach[evaluate={\j=int(7+2*\i);}] \i in {1,2,3,4}{
								\draw (b-\i) edge[directededgedecorated,violet,quick curve through={($(b-\i)!0.5!(a-\i)+({360/16*\j}:0.5)$)}] (a-\i);
								\draw (b-\i) edge[directededgedecorated,violet,quick curve through={($(b-\i)!0.5!(a-\i)+({360/16*\j}:0.3)$)}] (a-\i);
							}
							
							\draw[directededge,gray!30!white,dashed] (bag-v-t) to (a-1);
							\draw[directededge,gray!30!white,dashed] (bag-v-t) to (a-2);
							\draw[directededge,gray!30!white,dashed] (bag-v-t) to (a-3);
							\draw[directededge,gray!30!white,dashed] (bag-v-t) to (a-4);
							
							\draw[directededge,gray!30!white,dashed] (b-1) to (bag-v-t);
							\draw[directededge,gray!30!white,dashed] (b-2) to (bag-v-t);
							\draw[directededge,gray!30!white,dashed] (b-3) to (bag-v-t);
							\draw[directededge,gray!30!white,dashed] (b-4) to (bag-v-t);
							
						\end{pgfonlayer}
						
						\node (a-1-label) at ($(a-1)+(160:0.3)$) {\textcolor{PastelOrange}{\(a_1\)}};
						\node (b-1-label) at ($(b-1)+(260:0.3)$) {\textcolor{PastelOrange}{\(b_1\)}};
						
						\node (a-2-label) at ($(a-2)+(220:0.4)$) {\textcolor{PastelOrange}{\(a_2\)}};
						\node (b-2-label) at ($(b-2)+(295:0.3)$) {\textcolor{PastelOrange}{\(b_2\)}};
						
						\node (a-3-label) at ($(a-3)+(250:0.3)$) {\textcolor{PastelOrange}{\(a_3\)}};
						\node (b-3-label) at ($(b-3)+(340:0.3)$) {\textcolor{PastelOrange}{\(b_3\)}};
						
						\node (a-4-label) at ($(a-4)+(300:0.3)$) {\textcolor{PastelOrange}{\(a_{\ell}\)}};
						\node (b-4-label) at ($(b-4)+(20:0.3)$) {\textcolor{PastelOrange}{\(b_{\ell}\)}};	
				\end{tikzpicture}};
			\end{tikzpicture}
		\end{center}
		\caption{Illustration of the construction of \(\torso{x}\) for a node \(x \in \V{T}\).
			The Chekhov-edges are depicted in violet.}
		\label{fig:torso}
	\end{figure}
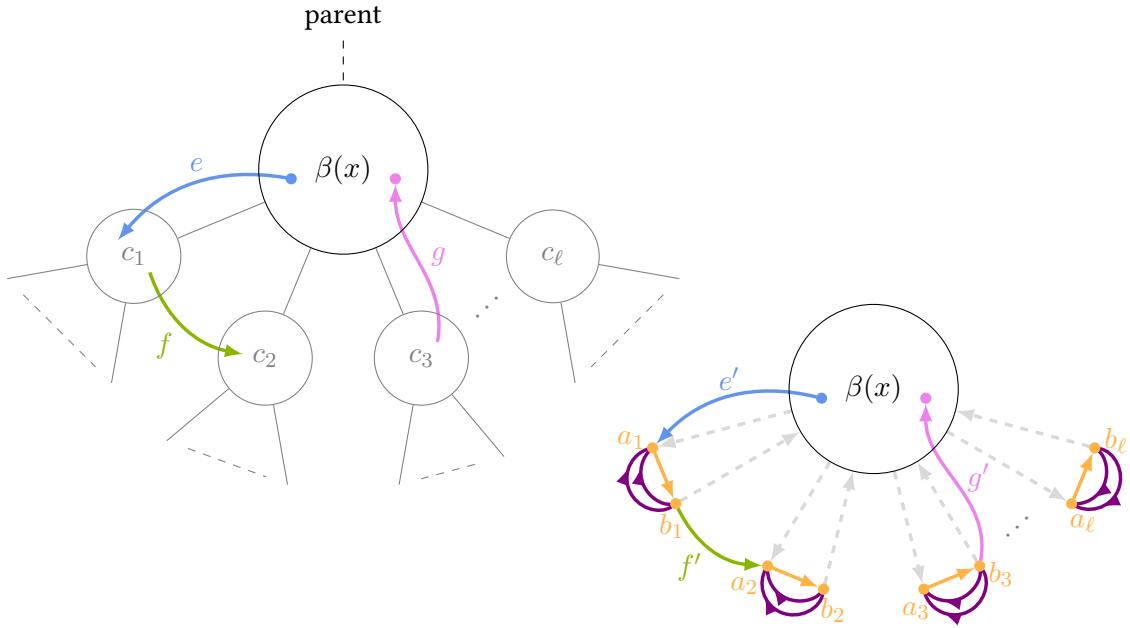
	
	The vertex set of \(\torso{x}\) is defined as \(\beta(x) \cup \bigcup_{i\in [\ell]} \Set{a_i,b_i}\).
	That is, for every child $y_i$ of \(x\) we add two new \emph{auxiliary vertices} $a_i$ and $b_i$.
	Denote $V_i=\beta(\subtree{T}{y_i})$ for brevity.

	Now, we let \(E^\star \subseteq E(D)\) be the set of all the edges $e\in E(D)$ such that the head and the tail of $e$ either both belong to $\beta(x)$, or belong to different sets among $\{\beta(x)\}\cup \{V_i\colon i\in [\ell]\}$.
	For every edge \(e \in E^\star\), we introduce a new edge \(e'\) to $\torso{x}$:
	\begin{itemize}
		\item If \(\tail{e},\head{e} \in \beta(x)\), then \(e'\coloneqq e\).
		\item If \(\tail{e} \in \beta(x)\) and \(\head{e} \in V_j\), then \(e' \coloneqq (\tail{e},a_j)\).
		\item If \(\tail{e} \in V_i\) and \(\head{e} \in \beta(x)\), then \(e' \coloneqq (b_i,\head{e})\).
		\item If \(\tail{e} \in V_i\) and \(\head{e} \in V_j\) with \(i\neq j\), then \(e' \coloneqq (b_i,a_j)\).
	\end{itemize}

		We denote by $\domain{\torsoEdgeMap{x}}$ the set of torso-edges we have constructed so far, and define a partial map \(\torsoEdgeMap{x}\colon \E{\torso{x}} \to E^\star \) mapping every \(e\in \domain{\torsoEdgeMap{x}}\) to its corresponding edge \(e'\) in \(D\) in order to formally keep track of the relation between the torso-edges and the original edges.
	Then (and this is why \(\torsoEdgeMap{x}\) is partial), for every $i\in [\ell]$ we add to $\torso{x}$ the edge $(a_i,b_i)$ and two edges $(b_i,a_i)$.
	The two edges $(b_i,a_i)$ will be called the \emph{Chekhov-edges}\footnote{\url{https://en.wikipedia.org/wiki/Chekhov's_gun}} of $z_i$.
	Thus, the multiset of edges of \(\torso{x}\) consists of \(\torsoEdgeMap{x}^{-1}(E^\star)\), the edges $(a_i,b_i)$ for $i\in [\ell]$, and all the Chekhov-edges.

	Note that for every $i\in [\ell]$, the three edges introduced between \(a_i\) and \(b_i\) yield an onion, which we refer to as \(O_i\).
	Similarly to calling \(a_i,b_i\) auxiliary vertices, we also call the edges  \(O_i\) the \emph{auxiliary edges}.

We next verify the following: an immersion model of a large cylindrical wall found in a torso can be conveniently lifted to the original digraph.

\begin{lemma}
	\label{lem:grid_imm_in_torsos}
	Let \(\Brace{T,\beta,\gamma}\) be a rooted tree-partition of a digraph \(D\), and let \(x \in \V{T}\).
Suppose that for every child $y$ of $x$ the set $\bigcup \Set{\beta(z) \ \colon z \text{ is a descendant of } y}\) induces a strongly connected subgraph of \(D\).
	If \(\torso{x}\) contains \(\gridImm{2k}\) as an immersion, then \(D\) contains \(\gridImm{k}\) as an immersion.
	
	Moreover, having an immersion model of \(\gridImm{2k}\) in \(\torso{x}\), we can find an immersion model of \(\gridImm{k}\) in~$D$ in time polynomial in size of $D$.
\end{lemma}
\begin{proof}
	Let $\mu'$ be an immersion model of \(\gridImm{2k}\) in \(\torso{x}\).
	Let $C'_1=C_1$, and for every $i \in \{2,\ldots,k\}$ we define the simple cycle $C'_i$ in \(\gridImm{2k}\) as
	\[p^+_{1,2i-2} \xrightarrow{P_1} p^+_{1,2i-1}\xrightarrow{C_{2i-1}}q^+_{1,2i-1}\xrightarrow{Q_1}q^+_{1,2i-2}\xrightarrow{C_{2i-2}} p^+_{2,2i-2} \xrightarrow{P_2} p^+_{2,2i-1} \ldots q^+_{2k,2i-2} \xrightarrow{C_{2i-2}} p^+_{1,{2i-2}}.\]
	
	Consider a subgraph $W$ of \(\gridImm{2k}\) obtained by taking the union of all edge and vertex sets of the elements of $\bigcup\{P_{2i-1},Q_{2i-1},C'_{i} \ \colon \ i \in [k]\}.$ 
	Define 
	\begin{align*}
	&V^-:=\{p^-_{2i-1,2j-1},q^-_{2i-1,2j} \ \colon \ i \in [k],j \in [k-1]\} \\ 
	&V^+:=\{p^+_{2i-1,2j},q^+_{2i-1,2j-1} \ \colon \ i \in [k],j \in [k-1]\}.
	\end{align*}
	Clearly, $W$ contains \(\gridImm{k}\) as an immersion, which in turn gives rise to an immersion model $\mu$ of \(\gridImm{k}\) in $\torso{x}$ where the image of the vertex set of \(\gridImm{k}\) is $\mu'\left(V^+ \cup V^-\right).$ We treat $V^+ \cup V^-$ also as the set of vertices of \(\gridImm{k}\).
	
	Note that for every $v \in V^+$, no out-neighbour $w$ of $v$ in \(\gridImm{2k}\) belongs to $V^-$, as for every $(i,j) \in [k] \times [k-1]$ the out-neighbour of $p^+_{2i-1,2j}$ is $p^-_{2i-1,2j}$, and the out-neighbour of $q^+_{2i-1,2j-1}$ is $p^-_{2i,2j-1}.$
	Thus, if $w$ is an out-neighbour of $v$ in \(\gridImm{k}\), the trail $\mu(v,w)$ is obtained by concatenating more than one trail given by $\mu'$, implying that $\mu(v,w)$ has an internal vertex.

	If $x$ is a leaf of $T$, $\mu$ is a model of \(\gridImm{k}\) in $D$.
	Thus, we can assume that $y_1,\ldots,y_\ell$ are the children of $x$.
	Fix $i \in [\ell]$.
	Since the out-degree of the auxiliary vertex $a_i$ and the in-degree of the auxiliary vertex $b_i$ are equal to 1, if there exists a vertex $v$ in $\gridImm{k}$ such that $\mu(v)=a_i$ (resp.~$\mu(v)=b_i$), then $v \in V^+$ (resp.~$v \in V^-$).
	Thus, there is no edge $(v,w)$ of $\gridImm{k}$ such that $\mu(v)=a_i$ and $\mu(w)=b_i$, as in that case $\mu(v,w)=(a_ib_i)$ would not have an internal vertex.
	Having established that property, we can modify $\mu$ to be an immersion model of \(\gridImm{k}\) in $D$ in a rather straightforward way.
	
	We assume first that there is an edge $(v,w)$ in $\gridImm{k}$ such that $\mu(v,w)$ uses the edge $a_ib_i$ but $\mu(v)\neq a_i$ and $\mu(w)\neq b_i$.
	Let $e$ and $f$ be the edges of \(\torso{x}\) that, respectively
	 directly precede and succeed $a_ib_i$ in $\mu(v,w)$.
	Since $\bigcup\{\beta(z) \ \colon \ z \textrm{ is a descendant of } y_i\}$ induces a strongly connected subgraph $D'$ of $D$, $D'$ contains a $\head{\torsoEdgeMap{x}(e)}$-$\tail{\torsoEdgeMap{x}(f)}$ path $(r_0,\ldots,r_m)$.
	We replace the subsequence $(e,a_ib_i,f)$ of $\mu(x,y)$ by $(\torsoEdgeMap{x}(e),r_0,\ldots,r_m,\torsoEdgeMap{x}(f))$.
	
	Thus, we consider the case in which $\mu(v,w)$ uses the edge $a_ib_i$ and $\mu(v)= a_i$ (the case $\mu(w)= b_i$ is symmetric).
	As we observed, in this case $v \in V^+$ and $\mu(w) \neq b_i$.
	Let $u_1,u_2 \in V^-$ be such that $(u_1,v),(u_2,v)$ are the edges of $\gridImm{k}$.	Let $e_1,e_2 \in E(\torso{x})$ be, respectively, last elements of $\mu(u_1,v)$ and $\mu(u_2,v)$, and let $f$ be the second (following $a_ib_i$) element of $\mu(v,w)$.
	Again, if $D'$ is a strongly connected subgraph $D$ induced by the sum of bags in the subtree rooted in $y_i$, $D'$ must contain an $\head{\torsoEdgeMap{x}(e_1)}$-$\tail{\torsoEdgeMap{x}(f)}$ path $R$, and a $\head{\torsoEdgeMap{x}(e_2)}$-$r$ path $S$ for some vertex $r$ that lays on $R$ (it may happen that $r=\head{\torsoEdgeMap{x}(e_2)}$ and $S$ is trivial).
	We set $\mu(v)$ to be $r$, replace $e_1$ and $e_2$ in the trails by, respectively, $\torsoEdgeMap{x}(e_1)$ followed by a prefix of $R$ that ends in $r$, and by $\torsoEdgeMap{x}(e_2)$ followed by $S$.
	Similarly, we replace $(a_ib_i,f)$ in $\mu(v,w)$ by the suffix of $R$ that starts in $r$, followed by $\torsoEdgeMap{x}(f)$.
	
	Since for different $i,i' \in [\ell]$ the bags of descendants of $i$ and of $i'$ are disjoint, and $\torsoEdgeMap{x}$ is a bijection, we never use the same edge of $D$ twice in the replacing procedure.
	Thus, after replacing in $\mu$ every vertex that is initially mapped to an auxiliary vertex, and every edge of $\torso{x}$ that does not belong to $D$, we get an immersion model $\mu$ of $\gridImm{k}$ in $D$.
	
	Last, we note that all these modifications can be performed in time polynomial in the size of $D$.
\end{proof}

\paragraph*{Directed treewidth and related notions.} As we discussed in \cref{sec:intro}, our reasoning ultimately relies on the duality between directed treewidth and topological minors of cylindrical walls, provided by the Directed Grid Theorem of Kawarabayashi and Kreutzer~\cite{kawarabayashi2022directed}.
Let us first recall the definition of directed tree decompositions and directed treewidth.
We use the definition proposed by Kawarabayashi and Kreutzer~\cite{kawarabayashi2022directed}.

\begin{definition}[{directed \treewidth~\cite{kawarabayashi2022directed}}]
	\label{def:dtw}
	A \emph{directed tree-decomposition} of a digraph \(D\) is a triple \(\Brace{T, \beta, \gamma}\), where \(T\) is a rooted tree, and \(\beta\colon V(T) \to 2^{V(D)}\) and \(\gamma\colon E(T) \to 2^{V(D)}\) are functions mapping every node $x$ of \(T\) to its \emph{bag} $\beta(x)\subseteq \V{D}$, and every edge $e$ of \(T\) to its \emph{guard} \(\gamma(e)\subseteq D\), so that the following conditions are satisfied.
	\begin{enumerate}[label=\textbf{(T\arabic*)},leftmargin=1.2cm]
		\item \label{item:directed-tree-width:beta}
		\(\Set{\beta(x) \colon x \in V(T)}\) is a near-partition of \(V(D)\); and
		\item \label{item:directed-tree-width:guard}
		for every \(e \in \E{T}\), there is no cycle in \(D - \gamma(e)\) visiting both \(\beta(\treeLower{T}{e})\) and~\(\beta(\treeUpper{T}{e})\).
	\end{enumerate}
	For \(x \in V(T)\) we define \(\Gamma(x) \coloneqq \beta(x) \cup \bigcup \set{\gamma(e) : e \sim x}\).
	The \emph{width} of \(\Brace{T, \beta, \gamma}\) is the least integer \(w\) such that \(\Abs{\Gamma(t)} \leq w + 1\) for all \(t \in V(T)\).
	The \emph{directed \treewidth} of \(D\) is the smallest width of any directed tree-decomposition of \(D\).
\end{definition}

With this definition in place, the Directed Grid Theorem reads as follows.

\begin{theorem}[{Directed Grid Theorem~\cite{kawarabayashi2015directed,kawarabayashi2022directed}}]
	\label{thm:gridtheorem}
	There is a function \(\thmFkt{thm:gridtheorem}{g}{}\colon \N\to \N\) such that for every digraph $D$, one of the following conditions holds:
	\begin{itemize}
		\item \(D\) contains \(\gridImm{k}\) as a topological minor; or
		\item \(D\) admits a directed tree-decomposition of width at most \(\thmFkt{thm:gridtheorem}{g}{k}\).
	\end{itemize}
\end{theorem}

We remark that the original formulation of Kawarabayashi and Kreutzer speaks about containing a cylindrical grid of order \(k\) as a \emph{butterfly minor}, which is a slightly different containment notion for digraphs.
However, it is standard and easy to see that the largest size of a cylindrical grid contained as a butterfly minor, and the largest size of a cylindrical wall contained as a topological minor, differ by a multiplicative factor at most $2$.
Therefore, the original formulation of Kawarabayashi and Kreutzer is equivalent to the one above.

We remark that the original proof of Kawarabayashi and Kreutzer yields a non-elementary function~$\thmFkt{thm:gridtheorem}{g}{}$.
This was recently improved to an elementary bound by Hatzel, Kreutzer, Milani, and Muzi~\cite{COSS2024FOCS}.

As shown by Campos, Lopes, Maia, and Sau~\cite{DBLP:journals/siamdm/CamposLMS22}, the proof of Directed Grid Theorem can be modified to yield the following fixed-parameter algorithm.
\begin{theorem}[\cite{DBLP:journals/siamdm/CamposLMS22}]\label{thm:grid-algorithmic}
There is a fixed-parameter algorithm to find, for any digraph \(D\) and parameter~$k$, either
\begin{itemize}
\item a topological minor model of \(\gridImm{k}\) in $D$; or
\item a directed tree-cut decomposition of $D$ of width at most \(\thmFkt{thm:gridtheorem}{g}{k}\).
\end{itemize}
\end{theorem}

In fact, in this paper we find it convenient not to work with the notion of directed treewidth, as defined in \cref{def:dtw}, but with a functionally equivalent parameter called \emph{directed branchwidth}~\cite{thesisWiederrecht,wiederrecht2019note}.

\begin{definition}[Directed Branchwidth~\cite{wiederrecht2019note,thesisWiederrecht}]
	\label{def:dbw}
	A \emph{directed branch-decomposition} of a digraph \(D\) is a tuple \((T,\lambda)\) where \(T\) is a subcubic tree and \(\lambda\) is a bijection between the vertices of \(D\) and the leaves of \(T\).
	
	Now define \(\gamma\colon E(T)\to 2^{V(D)}\) mapping every edge \(e=xy\) of \(T\) to a minimum vertex set \(\gamma(e) \subseteq \V{D}\), which we call the \emph{vertex-guard}, so that the following condition is satisfied: \(D-\gamma(e)\) does not contain any cycle containing both a vertex of \(\lambda(\treeSplit{T}{x}{e})\) and a vertex of \(\lambda(\treeSplit{T}{y}{e})\).
	
	The \emph{width} of \((T,\lambda)\) is the maximum size of \(\gamma(e)\) over all edges \(e \in \E{T}\).
	
	The \emph{directed branchwidth} of \(D\), denoted \(\HSwidth{D}\), is the minimum width over all valid directed branch-decompositions of \(D\).
\end{definition}

%


This concept is inspired by the parameter \emph{cyclewidth}, which was proposed by Hatzel, Rabinovich, and Wiederrecht~\cite{HatzelRW19}.
The original definition of \cyclewidth proposed differs by measuring the contribution of an edge \(e = xy \in E(T)\) to the width differently: instead of considering the minimum size of a vertex hitting sets for cycles visiting both \(\lambda(\treeSplit{T}{x}{e})\) and \(\lambda(\treeSplit{T}{y}{e})\), they considered vertex-disjoint packings of such cycles, where the contribution of every cycle is equal to the number of times it crosses the cut \(\Brace{\lambda(\treeSplit{T}{x}{e}),\lambda(\treeSplit{T}{y}{e})}\).
As proved in~\cite{HatzelRW19}, cyclewidth defined in this way is functionally equivalent to directed treewidth, in the sense that each of these parameters is always bounded by a function of the other.
See also~\cite{improvedCyclewidthBound,linearCyclewidthBound} for improved bounds.

Wiederrecht~\cite{wiederrecht2019note} showed that directed branchwidth is functionally equivalent to directed treewidth as well, even with a linear function relating the two parameters.

\begin{restatable}[{\cite[Theorem 3.5]{wiederrecht2019note}}]{theorem}{dtwHSw}
	\label{thm:cyw_dtw}
	Let \(\thmFkt{thm:cyw_dtw}{g}{k} \coloneqq 9k + 1\) and \(\thmFkt{thm:cyw_dtw}{h}{k} \coloneqq k+1\). Then for every digraph \(D\), we have
	\begin{equation*}
		\HSwidth{D} \leq \thmFkt{thm:cyw_dtw}{h}{\dtw{D}}\qquad\text{and}\qquad
		\dtw{D} \leq \thmFkt{thm:cyw_dtw}{g}{\HSwidth{D}}\text{.} 
	\end{equation*}
	Moreover, given a directed tree-decomposition of $D$ of width $k$, a directed branch-decomposition of $D$ of width at most $\thmFkt{thm:cyw_dtw}{h}{k}$ can be computed in polynomial time.
\end{restatable}

We remark that the algorithmic claim of \cref{thm:cyw_dtw} is not stated explicitly in \cite{wiederrecht2019note}, but follows easily from the construction in the proof.

\section{Large cylindrical wall immersions force large tree-cutwidth}

In this section we give a proof of \cref{thm:imm_implies_large_dtcw}, providing the easy direction of the duality between directed tree-cutwidth and immersions of directed cylindrical walls. 

First, let us establish an expected fact: directed tree-cutwidth is monotone under taking immersions.

\begin{restatable}{theorem}{ThmMonotone}
\label{thm:immersion-monotone}
Suppose \(D\) and \(D'\) are digraphs such that \(D\) contains \(D'\) as an immersion.
Then, the \dtcw of \(D'\) is at most the \dtcw of \(D\).
\end{restatable}
\begin{proof}
	Let \((T, \beta, \gamma, \theta)\) be a \dtcdecomp of \(D\) of optimum width, and
	let \(\mu\) be an immersion model of \(D'\) in \(D\).
	Let us define a new bag function \(\beta' \colon V(T) \rightarrow 2^{V(D')}\) by setting \[\beta'(x) \coloneqq \mu^{-1}(\beta(x)),\qquad\text{for every node }x\text{ of }T.\]
	Further, for a set of edges \(F \subseteq E(D)\) we define
	\[\zeta(F) \coloneqq \set{e \in E(D') \mid \mu(e) \cap F \neq \emptyset}.\]
	Observe that for any \(F \subseteq E(D)\) we have \(\Abs{\zeta(F)} \leq \Abs{F}\), as the trails used by \(\mu\) are
	pairwise edge-disjoint by definition.
	So, we finally define \(\gamma' \colon E(T) \rightarrow 2^{E(D')}\),
	and \(\theta' \colon V(T) \rightarrow 2^{E(D')}\) by setting, for each $e\in E(T)$ and $x\in V(T)$
	\begin{equation*}
		\gamma'(e) \coloneqq \zeta(\gamma(e))\qquad\text{and}\qquad
		\theta'(x) \coloneqq \zeta(\theta(x))\text{.}
	\end{equation*}
	We now argue that \((T, \beta', \gamma', \theta')\) is a \dtcdecomp of \(D'\). First, to argue \cref{item:dtcw:tp} we have to check that $(T, \beta', \gamma')$ is a tree-partition of $D'$.
	Clearly, \(\beta'\) is a near-partition of \(V(D')\), so \cref{item:tp:bags} holds.
	To argue \cref{item:tp:edgeguards}, consider any edge \(e = xy\) of \(T\).
	Let \(A \coloneqq \beta(\treeSplit{T}{x}{e})\), \(B \coloneqq \beta(\treeSplit{T}{y}{e})\),
	\(A' \coloneqq \beta'(\treeSplit{T}{x}{e})\), and \(B' \coloneqq \beta'(\treeSplit{T}{y}{e})\). Further,
	let \(C\) be a cycle in \(D'\) crossing the cut \((A',B')\),
	and let \(F \coloneqq \gamma(e)\) and \(F' \coloneqq \gamma'(e)\).
	Observe that \(\mu(E(C))\) is the edge set of a cycle (not necessarily simple) in \(D\)
	that crosses the cut $(A,B)$.
	Therefore, $C$ must contain an edge \(f \in F\).
	This edge must be contained in some \(\mu(f')\) where \(f' \in E(C)\),
	and therefore we have $f'\in F'\cap E(C)$; so $C$ is hit by $F'$. This establishes \cref{item:tp:edgeguards} and thereby also \cref{item:dtcw:tp}.
	The proof of \cref{item:dtcw:onionguards} follows similar lines as that of \cref{item:tp:edgeguards}; we leave it to the reader.
	
	From the construction it follows immediately that the width of \((T, \beta', \gamma', \theta')\)
	is at most that of \((T, \beta, \gamma, \theta)\). As the width of \((T, \beta, \gamma, \theta)\) is equal to the directed tree-cutwidth of $D$, this proves that the directed tree-cutwidth of $D'$ is bounded by the directed tree-cutwidth of $D$.
\end{proof}

Next, we check that cylindrical walls have large \dtcw.

\begin{lemma}\label{lemma:wall-dtcw}
	There exists a function \(\thmFkt{lemma:wall-dtcw}{f}{} \colon \N \rightarrow \N\),
	such that for any positive integer \(k\), the \dtcw
	of \(\gridImm{\thmFkt{lemma:wall-dtcw}{f}{k}}\) is at least \(k\).
\end{lemma}

To prove \cref{lemma:wall-dtcw} we need the following technical lemma describing roots and tips of onion immersion models in a cylindrical wall.

\begin{lemma} 
	\label{lem:another-onion-in-the-wall}
	Let $k \geq 2$ be an integer, and consider the cylindrical wall $\gridImm{k}$.
	We define $B_k$ to be the bipartite graph, with bipartition classes $P^- \cup Q^-$ and $P^+\cup Q^+$, so that vertices $u \in P^- \cup Q^-$ and $v \in P^+\cup Q^+$ are adjacent if and only if there exists an onion immersion model in $\gridImm{k}$ that maps the root to $u$ and the tip to $v$.
	Let $i \in [k], j\in [k-1]$.
	Then the neighbourhood of $q^+_{i,j}$ in $B_k$ contains
	\begin{itemize}
		\item $p^-_{a,b}$ for all $a \in [k] \setminus \{i\}$ and $b \in \{1,\ldots,j\}$, and
		\item $q^-_{a,b}$ for all $a \in [k] \setminus \{i\}$ and $b \in \{j+1,\ldots,k\}$.
	\end{itemize}
Similarly, for $i \in [k]$ and $j \in \{2,\ldots,k\}$, the neighbourhood of $p^+_{i,j}$ in $B_k$ contains
\begin{itemize}
	\item $p^-_{a,b}$ for all $a \in [k] \setminus \{i\}$ and $b \in \{1,\ldots,j-1\}$, and
	\item $q^-_{a,b}$ for all $a \in [k] \setminus \{i-1\}$ and $b \in \{j,\ldots,k\}$,
\end{itemize}
where we follow the convention that if $i=1$, then $i-1=k$.
	
\end{lemma}
\begin{proof}
	The four different cases of the construction are depicted in \cref{fig:another-onion-in-the-wall}.
	First, we analyse the neighbourhood of $q^+_{i,j}$ in $B_k$, where $j\leq k-1$. 
	To build an onion immersion model with tip $q^+_{i,j}$ and root $p^-_{a,b}$ with $a\neq i$ and $b\leq j$, we define paths 
	
	\begin{tabular}{ccccccccc}
		$W_1=$&$q^+_{i,j}$&$\xrightarrow{Q_i}$&$q^+_{i,b}$&$\xrightarrow{C_b}$&$p^-_{a,b},$&& & \\
		$W_2=$&$p^-_{a,b}$&$\xrightarrow{P_a}$&$p^+_{a,j+1}$&$\xrightarrow{C_{j+1}}$&$q^-_{i,j+1}$&$\xrightarrow{Q_i}$& $q^+_{i,j},$&\textrm{and} \\
		$W_3=$&$p^-_{a,b}$&$\xrightarrow{C_b}$&$p^-_{i,b}$&$\xrightarrow{P_i}$&$p^-_{i,j}$&$\xrightarrow{C_j}$&$q^+_{i,j}$.& 
	\end{tabular}
	
	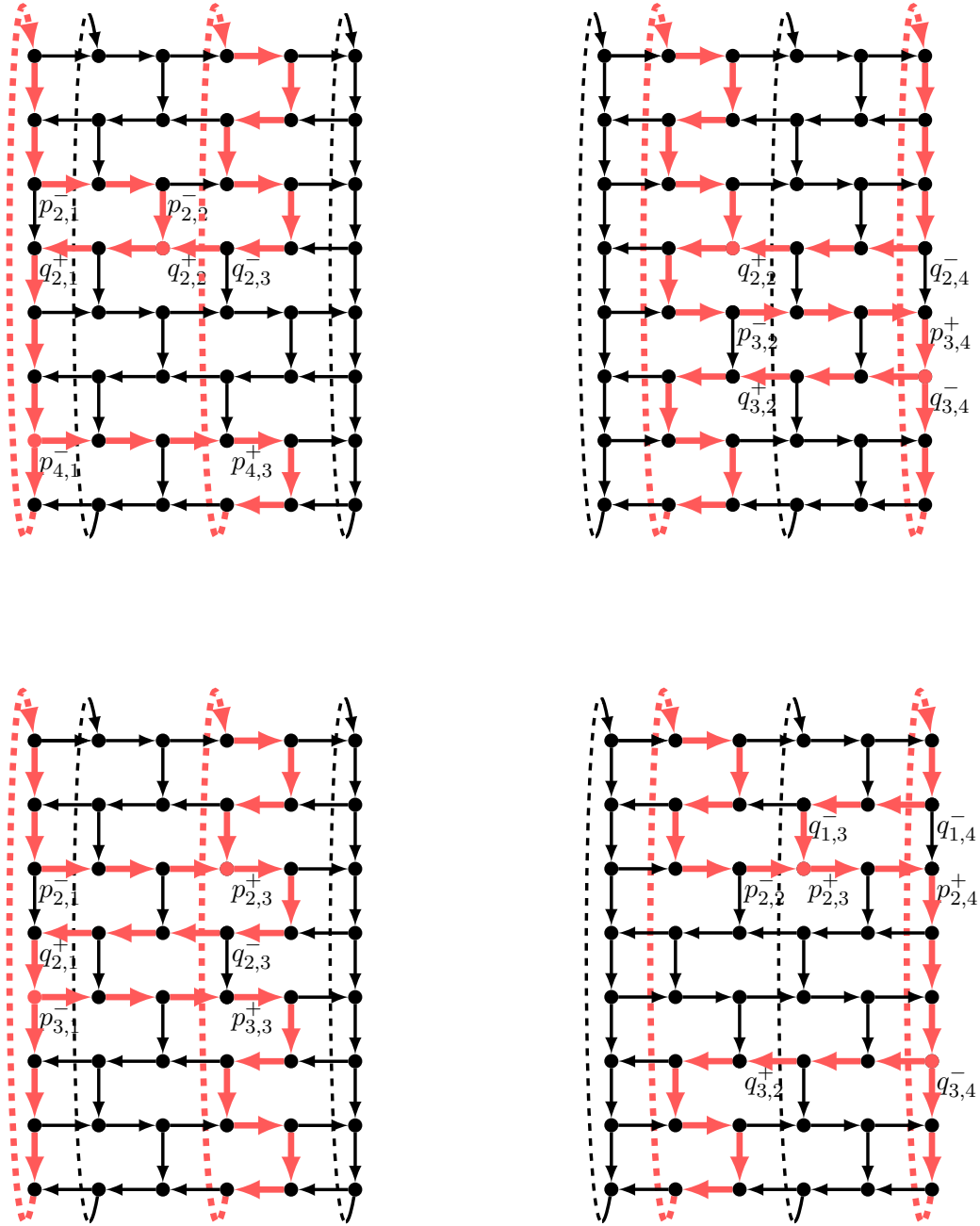
\begin{figure}
		\begin{tikzpicture}[scale=0.6, label distance=-2mm]
			\def\dist{1.5}
			\node (centre) at (0,0) {};
			\node[vertex] (p-1-1-m) at ($(0,0)$) {};
			\node[vertex,label=315:{\(p^{-}_{2,1}\)}] (p-2-1-m) at ($(p-1-1-m)+(270:{2*\dist})$) {};
			\node[vertex] (p-3-1-m) at ($(p-2-1-m)+(270:{2*\dist})$) {};
			\node[vertex,color=red,label=315:{\(p^{-}_{4,1}\)}] (p-4-1-m) at ($(p-3-1-m)+(270:{2*\dist})$) {};
			
			\node[vertex] (q-1-1-p) at ($(p-1-1-m)+(270:{\dist})$) {};
			\node[vertex,label=315:{\(q^{+}_{2,1}\)}] (q-2-1-p) at ($(q-1-1-p)+(270:{2*\dist})$) {};
			\node[vertex] (q-3-1-p) at ($(q-2-1-p)+(270:{2*\dist})$) {};
			\node[vertex] (q-4-1-p) at ($(q-3-1-p)+(270:{2*\dist})$) {};
			
			\foreach \i in {1,2,3,4}
			{
				\node[vertex] (p-\i-2-p) at ($(p-\i-1-m)+(0:{\dist})$) {};
				\node[vertex] (p-\i-2-m) at ($(p-\i-2-p)+(0:{\dist})$) {};
				\node[vertex] (p-\i-3-p) at ($(p-\i-2-m)+(0:{\dist})$) {};
				\node[vertex] (p-\i-3-m) at ($(p-\i-3-p)+(0:{\dist})$) {};
				\node[vertex] (p-\i-4-p) at ($(p-\i-3-m)+(0:{\dist})$) {};
			}
			\node[vertex,label=315:{\(p^{-}_{2,2}\)}] (p-2-2-m) at ($(p-2-2-p)+(0:{\dist})$) {};
				\node[vertex,label=315:{\(p^{+}_{4,3}\)}] (p-4-3-p) at ($(p-4-2-m)+(0:{\dist})$) {};
			\foreach \i in {1,2,3,4}
			{
				\node[vertex] (q-\i-2-m) at ($(q-\i-1-p)+(0:{\dist})$) {};
				\node[vertex] (q-\i-2-p) at ($(q-\i-2-m)+(0:{\dist})$) {};
				\node[vertex] (q-\i-3-m) at ($(q-\i-2-p)+(0:{\dist})$) {};
				\node[vertex] (q-\i-3-p) at ($(q-\i-3-m)+(0:{\dist})$) {};
				\node[vertex] (q-\i-4-m) at ($(q-\i-3-p)+(0:{\dist})$) {};
			}
			
			\node[vertex,color=red,label=315:{\(q^{+}_{2,2}\)}] (q-2-2-p) at ($(q-2-2-m)+(0:{\dist})$) {};
			\node[vertex,label=315:{\(q^{-}_{2,3}\)}] (q-2-3-m) at ($(q-2-2-p)+(0:{\dist})$) {};
			\foreach \x/\y in {q-4-2-m/p-1-2-p,q-4-4-m/p-1-4-p}{
				\begin{scope}
					\clip ($(\x)+(-0.2,-1)$) -- ($(\y)+(-0.2,1)$) --($(\y)+(5,1)$) -- ($(\x)+(5,-1)$) -- cycle;
					\draw[directededge,out=260,in=100] (\x) to (\y);
				\end{scope}
				\begin{scope}
					\clip ($(\x)+(-0.2,-1)$) -- ($(\y)+(-0.2,1)$) --($(\y)+(-5,1)$) -- ($(\x)+(-5,-1)$) -- cycle;
					\draw[directededge,out=260,in=100,dashed] (\x) to (\y);
				\end{scope}
			}
			\draw[directededge,line width=2.5,color=red,out=260,in=100,dashed] (q-4-1-p) to (p-1-1-m);
			\draw[directededge,line width=2.5,color=red,out=260,in=100,dashed] (q-4-3-m) to (p-1-3-p);
			\foreach \i in {1,2,3,4}{

					
					\draw[directededge] (p-\i-4-p) to (q-\i-4-m);
					\draw[directededge] (p-\i-3-m) to (p-\i-4-p);
					\draw[directededge] (q-\i-4-m) to (q-\i-3-p);
			}
			
			\draw[directededge,line width=2.5,color=red] (p-2-1-m) to (p-2-2-p);
			\draw[directededge,line width=2.5,color=red] (p-2-2-p) to (p-2-2-m);
			\draw[directededge,line width=2.5,color=red] (q-2-2-m) to (q-2-1-p);
			\draw[directededge,line width=2.5,color=red] (q-2-2-p) to (q-2-2-m);
			\draw[directededge,line width=2.5,color=red] (p-1-1-m) to (q-1-1-p);
			\draw[directededge,line width=2.5,color=red] (p-3-1-m) to (q-3-1-p);
			\draw[directededge,line width=2.5,color=red] (p-4-1-m) to (q-4-1-p);
			\draw[directededge] (p-2-1-m) to (q-2-1-p);
			
			\draw[directededge] (p-1-1-m) to (p-1-2-p);
			\draw[directededge] (p-1-2-p) to (p-1-2-m);
			\draw[directededge] (p-1-2-m) to (p-1-3-p);
			\draw[directededge,line width=2.5,color=red] (p-1-3-p) to (p-1-3-m);
			
			\draw[directededge] (q-1-2-m) to (q-1-1-p);
			\draw[directededge] (q-1-2-p) to (q-1-2-m);
			\draw[directededge] (q-1-3-m) to (q-1-2-p);
			\draw[directededge,line width=2.5,color=red] (q-1-3-p) to (q-1-3-m);
			
			\draw[directededge] (p-2-2-m) to (p-2-3-p);
			\draw[directededge,line width=2.5,color=red] (p-2-3-p) to (p-2-3-m);
			
			\draw[directededge,line width=2.5,color=red] (q-2-3-m) to (q-2-2-p);
			\draw[directededge,line width=2.5,color=red] (q-2-3-p) to (q-2-3-m);
			
			\draw[directededge] (p-3-1-m) to (p-3-2-p);
			\draw[directededge] (p-3-2-p) to (p-3-2-m);
			\draw[directededge] (p-3-2-m) to (p-3-3-p);
			\draw[directededge] (p-3-3-p) to (p-3-3-m);
			
			\draw[directededge] (q-3-2-m) to (q-3-1-p);
			\draw[directededge] (q-3-2-p) to (q-3-2-m);
			\draw[directededge] (q-3-3-m) to (q-3-2-p);
			\draw[directededge] (q-3-3-p) to (q-3-3-m);
			
			\draw[directededge,line width=2.5,color=red] (p-4-1-m) to (p-4-2-p);
			\draw[directededge,line width=2.5,color=red] (p-4-2-p) to (p-4-2-m);
			\draw[directededge,line width=2.5,color=red] (p-4-2-m) to (p-4-3-p);
			\draw[directededge,line width=2.5,color=red] (p-4-3-p) to (p-4-3-m);
			
			\draw[directededge] (q-4-2-m) to (q-4-1-p);
			\draw[directededge] (q-4-2-p) to (q-4-2-m);
			\draw[directededge] (q-4-3-m) to (q-4-2-p);
			\draw[directededge,line width=2.5,color=red] (q-4-3-p) to (q-4-3-m);
			
			\foreach[evaluate={\j=int(\i+1);}] \i in {1,2,3}{
				\draw[directededge,line width=2.5, color=red] (q-\i-1-p) to (p-\j-1-m);
				\draw[directededge] (q-\i-2-m) to (p-\j-2-p);
				\draw[directededge] (q-\i-4-m) to (p-\j-4-p);
				}
			\draw[directededge,line width=2.5, color=red] (q-1-3-m) to (p-2-3-p);
			\draw[directededge] (q-2-3-m) to (p-3-3-p);
			\draw[directededge] (q-3-3-m) to (p-4-3-p);			
					
			\draw[directededge,line width=2.5,color=red] (p-2-2-m) to (q-2-2-p);
			\draw[directededge] (p-1-2-m) to (q-1-2-p);
			\draw[directededge] (p-3-2-m) to (q-3-2-p);
			\draw[directededge] (p-4-2-m) to (q-4-2-p);
			
			\draw[directededge,line width=2.5,color=red] (p-2-3-m) to (q-2-3-p);
			\draw[directededge,line width=2.5,color=red] (p-1-3-m) to (q-1-3-p);
			\draw[directededge] (p-3-3-m) to (q-3-3-p);
			\draw[directededge,line width=2.5,color=red] (p-4-3-m) to (q-4-3-p);

		\end{tikzpicture}\hskip -2.5cm
		\begin{tikzpicture}[scale=0.6, label distance=-2mm]
			\def\dist{1.5}
			\node (centre) at (0,0) {};
			\node[vertex] (p-1-1-m) at ($(0,0)$) {};
			\node[vertex] (p-2-1-m) at ($(p-1-1-m)+(270:{2*\dist})$) {};
			\node[vertex] (p-3-1-m) at ($(p-2-1-m)+(270:{2*\dist})$) {};
			\node[vertex] (p-4-1-m) at ($(p-3-1-m)+(270:{2*\dist})$) {};
			
			\node[vertex] (q-1-1-p) at ($(p-1-1-m)+(270:{\dist})$) {};
			\node[vertex] (q-2-1-p) at ($(q-1-1-p)+(270:{2*\dist})$) {};
			\node[vertex] (q-3-1-p) at ($(q-2-1-p)+(270:{2*\dist})$) {};
			\node[vertex] (q-4-1-p) at ($(q-3-1-p)+(270:{2*\dist})$) {};
			
			\foreach \i in {1,2,3,4}
			{
				\node[vertex] (p-\i-2-p) at ($(p-\i-1-m)+(0:{\dist})$) {};
				\node[vertex] (p-\i-2-m) at ($(p-\i-2-p)+(0:{\dist})$) {};
				\node[vertex] (p-\i-3-p) at ($(p-\i-2-m)+(0:{\dist})$) {};
				\node[vertex] (p-\i-3-m) at ($(p-\i-3-p)+(0:{\dist})$) {};
				\node[vertex] (p-\i-4-p) at ($(p-\i-3-m)+(0:{\dist})$) {};
			}
			\foreach \i in {1,2,3,4}
			{
				\node[vertex] (q-\i-2-m) at ($(q-\i-1-p)+(0:{\dist})$) {};
				\node[vertex] (q-\i-2-p) at ($(q-\i-2-m)+(0:{\dist})$) {};
				\node[vertex] (q-\i-3-m) at ($(q-\i-2-p)+(0:{\dist})$) {};
				\node[vertex] (q-\i-3-p) at ($(q-\i-3-m)+(0:{\dist})$) {};
				\node[vertex] (q-\i-4-m) at ($(q-\i-3-p)+(0:{\dist})$) {};
			}
			
			\node[vertex,color=red,label=315:{\(q^{+}_{2,2}\)}] (q-2-2-p) at ($(q-2-2-m)+(0:{\dist})$) {};
			\node[vertex,label=315:{\(p^{-}_{3,2}\)}] (p-3-2-m) at ($(p-3-2-p)+(0:{\dist})$) {};
			\node[vertex,label=315:{\(q^{+}_{3,2}\)}] (q-3-2-p) at ($(q-3-2-m)+(0:{\dist})$) {};
			
			\node[vertex,label=315:{\(q^{-}_{2,4}\)}] (q-2-4-p) at ($(q-2-3-p)+(0:{\dist})$) {};
			\node[vertex,label=315:{\(p^{+}_{3,4}\)}] (p-3-4-m) at ($(p-3-3-m)+(0:{\dist})$) {};
			\node[vertex,color=red,label=315:{\(q^{-}_{3,4}\)}] (q-3-4-p) at ($(q-3-3-p)+(0:{\dist})$) {};
			
			\foreach \x/\y in {q-4-1-p/p-1-1-m,q-4-3-m/p-1-3-p}{
				\begin{scope}
					\clip ($(\x)+(-0.2,-1)$) -- ($(\y)+(-0.2,1)$) --($(\y)+(5,1)$) -- ($(\x)+(5,-1)$) -- cycle;
					\draw[directededge,out=260,in=100] (\x) to (\y);
				\end{scope}
				\begin{scope}
					\clip ($(\x)+(-0.2,-1)$) -- ($(\y)+(-0.2,1)$) --($(\y)+(-5,1)$) -- ($(\x)+(-5,-1)$) -- cycle;
					\draw[directededge,out=260,in=100,dashed] (\x) to (\y);
				\end{scope}
			}
			\draw[directededge,line width=2.5,color=red,out=260,in=100,dashed] (q-4-2-m) to (p-1-2-p);
			\draw[directededge,line width=2.5,color=red,out=260,in=100,dashed] (q-4-4-m) to (p-1-4-p);
			
			\draw[directededge] (p-1-1-m) to (p-1-2-p);
			\draw[directededge] (p-2-1-m) to (p-2-2-p);
			\draw[directededge,line width=2.5,color=red] (p-2-2-p) to (p-2-2-m);
			\draw[directededge] (q-2-2-m) to (q-2-1-p);

			\draw[directededge] (p-1-1-m) to (q-1-1-p);
			\draw[directededge] (p-3-1-m) to (q-3-1-p);
			\draw[directededge] (p-4-1-m) to (q-4-1-p);
			\draw[directededge] (p-2-1-m) to (q-2-1-p);
			\draw[directededge] (q-1-1-p) to (p-2-1-m);
			\draw[directededge] (q-3-1-p) to (p-4-1-m);
			\draw[directededge] (q-2-1-p) to (p-3-1-m);
			
			\draw[directededge] (p-1-1-m) to (p-1-2-p);
			\draw[directededge,line width=2.5,color=red] (p-1-2-p) to (p-1-2-m);
			\draw[directededge] (p-1-2-m) to (p-1-3-p);
			\draw[directededge] (p-1-3-p) to (p-1-3-m);
			\draw[directededge] (p-1-3-m) to (p-1-4-p);
			
			\draw[directededge] (q-1-2-m) to (q-1-1-p);
			\draw[directededge,line width=2.5,color=red] (q-1-2-p) to (q-1-2-m);
			\draw[directededge] (q-1-3-m) to (q-1-2-p);
			\draw[directededge] (q-1-3-p) to (q-1-3-m);
			\draw[directededge] (q-1-4-m) to (q-1-3-p);
			
			\draw[directededge] (p-2-2-m) to (p-2-3-p);
			\draw[directededge] (p-2-3-p) to (p-2-3-m);
			\draw[directededge] (p-2-3-m) to (p-2-4-p);
			
			\draw[directededge,line width=2.5,color=red] (q-2-2-p) to (q-2-2-m);
			\draw[directededge,line width=2.5,color=red] (q-2-3-m) to (q-2-2-p);
			\draw[directededge,line width=2.5,color=red] (q-2-3-p) to (q-2-3-m);
			\draw[directededge,line width=2.5,color=red] (q-2-4-m) to (q-2-3-p);
			
			\draw[directededge] (p-3-1-m) to (p-3-2-p);
			\draw[directededge,line width=2.5,color=red] (p-3-2-p) to (p-3-2-m);
			\draw[directededge,line width=2.5,color=red] (p-3-2-m) to (p-3-3-p);
			\draw[directededge,line width=2.5,color=red] (p-3-3-p) to (p-3-3-m);
			\draw[directededge,line width=2.5,color=red] (p-3-3-m) to (p-3-4-p);
			
			\draw[directededge] (q-3-2-m) to (q-3-1-p);
			\draw[directededge,line width=2.5,color=red] (q-3-2-p) to (q-3-2-m);
			\draw[directededge,line width=2.5,color=red] (q-3-3-m) to (q-3-2-p);
			\draw[directededge,line width=2.5,color=red] (q-3-3-p) to (q-3-3-m);
			\draw[directededge,line width=2.5,color=red] (q-3-4-m) to (q-3-3-p);
			
			\draw[directededge] (p-4-1-m) to (p-4-2-p);
			\draw[directededge,line width=2.5,color=red] (p-4-2-p) to (p-4-2-m);
			\draw[directededge] (p-4-2-m) to (p-4-3-p);
			\draw[directededge] (p-4-3-p) to (p-4-3-m);
			\draw[directededge] (p-4-3-m) to (p-4-4-p);
			
			\draw[directededge] (q-4-2-m) to (q-4-1-p);
			\draw[directededge,line width=2.5,color=red] (q-4-2-p) to (q-4-2-m);
			\draw[directededge] (q-4-3-m) to (q-4-2-p);
			\draw[directededge] (q-4-3-p) to (q-4-3-m);
			\draw[directededge] (q-4-4-m) to (q-4-3-p);

			\draw[directededge] (q-1-3-m) to (p-2-3-p);
			\draw[directededge] (q-2-3-m) to (p-3-3-p);
			\draw[directededge] (q-3-3-m) to (p-4-3-p);

			\draw[directededge,line width=2.5,color=red] (p-2-2-m) to (q-2-2-p);
			\draw[directededge,line width=2.5,color=red] (p-1-2-m) to (q-1-2-p);
			\draw[directededge] (p-3-2-m) to (q-3-2-p);
			\draw[directededge,line width=2.5,color=red] (p-4-2-m) to (q-4-2-p);
			\draw[directededge,line width=2.5,color=red] (q-2-2-m) to (p-3-2-p);
			\draw[directededge,line width=2.5,color=red] (q-1-2-m) to (p-2-2-p);
			\draw[directededge,line width=2.5,color=red] (q-3-2-m) to (p-4-2-p);
			
			\draw[directededge] (p-2-3-m) to (q-2-3-p);
			\draw[directededge] (p-1-3-m) to (q-1-3-p);
			\draw[directededge] (p-3-3-m) to (q-3-3-p);
			\draw[directededge] (p-4-3-m) to (q-4-3-p);
			
			\draw[directededge] (q-2-4-p) to (p-3-4-m);
			\draw[directededge,line width=2.5,color=red] (q-1-4-m) to (p-2-4-p);
			\draw[directededge,line width=2.5,color=red] (q-3-4-m) to (p-4-4-p);
			\draw[directededge,line width=2.5,color=red] (p-1-4-p) to (q-1-4-m);
			\draw[directededge,line width=2.5,color=red] (p-2-4-p) to (q-2-4-m);
			\draw[directededge,line width=2.5,color=red] (p-3-4-p) to (q-3-4-m);
			\draw[directededge,line width=2.5,color=red] (p-4-4-p) to (q-4-4-m);

			\end{tikzpicture} \vskip -2cm
			
			\begin{tikzpicture}[scale=0.6, label distance=-2mm]
			\def\dist{1.5}
			\node (centre) at (0,0) {};
			\node[vertex] (p-1-1-m) at ($(0,0)$) {};
			\node[vertex,label=315:{\(p^{-}_{2,1}\)}] (p-2-1-m) at ($(p-1-1-m)+(270:{2*\dist})$) {};
			\node[vertex,color=red,label=315:{\(p^{-}_{3,1}\)}] (p-3-1-m) at ($(p-2-1-m)+(270:{2*\dist})$) {};
			\node[vertex] (p-4-1-m) at ($(p-3-1-m)+(270:{2*\dist})$) {};
			
			\node[vertex] (q-1-1-p) at ($(p-1-1-m)+(270:{\dist})$) {};
			\node[vertex,label=315:{\(q^{+}_{2,1}\)}] (q-2-1-p) at ($(q-1-1-p)+(270:{2*\dist})$) {};
			\node[vertex] (q-3-1-p) at ($(q-2-1-p)+(270:{2*\dist})$) {};
			\node[vertex] (q-4-1-p) at ($(q-3-1-p)+(270:{2*\dist})$) {};
			
			\foreach \i in {1,2,3,4}
			{
				\node[vertex] (p-\i-2-p) at ($(p-\i-1-m)+(0:{\dist})$) {};
				\node[vertex] (p-\i-2-m) at ($(p-\i-2-p)+(0:{\dist})$) {};
				\node[vertex] (p-\i-3-p) at ($(p-\i-2-m)+(0:{\dist})$) {};
				\node[vertex] (p-\i-3-m) at ($(p-\i-3-p)+(0:{\dist})$) {};
				\node[vertex] (p-\i-4-p) at ($(p-\i-3-m)+(0:{\dist})$) {};
			}
			\foreach \i in {1,2,3,4}
			{
				\node[vertex] (q-\i-2-m) at ($(q-\i-1-p)+(0:{\dist})$) {};
				\node[vertex] (q-\i-2-p) at ($(q-\i-2-m)+(0:{\dist})$) {};
				\node[vertex] (q-\i-3-m) at ($(q-\i-2-p)+(0:{\dist})$) {};
				\node[vertex] (q-\i-3-p) at ($(q-\i-3-m)+(0:{\dist})$) {};
				\node[vertex] (q-\i-4-m) at ($(q-\i-3-p)+(0:{\dist})$) {};
			}
			\node[vertex,color=red,label=315:{\(p^{+}_{2,3}\)}] (p-2-3-p) at ($(p-2-2-m)+(0:{\dist})$) {};
			\node[vertex,label=315:{\(p^{+}_{3,3}\)}] (p-3-3-p) at ($(p-3-2-m)+(0:{\dist})$) {};
			\node[vertex,label=315:{\(q^{-}_{2,3}\)}] (q-2-3-m) at ($(q-2-2-p)+(0:{\dist})$) {};
			
			\foreach \x/\y in {q-4-2-m/p-1-2-p,q-4-4-m/p-1-4-p}{
				\begin{scope}
					\clip ($(\x)+(-0.2,-1)$) -- ($(\y)+(-0.2,1)$) --($(\y)+(5,1)$) -- ($(\x)+(5,-1)$) -- cycle;
					\draw[directededge,out=260,in=100] (\x) to (\y);
				\end{scope}
				\begin{scope}
					\clip ($(\x)+(-0.2,-1)$) -- ($(\y)+(-0.2,1)$) --($(\y)+(-5,1)$) -- ($(\x)+(-5,-1)$) -- cycle;
					\draw[directededge,out=260,in=100,dashed] (\x) to (\y);
				\end{scope}
			}
			\draw[directededge,line width=2.5,color=red,out=260,in=100,dashed] (q-4-1-p) to (p-1-1-m);
			\draw[directededge,line width=2.5,color=red,out=260,in=100,dashed] (q-4-3-m) to (p-1-3-p);

			\draw[directededge] (p-1-1-m) to (p-1-2-p);
			\draw[directededge,line width=2.5,color=red] (p-2-1-m) to (p-2-2-p);
			\draw[directededge,line width=2.5,color=red] (p-2-2-p) to (p-2-2-m);
			\draw[directededge,line width=2.5,color=red] (q-2-2-m) to (q-2-1-p);

			\draw[directededge,line width=2.5,color=red] (p-1-1-m) to (q-1-1-p);
			\draw[directededge,line width=2.5,color=red] (p-3-1-m) to (q-3-1-p);
			\draw[directededge,line width=2.5,color=red] (p-4-1-m) to (q-4-1-p);
			\draw[directededge] (p-2-1-m) to (q-2-1-p);
			\draw[directededge,line width=2.5,color=red] (q-1-1-p) to (p-2-1-m);
			\draw[directededge,line width=2.5,color=red] (q-3-1-p) to (p-4-1-m);
			\draw[directededge,line width=2.5,color=red] (q-2-1-p) to (p-3-1-m);
			
			\draw[directededge] (p-1-1-m) to (p-1-2-p);
			\draw[directededge] (p-1-2-p) to (p-1-2-m);
			\draw[directededge] (p-1-2-m) to (p-1-3-p);
			\draw[directededge,line width=2.5,color=red] (p-1-3-p) to (p-1-3-m);
			\draw[directededge] (p-1-3-m) to (p-1-4-p);
			
			\draw[directededge] (q-1-2-m) to (q-1-1-p);
			\draw[directededge] (q-1-2-p) to (q-1-2-m);
			\draw[directededge] (q-1-3-m) to (q-1-2-p);
			\draw[directededge,line width=2.5,color=red] (q-1-3-p) to (q-1-3-m);
			\draw[directededge] (q-1-4-m) to (q-1-3-p);
			
			\draw[directededge,line width=2.5,color=red] (p-2-2-m) to (p-2-3-p);
			\draw[directededge,line width=2.5,color=red] (p-2-3-p) to (p-2-3-m);
			\draw[directededge] (p-2-3-m) to (p-2-4-p);
			
			\draw[directededge,line width=2.5,color=red] (q-2-2-p) to (q-2-2-m);
			\draw[directededge,line width=2.5,color=red] (q-2-3-m) to (q-2-2-p);
			\draw[directededge,line width=2.5,color=red] (q-2-3-p) to (q-2-3-m);
			\draw[directededge] (q-2-4-m) to (q-2-3-p);
			
			\draw[directededge,line width=2.5,color=red] (p-3-1-m) to (p-3-2-p);
			\draw[directededge,line width=2.5,color=red] (p-3-2-p) to (p-3-2-m);
			\draw[directededge,line width=2.5,color=red] (p-3-2-m) to (p-3-3-p);
			\draw[directededge,line width=2.5,color=red] (p-3-3-p) to (p-3-3-m);
			\draw[directededge] (p-3-3-m) to (p-3-4-p);
			
			\draw[directededge] (q-3-2-m) to (q-3-1-p);
			\draw[directededge] (q-3-2-p) to (q-3-2-m);
			\draw[directededge] (q-3-3-m) to (q-3-2-p);
			\draw[directededge,line width=2.5,color=red] (q-3-3-p) to (q-3-3-m);
			\draw[directededge] (q-3-4-m) to (q-3-3-p);
			
			\draw[directededge] (p-4-1-m) to (p-4-2-p);
			\draw[directededge] (p-4-2-p) to (p-4-2-m);
			\draw[directededge] (p-4-2-m) to (p-4-3-p);
			\draw[directededge,line width=2.5,color=red] (p-4-3-p) to (p-4-3-m);
			\draw[directededge] (p-4-3-m) to (p-4-4-p);
			
			\draw[directededge] (q-4-2-m) to (q-4-1-p);
			\draw[directededge] (q-4-2-p) to (q-4-2-m);
			\draw[directededge] (q-4-3-m) to (q-4-2-p);
			\draw[directededge,line width=2.5,color=red] (q-4-3-p) to (q-4-3-m);
			\draw[directededge] (q-4-4-m) to (q-4-3-p);

			\draw[directededge,line width=2.5,color=red] (q-1-3-m) to (p-2-3-p);
			\draw[directededge] (q-2-3-m) to (p-3-3-p);
			\draw[directededge,line width=2.5,color=red] (q-3-3-m) to (p-4-3-p);

			\draw[directededge] (p-2-2-m) to (q-2-2-p);
			\draw[directededge] (p-1-2-m) to (q-1-2-p);
			\draw[directededge] (p-3-2-m) to (q-3-2-p);
			\draw[directededge] (p-4-2-m) to (q-4-2-p);
			\draw[directededge] (q-2-2-m) to (p-3-2-p);
			\draw[directededge] (q-1-2-m) to (p-2-2-p);
			\draw[directededge] (q-3-2-m) to (p-4-2-p);
			
			\draw[directededge,line width=2.5,color=red] (p-2-3-m) to (q-2-3-p);
			\draw[directededge,line width=2.5,color=red] (p-1-3-m) to (q-1-3-p);
			\draw[directededge,line width=2.5,color=red] (p-3-3-m) to (q-3-3-p);
			\draw[directededge,line width=2.5,color=red] (p-4-3-m) to (q-4-3-p);
			
			\draw[directededge] (q-2-4-p) to (p-3-4-m);
			\draw[directededge] (q-1-4-m) to (p-2-4-p);
			\draw[directededge] (q-3-4-m) to (p-4-4-p);
			\draw[directededge] (p-1-4-p) to (q-1-4-m);
			\draw[directededge] (p-2-4-p) to (q-2-4-m);
			\draw[directededge] (p-3-4-p) to (q-3-4-m);
			\draw[directededge] (p-4-4-p) to (q-4-4-m);
			\end{tikzpicture} \hskip -2.5cm
			\begin{tikzpicture}[scale=0.6, label distance=-2mm]
			\def\dist{1.5}
			\node (centre) at (0,0) {};
			\node[vertex] (p-1-1-m) at ($(0,0)$) {};
			\node[vertex] (p-2-1-m) at ($(p-1-1-m)+(270:{2*\dist})$) {};
			\node[vertex] (p-3-1-m) at ($(p-2-1-m)+(270:{2*\dist})$) {};
			\node[vertex] (p-4-1-m) at ($(p-3-1-m)+(270:{2*\dist})$) {};
			
			\node[vertex] (q-1-1-p) at ($(p-1-1-m)+(270:{\dist})$) {};
			\node[vertex] (q-2-1-p) at ($(q-1-1-p)+(270:{2*\dist})$) {};
			\node[vertex] (q-3-1-p) at ($(q-2-1-p)+(270:{2*\dist})$) {};
			\node[vertex] (q-4-1-p) at ($(q-3-1-p)+(270:{2*\dist})$) {};
			
			\foreach \i in {1,2,3,4}
			{
				\node[vertex] (p-\i-2-p) at ($(p-\i-1-m)+(0:{\dist})$) {};
				\node[vertex] (p-\i-2-m) at ($(p-\i-2-p)+(0:{\dist})$) {};
				\node[vertex] (p-\i-3-p) at ($(p-\i-2-m)+(0:{\dist})$) {};
				\node[vertex] (p-\i-3-m) at ($(p-\i-3-p)+(0:{\dist})$) {};
				\node[vertex] (p-\i-4-p) at ($(p-\i-3-m)+(0:{\dist})$) {};
			}
			\foreach \i in {1,2,3,4}
			{
				\node[vertex] (q-\i-2-m) at ($(q-\i-1-p)+(0:{\dist})$) {};
				\node[vertex] (q-\i-2-p) at ($(q-\i-2-m)+(0:{\dist})$) {};
				\node[vertex] (q-\i-3-m) at ($(q-\i-2-p)+(0:{\dist})$) {};
				\node[vertex] (q-\i-3-p) at ($(q-\i-3-m)+(0:{\dist})$) {};
				\node[vertex] (q-\i-4-m) at ($(q-\i-3-p)+(0:{\dist})$) {};
			}
			
			\node[vertex,label=315:{\(p^{-}_{2,2}\)}] (p-2-2-m) at ($(p-2-2-p)+(0:{\dist})$) {};
			\node[vertex,color=red,label=315:{\(p^{+}_{2,3}\)}] (p-2-3-p) at ($(p-2-2-m)+(0:{\dist})$) {};
			\node[vertex,label=315:{\(q^{+}_{3,2}\)}] (q-3-2-p) at ($(q-3-2-m)+(0:{\dist})$) {};
			\node[vertex,label=315:{\(q^{-}_{1,3}\)}] (q-1-3-m) at ($(q-1-2-p)+(0:{\dist})$) {};
			
			\node[vertex,label=315:{\(q^{-}_{1,4}\)}] (q-1-4-p) at ($(q-1-3-p)+(0:{\dist})$) {};
			\node[vertex,label=315:{\(p^{+}_{2,4}\)}] (p-2-4-m) at ($(p-2-3-m)+(0:{\dist})$) {};
			\node[vertex,color=red,label=315:{\(q^{-}_{3,4}\)}] (q-3-4-p) at ($(q-3-3-p)+(0:{\dist})$) {};
			
			\foreach \x/\y in {q-4-1-p/p-1-1-m,q-4-3-m/p-1-3-p}{
				\begin{scope}
					\clip ($(\x)+(-0.2,-1)$) -- ($(\y)+(-0.2,1)$) --($(\y)+(5,1)$) -- ($(\x)+(5,-1)$) -- cycle;
					\draw[directededge,out=260,in=100] (\x) to (\y);
				\end{scope}
				\begin{scope}
					\clip ($(\x)+(-0.2,-1)$) -- ($(\y)+(-0.2,1)$) --($(\y)+(-5,1)$) -- ($(\x)+(-5,-1)$) -- cycle;
					\draw[directededge,out=260,in=100,dashed] (\x) to (\y);
				\end{scope}
			}
			\draw[directededge,line width=2.5,color=red,out=260,in=100,dashed] (q-4-2-m) to (p-1-2-p);
			\draw[directededge,line width=2.5,color=red,out=260,in=100,dashed] (q-4-4-m) to (p-1-4-p);
			
			\draw[directededge] (p-1-1-m) to (p-1-2-p);
			\draw[directededge] (p-2-1-m) to (p-2-2-p);
			\draw[directededge,line width=2.5,color=red] (p-2-2-p) to (p-2-2-m);
			\draw[directededge] (q-2-2-m) to (q-2-1-p);

			\draw[directededge] (p-1-1-m) to (q-1-1-p);
			\draw[directededge] (p-3-1-m) to (q-3-1-p);
			\draw[directededge] (p-4-1-m) to (q-4-1-p);
			\draw[directededge] (p-2-1-m) to (q-2-1-p);
			\draw[directededge] (q-1-1-p) to (p-2-1-m);
			\draw[directededge] (q-3-1-p) to (p-4-1-m);
			\draw[directededge] (q-2-1-p) to (p-3-1-m);
			
			\draw[directededge] (p-1-1-m) to (p-1-2-p);
			\draw[directededge,line width=2.5,color=red] (p-1-2-p) to (p-1-2-m);
			\draw[directededge] (p-1-2-m) to (p-1-3-p);
			\draw[directededge] (p-1-3-p) to (p-1-3-m);
			\draw[directededge] (p-1-3-m) to (p-1-4-p);
			
			\draw[directededge] (q-1-2-m) to (q-1-1-p);
			\draw[directededge,line width=2.5,color=red] (q-1-2-p) to (q-1-2-m);
			\draw[directededge] (q-1-3-m) to (q-1-2-p);
			\draw[directededge,line width=2.5,color=red] (q-1-3-p) to (q-1-3-m);
			\draw[directededge,line width=2.5,color=red] (q-1-4-m) to (q-1-3-p);
			
			\draw[directededge,line width=2.5,color=red] (p-2-2-m) to (p-2-3-p);
			\draw[directededge,line width=2.5,color=red] (p-2-3-p) to (p-2-3-m);
			\draw[directededge,line width=2.5,color=red] (p-2-3-m) to (p-2-4-p);
			
			\draw[directededge] (q-2-2-p) to (q-2-2-m);
			\draw[directededge] (q-2-3-m) to (q-2-2-p);
			\draw[directededge] (q-2-3-p) to (q-2-3-m);
			\draw[directededge] (q-2-4-m) to (q-2-3-p);
			
			\draw[directededge] (p-3-1-m) to (p-3-2-p);
			\draw[directededge] (p-3-2-p) to (p-3-2-m);
			\draw[directededge] (p-3-2-m) to (p-3-3-p);
			\draw[directededge] (p-3-3-p) to (p-3-3-m);
			\draw[directededge] (p-3-3-m) to (p-3-4-p);
			
			\draw[directededge] (q-3-2-m) to (q-3-1-p);
			\draw[directededge,line width=2.5,color=red] (q-3-2-p) to (q-3-2-m);
			\draw[directededge,line width=2.5,color=red] (q-3-3-m) to (q-3-2-p);
			\draw[directededge,line width=2.5,color=red] (q-3-3-p) to (q-3-3-m);
			\draw[directededge,line width=2.5,color=red] (q-3-4-m) to (q-3-3-p);
			
			\draw[directededge] (p-4-1-m) to (p-4-2-p);
			\draw[directededge,line width=2.5,color=red] (p-4-2-p) to (p-4-2-m);
			\draw[directededge] (p-4-2-m) to (p-4-3-p);
			\draw[directededge] (p-4-3-p) to (p-4-3-m);
			\draw[directededge] (p-4-3-m) to (p-4-4-p);
			
			\draw[directededge] (q-4-2-m) to (q-4-1-p);
			\draw[directededge,line width=2.5,color=red] (q-4-2-p) to (q-4-2-m);
			\draw[directededge] (q-4-3-m) to (q-4-2-p);
			\draw[directededge] (q-4-3-p) to (q-4-3-m);
			\draw[directededge] (q-4-4-m) to (q-4-3-p);

			\draw[directededge,line width=2.5,color=red] (q-1-3-m) to (p-2-3-p);
			\draw[directededge] (q-2-3-m) to (p-3-3-p);
			\draw[directededge] (q-3-3-m) to (p-4-3-p);

			\draw[directededge] (p-2-2-m) to (q-2-2-p);
			\draw[directededge,line width=2.5,color=red] (p-1-2-m) to (q-1-2-p);
			\draw[directededge] (p-3-2-m) to (q-3-2-p);
			\draw[directededge,line width=2.5,color=red] (p-4-2-m) to (q-4-2-p);
			\draw[directededge] (q-2-2-m) to (p-3-2-p);
			\draw[directededge,line width=2.5,color=red] (q-1-2-m) to (p-2-2-p);
			\draw[directededge,line width=2.5,color=red] (q-3-2-m) to (p-4-2-p);
			
			\draw[directededge] (p-2-3-m) to (q-2-3-p);
			\draw[directededge] (p-1-3-m) to (q-1-3-p);
			\draw[directededge] (p-3-3-m) to (q-3-3-p);
			\draw[directededge] (p-4-3-m) to (q-4-3-p);
			
			\draw[directededge,line width=2.5,color=red] (q-2-4-p) to (p-3-4-m);
			\draw[directededge] (q-1-4-m) to (p-2-4-p);
			\draw[directededge,line width=2.5,color=red] (q-3-4-m) to (p-4-4-p);
			\draw[directededge,line width=2.5,color=red] (p-1-4-p) to (q-1-4-m);
			\draw[directededge,line width=2.5,color=red] (p-2-4-p) to (q-2-4-m);
			\draw[directededge,line width=2.5,color=red] (p-3-4-p) to (q-3-4-m);
			\draw[directededge,line width=2.5,color=red] (p-4-4-p) to (q-4-4-m);

			\end{tikzpicture}
		\caption{The construction of onion immersion models corresponding to four different cases considered in the proof (top left, top right, bottom left, bottom right).}
		\label{fig:another-onion-in-the-wall}
	\end{figure}

	As $b \leq j \leq k-1$, it is straightforward to verify that $W_1$, $W_2$ and $W_3$ are well-defined paths.
	Moreover, a direct inspection shows that they are pairwise internally vertex-disjoint, and thus also edge-disjoint. 
	Thus, $W_1,W_2$, and $W_3$ form an onion immersion model with root $p^-_{a,b}$ and tip~$q^+_{i,j}$.
	
	Next, we argue that for each $a\neq i$ and $b>j$, we may find an onion immersion model with root $q^-_{a,b}$ and tip $q^+_{i,j}$.
	
	\begin{tabular}{ccccccccc}
		$W_1=$&$q^+_{i,j}$&$\xrightarrow{C_j}$&$p^-_{i+1,j}$&$\xrightarrow{P_{i+1}}$&$p^+_{i+1,b}$&$\xrightarrow{C_b}$&$ q^-_{a,b}$, \\
		$W_2=$&$q^-_{a,b}$&$\xrightarrow{C_b}$&$q^-_{i,b}$&$\xrightarrow{Q_i}$& $q^+_{i,j},$&\textrm{and}& \\
		$W_3=$&$q^-_{a,b}$&$\xrightarrow{Q_a}$&$q^+_{a,j}$&$\xrightarrow{C_j}$&$q^+_{i,j}$.&&
	\end{tabular}
	
	Again, as $b>j$, these paths are well-defined and a direct inspection shows that they are internally vertex-disjoint. So they constitute an onion immersion model with root root $q^-_{a,b}$ and tip $q^+_{i,j}$.
	
	
	Next, consider a vertex $p^+_{i,j}$ where $j>1$. To argue that there is an onion immersion model with tip $p^+_{i,j}$ and root $p^-_{a,b}$, for any $a\neq i$ and $b<j$,  we consider the following paths:
	
	\begin{tabular}{ccccccccc}
		$W_1=$&$p^+_{i,j}$&$\xrightarrow{C_j}$&$q^-_{i,j}$&$\xrightarrow{Q_i}$&$q^+_{i,b}$&$\xrightarrow{C_b}$&$ p^-_{a,b}$, \\
		$W_2=$&$p^-_{a,b}$&$\xrightarrow{P_a}$&$p^+_{a,j}$&$\xrightarrow{C_j}$&$p^+_{i,j}$,&\textrm{and}& \\
		$W_3=$&$p^-_{a,b}$&$\xrightarrow{C_b}$&$p^-_{i,b}$&$\xrightarrow{P_i}$& $p^+_{i,j}$.&& 
	\end{tabular}

	And to argue that there is an onion immersion model with tip $p_{i,j}^+$ and root $q^-_{a,b}$, for any $a\neq i-1$ and $b\geq j$, we consider the following paths:
	
	
	\begin{tabular}{ccccccccc}
		$W_1=$&$p^+_{i,j}$&$\xrightarrow{P_i}$&$p^+_{i,b}$&$\xrightarrow{C_b}$&$q^-_{a,b},$&& & \\
		$W_2=$&$q^-_{a,b}$&$\xrightarrow{Q_a}$&$q^+_{a,j-1}$&$\xrightarrow{C_{j-1}}$&$p^-_{i,j-1}$&$\xrightarrow{P_i}$& $p^+_{i,j},$&\textrm{and} \\
		$W_3=$&$q^-_{a,b}$&$\xrightarrow{C_b}$&$q^-_{i-1,b}$&$\xrightarrow{Q_{i-1}}$&$q^-_{i-1,j}$&$\xrightarrow{C_j}$&$p^+_{i,j}$.& 
	\end{tabular}
	
	And again, a direct inspection shows that
	in all these cases, the paths are well-defined and internally vertex-disjoint, so they give us the desired onion immersion models.
\end{proof}

\begin{proof}[Proof of \cref{lemma:wall-dtcw}]
	We set \(\thmFkt{lemma:wall-dtcw}{f}{k}=k^2+1\). The case $k=1$ is straightforward to verify, so from now on we consider only $k \geq 2$.
	For the sake of contradiction, let $(T, \beta, \gamma, \theta)$ be a directed tree-cut decomposition of $\gridImm{k^2+1}$ of width at most $k-1$. 
	Note that $\left|V\left(\gridImm{k^2+1}\right)\right| = 4(k^2+1)k^2$.
	
	First, we observe that there must exist a node $x^\star \in V(T)$ such that for every neighbour $x$ of $x^\star$, the set $\beta(\treeSplit{T}{x}{xx^\star})$ contains at most $2(k^2+1)k^2$ vertices of $\gridImm{k^2+1}$.
	Indeed, to find such a node $x^\star$ it suffices to orient every edge $e=xy$ of $T$ towards $x$ or $y$ depending on which of the sets $\beta(T^x_{xy}),\beta(T^y_{xy})$ is larger (with ties broken arbitrarily), and select $x^\star$ to be a sink in this orientation (which exists since $T$ is a tree).

	As $|\theta(x^\star)| \leq k-1$, there exist sets $\cC' \subseteq \{C_1,\ldots,C_{k^2+1}\}$, and $\cP' \subseteq \{(P_1,Q_1),\ldots,(P_{k^2+1},Q_{k^2+1})\}$, each of size $\ell\coloneqq k^2-k+2$, such that every cycle in $\cC'$, and every path in every pair in $\cP'$ is disjoint with $\theta(x^\star)$. 
	Consider the digraph $D$ obtained by taking the union of all the edge and vertex sets of all the cycles from $\cC'$ and paths from the pairs in $\cP'$, and dissolving all vertices whose indegrees and outdegrees in the resulting subdigraph are at most one. 
	(Dissolving such a vertex means either removing it, if it has in-degree $0$ or out-degree~$0$, or replacing it, together with the incident edges, by a single edge from the unique in-neighbour to the unique out-neighbour.)
	Clearly, $\gridImm{k^2+1}$ contains $D$ as an immersion and $E(D) \cap \theta(x^\star) = \emptyset$.
	So in particular, there is no onion immersion model in $D$ whose root lies in a bag placed in one of the components of $T-x^\star$ and tip in the other.
	
	We observe that $D$ is isomorphic to the cylindrical wall $\gridImm{\ell}$.
	In what follows, the sets $Q^+$, $Q^-$, $P^+$, $P^-$ are with respect to $D$, with this isomorphism in mind.
	Note that each of these sets is of size $\ell(\ell-1)$.
	Define the graph $B_\ell$ as in~\cref{lem:another-onion-in-the-wall}.
	Observe that if two vertices $u,v \in V(D)$ are adjacent in the graph $B_\ell$, there must exist a neighbour $x$ of $x^\star$ such that $u,v \in \beta(T^x_{xx^\star}) \cup \beta(x^\star)$.
	
	By \cref{lem:another-onion-in-the-wall}, any two vertices $u,v \in P^+\cup Q^+$ (resp.~$u,v \in P^- \cup Q^-$) have at least $2(\ell-2) > k-1$ common neighbours in $B_\ell$, as they must have at least $\ell-2$ common neighbours on the first cycle and at least $\ell-2$ on the last one. 
	Since $|P^+\cup Q^+| = 2\ell(\ell-1)>k-1$, there exist $u \in P^+\cup Q^+$ and a neighbour $x$ of $x^\star$ such that $u \in \beta(T^x_{xx^\star})$.
	Denote by $N(u) \subseteq P^-\cup Q^-$ the neighbourhood of $u$ in~$B_\ell$, and note that $|N(u)|\geq \ell(\ell-1)$. Clearly, $N(u) \subseteq \beta(T^x_{xx^\star}) \cup \beta(x^\star)$, and since $|N(u)| > k-1$, there must be $w \in N(u)$ that belongs to $\beta(T^x_{xx^\star})$ as well.
	
	Since every $u' \in P^+\cup Q^+$ has more than $k-1$ common neighbours with $u$ in $B_\ell$, all elements of $P^+\cup Q^+$ must belong to $\beta(T^x_{xx^\star}) \cup \beta(x^\star)$.
	Similarly, since every $w' \in P^-\cup Q^-$ has more than $k-1$ common neighbours with $w$ in $B_\ell$, all elements of $P^-\cup Q^-$ must belong to $\beta(T^x_{xx^\star}) \cup \beta(x^\star)$, and therefore $V(D) \subseteq \beta(T^x_{xx^\star}) \cup \beta(x^\star)$. As \[|V(D) \setminus \beta(x^\star)|\geq 4\ell(\ell-1)-(k-1)>2(k^2+1)k^2 \geq |\beta(T^x_{xx^\star})|,\] we get a contradiction.
\end{proof}

We can now complete the proof of~\cref{thm:imm_implies_large_dtcw}.

\ThmgridImmImplhighDTCW*
\begin{proof}
	Set \(\thmFkt{thm:imm_implies_large_dtcw}{f}{k} \coloneqq \thmFkt{lemma:wall-dtcw}{f}{k}\).
	Let \(D\) be a digraph containing an immersion model of \(\gridImm{\thmFkt{thm:imm_implies_large_dtcw}{f}{k}}\).
	By \cref{thm:immersion-monotone}, the \dtcw of \(D\) is at least that of \(\gridImm{\thmFkt{thm:imm_implies_large_dtcw}{f}{k}}\).
	Now \cref{lemma:wall-dtcw} implies that \dtcw of \(\gridImm{\thmFkt{thm:imm_implies_large_dtcw}{f}{k}}\)
	is at least \(k\).
	This ends the proof.
\end{proof}

\section{Harvesting onions}\label{sec:harvesting}

In this section we recall the key tool introduced in \cite{bozyk2022digraphs}: the Onion Harvesting Lemma. In essence, the results says from two families of edge-disjoint trails that have a large mutual intersection and all share one endpoint, one can always extract a large onion-star immersion model. We will use this statement several times.

For technical reasons, we need to strengthen a bit the result from \cite{bozyk2022digraphs}. Precisely, we will use the following formulation.

\begin{theorem}[Congested Onion Harvesting]
	\label{lem:onion_harvesting}
	There exists a function \(\thmFkt{lem:onion_harvesting}{g}{} \colon \N \times \N \to \N\) such that the following holds.
	Let \(D\) be a digraph, \(x\) be a vertex of \(D\), \(k\) and \(c\) be a positive integers,
	and let \(\mathcal{P}\) and~\(\mathcal{Q}\) be two families of trails in \(D\), each of size at least \(\thmFkt{lem:onion_harvesting}{g}{k,c}\), satisfying the following:
	\begin{itemize}
		\item \(\Start{P} = \End{Q} = x\) for all \(P \in \mathcal{P}\) and \(Q\in \mathcal{Q}\),
		\item \(\E{P} \cap \E{Q} \neq \emptyset\) for all \(P \in \mathcal{P}\) and \(Q \in \mathcal{Q}\), 
		\item $x$ only occurs once and each other vertex only occurs at most \(c\) times within every single trail \(P \in \mathcal{P}\) or \(Q \in \mathcal{Q}\), and
		\item the trails of \(\mathcal{P}\) are pairwise edge-disjoint and the trails of \(\mathcal{Q}\) are pairwise edge-disjoint.
	\end{itemize}
	Then \(D\) contains an onion-star immersion model of order \(k\) with centre \(x\).
	Moreover, there exists a fixed-parameter algorithm (with parameter $k$) to find such a model in~$D$.
\end{theorem}

The variant presented in \cite{bozyk2022digraphs} is for $c=1$.
That is, each of $\mathcal{P}$ and $\mathcal{Q}$ consists of edge-disjoint paths that can visit every vertex at most once.
The formulation above allows every vertex except for $x$ to be visited at most $c$ times, where $c$ is a parameter on which the bounds can depend.
We use it for $c=2$.
In~\cref{sec:appendix} we argue that the argumentation of \cite{bozyk2022digraphs} can be easily amended to also yield~\cref{lem:onion_harvesting}. Also, the algorithmic statement is not explicitly mentioned in \cite{bozyk2022digraphs}, both follows easily from the proof; we argue this in~\cref{sec:appendix} as well. 

For convenience, we now establish a basic corollary of~\cref{lem:onion_harvesting}: To find a large onion-star immersion model it suffices to find a large out-onion-star immersion model and a large in-onion-star immersion model with coinciding centres.

\begin{lemma}
	\label{lem:two_onion_stars}
	There is a function \(\thmFkt{lem:two_onion_stars}{f}{} \colon \N \to \N\) such that the following holds.
	Let \(D\) be a digraph, \(k\) be a positive integer, and \(x \in \V{D}\) be such that there
  is an out-onion-star model immersion of order \(\thmFkt{lem:two_onion_stars}{f}{k}\)
  centred at \(x\) as well as an in-onion-star model immersion of order \(\thmFkt{lem:two_onion_stars}{f}{k}\) centred at \(x\).
	
	Then \(D\) contains an onion-star immersion model of order \(t\) centred at \(x\).
	Moreover, there exists a fixed-parameter algorithm (with parameter $k$) to find such a model in~$D$.
\end{lemma}
\begin{proof}
	We set $\thmFkt{lem:two_onion_stars}{f}{k} \coloneqq 2 \cdot \bRamsey(3((\thmFkt{lem:onion_harvesting}{g}{k,2}-1)2^{8 \thmFkt{lem:onion_harvesting}{g}{k,2}}+1)).$
	
	Let $\OO'_i$ and $\OO'_o$ be the respective families of onion immersion models that make up the in-onion-star and the out-onion-star immersion models of order \(\thmFkt{lem:two_onion_stars}{f}{k}\) in $D$ that are centred at $x$.
	In what follows, we treat the elements of $\OO'_i$ and of $\OO'_o$ as subgraphs of $D$.
	Each of these subgraph is the union of three edge-disjoint trails which we can assume without loss of generality to be paths.
	
	By definition of an in-onion-star immersion, the elements of $\OO'_i$ are pairwise edge-disjoint, and the roots of the elements of $\OO'_i$ are pairwise distinct. 
	Analogously, the elements of $\OO'_o$ are pairwise edge-disjoint, and the tips of the elements of $\OO'_o$ are pairwise distinct.
	We arbitrarily choose a subfamily $\OO_i$ of \(\thmFkt{lem:two_onion_stars}{f}{k}/2=\bRamsey(3((\thmFkt{lem:onion_harvesting}{g}{k,2}-1)2^{8 \thmFkt{lem:onion_harvesting}{g}{k,2}}+1))\) elements of $\OO'_i$, and then choose a subset $\OO_o \subseteq \OO'_o$ of the same size so that the set of all root vertices of the elements in $\OO_i$ is disjoint with the set of all tip vertices of the elements in $\OO_o$. 
	
	Next, we construct a bipartite graph $G_1$ with bipartition classes $\OO_i$ and $\OO_o$ so that the elements of $\OO_i$ and $\OO_o$ are adjacent in $G_1$ if they are not edge-disjoint.
	By \cref{thm:BipRamsey}, there exist $\textsf{In} \subseteq \OO_i$ and $\textsf{Out} \subseteq \OO_o$, each of size $3((\thmFkt{lem:onion_harvesting}{g}{k,2}-1)2^{8 \thmFkt{lem:onion_harvesting}{g}{k,2}}+1)$, that are either anticomplete or complete to each other in $G_1$. 
	
	In the first case, since we have already guaranteed that the roots and tips of onion immersion models other than $x$ are pairwise different, we get that the elements of $\textsf{In}$ and $\textsf{Out}$ form an immersion model of an onion-star of order $3((\thmFkt{lem:onion_harvesting}{g}{k,2}-1)2^{8 \thmFkt{lem:onion_harvesting}{g}{k,2}}+1) \geq k$ in $D$ centred at $x$.
	
	Therefore, assume we have $\textsf{In} \subseteq \OO_i$ and $\textsf{Out} \subseteq \OO_o$, each of size $3((\thmFkt{lem:onion_harvesting}{g}{k,2}-1)2^{8 \thmFkt{lem:onion_harvesting}{g}{k,2}}+1)$ such that each $I \in \textsf{In}$ and $O \in \textsf{Out}$ are adjacent in $G_1$. 
	Note, however, that since all the roots and tips of the elements of $\textsf{In} \cup \textsf{Out}$ other than $x$ are pairwise distinct, so must be the elements of $\textsf{In} \cup \textsf{Out}$.  
	
	For every $O \in \textsf{In}$ (resp.~$O \in \textsf{Out}$) we denote by $\ell(O)$ the unique edge of $O$ whose tail (resp. head) is $x$.
	We construct a digraph $G_2$ on vertices $\textsf{In} \cup \textsf{Out}$ by adding an edge $(O_1,O_2)$ if $O_1$ contains $\ell(O_2)$.
	As the elements of $\textsf{In}$ (resp. $\textsf{Out}$) are pairwise edge-disjoint, $G_2$ is bipartite with bipartition classes $\textsf{In}$ and $\textsf{Out}$. Moreover, the indegree of each vertex is at most one.
	
	Let $\textsf{In}'$ be a subset of $\textsf{In}$ of size $|\textsf{In}|/3=( \thmFkt{lem:onion_harvesting}{g}{k,2}-1)2^{8 \thmFkt{lem:onion_harvesting}{g}{k,2}}+1$ that minimizes \(|\Out{G_2}{\textsf{In}'}|\).
	Since the number of edges incoming to $\textsf{Out}$ is at most $|\textsf{Out}|=|\textsf{In}|$, we have that \(|\Out{G_2}{\textsf{In}'}| \leq |\textsf{In}|/3\). However, $|\In{G_2}{\textsf{In}'}|\leq |\textsf{In}|/3$ as well, because every vertex has indegree at most~$1$ in $G_2$.
	
	Therefore, we can choose a subset $\textsf{Out}' \subseteq \textsf{Out}$ of size $|\textsf{In}|/3=|\textsf{Out}|/3=( \thmFkt{lem:onion_harvesting}{g}{k,2}-1)2^{8 \thmFkt{lem:onion_harvesting}{g}{k,2}}+1$ such that there are no edges between $\textsf{In}'$ and $\textsf{Out}'$. 
	
	Define $\textsf{In}'_{\ell} \coloneqq \{O - \ell(O) \colon O \in \textsf{In}'\}$, and $\textsf{Out}'_{\ell} \coloneqq \{O - \ell(O) \colon O \in \textsf{Out}'\}$.
	We know that for each $(O_i,O_o) \in \textsf{In}'_{\ell} \times \textsf{Out}'_{\ell}$ it still holds that $O_i$ and $O_o$ share an edge, as otherwise their corresponding onion immersion models would be adjacent in $G_2$. 
	We arbitrarily restrict $\textsf{Out}'_{\ell}$ to have only $4\thmFkt{lem:onion_harvesting}{g}{k,2}$ elements.
	For $O \in \textsf{In}'_{\ell} \cup \textsf{Out}'_{\ell}$, denote by $T_1(O)$ and $T_2(O)$ the two inclusion-wise maximal trails contained in $O$.
	We create a bipartite graph $G_3$ with bipartition classes 
	\begin{align*}
	\cP_{\textsf{In}}\coloneqq & \ \bigcup_{O \in \textsf{In}'_{\ell}}\{T_1(O),T_2(O)\}, \\
	\cQ_{\textsf{Out}}\coloneqq & \ \bigcup_{O \in \textsf{Out}'_{\ell}}\{T_1(O),T_2(O)\},
	\end{align*}
	and such that two trails from $\cP_{\textsf{In}}$ and $\cQ_{\textsf{Out}}$ are adjacent if they share an edge. 
	Without loss of generality we assume that for each $O \in \textsf{In}'_{\ell} \cup \textsf{Out}'_{\ell}$ the degree of $T_1(O)$ in $G_3$ is greater or equal that the degree of~$T_2(O)$. 
	
	For every $(O_i,O_o) \in \textsf{In}'_{\ell} \times \textsf{Out}'_{\ell}$, since $O_i$ and $O_o$ share at least one edge, there exists at least one edge in $G_3$ with one endpoint in $\{T_1(O_i),T_2(O_i)\}$ and the other in $\{T_1(O_o),T_2(O_o)\}$.
	Hence, for every $O_i \in \textsf{In}'_{\ell}$, $T_1(O_i)$ has degree at least $|\cQ_{\textsf{Out}}|/4=2|\textsf{Out}'_{\ell}|/4=2 \thmFkt{lem:onion_harvesting}{g}{k,2}$ in $G_3$.
	
	For every $O_i \in \textsf{In}'_{\ell}$ we remove $T_2(O_i)$ from $\cP_{\textsf{In}}$.
	Since $|\textsf{Out}'_{\ell}| = 8 \thmFkt{lem:onion_harvesting}{g}{k,2}$, there are at most $2^{8 \thmFkt{lem:onion_harvesting}{g}{k,2}}$ possible neighbourhoods that the elements of $\cP_{\textsf{In}}$ might have in $\cQ_{\textsf{Out}}$, in $G_3$. 
	
	Now, since there are $|\textsf{In}'_{\ell}|=( \thmFkt{lem:onion_harvesting}{g}{k,2}-1)2^{8 \thmFkt{lem:onion_harvesting}{g}{k,2}}+1$ vertices left in $\cP_{\textsf{In}}$, we know that we can find $\thmFkt{lem:onion_harvesting}{g}{k,2}$ of them, denote their set by $\cP$, that have the same neighbourhood in $G_3$, say $\cQ'$. As we already noticed, $|\cQ'| \geq 2\thmFkt{lem:onion_harvesting}{g}{k,2}$.
	
	Lastly, let $\cQ\coloneqq \cQ' \setminus \{T_2(O) \colon O \in \textsf{Out}'_{\ell} \textrm{ and } T_1(O),T_2(O) \in \cQ'\}$. Clearly, $|\cQ| \geq \thmFkt{lem:onion_harvesting}{g}{k,2}$.
	We arbitrarily trim the sets $\cP$ and $\cQ$ to contain $\thmFkt{lem:onion_harvesting}{g}{k,2}$ elements each. 
	Observe that $\cP \cup \cQ$ induces a biclique in $G_3$ with $\thmFkt{lem:onion_harvesting}{g}{k,2}$ vertices on either side, and that, for every $O \in \textsf{In}'_{\ell} \cup \textsf{Out}'_{\ell}$, this biclique contains at most one of $T_1(O)$ and $T_2(O)$. 
	Now it is enough to note that
	\begin{itemize}
		\item \(\End{P} = \Start{Q} = x\) for all \(P \in \mathcal{P}\) and \(Q\in \mathcal{Q}\),
		\item \(\E{P} \cap \E{Q} \neq \emptyset\) for all \(P \in \mathcal{P}\) and \(Q \in \mathcal{Q}\),
		\item the trails of \(\mathcal{P}\) are pairwise edge-disjoint and the trails of \(\mathcal{Q}\) are pairwise edge-disjoint, and
		\item each trail of \(\mathcal{P}\cup \mathcal{Q}\) is the union of two paths, at most one of which visits $x$.
	\end{itemize}
	Therefore, we may use \cref{lem:onion_harvesting} with $c=2$ to conclude that $D$ contains an onion-star immersion model of order $k$ with centre $x$.
	Finally, thanks to the algorithmic statement of \cref{lem:onion_harvesting}, the reasoning given above can be turned into a fixed-parameter algorithm in a straightforward manner.
\end{proof}

\section{Potato Theorem}
\label{sec:potato}


In this section we prove the following crucial structural result. Intuitively, it states that every digraph can be tree-decomposed into flow-connected pieces using cuts of small porosity.
This idea is formalised by the notion of \linked tree-partitions, see~\cref{def:linked}.


\PotatoLemma*


\subsection{Auxiliary properties of cuts}
We start by providing several auxiliary lemmas and constructions related to cuts in digraphs, which are useful for working with them later on.
We start with some results on porosity.

\begin{lemma}[Porosity-uncrossing-lemma]
	\label{lemma:porosity-boolean}
  Let \(D\) be a digraph, \((A,B), (A',B')\) be cuts in \(D\), and $p,p'$ be the porosity of \((A,B)\) and \(A',B'\), respectively.
  Then \((A \cup A', B \cap B'), (A \cap A', B \cup B'), (A \cap B', B \cup A'), (A' \cap B, B' \cup A)\)
  are all cuts in \(D\) of porosity at most \(p+p'\).
\end{lemma}
\begin{proof}
  Since the porosity of \((A,B)\) is the same as that of \((B,A)\),
  and the roles of \((A,B)\) and \((A',B')\) are symmetric,
  it suffices to argue about \((A \cap A', B \cup B')\).

  Let \(F\) be the set of \(p\) edges hitting all cycles visiting  both \(A\) and \(B\), and let \(F'\) be the set of \(p'\) edges hitting all cycles visiting both \(A'\) and \(B'\).
  We argue next that \(F \cup F'\) hits all cycles containing a vertex of both \(A \cap A'\) and \(B \cup B'\).
  Let \(C\) be such a cycle.
  It contains a vertex \(a \in A \cap A'\), and a vertex \(b \in B \cup B'\).
  If \(b \in B\), then the cycle contains a vertex of both \(A\) and \(B\) and thus is hit by \(F\).
  If \(b \in B'\) then \(C\) contains a vertex of \(A'\) and a vertex of \(B'\) and thus is hit by \(F'\).
\end{proof}

For the next \namecref{lemma:porosity-onecl}, we need the following definition. Let $D$ be a digraph and $X\subseteq V(D)$ be a subset of vertices of $D$.
The \emph{closure} of \(X\) is the set defined as
\begin{align*}
	\onecl{X}{D} \coloneqq \{v \in V(D)\ \mid\ \ & \textrm{there exists } x \in X \text{ and paths } P_1,P_2\\  & \text{such that }
	P_1 \text{ is an } (x,v)\text{-path and }
	P_2 \text{ is a } (v,x)\text{-path}\}
	\text{.}
\end{align*}
If \(D\) is clear from the context it is omitted.
Note that the closure of \(X\) is equal to the union of all the strongly connected components of \(D\) that intersect \(X\).

We note the following observation that explains how the porosity of a cut behaves under restriction to a closed set.

\begin{lemma}\label{lemma:porosity-onecl}
  Let \(D\) be a digraph and let \((A,B)\) be a cut in \(D\).
  Let \(X \subseteq A\) and \(A' \coloneqq \onecl{X}{D[A]}\).
  Then the porosity of the cut \((A', V(D) \setminus A')\)
  is at most the porosity of the cut \((A,B)\).
\end{lemma}
\begin{proof}
  Let \(F\) be any set of edges hitting all cycles visiting both \(A\) and \(B\).
  We now argue that \(F\) is also a hitting set for cycles visiting both \(A'\) and \(V(D) \setminus A'\).

  Let \(C\) be a cycle visiting a vertex \(a \in A' \subseteq A\) and \(b \in V(D) \setminus A'\).
  This cycle clearly contains a vertex of \(A\).
  Observe, however, that \(C\) cannot be entirely contained in $D[A]$, since then \(b\) would also be a member of the closure of \(X\) within \(D[A]\), that is, of $A'$.
  So \(C\) must visit \(B\), and thus, is a cycle crossing the cut $(A,B)$, that therefore is hit by \(F\).
\end{proof}

\begin{lemma}
  \label{lemma:potato_uncrossing}
  Let \(D\) be a digraph and \(k\) be a non-negative integer.
  Let \(X\) be a subset of vertices of \(D\) which is \(k\)-flow-connected.
  Suppose \((A, B)\) and \((A', B')\) are cuts in \(D\) that both split \(X\), and moreover \(\Abs{\Out{}{A}} = k\) and \(\Abs{\Out{}{A'}} = k\).
  Then one of the following outcomes holds:
  \begin{enumerate}[label=\textbf{(C\arabic*)},leftmargin=1.2cm]
    \item \(A \cap X \subseteq B' \cap X\) or \(B' \cap X \subseteq A \cap X\), or \label{lemma:potato_uncrossing:out_uninteresting}
    \item \((A \cap A', B \cup B')\) and \((A \cup A', B \cap B')\) are both cuts of order \(k\) that split \(X\). \label{lemma:potato_uncrossing:out_uncrossing}
  \end{enumerate}
\end{lemma}
\begin{proof}
  If~\cref{lemma:potato_uncrossing:out_uninteresting} does not hold, then \(A \cap A' \cap X = (A \cap X) \setminus (B' \cap X) \neq \emptyset\) and \(B \cap B' \cap X = (B' \cap X) \setminus (A \cap X) \neq \emptyset\).
  Thus \((A \cap A', B \cup B')\) and \((A \cup A', B \cap B')\) are both cuts in \(D\) that split \(X\).
  Since $X$ is $k$-flow-connected, both these cuts must have order at least \(k\).
  However, by submodularity we have \(\Abs{\Out{}{A \cap A'}} + \Abs{\Out{}{A \cup A'}} \leq \Abs{\Out{}{A}} + \Abs{\Out{}{A'}} = 2k\), so in fact they must both have order exactly \(k\).
  This shows~\cref{lemma:potato_uncrossing:out_uncrossing}.
\end{proof}

\paragraph*{Laminar families.}
We also work with \emph{laminar} families of cuts.
The following definitions are usually applied to cuts in digraphs, but it is convenient to develop the terminology in a slightly more abstract manner.

Let $U$ be some finite universe (typically, the vertex set of a digraph).
A \emph{cut} of $U$ is an ordered bi-partition of $U$, i.e.~a pair $(A,B)$ of disjoint subsets of $U$ with $U=A\cup B$.
A cut is \emph{non-trivial} if both $A$ and $B$ are non-empty.
We say that two cuts $(A,B)$ and $(A',B')$ \emph{cross} if all of the four sets
\begin{equation*}
	A\cap A',\quad A\cap B',\quad B\cap A',\quad  B\cap B',
\end{equation*}
are non-empty.
Note that non-trivial cuts $(A,B)$ and $(A',B')$ do not cross if and only if among the four sets $A,B,A',B'$ there is some non-trivial inclusion relation: one of the sets contains another one.
A family $\FF$ of cuts of $U$ is \emph{laminar} if the cuts of $\FF$ are non-trivial and pairwise non-crossing.
Note that for technical reasons, we disallow trivial cuts in laminar families.

It is also useful to \say{localise} the above definition to some subset of the universe. 
If $S\subseteq U$ is a non-empty subset of the universe, then we say that a cut $(A,B)$ of $U$ \emph{splits} $S$ if both $A\cap S$ and $B\cap S$ are non-empty. 
We also define the cut \emph{induced} on $S$ by $(A,B)$ as $(A,B)|_S\coloneqq (A\cap S,B\cap S)$.
Note that $(A,B)$ splits $S$ if and only if $(A,B)|_S$ is a non-trivial cut of $S$. For a family $\FF$ of cuts of $U$, define
\begin{equation*}
	\FF_S\coloneqq \{(A,B)|_S\colon (A,B)\in \FF)\}.
\end{equation*}
We say that $\FF$ is \emph{laminar on $S$} if $|\FF|=|\FF_S|$ and $\FF|_S$ is a laminar family of cuts of $S$.
Note that the first condition is equivalent to saying that the elements of $\FF$ induce pairwise different cuts on $S$, while the latter condition entails that all the elements of $\FF$ split $S$ (for trivial cuts are disallowed in laminar families).
We also say that two cuts $(A,B),(A',B')$ of $U$ \emph{cross on $S$} if $(A,B)|_S,(A',B')|_S$ are crossing cuts of $S$.


\newcommand*{\familyTree}[1]{\mathsf{LaminarTree}(#1)}
A laminar family \(\FF\) of cuts of a universe $U$ naturally yields a rooted tree $T$ on \(\Abs{\FF}+1\) nodes constructed as follows.
Fix an element \(v\) of \(U\).
We define \(\familyTree{\FF,v} \coloneqq \Brace{T,\lambda, \eta}\) where \(T\) is a rooted tree, \(\lambda\colon \V{T} \to 2^{U}\), and \(\eta \colon E(T) \rightarrow \FF\) is a bijection.
As \(\Brace{A,B}\) is a cut if and only if \(\Brace{B,A}\) is a cut, up to inverting every cut in $\FF$ we may assume that \(v \not\in A\) for all members \(\Brace{A,B} \in \FF\).
Now we define a partial order \(\preccurlyeq\) on the elements of \(\FF\): we set \((A,B) \preccurlyeq (A',B')\) if \(A \subseteq A'\).
As \(A \not\subseteq A'\) and \(A' \not\subseteq A\) imply \(A \cap A' = \emptyset\), we get that \(\preccurlyeq\) is a tree-order (a partial order where any two incomparable elements do not have any common smaller element).
Therefore, we may use \(\preccurlyeq\) to build \(T\) as follows:
\begin{align*}
	\V{T} &\coloneqq \Set{x_F \mid F \in \FF} \cup \Set{x_{\variablestyle{max}}},\\
	\E{T} &\coloneqq \Set{(x_F,x_{F'}) \mid F' \preccurlyeq F \text{ and there is no }F''\in \FF\setminus\Set{F,F'}\text{ such that } F' \preccurlyeq F'' \preccurlyeq F} \\
	& \qquad \cup \Set{(x_{\variablestyle{max}},x_F) \mid F \text{ is a maximal element of }\preccurlyeq}.
\end{align*}
We choose \(x_{\variablestyle{max}}\) to be the root of \(T\).
The function \(\eta\) is defined by setting, for every edge \((x,y)\) of~$T$, \(\eta((x,y))\) to be the unique cut \(F \in \FF\) such that \(y = x_F\).
The function \(\lambda\) is set accordingly: \(\lambda(x_{\variablestyle{max}}) = U\)
and \(\lambda(x_F) = A\), where \((A,B) = F\) and \(v \not \in A\).

It is easy to see that no matter which element \(v\) is chosen at the beginning of the construction, the undirected tree obtained from $T$ by forgetting the orientation of the edges is always the same.
However, different choices of $v$ lead to different rootings of this tree and different mappings $\lambda$ and $\eta$, assigning the cuts of $\FF$ and their sides to the edges and the nodes of $T$. 

\subsection{Important cuts}

In our proof of \cref{lemma:potato}, we crucially rely on the notion of an \emph{important cut}, defined below.

\begin{definition}
 Let $D$ be a digraph and $X,Y\subseteq V(D)$ be disjoint subsets of vertices. We say that an \(X\)-\(Y\)-cut $(A,B)$ is \emph{important} if the following conditions hold:
 \begin{itemize}
 	\item every vertex of $A$ is reachable from $X$ by a path whose all vertices belong to $A$; and
 	\item there is no $X$-$Y$-cut $(A',B')$ such that \(A \subsetneq A'\), every vertex of $A'$ is reachable from $X$ by a path whose all vertices belong to $A'$, and \(\Abs{\Out{}{A'}} \leq \Abs{\Out{}{A}}\).
 \end{itemize}
\end{definition}

We remark that usually in the literature, it is the set of edges $F=\Out{}{A}$ crossing the cut in the relevant direction that is called an important cut, rather than the bipartition $(A,B)$. However, the two definitions are easily equivalent.

The key property of important cuts is that there are not so many of them of small order. More precisely, we rely on the following upper bound due to Chen, Liu, Lu, O'Sullivan, and Razgon~\cite{importantSep2008}.
See also~\cite[Section~8.2]{PAbook} for an extended discussion, because the result is implicit in~\cite{importantSep2008}.

\begin{theorem}[{\cite{importantSep2008}}]
  \label{thm:important_cuts}
  Let \(D\) be a digraph, \(k\) be a non-negative integer, and \(X, Y\) be disjoint subsets of \(V(D)\).
  Then,  in \(D\) there are at most \(4^k\) distinct important \(X\)-\(Y\)-cuts $(A,B)$ with \(\Abs{\Out{}{A}}\leq k\).
\end{theorem}

Note that in the discussion above, the definition of an important cut is not symmetric: in an important $X$-$Y$-cut $(A,B)$ we care about the number of edges going from $A$ (the part containing $X$) to $B$ (the part containing~$Y$), and about the reachability of every vertex of $A$ from $X$.
To facilitate speaking about the other direction as well, we consider the following definition:
a \emph{$k$-key $X$-$Y$-cut} is an $X$-$Y$-cut \((A,B)\) that is either important in $D$ and satisfies $\Abs{\Out{}{A}}\leq k$, or is important in the digraph $\Rev(D)$ obtained from $D$ by reverting the orientation of every edge and satisfies $\Abs{\In{}{A}}\leq k$. Note that $(A,B)$ being important in $\Rev(D)$ means that (i) every vertex of $A$ can reach $X$ by a path entirely contained in $D[A]$, and (ii) there is no \(X\)-\(Y\)-cut \((A',B')\)
such that \(A \subsetneq A'\), every vertex of $A'$ can reach $X$ by a path entirely contained in $D[A']$, and \(\Abs{\In{}{A'}} \leq \Abs{\In{}{A}}\).

\begin{corollary}\label{coro:key_cuts}
  Let \(D\) be a digraph, \(k\) be a non-negative integer, and \(X, Y\) be disjoint subsets of \(V(D)\).
  Then, in \(D\) there are at most \(2\cdot4^k\) distinct $k$-key \(X\)-\(Y\)-cuts.
\end{corollary}

\subsection{Partial tree-partitions}
\label{sec:partial_potatoes}

The most crucial step in the proof of \nameref{lemma:potato} is to decompose \(k\)-flow connected parts of the given graph in a tree-like fashion each and then re-assemble these decompositions into one for the entire graph.
To that end, in this section we will show that a \(k\)-flow connected set can be decomposed into \((k + 1)\)-flow connected parts in a tree-like fashion,
and expose each cut of the input set exposed by the decomposition tree as a bounded porosity cut in the entire host digraph.



Let us define more formally what type of object we want to build.
Intuitively, we want the final \say{local} shape to match the target tree partition on a subset of vertices (so the near-partition simply
partitions the subset), but we need to record global connectivity information. We do it by labelling edges of the tree 
with cuts in the host digraph in a way that stays consistent with the near-partition of the subset.
We need this edge labelling by cuts and the near-partition to be consistent with each other.
%
%

\begin{definition}
	Given a digraph \(D\) and a set of vertices \(S \subseteq \V{D}\) a \emph{partial tree-partition of \(S\) in \(D\)} is a tuple \(\mathcal{T} =(T,\beta,\eta)\) such that
	\begin{itemize}
		\item \(T\) is a rooted tree,
		\item \(\beta:\V{T} \to 2^S\) is a near-partition of \(S\),
		\item \(\eta:\E{T} \to \FF\) is a bijection, where \(\FF\) is some family of cuts in \(D\),
    \item \(\FF\) is a family of cuts in \(D\) that is laminar on \(S\),
    \item if \(T\) has an edge, then \(T\) is the underlying tree of \(\familyTree{\FF|_S, s}\), for some vertex \(s\) of~\(S\)
		\item for every edge \(e \in E(T)\) we have \(\beta(\treeLower{(T)}{e}) \subseteq A\),
		and \(\beta(\treeUpper{(T)}{e}) \subseteq B\), where  \((A,B) = \eta(e)\).
	\end{itemize}
\end{definition}
Note that for the final bullet point, the ordering of the cut \((A,B)\) matters.

Let \((T,\beta,\eta)\) be a partial-tree partition of some set \(S\) in a digraph \(D\),
and let \(\FF \coloneq \eta(E(T))\).
We consider a natural ancestor/descendant relation on the cuts of $\FF$ induced by $T$, defined as follows.
Let $\kappa_{\mathcal{T}}\colon \FF\to V(T)$ be defined by the function composition \(\kappa_{\mathcal{T}} \coloneqq \head{} \cdot \eta^{-1} \).
Thus, $\kappa_{\mathcal{T}}$ is a bijection between $\FF$ and the non-root nodes of $T$.
For two cuts $(A,B),(A',B')\in \FF$, we say that $(A',B')$ is an \emph{ancestor} of $(A,B)$ in $T$, and that
$(A,B)$ is a \emph{descendant} of $(A',B')$ in $T$, if $\kappa_{\mathcal{T}}(A',B')$ is an ancestor of $\kappa_{\mathcal{T}}(A,B)$ in $T$.
We also say that $(A,B)$ and $(A',B')$ are \emph{incomparable} in $T$ if neither of them is an ancestor of the other.
%

In the next lemma, we establish that for \(k\)-flow-connected sets
there always exists a partial-tree partition.

\begin{lemma}
	\label{lem:partial_potatoes}
	Let \(D\) be a digraph and \(S \subseteq \V{D}\) a \(k\)-flow-connected set.
	Then, there exists a partial tree-partition \((T, \beta, \eta)\) of \(S\) in \(D\),
  where each \(\beta(x)\), for \(x \in V(T)\) is a \((k + 1)\)-flow-connected set.
	Furthermore, all cuts \(\eta(e)\), for \(e \in E(T)\), have order \(k\).
\end{lemma}
\begin{proof}
	If \(S = \emptyset\), then \(S\) is already \((k + 1)\)-flow-connected,
	and we simply set
	\begin{itemize}
		\item \(T\) to be a single-node tree,
		\item \(\beta\) to map the single node to the empty set,
		\item \(\eta\) to be the empty mapping, and
	\end{itemize} 
	Let us assume henceforth that \(S\) is not empty, and fix any \(s \in S\).
	Note that since \(S\) is $k$-flow-connected, every cut of $D$ that splits $S$ has order at least~$k$.
	
	Let \(\FF\) be an inclusion-wise maximal family of cuts of $D$ of order \(k\) that is laminar on \(S\).
	By reverting every cut in \(\FF\) if necessary we may assume that for each $(A,B)\in \FF$, we have $s\notin A$ (and thus $s\in B$).
	Let \(\sim_\FF\) be the following equivalence relation on \(S\):
	\(u \sim_\FF v\) if and only if no cut in \(\FF\) splits \(\Set{u,v}\).
	Let \(\XX\) be the set of equivalence classes under~\(\sim_\FF\).
	
	\begin{claim}
		\label{potlemma:claim:connectivity}
		Every element \(X\) of \(\XX\) is \((k + 1)\)-flow-connected in $D$.
	\end{claim}
	\begin{claimproof}
		Suppose this not the case and we have \(u,v \in X\) and a cut
		\((A, B)\) of order \(k\) in \(D\) such that \(u \in A\) and  \(v \in B\).
		Without loss of generality, assume that \(\Abs{\Out{}{A}} = k\).
		Since no cut in \(\FF\) splits \(\Set{u,v}\) while \((A, B)\) does,
		this in particular means that \((A, B) \not\in \FF\) and \((B,A)\notin \FF\).
		Let us exploit this
		fact to exhibit a larger family of \(k\)-cuts that is laminar on \(S\).
		
		We construct a sequence of cuts \(F_0, F_1, \dots\).
		Let us start with \(F_0 = (A, B)\).
		We ensure that our sequence maintains the following invariants on each cut \(F_i = (A_i, B_i)\):
		\begin{itemize}
			\item \(\Abs{\Out{}{A_i}}= k\);
			\item \(u \in A_i\) and \(v \in B_i\); and
			\item when \(i > 0\), if a cut \((A',B') \in \FF\) does not cross
			with \(F_{i - 1}\) on \(S\), then it also does not cross with \(F_i\) on \(S\).
		\end{itemize}
		Suppose we have already defined \(F_i = (A_i, B_i)\).
		If the family \(\FF \cup \Set{F_i}\) is laminar on \(S\), then the process terminates and we have arrived at a contradiction with the maximality of $\FF$; note here that $F_i$ splits $\{u,v\}$ while no other cut of $\FF$ does, so the cuts of $\FF \cup \Set{F_i}$ have pairwise different intersections with~$S$.
		So, suppose that \((A_i, B_i)\) crosses some \((A', B') \in \FF\) on $S$. Without loss of generality assume that
		\(\Abs{\Out{}{A'}} = k\).
		Apply~\cref{lemma:potato_uncrossing} to \((A_i , B_i)\),
		\((A', B')\) and the set \(S\).
		In outcome~\cref{lemma:potato_uncrossing:out_uninteresting} we get either:
		\(A_i \cap S \subseteq B' \cap S\), or \(B' \cap S \subseteq A_i \cap S\).
		Either way, we get a contradiction that the two cuts were crossing on \(S\).
		Thus we must get outcome~\cref{lemma:potato_uncrossing:out_uncrossing}.
		If \(u \in A'\) then let us set \(F_{i + 1} =(A_{i+1,B_{i+1}})\coloneqq (A_i \cap A', B_i \cup B')\),
		and otherwise let us set \(F_{i + 1} =(A_{i+1,B_{i+1}})\coloneqq (A_i \cup A', B_i \cap B')\).
		By~\cref{lemma:potato_uncrossing}, \(F_{i + 1}\) is a cut of order k.
		By construction and by the fact that $(A',B')\in \FF$ does not split $\{u,v\}$, we have \(u \in A_{i+1}\) and \(v \in B_{i+1}\). 
		Thus first two invariants hold. We also note that, clearly, \(C_{i + 1}\)
		does not cross \((A', B')\).
		We must argue the third invariant.
		
		Let \((A'', B'')\in \FF\) be a cut that does not cross \(F_i\) on \(S\).
		Without loss of generality, let \(A''\) be the side that does not contain \(u\) and \(v\).
		Assume first that \(A'\) is the side of \((A', B')\) not containing \(u\) and \(v\).
		Then, by construction, \(A_{i + 1} = A_{i} \cup A'\).
		If \(A_{i}\cap S\) contains \(A''\cap S \) then so does \(A_{i + 1}\cap S\), and $(A'',B'')$ and $F_{i+1}$ do not cross on $S$.
		So \(A_i\cap S\) and \(A''\cap S\) are disjoint. 
		Now, there are three cases to consider: 
		\begin{itemize}
			\item If \(A'\cap S\) and \(A''\cap S\) are disjoint as well, then so are \(A_{i + 1}\cap S\) and \(A''\cap S\), and again $(A'',B'')$ and $F_{i+1}$ do not cross on $S$.
			\item If \(A'\cap S \subseteq A''\cap S\), then \(A_i\cap S\) and \(A'\cap S\) are disjoint (due to \(A_i\cap S\) and \(A''\cap S\) being disjoint), so $F_i$ and $(A',B')$ were not crossing on $S$ to begin with; a contradiction.
			\item Finally, if \(A''\cap S \subseteq A'\cap S\), then \(A''\cap S \subseteq A_{i + 1}\cap S\) and again, $(A'',B'')$ and $F_{i+1}$ do not cross on~$S$.
		\end{itemize}    
		The case in which \(B'\) is the side of \((A', B')\) not containing \(u\) and \(v\) follows by a symmetric argument.
		
		Since in each step the cut \(F_{i + 1}\) crosses on $S$ with at least one cut less than \(F_{i}\) does, the process must terminate in a finite number of steps giving
		the desired contradiction.
	\end{claimproof}
	
	As $\FF$ is laminar on $S$, we have that \(\FF|_S\) is laminar, hence we may define
	\begin{equation*}
		(T,\lambda,\eta') \coloneqq \familyTree{\FF|_S,s}.
	\end{equation*}
	Let $r$ be the root of $T$. 
	We endow $T$ with a bag function \(\beta \colon V(T)\to 2^{S}\) defined as follows:
	\begin{equation*}
		\beta(t) \coloneqq \lambda(t) \setminus \bigcup_{c\colon \text{ child of } t\text{ in }T} \lambda(c),\qquad\text{for every }t\in V(T)\text{.}
	\end{equation*}
	\begin{claim}\label{potlemma:near-partition}
		For every node \(t \in V(T)\), the set \(\beta(t)\) is either empty or satisfies
		\(\beta(t) \in \XX\).
		Moreover, for every element \(X \in \XX\) there is an unique node
		\(t_X \in V(T)\) such that \(\beta(t_X) = X\).
	\end{claim}
	\begin{claimproof}
		First, we verify that \(\{\beta(t)\colon t\in V(T)\}\) is a near-partition of \(S\).
		Consider any \(u \in S\) and suppose \(u \in \beta(t)\)
		and \(u \in \beta(t')\) for some distinct $t,t'\in V(T)$.
		It follows that \(u \in \lambda(t)\) and \(u \in \lambda(t')\),
		so either \(t\) is a descendant of \(t'\)
		or \(t'\) is a descendant of \(t\).
		Without loss of generality, assume the former.
		Let \(t''\) be the child of \(t'\) that is an ancestor of $t$.
		It follows that \(u \in \lambda(t'')\subseteq \lambda(t')\), and
		thus, by construction, \(u \not \in \beta(t)\); a contradiction.
		So \(\{\beta(t)\colon t\in V(T)\}\)	is a near-partition of \(S\).
		
		Note that for any node \(t \in V(T)\) we have
		\[\bigcup_{t'\colon \text{ a descendant of } t\text{ in }T} \beta(t') = \lambda(t)\text{.}\]
		Hence, if \(e=(t,c) \in E(T)\) (so \(t\) is the parent of \(c\)) and
		\((A,B) = \eta((t,c)) \in \FF|_S\) with \(s \not \in A\), then we have
		\(A = \lambda(z)\) and consequently
		\(A = \beta(\treeLower{(T)}{e})\), \(B = \beta(\treeUpper{(T)}{e})\).
		
		Observe now, that if $u \sim_\FF v$, then there must exist $t\in V(T)$ with $u,v\in \beta(t)$.
		Indeed, otherwise there would be an edge $e\in E(T)$ such that $u$ and $v$ belong to the opposite sides of the cut $(\beta(\treeLower{(T)}{e}),\beta(\treeUpper{(T)}{e}))=\eta'(e)$, implying that $\{u,v\}$ would be split by $\eta'(e)\in \FF|_S$;
		a contradiction to $u\sim_\FF v$.
		Consequently, every element of $\XX$ is entirely contained in a single element of the near-partition \(\{\beta(t)\colon t\in V(T)\}\).
		
		It now remains to show that if $u,v\in S$ are such that \(u \nsim_\FF v\), then $u,v$ belong to different parts of the near-partition \(\{\beta(t)\colon t\in V(T)\}\).
		Since $u\nsim_\FF v$, there is a cut \((A,B) \in \FF|_S\) such that $u$ and $v$ belong to the opposite sides of $(A,B)$.
		Since $\eta$ is a bijection, there is an edge \(e \in E(T)\) such that \(\eta'(e) = (A,B)\).
		It follows that the nodes whose bags contain $u$ and $v$ must belong to different components of \(T \setminus e\), so these nodes in particular must be different.
		This ends the proof.
	\end{claimproof}
	
	Choosing \(\eta\) as the unique extension in \(\FF\) of the cuts in \(\eta'\) to cuts in \(D\) finishes this proof
  (note that this is possible as \(\Abs*{\FF} = \Abs*{\FF|_S}\))
\end{proof}

Note that \((T,\beta,\eta)\) obtained by the proof above yields a decomposition of \(S\) that behaves well in \(D[S]\), however the cuts of $\FF = \eta(\E{T})$ are laminar only on $S$, and not necessarily laminar in all of $D$.
The next lemma constructs a new family of cuts extending \(\FF|_S\) but being well-behaved with respect to \(T\) in the entire graph.

\begin{lemma}
	\label{lem:partial_potatoes_uncrossed}
	Let \(D\) be a digraph, and \(S \subseteq \V{D}\) a \(k\)-flow-connected set.
  Let \((T, \beta, \eta')\) be a partial tree-partition of \(S\),
  such that all cuts \(\eta'(e)\), for \(e \in E(T)\), have order \(k\).
	Then, there exists a partial tree-partition \((T, \beta, \eta)\) of \(S\) in \(D\),
  such that the family of cuts \(\eta(E(T))\) is a laminar family
  of cuts.
	Furthermore, the porosity of \(\eta(e)\), for \(e \in E(T)\), is at most \(k^2\cdot 4^{k+1} + 2k\).
\end{lemma}
\begin{proof}
  If \(S = \emptyset\) there is nothing to be done.
  Otherwise, let \(\FF' \coloneq \eta'(E(T))\) and let \(s \in S\) be any
  vertex such that \(T\) is the underlying tree of \(\familyTree{\FF'|_S,s}\).
	The first step is to clean the family \(\FF'\) so that cuts having the same \say{direction} are also laminar in \(D\).
	\begin{claim}
		\label{potlemma:claim:same_orientation_laminar}
		There exists a family \(\FF\) of cuts of $D$ of order \(k\) and a bijection \(\nu: \FF' \to \FF\), such that
		\begin{itemize}
			\item for every \((A,B) \in \FF'\), the cut \((A_{\nw}, B_{\nw}) = \nu((A,B))\in \FF\)
			fulfils
			\begin{equation*}
				A \cap S = A_{\nw} \cap S\qquad\text{and}\qquad B \cap S = B_{\nw} \cap S, \text{ and}
			\end{equation*}
			\item for all \((A, B),(A', B') \in \FF\),
			where \(\nu^{-1}((A,B))\) is a descendant of \(\nu^{-1}((A',B'))\) in \(T\),
			\begin{itemize}
				\item if \(\Abs{\Out{}{A}} = \Abs{\Out{}{A'}} = k\), then \(A \subseteq A'\); and
				\item if \(\Abs{\In{}{A}} = \Abs{\In{}{A'}} = k\), then \(A \subseteq A'\).
			\end{itemize}
		\end{itemize}
	\end{claim}
	\begin{claimproof}
		Consider any \((A,B) \in \FF'\).
		Assume first that \(\Abs{\Out{}{A}} = k\).
		Let \(\mathcal{A}\) be the set of all ancestors \((A',B')\) of \((A,B)\) in \(T = T'\) such that
		\(\Abs{\Out{}{A'}} = k\) as well. ($\mathcal{A}$ includes $(A,B)$ as well.)
		We set
		\begin{equation*}
			A_{\nw} \coloneqq \bigcap_{(A',B') \in \mathcal{A}} A'\qquad\textrm{and}\qquad
			B_{\nw} \coloneqq V(D) \setminus A_{\nw} \qquad\textrm{and}\qquad \nu((A,B))) \coloneqq (A_{\nw},B_{\nw}).
		\end{equation*}
		In case that \(\Abs{\In{}{A}} = k\), we use a symmetric definition, but take $\mathcal{A}$ to be the set of all ancestors $(A',B')$ of $(A,B)$ satisfying \(\Abs{\In{}{A'}} = k\).
		
		We next show that \((A_{\nw}, B_{\nw})\) is a cut of order \(k\).
		To that end it suffices to work with a single ancestor \((A',B')\) of \((A,B)\), and to assume by symmetry that \(\Abs{\Out{}{A}}=\Abs{\Out{}{A'}} = k\).
		By submodularity,
		\begin{equation*}
			\Abs{\Out{}{A \cup A'}} + \Abs{\Out{}{A \cap A'}} \leq \Abs{\Out{}{A}} + \Abs{\Out{}{A'}} = 2k.
		\end{equation*}
		Since \(\FF'\) is laminar on \(S\), and \((A',B')\) is an ancestor of \((A,B)\),
		it follows that \((A \cup A') \cap S = A' \cap S\) and \(A \cap A' \cap S = A \cap S\).
		As a laminar cut family does not contain trivial cuts, we have
		\(A' \cap S \neq S\), \(A' \cap S \neq \emptyset\) and
		\(A \cap S \neq S\), \(A \cap S \neq \emptyset\).
		Therefore, as \(S\) is \(k\)-flow-connected, neither \(\Abs{\Out{}{A \cup A'}}\)
		nor \(\Abs{\Out{}{A \cap A'}}\) can be smaller than \(k\).
		Thus, from the inequality above it follows that, actually, \(\Abs{\Out{}{A \cup A'}}=\Abs{\Out{}{A \cap A'}}=k\).
		By applying this reasoning to each cut of $\mathcal{A}$ consecutively, we conclude that $\Abs{\Out{}{A_\nw}} = k$, as required.
		
		In particular the above argument also shows \(A \cap S = A_{\nw} \cap S\),
		and \(B \cap S = B_{\nw} \cap S\).
		Now let \(\nu((A, B)) = (A_{\nw},B_{\nw}),\nu((A', B')) = (A'_{\nw},B'_{\nw}) \in \FF\), where \((A,B)\) is a descendant of \((A',B')\) in \(T\).
		If \(\Abs{\Out{}{A_{\nw}}} = \Abs{\Out{}{A'_{\nw}}} = k\) or \(\Abs{\In{}{A_{\nw}}} = \Abs{\In{}{A'_{\nw}\}}} = k\),
		then the set \(\mathcal{A}\) of ancestors used for obtaining \(A_{\nw}\) is a subset of the set \(\mathcal{A}'\) used for obtaining \(A'_{\nw}\).
		Thus, \(A_{\nw} \subseteq A'_{\nw}\).
		%
	\end{claimproof}
	
	Note that \(\familyTree{\FF|_S,s} = \familyTree{\FF'|_S,s}\) for the family $\FF$ yielded by \cref{potlemma:claim:same_orientation_laminar},
  as no intersection with \(S\) has changed.
  So we continue working with the cut family \(\FF\).
	
	For a cut \((A,B) \in \FF\) 
	we define a set $Z(A)\subseteq A$ as follows:
	\begin{itemize}
		\item If \(\Abs{\Out{}{A}} = k\), then \(Z(A)\) is the set of at most \(k\) vertices in \(A\) that are the tails of edges starting in \(A\) and ending \(B\).
		\item If \(\Abs{\In{}{A}} = k\), then \(Z(A)\) is the set of at most \(k\) vertices in \(A\) that are the heads of edges starting in \(B\) and ending in \(A\).
	\end{itemize} 
	
	Consider any \((A,B) \in \FF\).
	For a non-empty set \(Z' \subseteq Z(A) \setminus (S \cap A)\), let \(\CC_{A}(Z')\) be the set of
	all $k$-key \(Z'\)-\((S \cap A)\)-cuts in \(D\).
	(In most cases, we investigate \(\CC_{A}(Z')\) for a
	specific cut \((A,B)\), and refer to it as just \(\CC(Z')\).)
	Also, let us define two sets, \(\mu(A)\) and \(\mu'(A)\), by the following formulas:
	\begin{align*}
		\mu'(A) &\coloneqq A \setminus \bigcup\limits_{z \in Z(A) \setminus (S \cap A)} \ \bigcup\limits_{(E,F) \in \CC_A(\Set{z})} \onecl{Z(A) \cap E}{D[E]}\\
		\mu(A) &\coloneqq \onecl{S \cap A}{D[\mu'(A)]}\text{.}
	\end{align*}
	Note that by definition,
	\[S\cap A\subseteq \mu(A)\subseteq \mu'(A)\subseteq A.\]
	First, let us investigate the porosity of the cut \((\mu(A), V(D) \setminus \mu(A))\).
	
	\begin{claim}\label{claim:potato_cycle_order}
		The porosity of the cut \((\mu(A), V(D) \setminus \mu(A))\) is at most \(2 k^2 \cdot 4^k + k\).
	\end{claim}
	\begin{claimproof}
		By~\cref{coro:key_cuts}, \((\mu'(A), V(D) \setminus \mu'(A))\)
		is a boolean combination of at most \( 2 k \cdot 4^k + 1\) different cuts,
		each of order at most \(k\), and thus also of porosity at most \(k\).
		So by~\cref{lemma:porosity-boolean}, the porosity of
		\((\mu'(A), V(D)\setminus \mu'(A))\) is at most \(2 k^2 \cdot 4^k + k\).
		Finally, by \cref{lemma:porosity-onecl}, the claim follows.
	\end{claimproof}
	
	With \(\mu(A)\) at hand, we almost have a globally laminar family of cuts, as the next two claims observe.
	
	\begin{claim}\label{cl:incomparable}
		If two cuts \((A,B), (A',B') \in \FF\) 
		are incomparable in~\(T\), then \(\mu(A) \cap \mu(A') = \emptyset\).
	\end{claim}
	\begin{claimproof}
		\begin{figure}[h!]
			\centering
			\begin{tikzpicture}[scale=1.2]
				\node (A) at (-2.5,0) {};
				\draw ($(A)+(0:2.9)$) edge[edge,closed,curve through={($(A)+(30:3)$) ($(A)+(90:2)$) ($(A)+(150:3)$) ($(A)+(180:2.7)$) ($(A)+(210:3)$) ($(A)+(270:2)$)}] ($(A)+(-30:3)$);
				\node (A-label) at ($(A)+(90:2.5)$) {\(A\)};
				
				\node (Ap) at (2.5,0) {};
				\draw ($(Ap)+(0:2.7)$) edge[edge,CornflowerBlue,closed,curve through={($(Ap)+(30:3)$) ($(Ap)+(90:2)$) ($(Ap)+(150:3)$) ($(Ap)+(180:2.9)$) ($(Ap)+(210:3)$) ($(Ap)+(270:2)$)}] ($(Ap)+(330:3)$);
				\node[CornflowerBlue] (Ap-label) at ($(Ap)+(90:2.5)$) {\(A'\)};
				
				\node[vertex,violet] (v) at ($(A)!0.5!(Ap)$) {};
				\node[violet] (v-label) at ($(v)+(240:0.4)$) {\(v\)};
				
				\node (mpA-a) at ($(A)+(90:1.3)$) {};
				\node (mpA-b) at ($(A)+(160:1)$) {};
				\node (mpA-c) at ($(A)+(200:2)$) {};
				\draw ($(A)+(0:2.9)$) edge[edge,closed,curve through={($(A)+(30:3)$) ($(A)+(90:2)$) ($(A)+(150:3)$) ($(A)+(180:2.7)$) ($(A)+(210:3)$) ($(A)+(270:2)$)}] ($(A)+(-30:3)$);
				\node (A-label) at ($(A)+(90:2.5)$) {\(A\)};
				\node[LavenderMagenta] (SA) at ($(mpA-b)!0.5!(mpA-c)$) {\(S\cap A\)};
				\draw ($(SA)+(0:0.7)$) edge[edge,LavenderMagenta,closed,curve through={($(SA)+(45:1)$) ($(SA)+(90:0.7)$) ($(SA)+(135:0.7)$) ($(SA)+(180:0.7)$) ($(SA)+(225:1)$) ($(SA)+(270:0.7)$)}] ($(SA)+(315:0.7)$);
				\draw ($(mpA-a)+(90:0.4)$) edge[edge,dashed,closed,curve through={($(mpA-a)+(0:0.4)$) ($(v)+(90:0.4)$) ($(v)+(0:0.2)$) ($(v)+(270:0.4)$) ($(mpA-b)+(270:1.3)$) ($(mpA-c)+(270:0.7)$) ($(mpA-c)+(180:0.7)$)}] ($(mpA-b)+(90:1)$);
				
				\node (mpA-label) at ($(A)+(260:1.2)$) {\(\mu'(A)\)};
				
				\node[vertex,AppleGreen] (w) at ($(Ap)+(20:2)$) {};
				\node[AppleGreen] (w-label) at ($(w)+(90:0.3)$) {\(w\)};
				
				\node[vertex,AppleGreen] (z) at ($(v)+(70:0.6)$) {};
				\node (zp) at ($(z)+(350:0.6)$) {};
				\node[AppleGreen] (z-label) at ($(z)+(280:0.3)$) {\(z\)};
				
				\node (App-a) at ($(Ap)+(140:1.3)$) {};
				\node (App-b) at ($(Ap)+(240:1.5)$) {};
				\node[violet] (App-label) at ($(App-b)$) {\(A''\)};
				\draw ($(App-a)+(90:0.4)$) edge[edge,violet,closed,curve through={($(App-a)!0.5!(w)$) ($(w)+(90:0.65)$) ($(w)+(0:0.4)$) ($(w)+(270:0.6)$) ($(App-b)+(10:0.3)$) ($(App-b)+(270:0.3)$) ($(App-b)+(190:0.3)$) ($(v)+(270:0.3)$) ($(v)+(180:0.3)$) ($(z)+(130:0.3)$) ($(z)+(90:0.4)$)}] ($(App-a)+(135:0.6)$);

				\node (mpAp-c) at ($(w)+(260:2)$) {};
				\node (mpAp-b) at ($(mpAp-c)+(200:1)$) {};
				\node (mpAp-a) at ($(mpAp-b)+(220:1.3)$) {};
				\draw ($(w)+(90:0.8)$) edge[edge,CornflowerBlue,dashed,closed,curve through={($(w)+(0:0.7)$) ($(w)!0.5!(mpAp-c)+(315:0.9)$) ($(mpAp-c)+(350:0.7)$) ($(mpAp-c)+(270:0.6)$) ($(mpAp-b)+(270:0.2)$) ($(mpAp-b)+(180:0.2)$) ($(App-b)+(70:0.3)$) ($(v)+(270:0.2)$) ($(v)+(180:0.2)$) ($(z)+(130:0.2)$) ($(z)+(90:0.2)$)}] ($(App-a)$);
				\node[CornflowerBlue] (mpAp-label) at ($(mpAp-b)+(100:0.5)$) {\(\mu'(A')\)};
				
				\draw ($(w)+(90:0.6)$) edge[edge,LavenderMagenta,closed,curve through={($(w)+(0:0.5)$) ($(w)!0.5!(mpAp-c)+(315:0.7)$) ($(mpAp-c)+(350:0.4)$) ($(mpAp-c)+(250:0.3)$) ($(mpAp-c)+(160:0.5)$)}] ($(w)+(180:0.5)$);

				\node (E-d) at ($(v)+(250:1)$) {};
				\node[PastelOrange] (E-label) at ($(App-b)+(200:1)$) {\(E\)};
				\draw ($(w)+(90:0.8)$) edge[edge,PastelOrange,closed,curve through={($(w)+(0:0.6)$) ($(w)+(270:0.75)$) ($(mpAp-b)+(350:0.2)$) ($(mpAp-a)+(0:0.5)$) ($(mpAp-a)+(270:0.2)$) ($(mpAp-a)+(190:0.3)$) ($(App-b)+(280:0.7)$) ($(App-b)+(225:0.7)$) ($(App-b)+(170:0.7)$) ($(E-d)+(270:0.2)$) ($(E-d)+(180:0.2)$) ($(z)+(130:0.5)$) }] ($(App-a)+(90:0.8)$);

				\draw[edge,AppleGreen,out=45,in=170] (v) to (z);
				\draw[directededge,AppleGreen,out=350,in=180] (z) to (zp.east);
				\draw (zp) edge[edge,AppleGreen,quick curve through={($(zp)+(0:0.2)$) .. ($(v)!0.5!(Ap)+(90:0.3)$) .. ($(Ap)+(90:0.6)$) .. ($(w)!0.5!(Ap)+(90:0.3)$)}] (w);
				\draw[directededge,AppleGreen] ($(w)+(185:0.3)$) to (w);
				
				\draw (w) edge[edge,PastelOrange,quick curve through={($(w)+(150:0.4)$) .. ($(Ap)$) .. ($(v)+(350:0.7)$)}] (v.east);
				\draw[directededge,PastelOrange] ($(v)+(350:0.3)$) to (v);
				
				\node[LavenderMagenta,fill=white,rounded corners] (SAp) at ($(mpAp-c)+(82:0.6)$) {\(S\cap A'\)};
				
				\node[PastelOrange] (Q) at ($(Ap)+(315:0.3)$) {\(Q\)};
			\end{tikzpicture}
			\caption{Depiction of objects used in the proof of~\cref{cl:incomparable}.}
			\label{fig:incomparable}
		\end{figure}
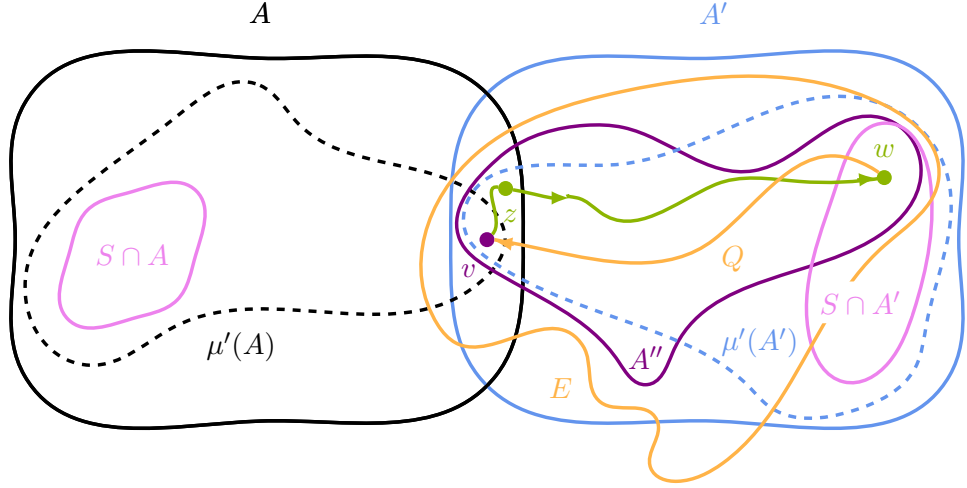
		Refer to~\cref{fig:incomparable} for an illustration of the objects used in this proof.
		
		Suppose for the sake of contradiction that \(v\) is an element of \(\mu(A) \cap \mu(A')\).
		Without loss of generality, assume that \(\Abs{\Out{}{A}} = k\).
		From the definition of closure,
		it follows that there is a path \(P\) from \(v\) to an element \(w \in S \cap A'\) contained
		entirely within \(D[\mu'(A')]\) and a path \(Q\) from \(w\) to \(v\) contained
		entirely within \(D[\mu'(A')]\).
		Since \(S \cap A \cap A' = \emptyset\) (as $(A,B),(A',B')$ are incomparable in the tree \(T\)),
		and \(v \in \mu(A) \subseteq A\), the path \(P\) starts in a vertex
		of \(A\), and ends in a vertex not in \(A\). So it must contain a vertex
		\(z \in Z(A)\). Note that we have \(z \not \in S \cap A\), as \(P\)
		is contained within \(\mu'(A') \subseteq A'\) and \(A' \cap S \cap A = \emptyset\).

		Observe that \((A',B')\) is a \(\Set{z}\)-\((S \cap A)\)-cut, as \(z \in A'\) and \(S \cap A \cap A' = \emptyset\). 
		Define $A''$ to be:
		\begin{itemize}
			\item  the set of all vertices reachable from \(\Set{z}\) within \(D[A']\), if \(\Abs{\Out{}{A'}} = k\);
			\item the set of all vertices that can reach \(\Set{z}\) within \(D[A']\), if \(\Abs{\In{}{A'}} = k\).
		\end{itemize}
		In either case we have $\Abs{\Out{}{A''}} \leq k$ or $\Abs{\In{}{A''}} \leq k$, respectively, and
		$(A'',V(D)\setminus A'')$ is still a \(\Set{z}\)-\((S \cap A)\)-cut. Moreover, the paths
		\(P\) and \(Q\) are contained within \(D[A'']\) regardless of the orientation
		of the cut \((A', B')\); in particular \(v \in A''\).
		
		Let \((E,F)\) be a $k$-key \(\Set{z}\)-\((S \cap A)\)-cut that
		extends the cut \((A'', V(D) \setminus A'')\) (so with \(A'' \subseteq E\)).
		The paths \(P\) and \(Q\) show that \(v\) belongs to
		\(\onecl{\set{z}}{D[E]}\), and so \(v\) does
		not belong to \(\mu'(A)\) by construction of $\mu'(A)$. This is a contradiction.
	\end{claimproof}

	\begin{claim}\label{claim:potato_nesting}
		Let \((A,B), (A',B') \in \FF\). 
		If \((A,B)\) is a descendant of \((A', B')\) in \(T\), and
		\(\Abs{\Out{}{A}} = \Abs{\Out{}{A'}} = k\) or \(\Abs{\In{}{A}} = \Abs{\In{}{A'}} = k\),
		then \(\mu(A) \subseteq \mu(A')\).
	\end{claim}
	\begin{claimproof}
		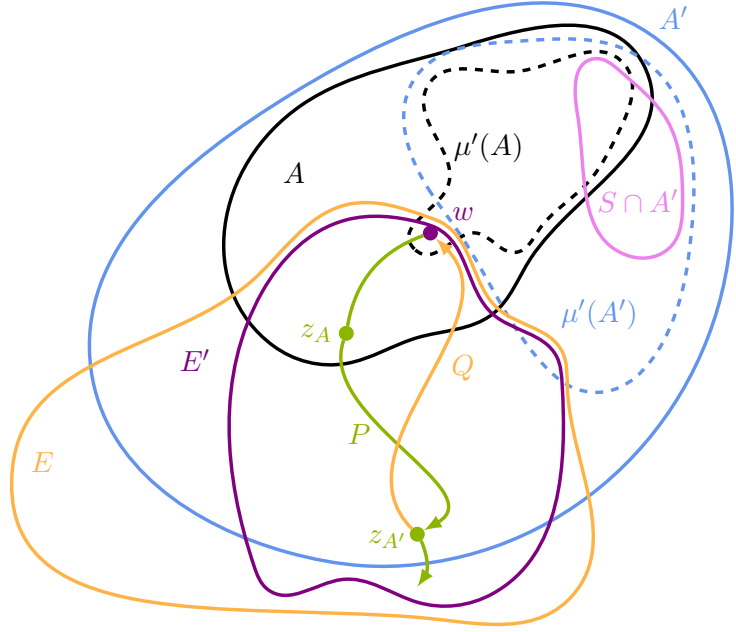
\begin{figure}[!ht]
			\centering
			\begin{tikzpicture}
				\node (Ap) at (0,0) {};
				
				\node[vertex,AppleGreen] (zA) at ($(Ap)+(200:1)$) {};
				\node[AppleGreen] (zA-label) at ($(zA)+(180:0.4)$) {\(z_A\)};
				
				\node[vertex,AppleGreen] (zAp) at ($(Ap)+(270:3)$) {};
				\node[AppleGreen] (zAp-label) at ($(zAp)+(190:0.4)$) {\(z_{A'}\)};
				
				\node (Ap-z) at ($(Ap)+(190:4)$) {};
				
				\node[vertex,violet] (w) at ($(Ap)+(80:1)$) {};
				\node[violet] (w-label) at ($(w)+(30:0.5)$) {\(w\)};
				\node (w-n) at ($(w)+(70:0.6)$) {};
				\node (w-s) at ($(w)+(0:0.6)$) {};
				\node (w-np) at ($(w-n)+(100:1.4)$) {};
				\node (w-sp) at ($(w-s)+(335:0.6)$) {};
				
				\node (A-a) at ($(zA)+(170:0.7)$) {};
				\node (A-x) at ($(Ap)+(355:1)$) {};
				\node (A-b) at ($(w)+(135:2)$) {};
				\node (A-c) at ($(w)+(40:3)$) {};
				
				\node (Ep-c) at ($(zAp)+(270:0.7)$) {};
				\node (Ep-b) at ($(Ep-c)+(180:1)$) {};
				\node (Ep-a) at ($(Ep-b)+(170:1)$) {};
				
				\node (E-a) at ($(Ap-z)+(225:2)$) {};
				\node (E-b) at ($(Ep-c)+(355:2)$) {};
				
				\node (SA-b) at ($(A-c)+(290:2)$) {};
				
				\node (mpAp-y) at ($(A-x)+(315:1)$) {};
				
				\draw ($(zAp)+(270:0.4)$) edge[edge,CornflowerBlue,closed,curve through={ ($(Ep-b)+(45:0.4)$) ($(Ep-a)+(45:0.4)$) ($(Ap-z)+(200:0.4)$) ($(A-b)+(135:0.6)$) ($(A-c)+(30:1)$)}] ($(SA-b)+(0:1)$);
				\node[CornflowerBlue] (Ap-label) at ($(A-c)+(45:1.3)$) {\(A'\)};
				
				\draw ($(Ap)+(270:0.4)$) edge[edge,closed,curve through={ ($(zA)+(270:0.4)$) ($(A-a)+(225:0.4)$) ($(A-b)+(135:0.4)$) ($(A-b)!0.5!(A-c)+(90:0.8)$) ($(A-c)+(30:0.6)$) ($(w-sp)+(335:0.2)$)}] (A-x);
				\node (A-label) at ($(A-a)+(90:2)$) {\(A\)};
				
				\draw ($(w)+(225:0.3)$) edge[edge,dashed,closed,curve through={ ($(w-n)$) ($(w-np)$) ($(w-np)!0.5!(A-c)+(90:0.2)$) ($(A-c)+(45:0.4)$) ($(A-c)+(290:0.6)$) ($(w-sp)!0.5!(A-c)+(335:0.4)$) ($(w-sp)+(45:0.4)$) ($(w-sp)$)}] ($(w-s)$);
				\node (mA-label) at ($(w-n)+(45:0.8)$) {\(\mu'(A)\)};
				
				\draw ($(w-n)+(190:0.2)$) edge[edge,dashed,CornflowerBlue,closed,curve through={ ($(w-np)+(135:0.2)$) ($(w-np)!0.5!(A-c)+(90:0.6)$) ($(A-c)+(45:0.6)$) ($(SA-b)+(315:0.6)$) ($(mpAp-y)+(325:0.5)$) (mpAp-y)}] ($(w-s)+(190:0.2)$);
				\node[CornflowerBlue] (mAp-label) at ($(mpAp-y)+(45:1)$) {\(\mu'(A')\)};
				
				\draw ($(A-c)+(110:0.4)$) edge[edge,LavenderMagenta,closed,curve through={($(A-c)+(45:0.3)$) ($(A-c)!0.5!(SA-b)+(45:0.7)$) ($(SA-b)+(45:0.5)$) ($(SA-b)+(315:0.3)$) ($(SA-b)+(225:0.5)$)}] ($(A-c)+(225:0.5)$);
				\node[LavenderMagenta] at ($(A-c)!0.8!(SA-b)+(180:0.1)$) {\(S\cap A'\)};
				
				\draw ($(w)+(45:0.1)$) edge[edge,violet,closed,curve through={ ($(A-x)$) ($(mpAp-y)+(45:0.2)$) ($(mpAp-y)+(315:0.3)$) ($(E-b)!0.5!(Ep-c)$) ($(Ep-c)+(270:0.2)$) ($(Ep-b)+(90:0.1)$) ($(Ep-a)+(270:0.2)$) ($(Ep-a)+(150:0.2)$) ($(A-a)+(135:0.6)$) ($(zA)+(90:1.5)$) ($(w)+(135:0.2)$)}] ($(w)+(90:0.1)$);
				\node[violet] (Ep-label) at ($(Ap-z)+(0:1)$) {\(E'\)};
				
				\draw ($(w)+(45:0.2)$) edge[edge,PastelOrange,closed,curve through={ ($(A-x)+(45:0.1)$) ($(mpAp-y)+(45:0.3)$) ($(mpAp-y)+(315:0.4)$) ($(E-b)$) ($(Ep-c)+(270:0.4)$) ($(E-a)$) ($(A-a)+(135:0.8)$) ($(zA)+(90:1.7)$)}] ($(w)+(90:0.2)$);
				\node[PastelOrange] (E-label) at ($(E-a)+(10:0.4)$) {\(E\)};
				
				\draw[edge,AppleGreen,out=200,in=80] (w) to (zA);
				\draw[directededge,AppleGreen,out=250,in=30] (zA) to (zAp);
				
				\draw[directededge,AppleGreen,bend left] (zAp) to (Ep-c.center);
				
				\draw[directededge,PastelOrange,out=135,in=315] (zAp) to (w);
				
				\node[AppleGreen] (P) at ($(zA)!0.5!(zAp)+(180:0.3)$) {\(P\)};
				\node[PastelOrange] (Q) at ($(mpAp-y)+(180:1.1)$) {\(Q\)};
			\end{tikzpicture}
			\caption{Depiction of objects used in the proof of~\cref{claim:potato_nesting}.}
      \label{fig:potato_nesting}
		\end{figure}
    Refer to \cref{fig:potato_nesting} for an illustration of the objects used in this proof.
    
		It suffices to prove that \(\mu'(A) \subseteq \mu'(A')\).
		Without loss of generality, let us assume that \(\Abs{\Out{}{A}} = \Abs{\Out{}{A'}} = k\).
		By the property achieved in~\cref{potlemma:claim:same_orientation_laminar}, we have \(A \subseteq A'\).
		
		For the sake of contradiction, suppose that there exists a vertex \(w \in \mu'(A) \setminus \mu'(A')\).
		Note that this in particular means that \(w \in A \subseteq A'\).
		Since \(w \not \in \mu'(A')\), there must be (as per definition of \(\mu'(A')\))
		an element \(z' \in Z(A') \setminus (S \cap A')\) and a $k$-key cut \((E,F) \in \CC_{A'}(\Set{z'})\),
		where \(S \cap A' \subseteq F\), such that
		\(w \in \onecl{Z(A') \cap E}{D[E]}\).
		As \(w \in \onecl{Z(A')}{D[E]}\), there are two paths \(P\) and \(Q\), leading from \(w\) to some \(z_{A'} \in Z(A')\) and back to $w$, respectively, that are entirely contained within \(D[E]\).
		Consider a path $P'$ obtained by extending \(P\) by an edge exiting \(A'\) whose tail is \(z_{A'}\).
		Then $P'$ starts in \(w\) and ends outside of \(A'\) which is a superset of \(A\).
		Hence, $P'$ must pass through a vertex from \(Z(A)\), and this vertex cannot be the
		last vertex of $P'$.
		Call this vertex \(z_A\).
		Since \(z_A \in E\) and \(S \cap A \subseteq S \cap A' \subseteq F\), it follows that \(z_A \not \in S \cap A\).
		Observe then, that \((E,F)\) is an \(\Set{z_A}\)-\((S \cap A)\)-cut of order at most \(k\).
		If \(\Abs{\Out{}{E}} \leq k\), then let \(E'\) be the set of all
		vertices reachable from \(z_A\) within \(D[E]\); and if \(\Abs{\In{}{E}} \leq k\),
		then let \(E'\) be the set of all vertices that can reach \(z_A\) within \(D[E]\).
		Note that the order of the cut \((E', V(D) \setminus E')\) is still at most $k$: in the first case we have \(\Abs{\Out{}{E'}}\leq \Abs{\Out{}{E}} \leq k\) and in the second case we have \(\Abs{\In{}{E'}}\leq \Abs{\In{}{E}} \leq k\). Moreover, regardless of the orientation of the cut \((E,F)\),
		the paths \(P\), \(Q\) are contained within \(D[E']\); in particular \(w \in E'\).
		Let \((E'', F'') \in \CC_{A}(\Set{z_A})\) be a $k$-key \(\Set{z_A}\)-\((S \cap A)\)-cut
		that extends \((E', V(D) \setminus E')\) (so with $E''\supseteq E'$).
		The paths \(P, Q\) witness that \(w \in \onecl{\Set{z_A}}{D[E]} \subseteq \onecl{Z(A) \cap E''}{D[E'']}\).
		But this implies that \(w \not \in \mu'(A)\), which contradicts our assumption.
	\end{claimproof}
	
	Finally, we can define a family of cuts that is globally laminar and of bounded  porosity.
	Consider any \((A,B) \in \FF\).
	Let \(\mathcal{A} \subseteq \FF\) be the set of all ancestors of \((A,B)\) in the tree \(T\).
	(In particular \((A,B) \in \mathcal{A}\), and $s\notin A'$ for each $(A',B')\in \mathcal{A}$.)
	We define
	\begin{equation*}
		\sigma(A) \coloneqq \bigcap_{(A',B') \in \mathcal{A}} \mu(A')\text{.}
	\end{equation*}
	It might seem at first glance that the set \(\sigma(A)\) is unwieldy to work with,
	as it is defined as an intersection of an unbounded number of sets.
	However, thanks to~\cref{claim:potato_nesting}, each set \(\sigma(A)\) is an intersection of at most two sets of the form \(\mu(A')\).
	We use this observation to establish a bound on the porosity of the cuts imposed by the sets $\sigma(A)$.
	
	\begin{claim}\label{cl:sigma-porosity}
		For every cut $(A,B)\in \FF$, the porosity of the cut $(\sigma(A),V(D)\setminus \sigma(A))$ is at most \(k^2 \cdot 4^{k + 1} + 2k\).
	\end{claim}
	\begin{claimproof}
		Let $\mathcal{A}^{\mathsf{out}}\coloneqq \{(A',B')\in \mathcal{A} \mid \Abs{\Out{}{A'}}=k\}$ and\mbox{ 
			$\mathcal{A}^{\mathsf{in}}\coloneqq \{(A',B')\in \mathcal{A} \mid \Abs{\In{}{A'}}=k\}$.}
		Observe that by \cref{claim:potato_nesting}, all the elements of the family $\{\mu(A')\colon (A',B')\in \mathcal{A}^{\mathsf{out}}\}$ are pairwise comparable by inclusion; and the same also holds for the family $\{\mu(A')\colon (A',B')\in \mathcal{A}^{\mathsf{in}}\}$.
		Therefore, $\sigma(A)$ is the intersection of at most two sets of the form $\mu(A')$ for some $(A',B')\in \mathcal{A}$.
		As by~\cref{claim:potato_cycle_order}, each such cut $(\mu(A'),V(D)\setminus \mu(A')))$ has porosity at most \(2k^2 \cdot 4^{k} + k\), it follows from~\cref{lemma:porosity-boolean} that the porosity of $(\sigma(A),V(D)\setminus \sigma(A))$ is at most \(4k^2 \cdot 4^k + 2k=k^2 \cdot 4^{k + 1} + 2k\).
	\end{claimproof}

	
	For a cut $(A,B)\in \FF$, we define $\sigma((A,B))\coloneqq (\sigma(A),V(D)\setminus \sigma(A))$.
	Note that by construction we have $S\cap A=S\cap \sigma(A)$, hence $(A,B)|_S=\sigma((A,B))|_S$.
	Therefore, if we define
	\begin{equation*}
		\sigma(\FF)\coloneqq \{\sigma((A,B))\colon (A,B)\in \FF\},
	\end{equation*}
	then $\sigma(\FF)|_S=\FF|_S$.
	In particular, $\sigma(\FF)|_S$ is laminar on $S$ and we have that $\familyTree{\FF|_S,s} = \familyTree{\sigma(\FF)|_S,s}$.
	Additionally, the family of cuts $\sigma(\FF)$ is globally laminar.
	
	\begin{claim}\label{cl:sigma-laminar}
		$\sigma(\FF)$ is a laminar family of cuts of $D$.
	\end{claim}
	\begin{claimproof}
		Consider any two cuts $(A,B),(A',B')\in \FF$.
		If these cuts are comparable in $T$, say $(A',B')$ is an ancestor of $(A,B)$, then by construction we have $\sigma(A)\subseteq \sigma(A')$.
		On the other hand, if $(A,B)$ and $(A',B')$ are incomparable in $T$, then, by~\cref{cl:incomparable}, it follows that $\sigma(A)$ and $\sigma(A')$ are disjoint.
	\end{claimproof}
	
	Thus we can finally define the mapping \(\eta\) by the following formula.
  For any edge \(e \in E(T)\), let \((X,Y) = \eta'(e)\),
  and let \((A,B)\) be the unique element of \(\FF\)
  such that \(A \cap S = X \cap S\). We define
	\begin{equation*}
		\eta(e) \coloneqq \sigma((A,B))= (\sigma(A), V(D) \setminus \sigma(A))
	\end{equation*}
	to obtain the demanded properties.
\end{proof}

\subsection{Proof of the Potato Theorem}

We have now gathered all the ingredients necessary for proving the Potato Theorem.

\PotatoLemma*

\begin{proof}[Proof of~\cref{lemma:potato}]
  We prove the assertion by induction on the desired flow-connectivity~\(k\).
  The base case is \(k = 1\).
  Let \(D\) be an arbitrary digraph, and let \(\CC = \Set{C_1, C_2, \dots, C_\ell}\)
  be the strongly connected components of \(D\).
  Let \(T\) be any rooted tree with \(\ell\) nodes.
  Fix any bijection \(\beta \colon V(T) \to \CC\) and let \(\gamma \colon E(T) \ni v \mapsto \emptyset \subseteq E(D)\).
  Clearly, any cycle in \(D\) must be fully contained within only one element of \(\CC\).
  Further, any two vertices of a single strongly connected component lie on a common cycle, hence
  every \(C_i \in \CC\) is \(1\)-flow-connected.
  Thus, \((T, \beta, \gamma)\) is a \(1\)-\linked tree-partition of \porosity~\(0\).
  Hence we may set \(\thmFkt{lemma:potato}{f}{1} \coloneqq 0\).

  Now, let us assume that the statement holds for \(k\) and our goal is to prove it for \(k + 1\).
  Let \(D\) be an arbitrary digraph.
  Our goal is to obtain a tree-partition \((T_{\final}, \beta_{\final}, \gamma_{\final})\) of \(D\) where each bag \(\beta_{\final}(x)\) for $x\in V(T_\final)$ is \((k + 1)\)-flow-connected.
  By induction hypothesis for \(D\) and \(k\) we obtain a $k$-\linked tree-partition~\(\TT^{\initial}=(T, \beta, \gamma)\) of \porosity at most
  \(\thmFkt{lemma:potato}{f}{k}\). By rooting $\TT^\initial$ at any node~$r$ and applying~\cref{lem:rooting_potdecomps}, we may assume that $\TT^\initial$ is \coherent and rooted at $r$.
  Note that this operation does not break the $k$-flow-connectedness of the bags, for a subset of a $k$-flow-connected set is also $k$-flow-connected, and also does not increase the porosity of the tree-partition.

  
  For every node \(x \in V(T)\) the set of vertices \(\beta(x)\) is \(k\)-flow-connected by the induction assumption.
  Let \((T_x, \beta_x, \eta'_x)\) be the outcome of \cref{lem:partial_potatoes}.
  Now, we can apply~\cref{lem:partial_potatoes_uncrossed} to obtain a partial tree-partition \((T_x,\beta_x,\eta_x)\) of \(\beta(x)\) in \(D\).
  Thus we obtain \(\set{(T_x, \beta_x, \eta_x) \colon x \in V(T)}\),
  where each \(\beta_x\) has in its image only \((k + 1)\)-flow-connected sets by~\cref{lem:partial_potatoes},
  and each \(\eta_x(E(T_x))\) is a laminar family of cuts consistent with the near-partition \(\beta_x\) by~\cref{lem:partial_potatoes_uncrossed}.


	\begin{figure}
		\centering
		\begin{tikzpicture}
			\def\innerdist{0.7}
			\def\outerdist{3}
			\tikzstyle{child} = [draw,thick,isosceles triangle,rotate=90,scale=0.9]	
			
			\node (before) at (-3,3) {
				\begin{tikzpicture}
				\node (centre) at (0,0) {};
				
				\node[draw,thick,ellipse,minimum height=2.7cm,minimum width=3.2cm] (A-centre) at ($(centre)+(90:3)$) {};
				\node[vertex,CornflowerBlue] (A-1) at ($(A-centre)+(120:0.8)$) {};
				\node[vertex,CornflowerBlue] (A-2) at ($(A-1)+(190:\innerdist)$) {};
				\node[vertex,CornflowerBlue] (A-3) at ($(A-1)+(250:\innerdist)$) {};
				\node[vertex,CornflowerBlue] (A-4) at ($(A-1)+(320:\innerdist)$) {};
				\node[vertex,CornflowerBlue] (A-5) at ($(A-3)+(230:\innerdist)$) {};
				\node[vertex,CornflowerBlue] (A-6) at ($(A-3)+(275:\innerdist)$) {};
				\node[vertex,CornflowerBlue] (A-7) at ($(A-4)+(260:\innerdist)$) {};
				\node[vertex,CornflowerBlue] (A-8) at ($(A-4)+(330:\innerdist)$) {};
				\node[vertex,CornflowerBlue] (A-9) at ($(A-4)+(10:\innerdist)$) {};
				\node[vertex,CornflowerBlue] (A-10) at ($(A-8)+(270:\innerdist)$) {};
				\node[vertex,CornflowerBlue] (A-11) at ($(A-8)+(330:\innerdist)$) {};
				\draw[edge,CornflowerBlue] (A-1) to (A-2);
				\draw[edge,CornflowerBlue] (A-1) to (A-3);
				\draw[edge,CornflowerBlue] (A-1) to (A-4);
				\draw[edge,CornflowerBlue] (A-3) to (A-5);
				\draw[edge,CornflowerBlue] (A-3) to (A-6);
				\draw[edge,CornflowerBlue] (A-4) to (A-7);
				\draw[edge,CornflowerBlue] (A-4) to (A-8);
				\draw[edge,CornflowerBlue] (A-4) to (A-9);
				\draw[edge,CornflowerBlue] (A-8) to (A-10);
				\draw[edge,CornflowerBlue] (A-8) to (A-11);
				
				\node[draw,thick,ellipse,minimum height=2.7cm,minimum width=3.2cm] (B-centre) at ($(A-centre)+(270:1)+(210:3)$) {};
				\node[vertex,CornflowerBlue] (B-1) at ($(B-centre)$) {};
				\node[vertex,CornflowerBlue] (B-2) at ($(B-1)+(100:\innerdist)$) {};
				\node[vertex,CornflowerBlue] (B-3) at ($(B-1)+(10:\innerdist)$) {}; 
				\node[vertex,CornflowerBlue] (B-4) at ($(B-1)+(270:\innerdist)$) {};
				
				\draw[edge,CornflowerBlue] (B-1) to (B-2);
				\draw[edge,CornflowerBlue] (B-1) to (B-3);
				\draw[edge,CornflowerBlue] (B-1) to (B-4);
				
				\node[draw,thick,ellipse,minimum height=2.7cm,minimum width=3.2cm] (C-centre) at ($(A-centre)+(270:1)+(330:3)$) {};
				\node[vertex,CornflowerBlue] (C-1) at ($(C-centre)$) {};
				\node[vertex,CornflowerBlue] (C-2) at ($(C-1)+(90:\innerdist)$) {};
				\node[vertex,CornflowerBlue] (C-3) at ($(C-1)+(210:\innerdist)$) {};
				\node[vertex,CornflowerBlue] (C-4) at ($(C-1)+(330:\innerdist)$) {};
				\node[vertex,CornflowerBlue] (C-5) at ($(C-2)+(180:\innerdist)$) {};
				\node[vertex,CornflowerBlue] (C-6) at ($(C-2)+(30:\innerdist)$) {};
				\node[vertex,CornflowerBlue] (C-7) at ($(C-6)+(320:\innerdist)$) {};
				\node[vertex,CornflowerBlue] (C-8) at ($(C-3)+(180:\innerdist)$) {};
				\node[vertex,CornflowerBlue] (C-9) at ($(C-3)+(280:\innerdist)$) {};
				\node[vertex,CornflowerBlue] (C-10) at ($(C-8)+(80:\innerdist)$) {};
				\node[vertex,CornflowerBlue] (C-11) at ($(C-4)+(270:\innerdist)$) {};
				\node[vertex,CornflowerBlue] (C-12) at ($(C-4)+(10:\innerdist)$) {};
				
				\draw[edge,CornflowerBlue] (C-1) to (C-2);
				\draw[edge,CornflowerBlue] (C-1) to (C-3);
				\draw[edge,CornflowerBlue] (C-1) to (C-4);
				\draw[edge,CornflowerBlue] (C-2) to (C-5);
				\draw[edge,CornflowerBlue] (C-2) to (C-6);
				\draw[edge,CornflowerBlue] (C-6) to (C-7);
				\draw[edge,CornflowerBlue] (C-3) to (C-8);
				\draw[edge,CornflowerBlue] (C-3) to (C-9);
				\draw[edge,CornflowerBlue] (C-8) to (C-10);
				\draw[edge,CornflowerBlue] (C-4) to (C-11);
				\draw[edge,CornflowerBlue] (C-4) to (C-12);
				
				\node[child] (T-B-1) at ($(B-centre)+(210:3)$) {};
				\node[child] (T-B-2) at ($(B-centre)+(250:3)$) {};
				\node[child] (T-B-3) at ($(B-centre)+(290:3)$) {};
				\node[child] (T-B-4) at ($(B-centre)+(320:3)$) {};
				
				\draw[edge] (A-centre) to (B-centre);
				\draw[edge] (A-centre) to (C-centre);
				
				\draw[edge] (T-B-1.east) to (B-centre);
				\draw[edge] (T-B-2.east) to (B-centre);
				\draw[edge] (T-B-3.east) to (B-centre);
				\draw[edge] (T-B-4.east) to (B-centre);
				
				\node[child] (T-C-1) at ($(C-centre)+(240:3)$) {};
				\node[child] (T-C-2) at ($(C-centre)+(275:3)$) {};
				\node[child] (T-C-3) at ($(C-centre)+(300:3)$) {};
				
				\draw[edge] (T-C-1.east) to (C-centre);
				\draw[edge] (T-C-2.east) to (C-centre);
				\draw[edge] (T-C-3.east) to (C-centre);
				
				\node (parent) at ($(A-centre)+(110:2)$) {};
				\node (parent-p) at ($(parent)+(110:0.5)$) {};
				\draw[edge] (A-centre) to (parent.center);
				\draw[edge,dotted] (parent-p.center) to (parent.center);
				\end{tikzpicture}};
				
			\node (after) at (3,-4) {
				\begin{tikzpicture}
				\node (centre) at (0,0) {};
				
				\node
				 (A-centre) at ($(centre)+(90:3)$) {};
				\node[vertex,CornflowerBlue] (A-1) at ($(A-centre)+(120:0.8)$) {};
				\node[vertex,CornflowerBlue] (A-2) at ($(A-1)+(190:\innerdist)$) {};
				\node[vertex,CornflowerBlue] (A-3) at ($(A-1)+(250:\innerdist)$) {};
				\node[vertex,CornflowerBlue] (A-4) at ($(A-1)+(320:\innerdist)$) {};
				\node[vertex,CornflowerBlue] (A-5) at ($(A-3)+(230:\innerdist)$) {};
				\node[vertex,CornflowerBlue] (A-6) at ($(A-3)+(275:\innerdist)$) {};
				\node[vertex,CornflowerBlue] (A-7) at ($(A-4)+(260:\innerdist)$) {};
				\node[vertex,CornflowerBlue] (A-8) at ($(A-4)+(330:\innerdist)$) {};
				\node[vertex,CornflowerBlue] (A-9) at ($(A-4)+(10:\innerdist)$) {};
				\node[vertex,CornflowerBlue] (A-10) at ($(A-8)+(270:\innerdist)$) {};
				\node[vertex,CornflowerBlue] (A-11) at ($(A-8)+(330:\innerdist)$) {};
				\draw[edge,CornflowerBlue] (A-1) to (A-2);
				\draw[edge,CornflowerBlue] (A-1) to (A-3);
				\draw[edge,CornflowerBlue] (A-1) to (A-4);
				\draw[edge,CornflowerBlue] (A-3) to (A-5);
				\draw[edge,CornflowerBlue] (A-3) to (A-6);
				\draw[edge,CornflowerBlue] (A-4) to (A-7);
				\draw[edge,CornflowerBlue] (A-4) to (A-8);
				\draw[edge,CornflowerBlue] (A-4) to (A-9);
				\draw[edge,CornflowerBlue] (A-8) to (A-10);
				\draw[edge,CornflowerBlue] (A-8) to (A-11);
				
				\node
				 (B-centre) at ($(A-centre)+(270:1)+(210:3)$) {};
				\node[vertex,CornflowerBlue] (B-1) at ($(B-centre)$) {};
				\node[vertex,CornflowerBlue] (B-2) at ($(B-1)+(100:\innerdist)$) {};
				\node[vertex,CornflowerBlue] (B-3) at ($(B-1)+(10:\innerdist)$) {}; 
				\node[vertex,CornflowerBlue] (B-4) at ($(B-1)+(270:\innerdist)$) {};
				
				\draw[edge,CornflowerBlue] (B-1) to (B-2);
				\draw[edge,CornflowerBlue] (B-1) to (B-3);
				\draw[edge,CornflowerBlue] (B-1) to (B-4);
				
				\node
				 (C-centre) at ($(A-centre)+(270:1)+(330:3)$) {};
				\node[vertex,CornflowerBlue] (C-1) at ($(C-centre)$) {};
				\node[vertex,CornflowerBlue] (C-2) at ($(C-1)+(90:\innerdist)$) {};
				\node[vertex,CornflowerBlue] (C-3) at ($(C-1)+(210:\innerdist)$) {};
				\node[vertex,CornflowerBlue] (C-4) at ($(C-1)+(330:\innerdist)$) {};
				\node[vertex,CornflowerBlue] (C-5) at ($(C-2)+(180:\innerdist)$) {};
				\node[vertex,CornflowerBlue] (C-6) at ($(C-2)+(30:\innerdist)$) {};
				\node[vertex,CornflowerBlue] (C-7) at ($(C-6)+(320:\innerdist)$) {};
				\node[vertex,CornflowerBlue] (C-8) at ($(C-3)+(180:\innerdist)$) {};
				\node[vertex,CornflowerBlue] (C-9) at ($(C-3)+(280:\innerdist)$) {};
				\node[vertex,CornflowerBlue] (C-10) at ($(C-8)+(80:\innerdist)$) {};
				\node[vertex,CornflowerBlue] (C-11) at ($(C-4)+(270:\innerdist)$) {};
				\node[vertex,CornflowerBlue] (C-12) at ($(C-4)+(10:\innerdist)$) {};
				
				\draw[edge,CornflowerBlue] (C-1) to (C-2);
				\draw[edge,CornflowerBlue] (C-1) to (C-3);
				\draw[edge,CornflowerBlue] (C-1) to (C-4);
				\draw[edge,CornflowerBlue] (C-2) to (C-5);
				\draw[edge,CornflowerBlue] (C-2) to (C-6);
				\draw[edge,CornflowerBlue] (C-6) to (C-7);
				\draw[edge,CornflowerBlue] (C-3) to (C-8);
				\draw[edge,CornflowerBlue] (C-3) to (C-9);
				\draw[edge,CornflowerBlue] (C-8) to (C-10);
				\draw[edge,CornflowerBlue] (C-4) to (C-11);
				\draw[edge,CornflowerBlue] (C-4) to (C-12);
				
				\node[child,CornflowerBlue] (T-B-1) at ($(B-centre)+(210:3)$) {};
				\node[child,CornflowerBlue] (T-B-2) at ($(B-centre)+(250:3)$) {};
				\node[child,CornflowerBlue] (T-B-3) at ($(B-centre)+(290:3)$) {};
				\node[child,CornflowerBlue] (T-B-4) at ($(B-centre)+(320:3)$) {};
				
				\draw[edge,in=30,out=260,CornflowerBlue] (A-6) to (B-2);
				\draw[edge,in=90,out=20,CornflowerBlue] (A-8) to (C-6);
				
				\draw[edge,in=180,out=90,CornflowerBlue] (T-B-1.east) to (B-2);
				\draw[edge,in=190,out=90,CornflowerBlue] (T-B-2.east) to (B-1);
				\draw[edge,in=260,out=90,CornflowerBlue] (T-B-3.east) to (B-4);
				\draw[edge,in=10,out=90,CornflowerBlue] (T-B-4.east) to (B-4);
				
				\node[child,CornflowerBlue] (T-C-1) at ($(C-centre)+(240:3)$) {};
				\node[child,CornflowerBlue] (T-C-2) at ($(C-centre)+(275:3)$) {};
				\node[child,CornflowerBlue] (T-C-3) at ($(C-centre)+(300:3)$) {};
				
				\draw[edge,in=280,out=90,CornflowerBlue] (T-C-1.east) to (C-8);
				\draw[edge,in=260,out=90,CornflowerBlue] (T-C-2.east) to (C-11);
				\draw[edge,in=330,out=90,CornflowerBlue] (T-C-3.east) to (C-12);
				
				\node (parent) at ($(A-centre)+(110:2)$) {};
				\node (parent-p) at ($(parent)+(110:0.5)$) {};
				\draw[edge,CornflowerBlue,out=90,in=290] (A-1) to (parent.center);
				\draw[edge,dotted,CornflowerBlue] (parent-p.center) to (parent.center);
				\end{tikzpicture}};
			
			\node[rotate=-45,scale=2] at ($(before)!0.7!(after)$) {\(\Rightarrow\)};
		\end{tikzpicture}
		\label{fig:potdecomp_combination}
    \caption{Illustration of the tree-partition of \(D\) obtained by combining the constructed trees \(T_{x}\), for $x\in V(T)$, into a single tree.
      Black tree on the top-left is the tree \(T\) of the initial decomposition, the blue trees on the top-left are the trees \(T_{x}\) associated with the corresponding nodes \(x \in V(T)\).
      The blue tree on the bottom-right depicts the
      final tree-partition of \(D\).
    }
	\end{figure}
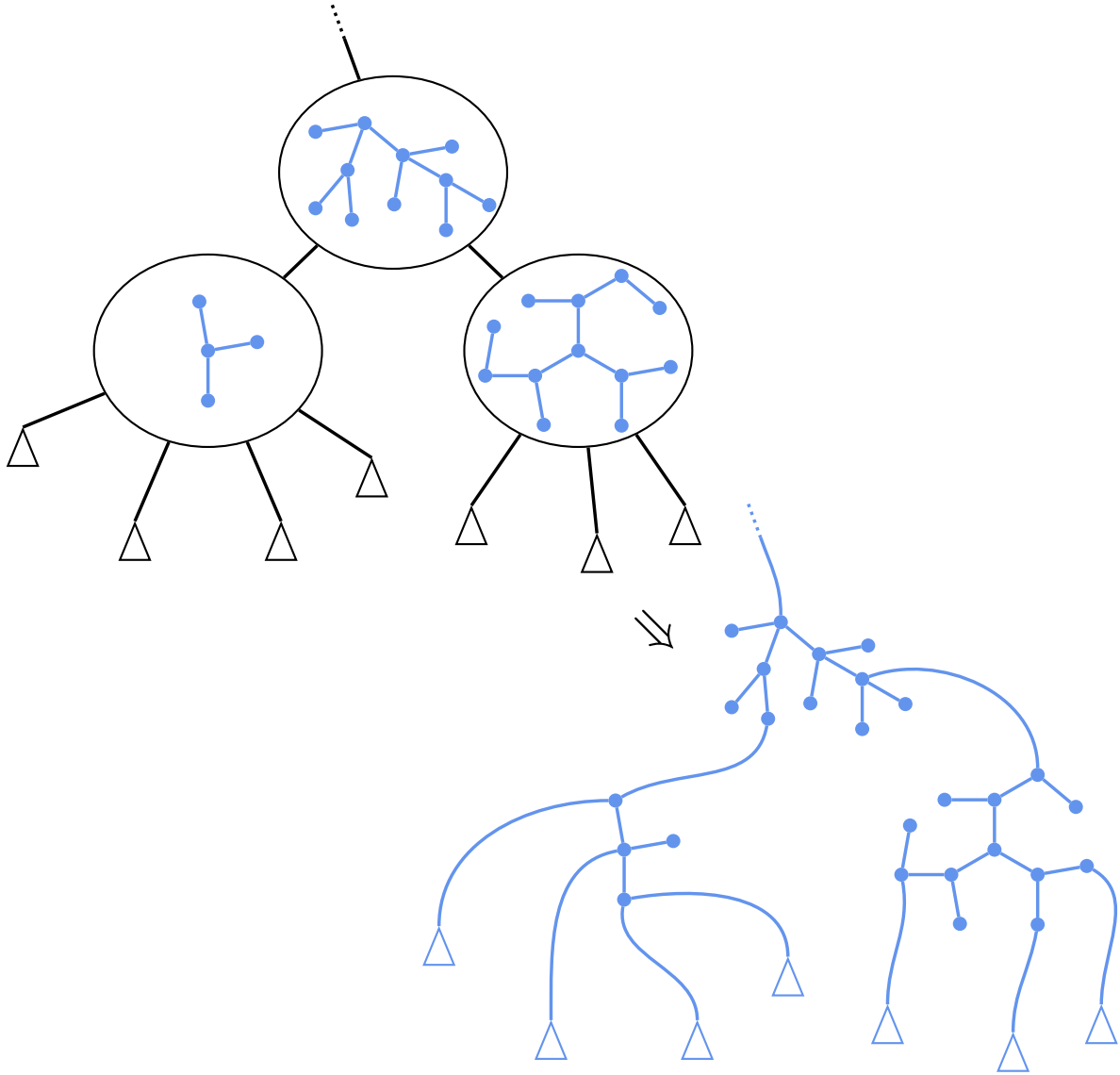

	Now we want to build a single $(k+1)$-\linked tree-partition of \(D\) out of the set \(\set{ (T_x, \beta_x, \eta_x) \ \colon \ x \in V(T)}\).
  We do so by adding, for every edge \((x,y) \in E(T)\)
  a single edge connecting some node of \(T_{x}\) with
  some node of \(T_{y}\).

  Consider any edge \((x',x) \in E(T)\); so \(x'\) is the parent of \(x\) in $T$.
  Denote \(R \coloneqq \beta(\subtree{T}{x})\) for brevity.
  Also, let~$T'$ be the tree $T_{x'}$ with the orientation of the edges forgotten; thus $T'$ is an undirected tree.
  We now use~\(R\) to construct a new orientation of (a subset of) the edges of \(T'\) (possibly different from the orientation of $T_{x'}$).
  Consider any edge \(f=tt' \in E(T')\) and let \((A,B)\) be the cut associated with $f$ by $\eta_{x'}$, where $A\supseteq \beta_{x'}(\treeSplit{T}{t}{f})$ and $B\supseteq \beta_{x'}(\treeSplit{T}{t'}{f})$.
  There are three possibilities:
  \begin{itemize}
    \item \(R\) is a subset of \(A\): we then orient the edge \(f\) towards \(t\).
    \item \(R\) is a subset of \(B\): we then orient the edge \(f\) towards \(t'\).
    \item \(R\) is neither a subset of \(A\) nor of \(B\): we leave \(f\) unoriented.
  \end{itemize}
  In the last case we say that \(x\) is \emph{ambiguous} over the edge \(f\).
  Let \(\FktO{Amb}{f}\) denote the set of all children~\(x\) of~\(x'\) that are ambiguous over the edge~\(f\). 
  
  The laminarity of the family of cuts $\eta_{x'}(E(T_{x'}))$ implies that there is a connected subtree $T^{\star}_x$ of $T'$ so that the edges within $T^{\star}_x$ are left unoriented, and all the other edges of $T'$ are oriented towards $T^{\star}_x$.
  (It may happen that $T^{\star}_y$ consists of one node, which is then the unique sink in the orientation of $T'$ that is total.)
  

  We now connect the trees $T_{x'}$ and $T_{x}$ by making the root of \(T_{x}\) a child of an arbitrarily chosen node of \(T^{\star}_{x'}\).
  By performing this operation for all the edges of the tree \(T\),
  we combine all the trees \(\set{T_{x} \colon x \in V(T)}\) into a single (rooted) tree \(T_{\final}\), with \(\bigcup_{x \in V(T)} V(T_{x})\) as the node set.
  We endow $T_\final$ with a bag function $\beta_\final$ constructed naturally from the partial tree-partitions:
  for a node $t$ of $T_\final$, we set $\beta_\final(t)\coloneqq \beta_{x}(t)$, where $x\in V(T)$ is such that $t\in V(T_{x})$.
  It follows straightforwardly from the construction that $\{\beta_\final(t)\colon t\in V(T_\final)\}$ is a near-partition of the vertex set of $D$.
  It also follows immediately from~\cref{potlemma:claim:connectivity} that for each $t\in V(T)$, the set $\beta_\final(t)$ is $(k+1)$-flow-connected.
  
  To finish the description of the final tree-partition of \(D\),
  we need to specify the edge-guard function \(\gamma_{\final}\) and bound the cardinalities of the edge-guards.
  We define it as follows.
  For a node $x$ of $T$, set
  \(W_x \coloneqq \gamma(e)\) where $e$ is the edge connecting the parent of $x$ with $x$; if $x$ is the root of $T$, then set \(W_x \coloneqq \emptyset\).
  Let \(f=(t',t)\) be an edge of \(T_{\final}\).
  If \(t' \in V(T_{x}), t \in V(T_{y})\) for some $x\neq y$, then we set \(\gamma_{\final}(f) \coloneqq \gamma((x,y))\).
  Otherwise, we have that \(t',t\in V(T_{x})\) for some $x\in V(T)$.
  Let \((A,B) \coloneqq \eta_{x}(f)\) and let \(H\)
  be a minimum size set of edges that hits all the cycles crossing the cut $(A,B)$.
  We set
  \begin{equation*}
    \gamma_{\final}(f) \coloneqq H\cup W_x \cup \bigcup_{y \in \FktO{Amb}{f}} W_y\text{.}
  \end{equation*}
  This completes the description of
  \((T_{\final}, \beta_{\final}, \gamma_{\final})\) of \(D\).

  We now argue that \((T_{\final}, \beta_{\final}, \gamma_{\final})\)
  is indeed a correct tree-partition of \(D\).
  Clearly, \(T_{\final}\) is a tree, and we have already argued that $\{\beta_\final(t)\colon t\in V(T_\final)\}$ is a near-partition of the vertex set of $D$; so this established \cref{item:tp:bags}.
  To argue \cref{item:tp:edgeguards}, we need to prove that for each edge $f=(t',t)\in E(T_\final)$, the set \(\gamma_{\final}(f)\)
  intersects every cycle visiting both sets
  \(\beta_{\final}(\treeLower{(T_{\final})}{f})\) and \(\beta_{\final}(\treeUpper{(T_{\final})}{f})\).
  This clearly holds if \(t' \in V(T_{x})\) and \(t \in V(T_{y})\) for distinct \(x,y \in V(T)\)
  by induction, for $(T,\beta,\gamma)$ is a tree-partition of $D$.
  So we may assume that \(t,t' \in V(T_{x})\) for some node $x$ of $T$.
  Let \((A,B) \coloneqq \eta_{x}(f)\).
  Further, let \(X \coloneqq \beta_{\final}(\treeLower{(T_{\final})}{f})\)
  and \(Y \coloneqq V(D) \setminus X=\beta_{\final}(\treeUpper{(T_{\final})}{f})\).
  Observe that by construction, \(Y\) is the union of the following sets:
  \begin{itemize}
  	\item \(\beta(\treeUpper{T}{x})\),
  	\item \(\beta(x) \cap B\),
  	\item sets \(\beta(\subtree{T}{y})\) for all children \(y\) of \(x\) such that \(\beta(\subtree{T}{y}) \subseteq B\), and
  	\item sets \(\beta(\subtree{T}{y})\) for $y$ ranging over some subset of \(\FktO{Amb}{f}\)
  	(i.e.~some subset of the children of \(x\) that are ambiguous over $f$).
  \end{itemize} 
  Similarly, \(X\) is the union of the following sets:
  \begin{itemize}
  	\item \(\beta(x) \cap A\),
  	\item sets \(\beta(\subtree{T}{y})\) for all children \(y\) of \(x\) such that  \(\beta(\subtree{T}{y}) \subseteq A\), and
  	\item sets \(\beta(\subtree{T}{y})\) for $y$ ranging over a subset of \(\FktO{Amb}{f}\) complementary to that used for \(Y\).
  \end{itemize} 
  Let \(C\) be a cycle visiting both a vertex of \(X\) and a vertex \(Y\).
  If \(C\) contains a vertex of \(\beta(\treeUpper{T}{x})\), then $C$ must be hit by $W_x$, because $C$ cannot be entirely
  contained within \(\beta(\treeUpper{T}{x})\) due to this set being disjoint with $X$. So assume that \(C\) is disjoint with \(\beta(\treeUpper{T}{x})\).
  If \(C\) contains a vertex of \(\beta(\treeLower{T}{y})\)
  for some node \(y\in \FktO{Amb}{f}\), then \(C\) cannot be entirely contained in \(\beta(\treeLower{T}{y})\) due to this set being disjoint with $X$ or $Y$, and consequently $C$ is hit by $W_y$.
  Thus, the only possibility remaining is that \(C\) contains
  both a vertex of \(A\) and a vertex of \(B\), and is thus hit by the set \(H\).
  Since all the mentioned hitting set were included in the definition of \(\gamma_{\final}(f)\), we conclude that in any case, \(C\) is hit by \(\gamma_{\final}(f)\).
  
  This concludes the verification of~\cref{item:tp:edgeguards}, so $(T_\final,\beta_\final,\gamma_\final)$ is indeed a rooted tree-partition of $D$.
  As we have already argued, the set $\beta_\final(t)$ is $(k+1)$-flow-connected for each $t\in V(T_\final)$, hence $(T_\final,\beta_\final,\gamma_\final)$ is $(k+1)$-\linked.
  It remains to bound the porosity of $(T_\final,\beta_\final,\gamma_\final)$.


  To this end, consider any edge $f=(t',t)\in E(T_\final)$; we have to bound the size of \(\gamma_{\final}(f)\).
  Again, if~\(t\) and~\(t'\) belong to trees $T_{x}$ and $T_{y}$ for $x\neq y$, then
  \(\Abs{\gamma_{\final}(f)} = \Abs{\gamma((x,y))} \leq \thmFkt{lemma:potato}{f}{k}\).
  So, assume that \(t,t' \in V(T_{x})\) for some node $x$ of $T$,
  and let \((A,B) \coloneqq \eta_{x}(f)\).
  Clearly, \(\Abs{W_x} \leq \thmFkt{lemma:potato}{f}{k}\).
  By~\cref{claim:potato_cycle_order}, \(\Abs{H} \leq k^2 \cdot 4^{k + 1} + 2 k\).
  All that remains is to bound the size of \(\FktO{Amb}{f}\). 
  For this, we make use of the \coherence of $(T,\beta,\gamma)$ --- a property that we established in the second paragraph of the proof using~\cref{lem:rooting_potdecomps}.
  
  \begin{claim}\label{cl:bound-ambiguous}
  	For every node $x\in V(T)$ and edge $f\in V(T_{x})$, we have $\Abs{\FktO{Amb}{f}}\leq k^2 \cdot 4^{k + 1} + 2 k$.
  \end{claim}
  \begin{claimproof}
  Consider any child $y\in \FktO{Amb}{f}$ and let \(R_y \coloneqq \beta(\treeLower{T}{y})\).
  By the \coherence of $(T,\beta,\gamma)$, each set $R_y$
  induces a strongly connected subgraph of $D$.
  By the ambiguity of \(y\) over the edge \(f\), it follows that $(A,B)$ splits $R_y$.
  Moreover, the sets \(\{R_y\colon y\in \FktO{Amb}{f}\}\) are pairwise vertex-disjoint. Hence, for each \(y\in \FktO{Amb}{f}\) we can find a cycle contained in \(R_y\)
  that visits both \(A\) and \(B\).
  These cycles are pairwise vertex-disjoint and thus also
  pairwise edge-disjoint, showing that the porosity of \((A,B)\)
  is at least \(|\FktO{Amb}{f}|\).
  However, by~\cref{claim:potato_cycle_order}, the porosity of $(A,B)$ is at most $k^2 \cdot 4^{k + 1} + 2 k$.
  Hence we get \(\Abs{\FktO{Amb}{f}}\leq k^2 \cdot 4^{k + 1} + 2 k\).
\end{claimproof}

  Thus, by~\cref{cl:bound-ambiguous} we have
  \begin{align*}
    \Abs*{\bigcup_{y \in \FktO{Amb}{f}} W_y} \leq \ & \Abs{\FktO{Amb}{f}} \cdot \thmFkt{lemma:potato}{f}{k}\\ 
    \leq \ & \Paren{k^2 \cdot 4^{k + 1} + 2 k} \cdot \thmFkt{lemma:potato}{f}{k}\text{.}
  \end{align*}
  This yields the final bound:
  \begin{align*}
    \Abs*{\gamma_{\final}(f)} \leq \ & \Paren{k^2 \cdot 4^{k + 1} + 2 k}+\thmFkt{lemma:potato}{f}{k}+ \Paren{k^2 \cdot 4^{k + 1} + 2 k} \cdot \thmFkt{lemma:potato}{f}{k}\\
    = \ & \Paren{\thmFkt{lemma:potato}{f}{k} + 1} \cdot \Paren{k^2 \cdot 4^{k+1} + 2 k + 1} - 1\text{.}
  \end{align*}
  
  Therefore, we can set \(\thmFkt{lemma:potato}{f}{k + 1} \coloneqq \Paren{\thmFkt{lemma:potato}{f}{k} + 1} \cdot \Paren{k^2 \cdot 4^{k + 1} + 4k + 1} - 1\).
  
We note that the above reasoning can yield the algorithmic counterpart of the lemma.
Indeed, 
The base case of the induction can be performed in time polynomial in the size of~$D$, so can be finding the set~$\FF'$ of cuts and modifying it into~$\FF$ (\cref{potlemma:claim:same_orientation_laminar}).
Defining the sets $\mu(A)$ and $\mu'(A)$ for every cut $(A,B) \in \FF$ can be done in fixed-parameter time, as finding important cuts can be done in fixed-parameter time \cite{PAbook,importantSep2008}.
Having this, for each $x \in V(T)$ we can construct $(T_{x}, \beta_{x},\eta_{x})$, as well as $T_\final$ and $\beta_\final$ in a straightforward~way. 

The place that benefits from more advanced algorithmic insight is the construction of $\gamma_\final$ for an edge of $T_\final$ that connects two vertices from $T_{x}$ for some $x \in V(T)$. 
Observe that we require computing a minimum set $H$ of edges that hits all the cycles crossing a cut $(A,B)$, but trying to adapt the proof steps directly only yields a set of size bounded by a function of $k$, without guaranteeing its minimality.
We note, however, that $H$ can be still computed fixed-parameter time, using the algorithm of Chitnis, Cygan, Hajiaghayi and Marx~\cite{DBLP:journals/talg/ChitnisCHM15} for the \textsc{Directed Subset Feedback Vertex Set}.
We define a line digraph $L(D)$ of $D$ to have the vertex set $E(D)$ and an edge from $(x,y)$ to $(z,w)$ if $y=z$. We also let $S \subseteq V(L(D))=E(D)$ to be the edges crossing $(A,B)$.
Then, by~\cite{DBLP:journals/talg/ChitnisCHM15}, we can find a minimum set $W$ of vertices of $L(D)$ that hits all the cycles that contain an element of $S$; clearly $W$ corresponds to a set of arcs of $D$ that define~$H$. 

Last, we note that the number of inductive steps in the proof of \cref{lemma:potato} is bounded by $k$.
\end{proof}


\section{Bounded degree case}\label{sec:boundeddeg}

With all the tools prepared, we now start working towards the proof of \cref{thm:grid_or_small_dtcw}.
In this section we treat the simplest setting where all the vertices have bounded cycle-degree.
As expected, in this setting directed tree-cutwidth and directed branchwidth are functionally related, so the result becomes a straightforward consequence of the Directed Grid Theorem (\cref{thm:gridtheorem}), using the equivalence of directed branchwidth and directed treewidth provided by \cref{thm:cyw_dtw}.

\begin{lemma}
	\label{lem:cyclewidth_to_dtcw}
	Let \(D\) be a digraph with maximum cycle-degree at most \(d\) and directed branchwidth at most~\(w\).
	Then \(D\) has \dtcw at most \(\max \set{3 d \cdot w, 1}\).
\end{lemma}
\begin{proof}
	Let \((\hat{T},\hat{\lambda})\) be an directed branch-decomposition of \(D\) of width at most \(w\).
	Thus there exists a vertex-guard \(\hat{\gamma}(e)\) of size at most \(w\) for every \(e \in \E{T}\).
	For a vertex \(v \in V(D)\), let \(h(v)\) be a set of at most \(d\) edges of \(D\)
	that hits all cycles containing \(v\).
	
	We define a \dtcdecomp \(\Brace{T,\beta,\gamma,\theta}\) of \(D\) using the following formulas, for all $x\in V(T)$ and $e\in E(T)$:
	\begin{align*}
		T &\coloneqq \hat{T};\\
		\beta(x) &\coloneqq \begin{cases*}
			\Set{\hat{\lambda}(x)} & \text{if \(x\) is a leaf,}\\
			\emptyset & \text{otherwise;}
		\end{cases*}\\
		\gamma(e) &\coloneqq \bigcup \set{h(v) \colon v \in \hat{\gamma}(e)};\\
		\theta(x) &\coloneqq \bigcup_{\inc{e}{x}} \gamma(e).
	\end{align*}
	
	We now argue that this yields a valid \dtcdecomp.
	\Cref{item:tp:bags} is obvious, as the mapping \(\hat{\lambda}\) bijectively maps \(V(D)\) to the set of leaves of \(T\).
	To argue \cref{item:tp:edgeguards}, take any edge \(e \in E(T)\). The cuts induced by this edge in \((T,\beta)\)
	and in \((T,\hat{\lambda})\) are the same. Thus, any cycle \(C\) crossing this cut contains a vertex \(v \in \hat{\lambda}(e)\) by the definition of directed branchwidth.
	Thus it contains some edge \(f \in h(v)\), and thus also of \(\gamma(e)\); so this establishes \cref{item:tp:edgeguards}, and hence also \cref{item:dtcw:tp}.
	To argue~\cref{item:dtcw:onionguards}, consider any \(x \in V(T)\) and observe that every in-onion of $x$ and every out-onion of $x$ contains a cycle that crosses one of the cuts induced in $(T,\beta)$ by some edge $e$ of $T$ incident to $x$.
	So this cycle, and therefore also the onion, is hit by \(\gamma(e)\subseteq \theta(x)\).
	
	We finish by providing an upper bound on the width of the decomposition obtained.
	Note that for every edge $e$ of $T$ we have \(\Abs{\gamma(e)} \leq d \cdot w\), as it is the union of at most \(w\) edge sets (as \(\Abs{\hat{\gamma}(e)} \leq w\)),	each of size at most \(d\).
	It follows that \(\Abs{\theta(x)} \leq 3d \cdot w\) for every node $x$, as the tree \(T=\widehat{T}\) is subcubic.
	Finally, by construction we also have \(\Abs{\beta(x)} \leq 1\), providing the desired bound on the width.
\end{proof}

\begin{theorem}
	\label{thm:bd_degree_case}
	There exists a function \(\thmFkt{thm:bd_degree_case}{g}{}\colon \N\times \N\to \N\) such that for every non-negative integer~\(d\), positive integer~\(k\), and digraph~\(D\) with maximum cycle-degree at most~\(d\), one of the following~holds:
	\begin{itemize}
		\item \(D\) contains \(\gridImm{k}\) as an immersion, or
		\item \(\dtcwFkt{D} \leq \thmFkt{thm:bd_degree_case}{g}{d, k}\).
	\end{itemize}
	Moreover, if $d=d(k)$, there exists a fixed-parameter algorithm to find, for every digraph \(D\) with maximum cycle-degree at most $d$, and parameter $k$, either an immersion model of \(\gridImm{k}\) in \(D\), or a directed tree-cut decomposition of $D$ of width at most \(\thmFkt{thm:bd_degree_case}{g}{d, k}\).
\end{theorem}
\begin{proof}
	Let \(D\) be a digraph with maximum cycle-degree at most \(d\),	and assume that it does not contain \(\gridImm{k}\) as an immersion.
	It follows that it does not contain \(\gridImm{k}\) as a topological minor.
	By~\cref{thm:gridtheorem}, \(D\)~has directed tree-width at most \(\thmFkt{thm:gridtheorem}{g}{k}\); let \(k' \coloneqq \thmFkt{thm:gridtheorem}{g}{k}\).
	By~\cref{thm:cyw_dtw} it follows that the directed branchwidth of \(D\) is at most \(\thmFkt{thm:cyw_dtw}{h}{k'}\); let \(\ell \coloneqq \thmFkt{thm:cyw_dtw}{h}{k'}\).
	Finally, by~\cref{lem:cyclewidth_to_dtcw} we infer that the directed tree-cutwidth of \(D\) is at most	\(\max\{3d \cdot \ell, 1\}\).
	
	Thus, we conclude that the theorem holds with \(\thmFkt{thm:bd_degree_case}{g}{d,k} \coloneqq \max \set{3d \cdot \thmFkt{thm:cyw_dtw}{h}{\thmFkt{thm:gridtheorem}{g}{k}}, 1}\).
	
	Last, the construction of \cref{lem:cyclewidth_to_dtcw} and the algorithmic claim in~\cref{thm:cyw_dtw} 
	can be combined with \cref{thm:grid-algorithmic} in a straightforward way to yield the algorithmic statement.
\end{proof}

\section{Small councils}\label{sec:councils}


\newcommand*{\migrationAux}[1]{\Fkt{\mathcal{M}_d}{#1}}

In this section we consider the case when the considered digraph $D$ contains only a bounded number of vertices with high cycle-degree, which are moreover well-connected to each other (formally, form a $k$-flow-connected set for a large $k$). For this, it will be convenient to define an auxiliary migration digraph \(\migrationAux{D}\) that captures the connectivity in $D$ outside of the set of high-degree vertices.

Formally, fix a constant $d\in \N$, let $D$ be a digraph, and let $Z$ be the set of all vertices in $D$ whose cycle-degree is at least $d$. 
Let \(I \coloneqq \In{D}{Z}\) be the set of all in-edges of \(Z\) and \(O \coloneqq \Out{D}{Z}\) be the set of all out-edges of \(Z\).
Let \(D'\) be obtained from \(D-Z\) by adding a vertex set \(T \coloneqq \Set{t_i \mid e_i \in I}\) together with edges \(\Set{(\tail{e_i},t_i) \mid e_i \in I}\) and similarly a vertex set \(S \coloneqq \Set{s_i \mid e_i \in O}\) together with edges \(\Set{(s_i,\head{e_i}) \mid e_i \in O}\).
Note that \(\migrationAux{D} \coloneqq \Brace{D',S, T}\) is a migration digraph.

We show that if this migration digraph allows for a small hitting set of all tripods of one type, then we can decompose the entire graph into a \dtcdecomp of bounded width.

\begin{lemma}
	\label{lem:tripod_hitting}
	Let \(D\) be a digraph with no \(\gridImm{k}\)-immersion and with at most \(z\) vertices of cycle-degree at least \(d\), for some \(k,d,z\in \N\).
	If there is a hitting set of size at most \(x \in \N\) for all out-tripods (or for all in-tripods) of \(\migrationAux{D} = \Brace{D',S,T}\), then there is a \dtcdecomp of \(D\) of width at most \(\thmFkt{lem:tripod_hitting}{h}{d,x,z,k}\).
	
	Moreover, if $d=d(k)$, $x=x(k)$ and $z=z(k)$, there is a fixed-parameter algorithm with parameter $k$ to find a directed tree-cut decomposition of $D$ of width at most $\thmFkt{lem:tripod_hitting}{h}{d,x,z,k}$.
\end{lemma}
\begin{proof}
	Let \(Z\) be the set of at most \(z\) vertices of cycle-degree at least \(d\).
  Without loss of generality, let us assume that there is a hitting set \(F'\subseteq E(D)\) of size \(x\) for all in-tripods of \(\migrationAux{D}\).
	The proof for out-tripods works analogously.
	
	We define \(\thmFkt{lem:tripod_hitting}{h}{d,x,z,k} \coloneqq
	x + \thmFkt{thm:bd_degree_case}{g}{d,k} + z + 1\).

  Define the following set of edges \(F \subseteq E(D)\).
  For any edge \(e' \in F'\):
  \begin{itemize}
    \item if \(e' \in E(D)\), then add \(e'\) to \(F\);
    \item if \(e' = (\tail{e_i},t_i)\) for some \(e_i \in \In{D}{Z}\), then add \(e_i\) to \(F\);
    \item if \(e' = (s_j, \head{e_j})\) for some \(e_j \in \Out{D}{Z}\), then add \(e_j\) to \(F\).
  \end{itemize}
  By construction we clearly have \(\Abs{F} \leq \Abs{F'} = x\).
	Let \[D^\star \coloneqq (D- F) -Z.\]
	We consider a topological ordering \(\preccurlyeq\) on the strongly connected components of \(D^\star\) (called further \emph{strong components} for brevity). Thus, for every edge $(u,v)$ of $D^\star$, if $R_u$ and $R_v$ are the strong components containing $u$ and $v$ respectively, then we have $R_u\preccurlyeq R_v$.
	
	We now build a partition \(\mathcal{B}\) of $V(D)\setminus Z$ as follows.
  Let \(\RR\) be the set of all strong components \(R\) of \(D^\star\)
  satisfying the following property: there exists a strong component \(R'\) of \(D^\star\)
  which is incident to an edge of \(\Out{}{Z} \setminus F\) and satisfies \(R' \preccurlyeq R\).
  That is, \(\RR\) is the upwards closure of components incident to \(\Out{}{Z} \setminus F\).
  Let \(B_{\mathsf{rest}}\) be the union of all the remaining strong components of \(D^\star\).

  \begin{claim}
    \label{claim:atmost_1_edge}
    Let \(A\) be a strong component of \(D^\star\).
    If there is an edge not in \(F\) starting in \(Z\) and ending in some vertex of \(A\),
    then in \(D \setminus F\), there is at most one edge of \(\In{}{Z} \setminus F\)  that can be reached from $A$ by a directed path whose all internal vertices do not belong to $Z$.
  \end{claim}
  \begin{claimproof}
    Suppose towards a contradiction that the claim was false.
    Then, there are two edges \(f_1, f_2 \in E(D) \setminus F\)
    with \(\head{f_i} \in Z, \tail{f_i} \in V(D) \setminus Z\),
    and we can get from \(A\) both to \(\tail{f_1}\) and to \(\tail{f_2}\) by a directed path in \(D \setminus F\) that avoids $Z$ except for the last vertex.
    Thus, the subpaths that reach \(\tail{f_i}\) are actually contained within \(D^\star = (D \setminus Z) \setminus F\).
    Let \(e \in E(D) \setminus F\) be such that
    \(\head{e} \in A\) and \(\tail{e} \in Z\).
    Let \(s \in \migrationAux{D}\) be the vertex corresponding
    to the edge \(e\), and \(t_1, t_2 \in \migrationAux{D}\)
    be the vertices corresponding to the edges \(f_1, f_2\).
    Clearly, in \(\migrationAux{D} \setminus F'\),
    \(s\) can reach both \(t_1\) and \(t_2\) (as \(A\) is a strong component).
    This gives us an in-tripod in \(\migrationAux{D}\) that avoids \(F'\); this is a contradiction.
  \end{claimproof}
  This claim implies that we can induce a near-partition of \(\RR\) via \(\In{}{Z} \setminus F\).
  That is, for each edge \(f \in \In{}{Z} \setminus F\),
  let \(\BB_f\) be the set of all elements of \(\RR\) that can reach \(\tail{f}\) in \(D^\star\);
  additionally, let \(B_{\emptyset}\) be the set of all elements of \(\RR\)
  that cannot reach (in \(D^\star\)) any \(\tail{f}\), for \(f \in \In{}{Z} \setminus F\).
  By~\cref{claim:atmost_1_edge}, the family
  \(\set{\BB_f \colon f \in \In{}{Z} \setminus F} \cup \set{\BB_{\emptyset}}\)
  is a near-partition of \(\RR\).
  This is the basis of our construction of \(\BB\). In $\BB$, we collect:
  \begin{enumerate}
    \item \(B_f \coloneqq \bigcup \BB_f\), for \(f\in \In{}{Z} \setminus F\) provided that \(\BB_f\) is non-empty;
    \item \(B_{\emptyset} \coloneqq \bigcup \BB_{\emptyset}\), provided that \(\BB_{\emptyset}\) is non-empty;
    \item \(B_{\mathsf{rest}}\), provided that \(B_{\mathsf{rest}}\) is non-empty.
  \end{enumerate}

	See~\cref{fig:fdeleted_structure}
	for an illustration of this construction.

  For each set \(B \in \BB\), define
  \[
    h(B) \coloneqq \begin{cases}
      \set{f},& \text{ if \(B = B_f\) for some \(f \in \In{}{Z} \setminus F\),}\\
      \emptyset,& \text{ if \(B = B_{\mathsf{rest}}\) or \(B = B_{\emptyset}\).}
    \end{cases}
  \]

  \begin{claim}
    \label{claim:cycle_hitting}
    Let \(B \in \BB\). Suppose \(C\) is a cycle in $D$
    visiting a vertex of \(B\) and of \(V(D) \setminus B\).
    Then \(C\) contains an edge of \(F \cup h(B)\).
  \end{claim}
  \begin{claimproof}
    If \(C\) contains an edge of \(F\) we are immediately done.
    So, assume that \(C\) does not contain any edge of \(F\).
    As \(B\) is the union of a collection of strong components of \(D^\star = (D \setminus F) \setminus Z\),
    and \(C\) visits both a vertex of \(B\) and
    a vertex outside of \(B\), it follows that $C$ must contain a vertex
    of \(Z\).

    Let \(b \in B\) be a vertex visited by \(C\).
    Let \(e\) be an edge in \(C\) that is the last edge before
    \(C\) enters the vertex \(b\) that is an element of \(\Out{}{Z} \setminus F\).
    Clearly this edge enters some strong component \(A\) of \(D^\star\),
    and all the edges between $e$ and $b$ along $C$ are contained in \(D^\star\),
    showing that \(b\) is reachable from \(A\) in~$D^\star$.
    This means that the strong component of \(b\) in \(D^\star\)
    is an element of \(\RR\). Thus, \(B \neq B_{\mathsf{rest}}\).
    
    Let \(f\) be the first edge of \(C\) occurring
    after \(C\) leaves the vertex \(b\) that is an element of
    \(\In{}{Z} \setminus F\).
    Thus, the strong component of \(b\) in \(D^\star\)
    can reach the edge \(f\) by a path avoiding $F$ or $Z$, and by~\cref{claim:atmost_1_edge}
    it follows that \(B = B_f\), and so \(C\) is hit by \(\set{f} = h(B)\).
  \end{claimproof}

	Clearly, every \(B \in \mathcal{B}\) induces a subgraph \(D_B\coloneqq D[B]\) of \(D\) with maximum cycle-degree less than \(d\).
	So, by~\cref{thm:bd_degree_case} and as no \(D_B\) contains a \(\gridImm{k}\)-immersion, there is a
  \dtcdecomp \(\Brace{T_B,\beta_B,\gamma_B,\theta_B}\) for \(D_B\) of width at most \(\thmFkt{thm:bd_degree_case}{g}{d,k}\).
	The next step is to combine these decompositions into a \dtcdecomp \(\Brace{T,\beta,\gamma,\theta}\) for \(D\).
	
	Let \(T\) be the tree obtained from the union over all \(T_B\), for \(B \in \mathcal{B}\), by adding an extra node \(s\) connected to exactly one (arbitrary chosen) node of each \(T_B\). We endow $T$ with a bag function $\beta$ defined naturally:
	\begin{align*}
		\beta(t) &\coloneqq \begin{cases}
			\beta_B(t), & \text{ if \(t \in \V{T_B}\) for some \(B\in \BB\),}\\
			Z, & \text{ if \(t = s\).}
		\end{cases} 
	\end{align*}
	
	We build the edge-guards as follows:
	\begin{align*}
		\gamma(e) & \coloneqq 
			\begin{cases}
				\gamma_B(e) \cup h(B) \cup F, &\text{ if \(e \in \E{T_B}\) for some \(B \in \BB\),}\\
				F \cup h(B),&\text{ if \(\inc{e}{s}\) and \(\inc{e}{a}\) for some \(a \in \V{T_B}\).}
			\end{cases}
	\end{align*}
	
	For the onion-guards we add the hitting set and the edges into \(Z\) to the old onion-guards:
	\begin{align*}
		\theta(t) &\coloneqq \begin{cases}
			\theta_B(t) \cup F \cup h(B), &\text{ if \(t \in T_B\) for some \(B \in \BB\),}\\
			F,&\text{ if \(t = s\).}
		\end{cases}
	\end{align*}
	
	%
	
	Note that if $d=d(k)$, $z=z(k)$, and $x=x(k)$, the above can be modified in a straightforward way to obtain a fixed-parameter algorithm (with respect to $k$) to find \(\Brace{T,\beta,\gamma,\theta}\). What is left to verify is that \(\Brace{T,\beta,\gamma,\theta}\) is indeed a \dtcdecomp of $D$ of width at most \(\thmFkt{lem:tripod_hitting}{h}{d,x,z,k}\).
	
	\paragraph{Edge-guards}
	
	First, note that \(\Abs{\Fkt{\gamma}{e}} \leq \thmFkt{lem:tripod_hitting}{h}{d,x,z,k}\) for all \(e \in \E{T}\). We are left with verifying the cycle-hitting property of the sets $\Fkt{\gamma}{e}$.
	
	Consider the case that \(e \in \E{T_B}\) for some \(B \in \BB\).
  Let \(C\) be a cycle crossing the cut induced by \(e\).
  If \(C\) contains an edge of \(F\) then it is hit by \(\gamma(e)\),
  as it contains \(F\), so assume that is not the case.
  If \(C\) is contained in \(B\), then
  it is hit by \(\gamma_B(e) \subseteq \gamma(e)\),
  as \((T_B,\beta_B,\gamma_B,\theta_B)\) is a \dtcdecomp.
  As \(V(D) \setminus B\) is contained on one side
  of the cut induced by \(e\), to cross this cut
  \(C\) must contain a vertex of \(B\) and a vertex outside of \(B\).
  Now, by~\cref{claim:cycle_hitting},
  \(C\) is hit by \(F \cup h(B) \subseteq \gamma(e)\)
  and we are done.

  Now consider the case of \(e\) being an edge connecting \(s\)
  and a node of some \(T_B\). Thus, the cut we consider is
  \((B, V(D) \setminus B)\).
  This is exactly the setup of~\cref{claim:cycle_hitting},
  so \(C\) is hit by \(F \cup h(B) = \gamma(e)\).
	
	\paragraph{Onion-guards}
	Again, note that \(\Abs{\Fkt{\theta}{t}} \leq \thmFkt{lem:tripod_hitting}{h}{d,x,z}\) for all \(t \in \V{T}\). We are left with verifying the onion-hitting property of the sets $\Fkt{\theta}{t}$.
	
  For the node \(s\), we argue that all the in-onions of $s$ are hit.
  We consider the following cases. Let \(e\) be an edge of \(T\)
  incident to \(s\), connecting it to some node of \(T_B\), \(B \in \BB\).
  We thus we are considering the cut \((B, V(D) \setminus B)\).
  Let \(\mu\) be an onion immersion model, with root in \(B\)
  and tip in \(V(D) \setminus B\).
  We can reinterpret \(\mu\) as a strongly connected subdigraph
  that contains both a vertex of \(B\) and a vertex outside of \(B\).
  Thus, there is a cycle containing a vertex of \(B\) and a
  vertex outside of \(B\) contained in this subdigraph.
  By~\cref{claim:cycle_hitting},
  this cycle is hit by \(F \cup h(B)\).
  Thus, if \(B = B_{\mathsf{rest}}\) or \(B = B_{\emptyset}\) we are done,
  as in these cases \(F \cup h(B) = F = \theta(s)\).
  So, assume that \(B = B_f\) for some \(f \in \In{}{Z} \setminus F\).
  Let us analyse what edges can there be between \(B_f\) and other members \(B' \in \BB\)
  (outside of the set \(F\)).
  Observe that there cannot be any edges between \(B_f\) to \(B_{f'}\)
  for \(f' \in \In{}{Z} \setminus F\) distinct from \(f\),
  as then some strong component of \(\RR\) would reach both \(f\) and \(f'\)
  contradicting~\cref{claim:atmost_1_edge}.
  Likewise, there cannot be an edge from \(B_f\) to \(B_{\mathsf{rest}}\)
  as this would contradict the definition of \(B_{\mathsf{rest}}\).
  And finally there cannot be an edge from \(B_\emptyset\)
  to \(B_f\) as this would expose a component
  of \(B_\emptyset\) that can reach \(f\).
  Furthermore, there is a single edge connecting
  \(B_f\) to \(Z\) --- the edge \(f\).
  As the same applies to all other components
  \(B_{f'}\) for \(f' \in \In{}{Z} \setminus F\),
  see~\cref{fig:fdeleted_structure},
  this implies that \(\Out{D \setminus F}{B_f \cup B_\emptyset} = \set{f}\).
  Thus, as \(\mu\) contains two edge-disjoint paths
  from the root to the tip,
  it follows that the tip of \(\mu\) must be an element of \(B_\emptyset\).
  But \(\Out{D \setminus F}{B_\emptyset} = \emptyset\),
  which contradicts the existence of the path
  connecting the tip to the root within \(\mu\).
  This shows that there is no in-onion that avoids \(F\)
  at~\(s\).

  \begin{figure}[!ht]
    \centering
    \begin{tikzpicture}
    	\node (centre) at (0,0) {};
    	\node[edge,LavenderMagenta,ellipse,minimum height=1.5cm,minimum width=2.5cm] (Z) at ($(centre)+(0,2)$) {\(Z\)};
    	
    	\node[edge,ellipse,minimum height=1.5cm,minimum width=2.5cm] (B-r) at ($(centre)+(160:3)$) {\(B_{\mathsf{rest}}\)};
    	\node[edge,ellipse,minimum height=1.5cm,minimum width=2.5cm] (B-0) at ($(centre)+(20:3)$) {\(B_{\emptyset}\)};
    	
    	\node[edge,CornflowerBlue,ellipse,minimum height=1.5cm,minimum width=2.5cm] (B-f) at ($(centre)+(0,0)$) {\(B_f\)};
    	\node[edge,CornflowerBlue,ellipse,minimum height=1.5cm,minimum width=2.5cm] (B-fp) at ($(B-f)+(0,-2)$) {\(B_{f'}\)};
    	
    	\draw[directededge,dashed,LavenderMagenta,bend right] ($(Z)+(170:0.8)$) to ($(B-r)+(60:0.4)$);
    	\draw[directededge,dashed,LavenderMagenta,bend right] ($(B-0)+(120:0.5)$) to ($(Z)+(10:0.7)$);
    	
    	\draw[directededge,dashed,out=175,in=5] ($(B-0)+(180:0.7)$) to ($(B-r)+(360:0.7)$);
    	
    	\draw[directededge,PastelOrange!50,line width=3pt,bend right] ($(B-r)+(315:0.6)$) to ($(B-f)+(180:0.7)$);
    	\draw[directededge,PastelOrange!50,line width=3pt,bend right] ($(B-f)+(0:0.8)$) to ($(B-0)+(225:0.5)$);
    	
    	\draw[directededge,PastelOrange!50,line width=3pt,bend right] ($(B-r)+(290:0.5)$) to ($(B-fp)+(180:0.7)$);
    	\draw[directededge,PastelOrange!50,line width=3pt,bend right] ($(B-fp)+(0:0.8)$) to ($(B-0)+(250:0.4)$);
    	
    	\draw[directededge,dashed,out=180,in=45] ($(B-0)+(200:0.7)$) to ($(B-f)+(45:0.5)$);
    	\draw[directededge,dashed,out=125,in=0] ($(B-f)+(135:0.5)$) to ($(B-r)+(-20:0.7)$);
    	
    	\draw[directededge,dashed,bend right] ($(B-f)+(250:0.4)$) to ($(B-fp)+(110:0.4)$);
    	\draw[directededge,dashed,bend right] ($(B-fp)+(70:0.4)$) to ($(B-f)+(300:0.4)$);
    	
    	\draw[directededge,dashed,out=225,in=225] ($(B-r)+(225:0.5)$) to ($(B-fp)+(225:0.5)$);
    	\draw[directededge,dashed,out=315,in=315] ($(B-fp)+(315:0.5)$) to ($(B-0)+(315:0.5)$);
    	
    	\draw[directededge,LavenderMagenta,out=20,in=315] ($(B-f)+(20:0.6)$) to ($(Z)+(315:0.5)$);
    \end{tikzpicture}
    \caption{Depiction of the structure of the digraph \(D \setminus F\).
      Dashed arrows represent non-edges,
      while \textcolor{PastelOrange}{thick orange edges} represent arbitrary amount of edges (possibly zero) in the prescribed direction.
      The \textcolor{LavenderMagenta}{magenta} edge is the edge \textcolor{LavenderMagenta}{\(f\)}.
    }
    \label{fig:fdeleted_structure}
  \end{figure}
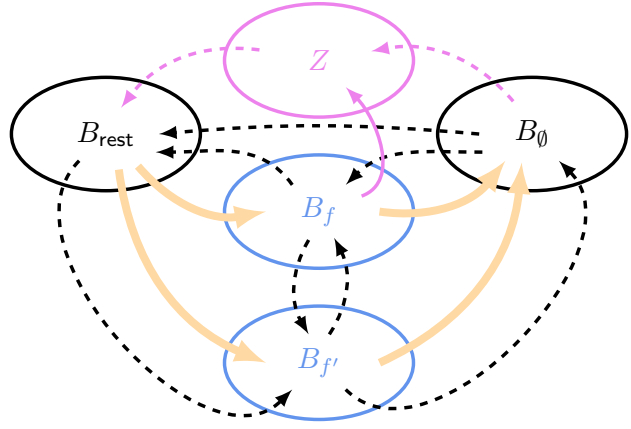

  For a node \(x \in V(T_B)\) for \(B \in \BB\),
  we hit onions in the same orientation
  as in the original \dtcdecomp \((T_B, \beta_B, \gamma_B, \theta_B)\).
  Let \(e\) be an edge incident to \(x\).
  Let \(\mu\) be an onion immersion model.
  Note that we can reinterpret \(\mu\) as a strongly connected
  subdigraph of \(D\).
  If this subdigraph contains
  a vertex of \(B\) and of \(V(D) \setminus B\)
  then it also contains a cycle with a vertex of \(B\) and a vertex outside of \(B\).
  By~\cref{claim:cycle_hitting},
  this cycle is hit by \(F \cup h(B) \subseteq \theta(x)\).
  So we are left with considering onion immersion models
  fully contained in \(B\) or fully contained in \(V(D)\setminus B\).
  In the latter case, the model does not cross the cut
  induced by \(e\); in the former we conclude (as \(\TT_B = (T_B, \beta_B, \gamma_B, \theta_B)\) is a \dtcdecomp)
  that if \(\mu\) has an appropriate orientation
  (the one hit by \(\TT_B\) at \(x\)) then
  \(\mu\) is hit by \(\theta_B(x) \subseteq \theta(x)\);
  in either case we obtain the desired conclusion.
\end{proof}

For the next step, we will need an auxiliary lemma that allows us to ``rebase'' out- and in-onion-star immersions within a flow-connected set, at the cost of losing a bit on their order. 

\begin{lemma}[Onion Transport Lemma]\label{lem:otl}
	There exists a function \(\thmFkt{lem:otl}{h}{}\colon \N\times \N\to \N\) such that for all integers $q,k\in \N$ and every digraph $D$ excluding the onion-star of order $k$ as an immersion, the following holds.
	Suppose $D$ contains a subset of vertices $Z$ that is $\thmFkt{lem:otl}{h}{q,k}$-flow-connected.
	Let $z,z'\in Z$ be distinct vertices, and suppose in $D$ there is an immersion model of an out-onion-star of order $\thmFkt{lem:otl}{h}{q,k}$ rooted at~$z$.
	Then there is also an immersion model of an out-onion-star of order $q$ rooted at $z'$.
	
	Moreover, if $q=q(k)$, there is a fixed-parameter algorithm with parameter $k$ to find such a model in $D$.
\end{lemma}
\begin{proof}
	Define
	\[h_k(x)\coloneqq \bRamsey(x,\thmFkt{lem:onion_harvesting}{g}{k,2}).\]
	We set 
	\begin{align*}
		q_1&\coloneqq h_k(h_k(h_k(h_k(2q))))\\
		q_2&\coloneqq 16\cdot q_1\\
		\thmFkt{lem:otl}{h}{q,k}&\coloneqq q_3\coloneqq h_k(q_2).
	\end{align*}
	Let ${\cal O}=(O_i)_{i\in [q_3]}$ be the family of $q_3$ edge-disjoint onion immersion models with root at $z$ that are assumed to exist. For each $i\in [q_3]$, we let $P^1_i, P^2_i$ be the two trails from $z$ to the tip of $O_i$ within $O_i$, and $P^3_i$ be the trail back from the tip of $O_1$ to $z$ within $O_i$. By shortening the trails if necessary, we may assume that $P^1_i,P^2_i,Q_i$ are paths. For $j\in [2]$, let ${\cal P}^j$ be the set of (pairwise edge-disjoint) trails obtained by taking, for each $i\in [q_3]$, the concatenation of $P^j_i$ and $Q_i$ and removing the last edge. We similarly define families ${\cal Q}^j$ for $j\in [2]$, just from the concatenation of $P^j_i$ and $Q_i$ we remove the first edge instead of the last. Note that each trail from the families ${\cal P}^j,{\cal Q}^j$, $j\in [2]$, visits every vertex at most twice. Moreover, the trails in ${\cal P}^j$ start at $z$ and never visit $z$ again, and the trails in ${\cal Q}^j$ end at $z$ and never visit $z$ before.
	
	Since $z,z'\in Z$ and $Z$ is $q_3$-flow-connected, we may find a family ${\cal K}$ of $q_3$ edge-disjoint paths from $z$ to $z'$ and a family ${\cal L}$ of $q_3$ edge-disjoint paths from $z'$ to $z$. By applying \cref{thm:BipRamsey} (in the bipartite intersection graph of $\cal K$ and $\cal L$) and \cref{lem:onion_harvesting} to ${\cal K}$ and ${\cal L}$, we may assume that all the paths in ${\cal K}\cup {\cal L}$ are pairwise edge-disjoint, at the cost of reducing the sizes of these families to $|{\cal K}|=|{\cal L}|=q_2$. (The other conclusion of the Bipartite Ramsey Theorem (\cref{thm:BipRamsey}) yields, through \cref{lem:onion_harvesting}, an immersion model of an onion-star of order $k$, which is a contradiction.) 
	
	Next, by losing at most half of the paths of ${\cal K}$ and at most half of the onions in ${\cal O}$, we may assume that the edges that the paths of ${\cal K}$ use to leave $z$ are all different from the edges that the paths of ${\cal P}^1$ use to leave $z$. By applying this argument also for the pairs of families $({\cal K},{\cal P}^2)$, $({\cal L},{\cal Q}^1)$, and $({\cal L},{\cal Q}^2)$, we may assume that no path of ${\cal K}$ leaves $z$ using the same edge as a path from ${\cal P}^1\cup {\cal P}^2$, and no path of ${\cal L}$ enters $z$ using the same edge as a path from ${\cal Q}^1\cup {\cal Q}^2$. This is at the cost of restricting ${\cal K}$ and ${\cal L}$ to size at least $q_2/16=q_1$, and ${\cal O}$ to size at least $q_3/16\geq q_1$. 
	
	Now, we may apply \cref{thm:BipRamsey} (in the intersection graph) and \cref{lem:onion_harvesting} to the families ${\cal K}$ and ${\cal Q}^1$ in order to assume that the paths in ${\cal K}\cup {\cal Q}^1$ are pairwise edge-disjoint. This is at a cost of losing the inverse of $h$ in the cardinalities of ${\cal K}$ and ${\cal O}$. By performing the same for pairs of families $({\cal K},{\cal Q}^2)$, $({\cal L},{\cal P}^1)$, and $({\cal L},{\cal P}^2)$, we draw the same conclusion about those pairs. This costs restricting ${\cal K}$, ${\cal L}$, and ${\cal O}$ to subfamilies of size $2q$ each.
	
	Now, the paths of ${\cal K}$ and of ${\cal L}$ are all pairwise edge-disjoint, and also edge-disjoint from the onion immersion models in ${\cal O}$. Since $|{\cal K}|=|{\cal L}|=|{\cal O}|=2q$, one can easily match the paths of ${\cal K}$ and ${\cal L}$ with the onions of ${\cal O}$ in order to construct an immersion model of an out-onion-star of order $q$ rooted at $z'$.
\end{proof}

Next, we show how to decompose the digraph assuming the set of vertices of high cycle-degree is small and well-connected.

\begin{lemma}[Pre-small council lemma]
	\label{lem:SCL_flow}
	There are functions \(\thmFkt{lem:SCL_flow}{f}{} : \N^3 \to \N\) and \(\thmFkt{lem:SCL_flow}{g}{} : \N^2 \to \N\) such that for every digraph \(D\) with at most \(z\) vertices of cycle-degree at least \(d\), which moreover form a \(\thmFkt{lem:SCL_flow}{g}{k,d}\)-flow-connected set, it holds that
	\begin{enumerate}[label=\textbf{(Z\arabic*)},leftmargin=1.2cm]
		\item \(D\) contains a \(\gridImm{k}\)-immersion or
		\label{item:SCL_flow:immersion}
		\item \(\dtcwFkt{D} \leq \thmFkt{lem:SCL_flow}{f}{k,d,z}\).
		\label{item:SCL_flow:decomposition}
	\end{enumerate}
	Moreover, if $d=d(k)$ and $z=z(k)$, there is a fixed-parameter algorithm with parameter $k$ to find either a \(\gridImm{k}\)-immersion model in $D$ or a directed tree-cut decomposition of $D$ of width at most $\thmFkt{lem:SCL_flow}{f}{k,d,z}$.
\end{lemma}
\begin{proof}
We define
\begin{align*}
    m &\coloneq \thmFkt{lem:two_onion_stars}{f}{2k^2}\\
    s &\coloneq \thmFkt{lem:otl}{h}{m, 2k^2}\\
    \ell &\coloneq h(h(s))\\
    \thmFkt{lem:SCL_flow}{g}{k,d} &\coloneqq \max \set{s, 2\ell} \\
    \thmFkt{lem:SCL_flow}{f}{k,d,z} &\coloneqq \thmFkt{lem:tripod_hitting}{h}{d,\thmFkt{thm:tripod-EP}{f}{z^3 \cdot d \cdot \ell},z,k},
\end{align*}
where \(h\) is some function we will define later (c.f.~\cref{claim:uncross}).
Let \(Z\) be the set of all vertices of cycle-degree at least \(d\) in \(D\).
By assumption \(\Abs{Z} \leq z\) and the vertices of \(Z\) are pairwise \(\thmFkt{lem:SCL_flow}{g}{k,d}\)-flow connected in \(D\).

We construct an migration digraph  \(\migrationAux{D} = \Brace{D',S, T}\) to which we can apply~\cref{thm:tripod-EP}.
We apply it twice, once for \(S\)-\(T\)-in-tripods and once for \(S\)-\(T\)-out-tripods.
This leaves us with the following two cases.

\paragraph*{Case 1: hitting set of size \(\thmFkt{thm:tripod-EP}{f}{z^3 \cdot d \cdot \ell}\) for all in-tripods or all out-tripods.}
In this case, we apply \cref{lem:tripod_hitting} to obtain a decomposition of width
\(\thmFkt{lem:tripod_hitting}{h}{d,\thmFkt{thm:tripod-EP}{f}{z^3 \cdot d \cdot \ell},z,k} = \thmFkt{lem:SCL_flow}{f}{k,d,z}\).

\paragraph*{Case 2: \(z^3 \cdot d \cdot \ell\) edge-disjoint in-tripods and \(z^3 \cdot d \cdot \ell\) edge-disjoint out-tripods.}

If the family of in-tripods contains $d$ in-tripods with the same centre $c$, we can extract from it a family of $d$ edge-disjoint paths in $D$ that start and end in $Z$ and go through $c$. 
Since $Z$ is $d$-flow-connected, by \cref{cor:cycleDeg-hitting} we get a contradiction with $c$ having cycle-degree smaller than $d$.
Thus, there are \(z^3 \ell\) in-tripods, and, by symmetric argument, \(z^3 \ell\) out-tripods in $D$ with distinct centres.


Finding a tripod in \(\migrationAux{D}\) yields a similar object in \(D\) as follows.
A \emph{\(Z\)-out-tripod} is the union of three edge-disjoint paths \(P_1,P_2,Q\) such that \(P_1\) and \(P_2\) start in \(Z\) and are otherwise disjoint from \(Z\), \(Q\) ends in \(Z\) and is otherwise disjoint from \(Z\) and \(\End{P_1}=\End{P_2}=\Start{Q}\), called the \emph{centre}.
The vertices \(\Start{P_1},\Start{P_2}\) we refer to as the \emph{sources} of the \(Z\)-onion, and \(\End{Q}\), we refer to as its \emph{sink}.
\emph{\(Z\)-in-tripods} are defined accordingly.

\begin{claim}
	\label{claim:PH}
	Having a packing of \(S\)-\(T\)-out-tripods of order \(k\cdot z^3\) in \(\migrationAux{D}\) yields a packing of \(Z\)-out-tripods of order \(k\) such that they all share the same two sources and the same sink.
	The same holds for in-tripods.
\end{claim}

\begin{claimproof}
	We prove the statement for out-tripods, the arguments for in-tripods are symmetric.
	In fact, this is obtained by a rather simple pigeon-hole-argument.
	
	Let \(\mathcal{P}\) be the packing of \(S\)-\(T\)-out-tripods of order \(k\cdot z^3\).
	Each of these sources and the sink of \(P\) correspond to a vertex in \(Z\), by construction of \(\migrationAux{D}\).
	So we consider the projection \(P_Z\) of \(P\) into \(D\), replacing the edges in \(P\cap \Out{D'}{S}\) by out-edges from \(Z\) in \(D\) and the edge in \(P\cap \In{D'}{T}\) by the corresponding in-edge into \(Z\) in \(D\).
	Note that \(P_Z\) is a \(Z\)-tripod and let \(\mathcal{P}_Z = \Set{P_Z\mid P\in \mathcal{P}}\) be the packing of \(Z\)-onions obtained.
	
	To simplify the argument assume that the \(Z\)-onions have a \emph{first} and a \emph{second} source, that is, making them distinguishable in an arbitrary way.
	
	As there is only \(z\) vertices in \(Z\), there is a vertex \(z_1 \in Z\) such that \(k\cdot z^2\) members \(\mathcal{P}_1\) of \(\mathcal{P}_Z\) have \(z_1\) as their first sources.
	Similarly, there is a vertex \(z_2 \in Z\) such that \(k\cdot z\) members \(\mathcal{P}_2\) of \(\mathcal{P}_1\) have \(z_2\) as their second source.
	
	Finally, there is a vertex \(z_3 \in Z\) such that \(k\) members \(\mathcal{P}_3\) of \(\mathcal{P}_2\) have \(z_3\) as their sink.
	
	Thus, the claim holds.
\end{claimproof}
Let us focus on the out-tripod packing. We will use it to exhibit a large out-onion-star
with a centre at some vertex in \(Z\).
By~\cref{claim:PH}, we may assume that we have a packing of size \(\ell\) of
\(Z\)-out-tripods, where each out-tripod has once source in \(z_1 \in Z\) and
the second source in \(z_2 \in Z\) (\(z_1\) and \(z_2\) are not necessarily distinct),
and each tripod has \(z_3 \in Z\) as its sink (again, \(z_3\) may
coincide with one or both of \(z_1,z_2\)). Moreover
all these tripods have distinct centres by construction.
Call this \(Z\)-out-tripod packing \(\TT_{\lout}\).

Let us consider the following auxiliary-digraph \(D_{\laux}\), obtained from \(D\)
by adding a new vertex \(x\), and \(\ell\) parallel edges from \(z_1\) to \(x\),
and \(\ell\) parallel edges from \(z_2\) to \(x\).
Consider removing a set \(F \subseteq E(D_{\laux})\) of size smaller than \(2\ell\).
Since the in-degree of \(x\) is equal to \(2\ell\), there is an in-edge
incident to \(x\) not in \(F\). Without loss of generality,
assume this is an edge from \(z_1\) to \(x\).
As \(Z\) is at least \((2\ell)\)-flow-connected in \(D \subseteq D_{\laux}\), in
\(D_\laux \setminus F\) there is still a \(z_3\)-\(z_1\) path.
Thus, by Menger's Theorem (c.f.~\cref{thm:menger}),
in \(D_{\laux}\) there is a \(z_3\)-\(x\) flow of size \(2\ell\).
This flow must correspond to exactly \(\ell\) paths from \(z_3\) to \(z_1\),
and \(\ell\) paths from \(z_3\) to \(z_2\).
Let \(\PP_1\) be the set of \(\ell\) paths from \(z_3\) to \(z_1\),
and \(\PP_2\) be the set of \(\ell\) paths from \(z_3\) to \(z_2\).
By construction \(\PP_1 \cup \PP_2\) is a set of pairwise edge-disjoint
trails all starting at \(z_3\).
If \(z_3 = z_1\), then we consider \(\PP_1\) to be a family
of \(\ell\) trivial trails (consisting of zero edges),
and similarly for \(z_3 = z_2\) and \(\PP_2\).

\begin{claim}
  \label{claim:uncross}
  There exists a function \(h : \N \rightarrow \N\) such that the following holds.

  Let \(\TT\) be a set of edge-disjoint \(Z\)-out-tripods in \(D\),
  all of which have
  the same vertex \(z \in Z\) as sink, and let \(\PP\) be a set of pairwise
  edge-disjoint trails in \(D\), all starting at \(z\).

  Then, for any positive integer \(n\), if
  \(\size*{\TT} \geq h(n)\),
  and \(\size*{\PP} \geq h(n)\),
  then either
  \begin{itemize}
    \item \(D\) contains an immersion of an onion star of order \(2k^2\), or
    \item there exists a subsets \(\TT' \subseteq \TT, \PP' \subseteq \PP\),
      each of size \(n\),
      such that each tripod in \(\TT'\) is edge-disjoint
      from every trail in \(\PP'\).
  \end{itemize}
\end{claim}
\begin{claimproof}
  Set 
  \begin{equation*}
    h(n) = \bRamsey(\thmFkt{lem:onion_harvesting}{g}{2k^2,2}, \bRamsey(\thmFkt{lem:onion_harvesting}{g}{2k^2,2}, n)).
  \end{equation*}
  Construct two sets of trails \(\QQ_1 = \set{Q^1_T : T \in \TT}, \QQ_2 = \set{Q^2_T : T \in \TT}\) indexed by \(\TT\)
  as follows:
  for each \(Z\)-out-tripod \(T \in \TT_{\lout}\) with the centre \(c \in V(D) \setminus Z\),
  let \(L_1\) be the \(z_1\)-\(c\) path contained in \(T\) where \(z_1 \in Z\) is one
  of the sources of~\(T\),
  let \(L_2\) be the \(z_2\)-\(c\) path contained in \(T\)
  where \(z_2\) is the second source of~\(T\),
  and let \(R\) be the \(c\)-\(z_3\) trail contained in \(T\);
  add \(Q^1_T\) defined as the concatenation of \(L_1\) and \(R\) to~\(\QQ_1\),
  and~\(Q^2_T\) the concatenation of \(L_2\) and~\(R\) to~\(\QQ_2\).
  Thus, each of \(\set{\QQ_1, \QQ_2}\) is a set of pairwise edge-disjoint
  trails in \(D\) all terminating at \(z\).

  Build an auxiliary bipartite graph on \(\TT \cup \PP\) as follows:
  \(T \in \TT\) and \(P \in \PP\) are adjacent if and only if~\(Q^1_T\) and~\(P\) share an edge.
  By~\cref{thm:BipRamsey}, either there is a complete bipartite
  subgraph of size at least \(\thmFkt{lem:onion_harvesting}{g}{2k^2,2}\)
  or an anti-complete bipartite subgraph of size at least \(\bRamsey(\thmFkt{lem:onion_harvesting}{g}{2k^2,2}, n)\).

  Assume we are in the former case.
  Thus, we have a subset \(\TT_c \subseteq \TT\),
  and \(\PP_c \subseteq \PP\) such that for each \(T \in \TT_c\) and \(P \in \PP_c\),
  the trails~\(Q^1_T\) and~\(P\) share an edge, and each is of size \(\thmFkt{lem:onion_harvesting}{g}{2k^2,2}\).
  This is exactly the setup of~\cref{lem:onion_harvesting},
  and in this case we conclude that \(D\) contains an onion star immersion
  of order~\(2k^2\).

  In the latter case, we have \(\TT_0 \subseteq \TT\) and \(\PP_0 \subseteq \PP\) such that for each \(T \in \TT_0\) and \(P \in \PP_0\),
  the trails~\(Q^1_T\) and~\(P\) are edge-disjoint,
  and each is of size \(\bRamsey(\thmFkt{lem:onion_harvesting}{g}{2k^2,2}, n)\).
  Now, build another auxiliary bipartite graph on~\(\TT_0 \cup \PP_0\),
  but now instead focus on \(Q_2\):
  \(T \in \TT_0\) and \(P \in \PP_0\) are adjacent if and only if
  \(Q^2_T\) and \(P\) share an edge.

  By~\cref{thm:BipRamsey}, either there is a complete bipartite
  subgraph of size at least \(\thmFkt{lem:onion_harvesting}{g}{2k^2,2}\)
  or a anti-complete bipartite subgraph of size at least \(n\).

  The former case unfolds exactly as in the previous
  step, with~\cref{lem:onion_harvesting} yielding a
  onion star immersion of order \(2k^2\).
  In the latter case, we obtain the desired subsets,
  as each tripod \(T\) is edge-covered by the two paths \(Q_1\) and \(Q_2\).
\end{claimproof}

Apply~\cref{claim:uncross} to \(\PP_1\) and \(\TT_{\lout}\).
These two set are each of size \(\ell = h(h(s))\).
Thus, we either obtain an onion-star of order \(2k^2\) as an immersion,
and thus also \(\gridImm{k}\),
or a subset \(\TT' \subseteq \TT_{\lout}\) and \(\PP'_1 \subseteq \PP_1\)
each of size \(h(s)\) that are pairwise edge-disjoint.
Again, apply~\cref{claim:uncross} to \(\PP_2\)
and \(\TT'\).
And either we conclude with a \(\gridImm{k}\) as an immersion in \(D\),
or we obtain a subset \(\TT''\) and \(\PP'_2 \subseteq \PP_2\),
each of size \(s\) and pairwise edge-disjoint.

We now use \(\TT'', \PP'_1 \PP'_2\) to assemble an out-onion-star
with a centre \(z_3\) and of order \(s\).
We do it as follows:
for each \(T\in \TT''\) let \(c \in V(D) \setminus Z\)
be the centre of \(T\), further let \(L_i\) be the \(z_i\)-\(c\)
path contained in \(T\).
Extend \(L_i\) by prepending a unique path from \(\PP'_i\),
thus obtaining a \(z_3\)-\(c\) trail.
This is an immersion model of an onion with root at \(z_3\)
and tip at \(c\).
As all tripods have distinct centres, and all objects
we used are pairwise edge-disjoint by construction,
this is a valid immersion model of an out-onion-star of order \(s\),
with centre at a vertex \(z \in Z\).

By a symmetric argument, applied to \(\TT_{\variablestyle{in}}\),
there exists also an immersion model of an in-onion start of
order~\(s\) with centre at \(z' \in Z\).

As the set \(Z\) is at least \(s\)-connected, and \(s = \thmFkt{lem:otl}{h}{m, 2k^2}\)
by~\cref{lem:otl},
either \(D\) contains an onion star of order \(2k^2\) as an immersion, or
there is a common \(z\) that is a centre of an
of both and in-onion-star and an out-onion-star immersion models in \(D\),
each of order \(m = \thmFkt{lem:two_onion_stars}{f}{2k^2}\),
and so by~\cref{lem:two_onion_stars} there is an
onion star immersion of order \(2k^2\) in \(D\).
Thus, \(D\) contains \(\gridImm{k}\).
\end{proof}

Finally, we prove the main result of this section, which is a variant of \cref{lem:SCL_flow} that no longer requires connectivity between the vertices of high cycle-degree.

\begin{lemma}[Small council lemma]
	\label{lem:SCL}
	There is a function \(\thmFkt{lem:SCL}{f}{} \colon \N^3 \to \N\) such that for every $k\in \N$ and a digraph \(D\) with at most \(z\) vertices of cycle-degree at least \(d\), it holds that
	\begin{enumerate}[label=\textbf{(S\arabic*)},leftmargin=1.2cm]
		\item \(D\) contains a \(\gridImm{k}\)-immersion or
		\label{item:SCL:immersion}
		\item \(\dtcwFkt{D} \leq \thmFkt{lem:SCL}{f}{k,d,z}\).
		\label{item:SCL:decomposition}
	\end{enumerate}
	Moreover, if $d=d(k)$ and $z=z(k)$, there is a fixed-parameter algorithm with parameter $k$ to find either a \(\gridImm{k}\)-immersion model in $D$ or a directed tree-cut decomposition of $D$ of width at most $\thmFkt{lem:SCL}{f}{k,d,z}$.
\end{lemma}
\begin{proof}
	Fix
	\begin{equation*}
		\ell\coloneqq \thmFkt{lem:SCL_flow}{g}{k,d}.
	\end{equation*}
	Let \(Z\) be the set of vertices of cycle-degree at least \(d\) in \(D\); then \(|Z|\leq z\).
	
	We construct a set of edges \(W\) in \(D\) by the following iterative procedure.
	\begin{itemize}
		\item Start with \(W\coloneqq \emptyset\).
		\item Find a minimal set \(F\) such that there exist distinct $a,b \in Z$ and an $a$-$b$-path in $D-W$, but no $a$-$b$-path in $(D-W)-F$; if $|F| < \ell$, add $F$ to $W$ and proceed to the next iteration, otherwise $W$ is the final set.  
	\end{itemize}
	Observe that every ordered pair \((a,b)\in Z^2\) is considered at most once in this procedure, because once a set \(F\) intersecting every \(a\)-\(b\)-path is added to \(W\), the condition that there exists an \(a\)-\(b\)-path in \(D-W\) will never hold again.
	Thus, the procedure eventually constructs a set \(W\)~with \[|W|\leq z(z-1)(\ell-1)\] elements,
	that, by Menger's Theorem (\cref{thm:menger}), satisfies the following assertion:
	\begin{quote}
		For every pair of distinct vertices \(a,b\in Z\), either there is no \(a\)-\(b\)-path in \(D-W\), or there is an \(a\)-\(b\)-flow of order \(\ell\) in \(D-W\).
	\end{quote}
	
	\newcommand{\CCs}{{\cal C}}
	
	Let \(\CCs\) be the family of strongly connected components of \(D-W\). Note that for each \(A\in \CCs\), the set \(Z^A\coloneqq V(A)\cap Z\) is \(\ell\)-flow-connected in \(A\). Also, \(|Z^A|\leq |Z|\leq z\) and all the vertices of \(A\) lying outside of \(Z^A\) have cycle-degrees smaller than \(d\). Therefore, we may apply \cref{lem:SCL_flow} to \(A\) and \(Z^A\), thus obtaining either a \(\gridImm{k}\)-immersion model in \(A\) or a \dtcdecomp \(\Brace{T^A,\beta^A,\gamma^A,\theta^A}\) of \(A\) of width at most \(\thmFkt{lem:SCL_flow}{f}{k,d,z}\). If the first case occurs for any \(A\in \CCs\), then we have exposed an \(\gridImm{k}\)-immersion model in \(D\), which is conclusion \cref{item:SCL:immersion}. So from now we may assume that we have constructed a suitable decomposition \(\Brace{T^A,\beta^A,\gamma^A,\theta^A}\) for each \(A\in \CCs\), and we work towards conclusion~\cref{item:SCL:decomposition}.
	
	We combine the decompositions \(\Set{\Brace{T^A,\beta^A,\gamma^A,\theta^A}\colon A\in \CCs}\) into a \dtcdecomp \(\Brace{T,\beta,\gamma,\theta}\) of $D$ as follows:
	\begin{itemize}
		\item For each \(A\in \CCs\), select an arbitrary node \(r^A\) of \(T^A\). Construct \(T\) by taking the disjoint union of all the trees \(\Set{T^A\colon A\in \CCs}\), adding a fresh node $r$, and making it adjacent to all the nodes \(r^A\) for \(A\in \CCs\).
		\item Set \(\beta(r)=\emptyset\), \(\theta(r)=W\), and \(\gamma(rr^A)=W\) for every \(A\in \CCs\).
		\item For each \(A\in \CCs\), node \(x\) of \(T^A\), and edge \(e\) of \(T^A\), set $\theta(x)\coloneqq \theta^A(x)\cup W$ and $\gamma(e)\coloneqq \gamma^A(e)\cup W$.
	\end{itemize}
	Clearly, the width of \(\Brace{T,\beta,\gamma,\theta}\) is bounded by \[\thmFkt{lem:SCL_flow}{f}{k,d,z}+|W|\leq \thmFkt{lem:SCL_flow}{f}{k,d,z}+z(z-1)(\ell-1)= \thmFkt{lem:SCL_flow}{f}{k,d,z}+z(z-1)(\thmFkt{lem:SCL_flow}{g}{k,d}-1).\]
	So, to reach conclusion \cref{item:SCL:decomposition}, set to the right hand side of the above to \(\thmFkt{lem:SCL}{f}{k,d,z}\).
	It suffices to argue that \(\Brace{T,\beta,\gamma,\theta}\) is indeed a \dtcdecomp of \(D\).
	
	It is clear from the construction that \(\{\beta(x)\colon x\in V(T)\}\) is a near-partition of the vertex set of \(D\).
	That \(W\) is an onion-guard for the node \(r\) and also an edge-guard for the edge \(rr^A\), for each \(A\in \CCs\), follows immediately from the observation that there is no directed cycle visiting two subgraphs of \(D-W\) induced by bags of different components of $T - r$.
	
	Next, consider an edge \(e\) of \(T^A\) for some \(A\in \CCs\).
	Observe that for every cycle \(C\) in \(D\) that crosses \(e\) (treated as an edge of \(T\)), either \(C\) also crosses the edge \(rr^A\) and consequently is intersected by \(W\), or \(C\) resides entirely within \(A\) and consequently is intersected by \(\gamma^A(e)\), due to this set being an edge-guard for \(e\) in \(\Brace{T^A,\beta^A,\gamma^A,\theta^A}\).
	In conclusion, \(\gamma(e)=\gamma^A(e)\cup W\) is a suitable edge-guard for $e$ in \(\Brace{T,\beta,\gamma,\theta}\).
	An analogous argument shows that \(\theta(x)=\theta^A(x)\cup W\) is a suitable onion-guard for any node \(x\) of \(T^A\).
	
	Last, the existence of the desired algorithm follows directly from the proof.
\end{proof}

\section{Digraphs of large tree-cutwidth contain large wall immersions}
\label{sec:hard}
%

In this section we complete the proof of our main result, \cref{thm:grid_or_small_dtcw}, that is, we prove that every digraph of large directed tree-cutwidth contains a large wall immersion.
We make use of one more result from the existent work on onion-star immersions~\cite{bozyk2022digraphs}, which says that a large highly connected set of vertices forces the existence of a large onion-star immersion.

\begin{lemma}[\cite{bozyk2022digraphs}]
	\label{lem:onion_star}
	There is a function \(\thmFkt{lem:onion_star}{h}{} \colon \N \to \N\) such that for every digraph \(D\) and $p\in \N$, the following holds: If \(D\) contains a \(\thmFkt{lem:onion_star}{h}{p}\)-flow-connected set of size $2p$, then \(D\) also contains an immersion model of the onion-star of order $p$.
	
	Moreover, there is a fixed-parameter algorithm (with respect to parameter $p$) that given $D$ and a flow-connected set as above, finds an immersion model of the onion-star of order $p$.
\end{lemma}

We remark that the algorithmic claim in \cref{lem:onion_star} is not explicitly mentioned in \cite{bozyk2022digraphs}, but follows easily from the (short) proof of this result in \cite{bozyk2022digraphs} and the algorithmic statement of \cref{lem:onion_harvesting}. In fact, \cref{lem:onion_star} is proved by one application of Menger's Theorem and several applications of \cref{lem:onion_harvesting}.

Next, we need to understand how to project onions in the entire graph to onions in the torsos of single bags of a tree-partition.

Let \(D\) be a digraph, $(X,Y)$ a partition of $V(D)$, and \(F\) be a set of edges in \(D\). If $u \in X$ and $v \in Y$, we call the union of three trails, two \(u\)-\(v\)-trails \(P_1,P_2\) and one \(v\)-\(u\)-trail \(P_3\), an \emph{\((F,(X,Y))\)-quasi-onion} if the three trails are edge-disjoint except for possibly sharing edges from \(F\).
We again refer to \(u\) as the \emph{root} of the quasi-onion and to \(v\) as its \emph{tip}.
We call an $(F,(X,Y))$-quasi-onion \emph{minimal} if there is no $(F,(X,Y))$-quasi-onion $O'$ consisting of $u'$-$v'$-trails $P'_1,P'_2$ and a $v'$-$u'$-trail $P'_3$, for some $u' \in X$ and $v' \in Y$, such~that 
\begin{align*}
\bigcup_{i \in [3]} \E{P'_i} \subseteq \bigcup_{i \in[3]} \E{P_i}\qquad\textrm{and}\qquad\sum_{i \in[3]} \Abs{\E{P'_i}} <\sum_{i \in[3]} \Abs{\E{P_i}}.
\end{align*}

\begin{lemma}[Quasi-onion Lemma]
	\label{lem:quasi_onions}
	Let \(D\) be a digraph with a given set of edges \(F\) that forms a matching (i.e.~they do not share any endpoints) and a bipartition \((X,Y)\) of its vertices such that no edge of \(F\) crosses \((X,Y)\).
	Then every minimal \((F,(X,Y))\)-quasi-onion with root \(u \in X\) and tip \(v \in Y\) is an immersion model of an onion.
\end{lemma}
\begin{proof}
	Let \(P_1\) and \(P_2\) be the two \(u\)-\(v\)-trails and \(P_3\) be the \(v\)-\(u\)-trail of a minimal \((F,(X,Y))\)-quasi-onion with root \(u \in X\) and tip \(v \in Y\).
	Note that the minimality asserts that $P_1,P_2$ and $P_3$ are paths. Also, note that since $P_1,P_2,P_3$ constitute a \((F,(X,Y))\)-quasi-onion, every edge shared by two of those trails must have both endpoints in $X$ (\emph{lie} in $X$) or both endpoints in $Y$ (\emph{lie} in $Y$).
	
	First we show that \(P_1\) is disjoint from \(P_2\).
	Assume otherwise and consider the first edge \(f \in F\) along \(P_1\) that it shares with \(P_2\).
	If \(f\) lies in \(Y\), then we define \(v' \coloneqq \tail{f}\) as the new tip and \(P'_1 \coloneqq P_1v'\), \(P'_2 \coloneqq P_2v'\) and \(P'_3 \coloneqq \red{v'P_2 \cdot P_3}\).
	This way the overall number of edges drops by the number of edges in \(v'P_1\) which is at least 1, as it contains \(f\). 
	The paths \(P'_1,P'_2\) and \(P'_3\) therefore yield a smaller $(F,(X,Y))$-quasi-onion with root \(u\) and tip \(v'\), a contradiction.
	
	If \(f\) lies in \(X\), then we define \(u' \coloneqq \head{f}\) as the new root.
	Then, we define \(P'_1 \coloneqq u'P_1\), \(P'_2 \coloneqq u'P_2\) and \(P'_3 \coloneqq \red{P_3 \cdot P_1u}\).
	Again, we obtain a smaller $(F,(X,Y))$-quasi-onion (by the number of edges in \(P_2u'\), in particular $f$) with root \(u'\) and tip \(v\), a contradiction.
	
	So, we may assume that the paths \(P_1\) and \(P_2\) are indeed edge-disjoint.
	We continue to consider the possibility of \(P_i\) intersecting \(P_3\) for \(i \in [2]\).
	Let \(j \coloneqq 3-i\).

	Consider the first intersection edge \(f \in F\) along \(P_i\) with \(P_3\).
	If \(f\) lies in \(Y\), then we define \(v' \coloneqq \tail{f}\) as the new tip.
	Then, we choose \(P'_i \coloneqq P_iv'\), \(P'_j \coloneqq \red{P_j \cdot P_3v'}\) and \(P'_3 \coloneqq v'P_3\).
	This yields a smaller $(F,(X,Y))$-quasi-onion, a contradiction.
	
	If \(f\) lies in \(X\), then we define \(u' \coloneqq \head{f}\) and \(P'_i \coloneqq u'P_i\), \(P'_j \coloneqq \red{u'P_3\cdot P_j}\) and \(P'_3 \coloneqq P_3u'\).
	This again would yield a smaller $(F,(X,Y))$-quasi-onion, a contradiction.
	
	So, indeed, the three paths \(P_1,P_2\) and \(P_3\) are all edge-disjoint and thus form an actual immersion model of an onion.
\end{proof}

Now, we can finally establish our main result.

\ThmImm*
\begin{proof}
	Throughout the proof we argue the combinatorial statement. At the end, we argue how the reasoning can be adapted to yield a fixed-parameter algorithm.
	
	If \(D\) contains an immersion model of \(\gridImm{k}\), then we are done, so we can assume it does not.
	
	Let \(q \coloneqq 2k^2\).
	We define
	\begin{equation*}
	\thmFkt{thm:grid_or_small_dtcw}{f}{k} \coloneqq	3 \cdot \thmFkt{lem:SCL}{f}{2k,2,2q} \cdot \Brace{1 + \thmFkt{lemma:potato}{f}{\thmFkt{lem:onion_star}{h}{q}}} 
		+ \thmFkt{lemma:potato}{f}{\thmFkt{lem:onion_star}{h}{q}}.
	\end{equation*}
	
	As an immersion model of an onion-star of order $q$ would also yield a \(\gridImm{k}\)-immersion model (\cref{lem:gridImm_from_onion_star}), \cref{lem:onion_star} implies that \(D\) does not contain any set of \(2q\) vertices that is \(\thmFkt{lem:onion_star}{h}{q}\)-flow-connected.
	
	By~\cref{lemma:potato}, there is a \(\thmFkt{lem:onion_star}{h}{q}\)-\linked \potdecomp \(\Brace{\hat{T},\hat{\beta},\hat{\gamma}}\) of \(D\) of \porosity \(\thmFkt{lemma:potato}{f}{\thmFkt{lem:onion_star}{h}{q}}\).
	As for every $x \in V(T)$ the bag $\hat{\beta}(x)$ is $\thmFkt{lem:onion_star}{h}{q}$-flow-connected, we have $|\hat{\beta}(x)| \leq 2q$.
	
	We root $\hat{T}$ in an arbitrary node.
	By~\cref{lem:rooting_potdecomps} we can assume that \(\Brace{\hat{T},\hat{\beta},\hat{\gamma}}\) is \coherent.
	This allows us to build a torso \(\torso{x}\) for every \(x \in \V{\hat{T}}\), as described in~\cref{sec:torso}.
	
	Note that since all auxiliary vertices of \(\torso{x}\) have cycle-degree exactly one, \(\torso{x}\) contains at most~\(2q\) vertices of cycle-degree at least two, i.e.~the original vertices of \(\hat{\beta}(x)\).
	
	Therefore, we can apply~\cref{lem:SCL} with parameter $2k$ and, as by~\cref{lem:grid_imm_in_torsos} we cannot get~\cref{item:SCL:immersion}, it yields~\cref{item:SCL:decomposition}, that is, a \dtcdecomp \(\TT_x^0\coloneqq\Brace{T^0_x,\beta^0_x,\gamma^0_x,\theta^0_x}\) of \(\torso{x}\) of width at most \(\thmFkt{lem:SCL}{f}{2k,2,2q}\).
	
	In order to be able to combine these decompositions into one decomposition of the entire graph, we first tidy it up a bit, ensuring that auxiliary vertices for the same subtree are placed into the same bag. To achieve this, we pay a multiplicative factor \(3\) in the width.
	
	\begin{claim}
		\label{claim:aux-vertices-together}
		We can construct a \dtcdecomp \(\TT_x\coloneqq \Brace{T_x,\beta_x,\gamma_x,\theta_x}\) from \(\TT_x^0\) such that for every pair \(a_i\) and \(b_i\) of auxiliary vertices, both of them lie in the same bag. The width of $\TT_x$ is at most \(3\cdot \thmFkt{lem:SCL}{f}{2k,2,2q}\).
	\end{claim}
	\begin{claimproof}
	If $x$ is a leaf of $\hat{T}$, there is nothing to prove, so assume that $y_1,\ldots, y_\ell$ are the children of $x$ in $\hat{T}$. 
		For every $i \in [\ell]$, let \(t_i \in \V{T^0_x}\) be such that \(a_i \in \beta^0_x(t_i)\) and let \(t'_i \in \V{T^0_x}\) be such that \(b_i \in \beta^0_x(t'_i)\).
		Let \(P_i\) be the unique \(t_i\)-\(t'_i\)-path in \(T^0_x\).
		
		Recall that we denote by $O_i$ the onion consisting of the three edges between $a_i$ and $b_i$.

	\begin{claim}
		\label{claim:edge_guards_splitting_contain_aux_edge}
		If \(e \in \E{P_i}\), then \(\gamma^0_x(e)\) contains at least one of the auxiliary edges \(\E{O_i}\).
	\end{claim}
	\begin{claimproof}
		This follows from \(O_i\) containing a cycle that visits \(a_i\) and \(b_i\), which  crosses \(e\) and thus has to be hit by the edge-guard \(\gamma^0_x(e)\).
	\end{claimproof}
	
	\begin{claim}
		\label{claim:onion_guards_splitting_contain_aux_edge}
		If \(t \in \V{P_i}\setminus\Set{t_i,t'_i}\), then \(\theta^0_x(t)\) contains at least one of the auxiliary edges \(\E{O_i}\).
	\end{claim}
	\begin{claimproof}
		This follows from \(O_i\) containing an onion with tip \(a_i\) and root \(b_i\), which crosses \(t\) and thus has to be hit by the onion-guard \(\theta^0_x(t)\).
	\end{claimproof}
		
		We define \(T_x \coloneqq T^0_x\).
		Then, for every \(t \in \V{T_x} = \V{T^0_x}\) we define 
		\begin{equation*}
			\beta_x(t) \coloneqq \beta^0_x(t)\setminus\Set{b_i\mid a_i \not\in \beta^0_x(t)} \cup \Set{b_i\mid a_i \in \beta^0_x(t).}
		\end{equation*}
		That is, we move each auxiliary vertex \(b_i\) to the bag containing \(a_i\).
		This increases the size of the bags at most by a factor of two.

		For every \(e \in \E{T_x} = \E{T^0_x}\) we define
		\begin{equation*}
			\gamma_x(e) \coloneqq \gamma^0_x(e) \cup \Set{\Brace{a_i,b_i}\mid e\in\E{P_i}}.
		\end{equation*}
		The size of the edge-guard of $e$ grows by the number of paths \(P_i\) the edge $e$ lies on.
	By~\cref{claim:edge_guards_splitting_contain_aux_edge}, these are at most as many as the original edge-guard size.
	So, the edge-guards grow by at most a factor two.
		
		Similarly, for every \(t \in \V{T_x} = \V{T^0_x}\) we define
		\begin{equation*}
			\theta_x(t) \coloneqq \theta^0_x(t) \cup \Set{\Brace{a_i,b_i}\mid t \in \V{P_i}}.
	\end{equation*}
	Each onion-guard grows by one edge for every \(P_i\) its corresponding node $t$ in $T_x=T_x^0$ lies on.
	Thus, if we denote by $n_i$ and $n_e$, respectively, the number of paths $P_i$ that contain $t$ as an internal vertex, and as an endpoint, we have $|\theta_x(t)| \leq |\theta^0_x(t)| + n_i + n_e$. 
	If $t$ is an internal node of \(P_i\), by~\cref{claim:onion_guards_splitting_contain_aux_edge}, $\theta_x^0(t)$ already contains an edge from \(O_i\), thus $n_i \leq |\theta^0_x(t)|$.
	Note that $t$ can be only an endpoint of \(P_i\) if $\beta^0_x(t)$ contains \(a_i\) or \(b_i\). Therefore, $n_e \leq |\beta^0_x(t)|$.
	Thus, the new onion-guards are at most twice their original size plus the maximum bag-size, i.e.~three times the width of $\TT^0_x$.
	
	It remains to show that both the new edge-guards and the new onion-guards are doing their job in hitting all crossing cycles or onions.
	Let \(e\in \E{T_x}\) be an edge in the decomposition and \(C\) be a cycle of \(\torso{x}\) crossing \(e\).
	If \(C\) is a cycle crossing \(e\) in \(\TT^0_x\) as well, then it is hit by \(\gamma^0_x(e) \subseteq \gamma_x(e)\).
	If \(C\) does not cross \(e\) in \(\TT^0_x\), then there is \(i \in [\ell]\) such that \(e \in \E{P_i}\) and \(C\) visits \(b_i\). Thus, $C$ contains the edge \(\Brace{a_i,b_i}\), which ensures that $C$ is hit by \(\gamma_x(e)\).
	
	For the onion-guards, let \(t \in \V{T^0_x} = \V{T_x}\) and assume without loss of generality hat \(\theta^0_x(t)\) hits all in-onions of $t$ in $\TT_x^0$.
	Let \(O\) be an in-onion of \(t\) in \(\TT_x\) (thus $O$ is an onion immersion model in $\torso{x}$).
	If \(O\) is also an in-onion of $t$ in $\TT_x^0$, then it is hit by \(\theta^0_x(t) \subseteq \theta_x(t)\).
	If \(O\) only became an in-onion of $t$ in $\TT_x$ due to moving the auxiliary vertices, then there is \(i \in [\ell]\) such that \(t \in \V{P_i}\) and \(b_i\) is the root or the tip of \(O\).
	In particular, $O$ must contain the edge $(a_i,b_i)$, and hence $O$ is hit by \(\Brace{a_i,b_i} \in \theta_x(t)\).
\end{claimproof}

Based on the cleaned-up \dtcdecomps \(\Brace{T_x,\beta_x,\gamma_x,\theta_x}\) of every torso \(\torso{x}\) of the tree-partition \(\Brace{\hat{T},\hat{\beta},\hat{\gamma}}\), we now construct a \dtcdecomp \(\Brace{T,\beta,\gamma,\theta}\) of the entire graph~\(D\).

\paragraph{Constructing the decomposition tree:}
We start with the union of the trees from all the torso decompositions: \(T'' \coloneqq \bigcup_{x \in \V{\hat{T}}} T_x\).
This yields a forest and we want to connect it into a single tree.
By \cref{claim:aux-vertices-together}, the pairs of auxiliary vertices always lie in the same bag, which we make use of for deciding how to connect the components of \(T''\).

For every vertex \(x \in \V{\hat{T}}\), we add exactly one edge between \(T_x\) and \(T_{y_i}\) for every child \(y_i\) of \(x\) in \(\hat{T}\), so that the resulting graph \(T\) is a tree, as follows.
As \(y_i\) is a child of \(x\), there is a node \(t \in \V{T_x}\) such that the bag \(\beta_x(t)\) contains both \(a_i\) and \(b_i\).
We add an edge between \(t\) and an arbitrary node of \(T_{y_i}\).
These edges correspond one-to-one to the edges of \(\hat{T}\) and we refer to them as \emph{outer edges} of \(T\), while we call the remaining edges, which come from inside the torso decompositions, the \emph{inner edges} of \(T\).


\paragraph{Constructing the bags:}
As every node \(t\) of \(T\) is also a node of some \(T_x\) for some \(x \in \hat{T}\), there is a bag \(\beta_x(t)\).
This bag potentially contains some auxiliary vertices of the torso \(\torso{x}\), thus we define \(\beta(t) \coloneqq \beta_x(t) \cap \V{D}\) in order to remove these.
Clearly the bags do not grow in size with doing this.

\paragraph{Constructing the edge-guards:}
We distinguish between the outer and the inner edges of \(T\).
If \(e \in \E{T}\) is an inner edge of \(T\), then there is a node \(x \in \V{\hat{T}}\) such that \(e \in \E{T_x}\). Thus, there is an edge-guard \(\gamma_x(e)\) assigned to it.
Let \(x'\) and $y_1,\ldots,y_\ell$ be, respectively, the parent and the children of $x \in \hat{T}$ (if they exist).  
We define 
\begin{align*}
	\gamma(e) \coloneqq{} &\big(\torsoEdgeMap{x}(\gamma_x(e) \cap \domain{\torsoEdgeMap{x}})\big) \cup{}\\
	&\hat{\gamma}(x',x) \cup{}\\
	&\bigcup \Set{\hat{\gamma}(x,y_i) \mid i\in [\ell]\text{ and }E(O_i)\cap \gamma_x(e)\neq \emptyset }.
\end{align*}
In other words, $\gamma(e)$ consists of the origins of every edge from $\gamma_x(e)$ that is not an auxiliary edge, 
the edge-guard of $(x',x) \in E(\hat{T})$, 
and edge-guards of $(x,y_i) \in E(\hat{T})$ for every $i \in [\ell]$ for which $\gamma_x(e)$ contains an edge of the onion $O_i$.

Next, consider an outer edge \(e \in \E{T}\) between a node of \(T_x\) and a node of \(T_{x'}\), with $x'$ being the parent of $x$ in $\hat{T}$.
As $e$ corresponds to an edge in \(\hat{T}\), we make its edge-guard inherit the edge-guard from the \potdecomp, i.e., set \(\gamma(e) \coloneqq \hat{\gamma}(x',x)\).
See~\cref{fig:edge_guards} for an illustration of the edge-guards as well as the construction of the tree for the \dtcdecomp.


\begin{figure}[!ht]
	\begin{center}
		\begin{tikzpicture}
			\tikzstyle{outervertex} = [draw=gray, circle, inner sep=0.8cm, thick,fill=gray!30!white]
			\tikzstyle{innervertex} = [draw=LavenderMagenta, circle, inner sep=1.5mm, thick,fill=LavenderMagenta!10!white]
			\tikzstyle{innervertexRIM} = [draw=LavenderMagenta, circle, inner sep=1.5mm, thick]
			\tikzstyle{subsetvertex} = [draw=PastelOrange, circle, inner sep=1.5mm, thick,fill=PastelOrange]
			\def\outerdist{4}
			\node (centre) at (0,0) {};
			\node[outervertex] (c-1) at ($(centre)+(90:2)$) {};
			\node[outervertex] (c-2) at ($(c-1)+(200:{\outerdist})$) {};
			\node[outervertex] (c-3) at ($(c-1)+(340:{\outerdist})$) {};
			\node[outervertex] (c-4) at ($(c-2)+(220:{\outerdist})$) {};
			\node[outervertex] (c-5) at ($(c-2)+(300:{0.9*\outerdist})$) {};
			\node[outervertex] (c-6) at ($(c-3)+(240:{0.9*\outerdist})$) {};
			\node[outervertex] (c-7) at ($(c-3)+(320:{\outerdist})$) {};
			
			\begin{pgfonlayer}{background}
				\draw[line width=5pt,CornflowerBlue!50!white] (c-1) to (c-2);
				
				\draw[edge,gray] (c-1) to (c-2);
				\draw[edge,gray] (c-1) to (c-3);
				
				\draw[edge,gray] (c-2) to (c-4);
				\draw[edge,gray] (c-2) to (c-5);
				
				\draw[edge,gray] (c-3) to (c-6);
				\draw[edge,gray] (c-3) to (c-7);
			\end{pgfonlayer}
			
			\begin{pgfonlayer}{main}
				\node[innervertex] (v-1-1) at ($(c-1)+(270:0.7)$) {};
				\node[innervertex] (v-1-2) at ($(c-1)+(190:0.6)$) {};
				\node[innervertex] (v-1-3) at ($(c-1)+(80:0.7)$) {};
				\node[innervertex] (v-1-4) at ($(c-1)+(360:0.8)$) {};
				
				\begin{scope}
					\clip (v-1-4) circle (2.1mm);
					\node[subsetvertex] at ($(v-1-4)+(340:0.25)$) {};
				\end{scope}
				\begin{scope}
					\clip (v-1-2) circle (2.1mm);
					\node[subsetvertex] at ($(v-1-2)+(200:0.25)$) {};
				\end{scope}
				
				\node[innervertex] (v-2-1) at ($(c-2)+(80:0.7)$) {};
				\node[innervertex] (v-2-2) at ($(c-2)+(225:0.6)$) {};
				\node[innervertex] (v-2-3) at ($(c-2)+(340:0.7)$) {};
				
				\begin{scope}
					\clip (v-2-3) circle (2.1mm);
					\node[subsetvertex] at ($(v-2-3)+(200:0.27)$) {};
					\node[subsetvertex] at ($(v-2-3)+(20:0.27)$) {};
				\end{scope}
				
				\node[innervertex] (v-3-1) at ($(c-3)+(80:0.7)$) {};
				\node[innervertex] (v-3-2) at ($(c-3)+(150:0.6)$) {};
				\node[innervertex] (v-3-3) at ($(c-3)+(220:0.7)$) {};
				\node[innervertex] (v-3-4) at ($(c-3)+(300:0.8)$) {};
				\node[innervertex] (v-3-5) at ($(c-3)+(10:0.8)$) {};
				
				\begin{scope}
					\clip (v-3-3) circle (2.1mm);
					\node[subsetvertex] at ($(v-3-3)+(240:0.25)$) {};
				\end{scope}
				\begin{scope}
					\clip (v-3-4) circle (2.1mm);
					\node[subsetvertex] at ($(v-3-4)+(340:0.25)$) {};
				\end{scope}
				
				\node[innervertex] (v-4-1) at ($(c-4)+(90:0.3)$) {};
				\node[innervertex] (v-4-2) at ($(c-4)+(220:0.6)$) {};
				\node[innervertex] (v-4-3) at ($(c-4)+(345:0.7)$) {};
				
				\node[innervertex] (v-5-1) at ($(c-5)+(270:0.2)$) {};
				\node[innervertex] (v-5-2) at ($(c-5)+(170:0.6)$) {};
				\node[innervertex] (v-5-3) at ($(c-5)+(10:0.7)$) {};
				
				\node[innervertex] (v-6-1) at ($(c-6)$) {};
				\node[innervertex] (v-6-2) at ($(c-6)+(60:0.6)$) {};
				\node[innervertex] (v-6-3) at ($(c-6)+(190:0.7)$) {};
				\node[innervertex] (v-6-4) at ($(c-6)+(335:0.7)$) {};
				
				\node[innervertex] (v-7-1) at ($(c-7)+(100:0.6)$) {};
				\node[innervertex] (v-7-2) at ($(c-7)+(190:0.6)$) {};
				\node[innervertex] (v-7-3) at ($(c-7)+(280:0.7)$) {};
				\node[innervertex] (v-7-4) at ($(c-7)+(30:0.7)$) {};
			\end{pgfonlayer}
			
			\begin{pgfonlayer}{foreground}
				\node[innervertexRIM] at (v-1-1) {};
				\node[innervertexRIM] at (v-1-2) {};
				\node[innervertexRIM] at (v-1-3) {};
				\node[innervertexRIM] at (v-1-4) {};
				
				\node[innervertexRIM] at (v-2-1) {};
				\node[innervertexRIM] at (v-2-2) {};
				\node[innervertexRIM] at (v-2-3) {};
				
				\node[innervertexRIM] at (v-3-1) {};
				\node[innervertexRIM] at (v-3-2) {};
				\node[innervertexRIM] at (v-3-3) {};
				\node[innervertexRIM] at (v-3-4) {};
				\node[innervertexRIM] at (v-3-5) {};
			\end{pgfonlayer}
			
			\begin{pgfonlayer}{main}
				\draw[edge,LavenderMagenta] (v-1-1) to (v-1-2);
				\draw[edge,LavenderMagenta] (v-1-1) to (v-1-3);
				\draw[edge,LavenderMagenta] (v-1-1) to (v-1-4);
				
				\draw[edge,LavenderMagenta] (v-2-1) to (v-2-2);
				\draw[edge,LavenderMagenta] (v-2-1) to (v-2-3);
				
				\draw[edge,LavenderMagenta] (v-3-1) to (v-3-2);
				\draw[edge,LavenderMagenta] (v-3-2) to (v-3-3);
				\draw[edge,LavenderMagenta] (v-3-3) to (v-3-4);
				\draw[edge,LavenderMagenta] (v-3-3) to (v-3-5);
				
				\draw[edge,LavenderMagenta] (v-4-1) to (v-4-2);
				\draw[edge,LavenderMagenta] (v-4-1) to (v-4-3);
				
				\draw[edge,LavenderMagenta] (v-5-1) to (v-5-2);
				\draw[edge,LavenderMagenta] (v-5-1) to (v-5-3);
				
				\draw[edge,LavenderMagenta] (v-6-1) to (v-6-2);
				\draw[edge,LavenderMagenta] (v-6-1) to (v-6-3);
				\draw[edge,LavenderMagenta] (v-6-1) to (v-6-4);
				
				\draw[edge,LavenderMagenta] (v-7-1) to (v-7-2);
				\draw[edge,LavenderMagenta] (v-7-2) to (v-7-3);
				\draw[edge,LavenderMagenta] (v-7-3) to (v-7-4);
			\end{pgfonlayer}
			
			\begin{pgfonlayer}{main}
				\draw[edge,PastelOrange,out=170,in=30] (v-1-2) to (v-2-1);
				\draw[edge,PastelOrange,out=350,in=120] (v-1-4) to (v-3-2);
				\draw[edge,PastelOrange,out=210,in=10] (v-2-3) to (v-4-3);
				\draw[edge,PastelOrange,out=350,in=80] (v-2-3) to (v-5-1);
				\draw[edge,PastelOrange,out=225,in=80] (v-3-3) to (v-6-2);
				\draw[edge,PastelOrange,out=10,in=80] (v-3-4) to (v-7-4);
			\end{pgfonlayer}
			
			\node[circle,fill=white,inner sep=0.2mm] (guard) at ($(c-1)!0.4!(c-2)+(300:0.5)$) {\textcolor{CornflowerBlue}{\(\hat{\gamma}\)}};
			\node (x-1) at ($(v-1-2)!0.5!(v-2-1)+(100:0.35)$) {};
			\node (x-2) at ($(v-2-1)!0.5!(v-2-2)$) {};
			\node (x-3) at ($(v-2-1)!0.6!(v-2-3)$) {};
			\draw[directededge,CornflowerBlue,out=130,in=315] (guard) to (x-1.center);
			\draw[directededge,CornflowerBlue,out=150,in=30] (guard) to (x-2.center);
			\draw[directededge,CornflowerBlue,out=170,in=30] (guard) to (x-3.center);
		\end{tikzpicture}
	\end{center}
	\caption{Illustration of the decomposition and how the new edge-guards use the edge-guards of the sub-decompositions and the potato-decomposition.}
		\label{fig:edge_guards}
\end{figure}
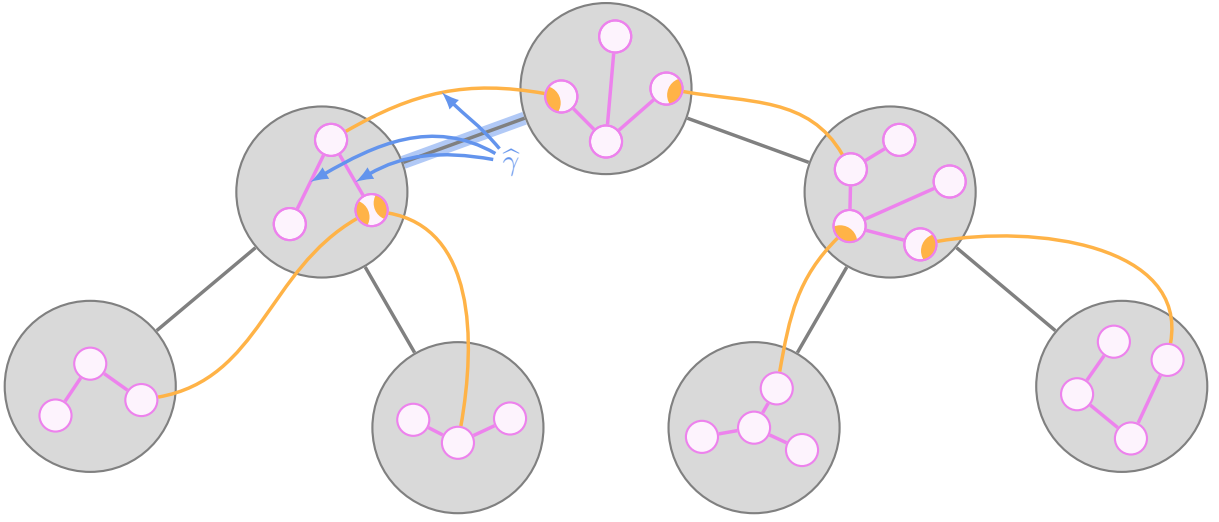

We observe that:
\begin{align}\label{eq:gamma-bound}
\begin{split}
\Abs{\gamma(e)} &{}\leq \Abs{\gamma_x(e)} 
			+ \thmFkt{lemma:potato}{f}{\thmFkt{lem:onion_star}{h}{q}}
			+ \Abs{\gamma_x(e)} \cdot \thmFkt{lemma:potato}{f}{\thmFkt{lem:onion_star}{h}{q}}\\
		&{}\leq 3 \cdot \thmFkt{lem:SCL}{f}{2k,2,2q} \cdot \Brace{1 + \thmFkt{lemma:potato}{f}{\thmFkt{lem:onion_star}{h}{q}}} 
		+ \thmFkt{lemma:potato}{f}{\thmFkt{lem:onion_star}{h}{q}}\\
		&{} = \thmFkt{thm:grid_or_small_dtcw}{f}{k}.
\end{split}
\end{align}


\begin{claim}
	For every \(e = tt' \in \E{T}\), the edge set \(\gamma(e)\) is indeed an edge-guard for \(e\), that is, it hits all cycles visiting both vertices from bags of \(\treeSplit{T}{t}{e}\) and of \(\treeSplit{T}{t'}{e}\).
\end{claim}
\begin{claimproof}
	Let \(e = tt'\in E(T)\). If $e$ is an outer edge, for some $(x, x') \in E(\hat{T})$ such that $t \in V(T_x)$, and $t' \in V(T_{x'})$, we have \(\beta(\treeSplit{T}{t}{e}) = \hat{\beta}(\treeSplit{{\hat{T}}}{x}{e})\) and \(\beta(\treeSplit{T}{t'}{e}) = \hat{\beta}(\treeSplit{{\hat{T}}}{x'}{e})\), so the statement follows because $\gamma(e)=\hat\gamma(x,x')$.
	
	Suppose now $e$ is an inner edge. Let $x$ be such that $t,t' \in V(T_x)$, and \(C\) be a cycle visiting vertices from the bags of both \(\treeSplit{T}{t}{e}\) and of \(\treeSplit{T}{t'}{e}\).
	
	If \(C\) visits a vertex that lies in \(\hat{\beta}(\treeUpper{T}{x})\), then it has to contain an edge from \(\hat{\gamma}\Brace{x',x}\), where \(x'\) is the parent of \(x\), by~\cref{item:tp:edgeguards}, so it is hit by \(\hat{\gamma}\Brace{x',x} \subseteq \gamma(e)\).
	
	
	So we can assume that $C$ does not visit vertices from \(\hat{\beta}(\treeUpper{T}{x})\).
	We consider a subgraph of $\torso{x}$ induced by all edges of the form $\variablestyle{origin}^{-1}_x(e)$, for every edge $e \in E(C)$ for which $\variablestyle{origin}^{-1}_x(e)$ exists, and edges $(a_i,b_i)$ for every child $y_i$ of $x$ in $\hat{T}$ such that $C$ intersects $\beta(\subtree{T}{y_i})$.
	This subgraph includes a cycle $C'$ that, if there exists $y_i$ such that $C$ intersects $\beta(\subtree{T}{y_i})$, contains the auxiliary edge $(a_i,b_i)$. If there is no such $y_i$, consider $C'=C$, as $C$ is a cycle in $D[\hat\beta(x)]\subseteq \torso{x}$.
	
	As \(C'\) is a cycle in \(\torso{x}\) that crosses \(e \in E(T_x)\), it is hit by \(\gamma_x(e)\).
	If $C'$ is hit by an edge \((a_i,b_i) \in \gamma_x(e)\), then we have \(C \cap \hat{\gamma}(x,y_i) \neq \emptyset\) and thus, \(C\) is hit by \(\hat{\gamma}(x,y_i)\subseteq \gamma(e)\).
	Otherwise \(C\) is hit by \(\torsoEdgeMap{x}(\gamma_x(e) \cap \domain{\torsoEdgeMap{x}}) \subseteq \gamma(e)\).
\end{claimproof}

\paragraph{Constructing the onion-guards:}
For every node \(t \in \V{T}\) there also is some \(T_x\) with \(x \in \V{\hat{T}}\) and \(t \in T_x\) and therefore there exists an onion-guard \(\theta_x(t)\) in the \dtcdecomp of the torso \(\torso{x}\).
Also, let \(x'\) be the parent of \(x\) in \(\hat{T}\) (unless \(x\) is the root).
We define the onion-guard of $t$ as follows:
\begin{align*}
	\theta(t) \coloneqq{} &\big(\torsoEdgeMap{x}(\theta_x(t) \cap \domain{\torsoEdgeMap{x}})\big) \cup{}\\
	&\hat{\gamma}(x',x)  \cup{}\\
	&\bigcup\Set{\hat{\gamma}(x,c_i) \mid i\in [\ell] \text{ and }\theta_x(t)\cap E(O_i)\neq \emptyset}.
\end{align*}
	We again start by proving that the obtained sizes match our claimed bounds.
	Note that we have
	\begin{align}\label{eq:theta-bound}
	\begin{split}
	\Abs{\theta(t)} &{}\leq \Abs{\theta_x(t)} 
	+ \thmFkt{lemma:potato}{f}{\thmFkt{lem:onion_star}{h}{q}}
			+ \Abs{\theta_x(t)} \cdot \thmFkt{lemma:potato}{f}{\thmFkt{lem:onion_star}{h}{q}}\\
			&{}\leq 3\cdot \thmFkt{lem:SCL}{f}{2k,2,2q} \cdot \Brace{1 + \thmFkt{lemma:potato}{f}{\thmFkt{lem:onion_star}{h}{q}}} 
			+ \thmFkt{lemma:potato}{f}{\thmFkt{lem:onion_star}{h}{q}}\\
			&{} = \thmFkt{thm:grid_or_small_dtcw}{f}{k}.
	\end{split}
	\end{align}

	All that is left to prove is that the new onion-guards do indeed hit the onions. Denote $\TT=(T,\beta,\gamma,\theta)$.
	
	\begin{claim}
		For every \(t \in \V{T}\) the set \(\theta(t)\) hits all in-onions or all out-onions of $t$ in $\TT$.
	\end{claim}
	\begin{claimproof}
		Let \(x \in V(\hat{T})\) be such that \(t \in \V{T_x}\).
		Note that each onion immersion model for which either root or tip is outside of the bags of the subtree of $\hat{T}$ rooted at \(x\), and the other one is inside, is hit by \(\hat{\gamma}(x',x)\), where $x'$ is the parent of $x$ in $\hat{T}$.
		So we only need to consider those onions that completely live in the subtree of $\hat{T}$ rooted at \(x\).
		For these we make use of the Quasi-onion Lemma~(\cref{lem:quasi_onions}).
		
		Assume in the decomposition $\TT_x$ of $\torso{x}$, the onion-guard \(\theta_x(t)\) hits all out-onions of \(t\).
		We aim to show that then \(\theta(t)\) also hits all out-onions of \(t\) in $\TT$.
		Let \(O\) be an out-onion of $t$ in $\TT$ (so $O$ is an onion immersion model in $D$) and let \(P_1,P_2,P_3\) be the three trails forming it.
		Let \(F \coloneqq \Set{(a_i,b_i) \mid y_i \text{ is a child of } x}\).
		As the mapping $\torsoEdgeMap{x}$ is a bijection, passing $P_1,P_2,P_3$ through $\torsoEdgeMap{x}$ and augmenting them with the edges of $F$, if necessary, naturally yields an \((F,(X,Y))\)-quasi-onion \(O'\) in \(\torso{x}\), for any cut $(X,Y)$ such that the root is in $X$ and the tip is in $Y$.
		
		If the root of \(O\) lies within \(\beta(t)\), then consider the bipartition \((X,Y) \coloneqq (\beta(t), \V{\torso{x}}\setminus\beta(t))\).
		Otherwise, let \(t'\) be the neighbour of \(t\) such that the bags of the subtree \((T_x)^{tt'}_{t'}\) contain the root of \(O\), then consider the bipartition \((X,Y) \coloneqq (\beta_x((T_x)^{tt'}_{t'}),\V{\torso{x}}\setminus\beta_x((T_x)^{tt'}_{t'}))\). Note that no edge of $F$ crosses the cut $(X,Y)$, and hence $O'$ is an $(F,(X,Y))$-quasi-onion.
		
		Now let \(O''\) be the minimal $(F,(X,Y))$-quasi-onion contained in the union of the paths of \(O'\) that has its root in $X$ and tip in $Y$.
		By~\cref{lem:quasi_onions}, \(O''\) is an onion and in particular, by our choice of \((X,Y)\), an out-onion of \(t\) in $\TT_x$. Therefore, $O''$ is hit by \(\theta_x(t)\).
		
		If the intersection of \(O''\) and \(\theta_x(t)\) contains an edge $e$ that is not an auxiliary edge in $\torso{x}$, then $\torsoEdgeMap{x}(e)$ also lies in \(\theta(t)\) and hits \(O\).
		Otherwise, \(O\) is hit by the set \(\bigcup\Set{\hat{\gamma}(x,c_i) \mid i\in [\ell] \text{ and }\theta_x(t)\cap E(O_i)\neq \emptyset}\), which was included in $\theta(t)$.
	\end{claimproof}
	
	We conclude that \(\TT=(T,\beta,\gamma,\theta)\) is indeed a \dtcdecomp of \(D\) of width at most \(\thmFkt{thm:grid_or_small_dtcw}{f}{k}\). This finishes the proof of the combinatorial bound.

	Finally, it is straightforward to verify that all the steps of the construction can be performed in fixed-parameter time, parameterized by $k$. This, together with the algorithmic statements asserted in \cref{lemma:potato,lem:SCL,lem:onion_star}, proves the algorithmic statement of \cref{thm:grid_or_small_dtcw}. 
\end{proof}

\section{Conclusions and open problems}\label{sec:conclusions}

In this work we have introduced a digraph parameter directed tree-cutwidth and we proved that the directed tree-cutwidth of a digraph $D$ is functionally equivalent to the largest order of a cylindrical wall that can be found in $D$ as an immersion.
This provides a suitable Grid Theorem for directed immersions. 
We expect that, similarly to other Grid Theorems found in structural graph theory, our theorem will lead to various follow-up results connected to immersions in digraphs.
One possibility for this are Erd\H{o}s-P\'osa-type statements.
Compared to other parameters, we understand relatively little about Erd\H{o}s-P\'osa-properties with respect to immersions.

The second natural direction are algorithmic consequences.
A major strength of treewidth and (undirected) tree-cutwidth is their connection to the rich world
of fixed-parameter tractability.
Note here that our main result, \cref{thm:grid_or_small_dtcw}, is already prepared for algorithmic usage, as it comes equipped with a fixed-parameter algorithm that either computes a directed tree-cut decomposition of bounded width or exposes an immersion model of a cylindrical wall of large order.
We leave working out such applications to future~work.

It is worth pointing out that sadly no advance can be made for the general \textsc{Edge-Disjoint Paths} problem.
As bounded directed tree-cutwidth implies bounded directed treewidth, one can directly obtain an XP-algorithm.
And as directed acyclic graphs (DAGs) have bounded directed tree-cutwidth, an FPT-algorithm is unlikely to exist under a
standard complexity theory assumption~\cite{edp-w1-hard}.

It follows from the algorithmic part of~\cref{thm:grid_or_small_dtcw} that there exists an FPT-time algorithm
that either correctly concludes that \dtcw of the input digraph \(D\) is greater than \(k\) (and provides a certificate in the form of a cylindrical wall immersion model),
or outputs a \dtcdecomp of \(D\) of width at most \(f(k)\) for some function \(f\).
However, the function \(f\) is non-elementary, and thus to make algorithmic results truly interesting, it would be important
to efficiently obtain decompositions of more reasonable width.
For directed treewidth there exists a simple constant factor approximation algorithm.
A similar result for undirected tree-cutwidth is established~\cite{tcw2018fptapprox}, it may provide a starting point for studying the directed setting.


In our presentation we refrained from giving an explicit bound on the function $\thmFkt{thm:grid_or_small_dtcw}{f}{k}$ bounding the directed tree-cutwidth of a digraph excluding \(\gridImm{k}\) as an immersion. Following the recent work of Hatzel, Kreuzer, Garlet Milani, and Muzi~\cite{COSS2024FOCS}, who improved the function in the Grid Theorem for directed treewidth (\cref{thm:gridtheorem}) to an elementary one, it is natural to ask whether $\thmFkt{thm:grid_or_small_dtcw}{f}{k}$ can also be bounded by an elementary function of $k$. An inspection of the arguments presented in this paper, as well as of the results from~\cite{bozyk2022digraphs,2024tripodEP} that we rely on, shows that we do \emph{not} obtain an elementary bound, but there is exactly one bottleneck preventing this: the proof of the Congested Onion Harvesting Lemma (\cref{lem:onion_harvesting}), presented in~\cite{bozyk2022digraphs} for $c=1$ and augmented in \cref{sec:appendix} for larger $c$, yields a non-elementary bound on the provided function $\thmFkt{lem:onion_harvesting}{g}{}$. Improving this bound to an elementary one, even just for $c=2$, would directly imply an elementary bound on the function $\thmFkt{thm:grid_or_small_dtcw}{f}{}$ provided by our main result. 

Finally, in~\cite{wollan2015structure} Wollan gave not only a Grid Theorem linking (undirected) tree-cutwidth of a graph~$G$ with the largest order of a wall immersion in $G$, but also proved a structure theorem approximately describing graphs excluding a complete graph $K_t$ as an immersion. A similar result was also proved independently by DeVos, McDonald, Mohar, and Scheide~\cite{DeVosMMS13}, also in the setting of Eulerian digraphs; however, as we discussed in \cref{sec:intro}, it is more in the spirit of our Potato Theorem (\cref{lemma:potato}). In essence, the theorem of Wollan says that if an undirected graph $G$ excludes $K_t$ as an immersion, then $G$ admits a tree-partition $\cal T$ such that:
\begin{itemize}
	\item every edge of $\cal T$ corresponds to a cut crossed by at most $t^2$ edges; and
	\item for every node $x$ of $\cal T$, the (appropriately defined) torso of $x$ contains only at most $t$ vertices of degree larger than $t^2$. 
\end{itemize}
In the light of this work, it is natural to ask whether there is some digraph analogue of this structural result of Wollan.
More precisely, if $\vec{K}_t$ denotes the complete digraph on $t$ vertices (i.e., with edges $(u,v)$ and $(v,u)$ present for every pair of vertices $u,v$), then can we prove a theorem approximately describing the structure of digraphs excluding $\vec{K}_t$ as an immersion?
Attempting to lift the approach of Wollan~\cite{wollan2015structure} directly immediately hits the following question: Is it true that if a digraph $D$ contains $k$ vertices that are pairwise $k$-flow-connected, for some large $k$ depending on $t$, then $D$ must contain $\vec{K}_t$ as an immersion? It turns out that the answer is negative even if $D$ is a plane digraph; we give an example below.

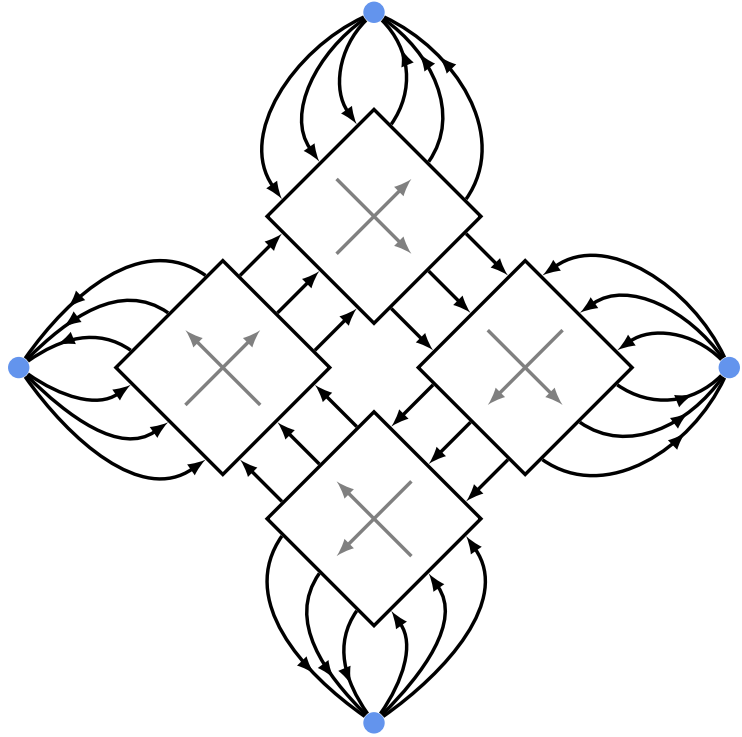
\begin{figure}[!ht]
	\centering
	\begin{tikzpicture}
		\node (centre) at (0,0) {};
		\node[draw,edge,rectangle,inner sep=1cm,rotate=45] (L) at ($(centre)+(180:2)$) {};
		\node[draw,edge,rectangle,inner sep=1cm,rotate=45] (T) at ($(centre)+(90:2)$) {};
		\node[draw,edge,rectangle,inner sep=1cm,rotate=45] (R) at ($(centre)+(0:2)$) {};
		\node[draw,edge,rectangle,inner sep=1cm,rotate=45] (B) at ($(centre)+(270:2)$) {};
		
		\node[vertex,CornflowerBlue,scale=1.5] (v-T) at ($(T)+(90:2.7)$) {};
		\node[vertex,CornflowerBlue,scale=1.5] (v-R) at ($(R)+(0:2.7)$) {};
		\node[vertex,CornflowerBlue,scale=1.5] (v-B) at ($(B)+(270:2.7)$) {};
		\node[vertex,CornflowerBlue,scale=1.5] (v-L) at ($(L)+(180:2.7)$) {};
		
		\draw[directededge,out=225,in=125] (v-T) to ($(T.north)+(45:0.7)$);
		\draw[directededge,out=215,in=125] (v-T) to ($(T.north)$);
		\draw[directededge,out=205,in=125] (v-T) to ($(T.north)+(225:0.7)$);
		
		\draw[directededgedecorated,out=55,in=315] ($(T.east)+(135:0.7)$) to (v-T);
		\draw[directededgedecorated,out=55,in=325] ($(T.east)$) to (v-T);
		\draw[directededgedecorated,out=55,in=335] ($(T.east)+(315:0.7)$) to (v-T);
		
		\draw[directededgedecorated,in=225,out=325] ($(R.south)+(45:0.7)$) to (v-R);
		\draw[directededgedecorated,in=235,out=325] ($(R.south)$) to (v-R);
		\draw[directededgedecorated,in=245,out=325] ($(R.south)+(225:0.7)$) to (v-R);
		
		\draw[directededge,out=115,in=35] (v-R) to ($(R.east)+(135:0.7)$);
		\draw[directededge,out=125,in=35] (v-R) to ($(R.east)$);
		\draw[directededge,out=135,in=35] (v-R) to ($(R.east)+(315:0.7)$);
		
		\draw[directededge,out=35,in=305] (v-B) to ($(B.south)+(45:0.7)$);
		\draw[directededge,out=45,in=305] (v-B) to ($(B.south)$);
		\draw[directededge,out=55,in=305] (v-B) to ($(B.south)+(225:0.7)$);
		
		\draw[directededgedecorated,out=235,in=145] ($(B.west)+(135:0.7)$) to (v-B);
		\draw[directededgedecorated,out=235,in=135] ($(B.west)$) to (v-B);
		\draw[directededgedecorated,out=235,in=125] ($(B.west)+(315:0.7)$) to (v-B);
		
		\draw[directededge,out=325,in=215] (v-L) to ($(L.west)+(135:0.7)$);
		\draw[directededge,out=315,in=215] (v-L) to ($(L.west)$);
		\draw[directededge,out=305,in=215] (v-L) to ($(L.west)+(315:0.7)$);
		
		\draw[directededgedecorated,out=145,in=55] ($(L.north)+(45:0.7)$) to (v-L);
		\draw[directededgedecorated,out=145,in=45] ($(L.north)$) to (v-L);
		\draw[directededgedecorated,out=145,in=35] ($(L.north)+(225:0.7)$) to (v-L);
		
		\draw[directededge] ($(T.south)+(45:0.7)$) to ($(R.north)+(45:0.7)$);
		\draw[directededge] ($(T.south)$) to ($(R.north)$);
		\draw[directededge] ($(T.south)+(225:0.7)$) to ($(R.north)+(225:0.7)$);
		
		\draw[directededge] ($(R.west)+(135:0.7)$) to ($(B.east)+(135:0.7)$);
		\draw[directededge] ($(R.west)$) to ($(B.east)$);
		\draw[directededge] ($(R.west)+(315:0.7)$) to ($(B.east)+(315:0.7)$);
		
		\draw[directededge] ($(B.north)+(45:0.7)$) to ($(L.south)+(45:0.7)$);
		\draw[directededge] ($(B.north)$) to ($(L.south)$);
		\draw[directededge] ($(B.north)+(225:0.7)$) to ($(L.south)+(225:0.7)$);
		
		\draw[directededge] ($(L.east)+(135:0.7)$) to ($(T.west)+(135:0.7)$);
		\draw[directededge] ($(L.east)$) to ($(T.west)$);
		\draw[directededge] ($(L.east)+(315:0.7)$) to ($(T.west)+(315:0.7)$);
		
		\draw[directededge,gray] ($(T)+(225:0.7)$) to ($(T)+(45:0.7)$);
		\draw[directededge,gray] ($(T)+(135:0.7)$) to ($(T)+(315:0.7)$);
		
		\draw[directededge,gray] ($(R)+(135:0.7)$) to ($(R)+(315:0.7)$);
		\draw[directededge,gray] ($(R)+(45:0.7)$) to ($(R)+(225:0.7)$);
		
		\draw[directededge,gray] ($(B)+(45:0.7)$) to ($(B)+(225:0.7)$);
		\draw[directededge,gray] ($(B)+(315:0.7)$) to ($(B)+(135:0.7)$);
		
		\draw[directededge,gray] ($(L)+(315:0.7)$) to ($(L)+(135:0.7)$);
		\draw[directededge,gray] ($(L)+(225:0.7)$) to ($(L)+(45:0.7)$);
	\end{tikzpicture}
	\caption{The plane digraph \(D\) built using $4$ grid-like constructions from \cite[Figure 2]{bozyk2022digraphs}.
		It has \textcolor{CornflowerBlue}{4 vertices} that are \(k\)-flow-connected but does not contain an immersion model of $\vec{K}_3$.}
	\label{fig:highly_connected_set_no_K5}
\end{figure}

Take several copies of the grid-like construction from \cite[Figure 2]{bozyk2022digraphs}, and arrange them as follows; see~\cref{fig:highly_connected_set_no_K5}.
In each copy, remove the eastern terminal.
Arrange the copies into a uni-directed cycle-of-grids by adding edges from the south-east side of each grid to the north-east side of the next grid.
Now, the remaining western terminals are pairwise highly flow-connected.
However, for any terminal \(t\) and paths \(P,Q\) such that \(P\) connects \(t\) to some other terminal and \(Q\) connects some other terminal to \(t\), the paths \(P\) and \(Q\) must share an edge.
Since the terminals are the only vertices of degree higher than 3, the resulting digraph does not contain \(\vec{K}_3\) as an immersion.

Given the example provided above, it is unclear to us what structure should be expected of digraphs excluding $\vec{K}_t$ as an immersion. Even more generally, what connectivity conditions on a digraph force the existence of an immersion model of $\vec{K}_t$? In~\cite{DeVosMMS12}, DeVos, McDonald, Mohar, and Scheide showed that assuming that the digraph has a large minimum outdegree and a large minimum indegree (both in terms of $t$) does \emph{not} suffice (but does, if the digraph is Eulerian). We conjecture that the stronger assumption that all the vertices have large cycle-degree is sufficient.

\begin{conjecture}\label{con:cycdeg}
	There exists a function $\thmFkt{con:cycdeg}{f}{}\colon \N\to \N$ such that for every $t\in \N$, every simple digraph of minimum cycle-degree at least $\thmFkt{con:cycdeg}{f}{t}$ contains $\vec{K}_t$ as an immersion.
\end{conjecture}

It can be verified that the example provided~in~\cite{DeVosMMS12} contains vertices of cycle-degree~$1$, hence it is not a counterexample to \cref{con:cycdeg}.



\paragraph*{Acknowledgments.}
We thank the reviewers for pointing us to relevant literature simplifying the presentation of our proof, as well as for finding a mistake in an earlier version of the counterexample depicted in \cref{fig:highly_connected_set_no_K5}. We also thank \L{}ukasz Bo\.zyk and Oscar Defrain for discussions at the early stage of this~project.

\paragraph*{AI disclosure.}
AI tools were used only for proofreading the manuscript and minor copy-editing. All the mathematical ideas are due to the authors and they take full responsibility for the reported findings.

\bibliographystyle{alphaurl}
\bibliography{literature}

\appendix 
\renewcommand*{\overline}[1]{\bar{#1}}

\section{Proof of \cref{lem:onion_harvesting}}
\label{sec:appendix}
Let $\cP$ and $\cQ$ be two families of trails in $D$ satisfying the four bullets in the  statement of \cref{lem:onion_harvesting}. 
We call $(\cP,\cQ)$ a \emph{well-crossing pair with congestion $c$ rooted at $x$}.

The statement of \cref{lem:onion_harvesting} is a strengthening of the Onion Harvesting Lemma from \cite{bozyk2022digraphs} that does not require that the trails in $\cP$ and $\cQ$ are actual paths, but every vertex except $x$ may repeat a fixed number of times within the same path. 
By closely examining the proof of OHL we may observe that the assumption about $\cP$ and $\cQ$ being paths is used only in the proof of the auxiliary Lemma 3.2  \cite{bozyk2022digraphs}. We sketch the proof of the analogue of Lemma 3.2 for our setting, although we omit the arguments that are the same as in the original proof of OHL. 

\begin{lemma}\label{lem:single-out}
Let $c,k\geq 1$ be integers. There exists a function $f_c\colon\N \to \N$ with the following property: If in a digraph $D$ there exists a well-crossing pair $(\cP,\cQ)$ with congestion $c$ rooted at $x$ with $|\cP|= f_c(k)$ and $|\cQ|= f_c(k)$, then in $D$ there exist an onion-immersion $\pi$ with root $x$, and a well-crossing pair $(\cP^*,\cQ^*)$ with congestion $c$ rooted at $x$ with $|\cP^*|= k$, $|\cQ^*|= k$ such that:
\begin{itemize}
\item $A(\cP^*) \subseteq A(\cP)$, and $A(\cQ^*) \subseteq A(\cQ)$;
\item all paths in $\cP^*\cup\cQ^*$ are arc-disjoint with $\pi$;
\item no arc in $A(\cP^*)\cap A(\cQ^*)$ has the sink of $\pi$ for tail.
\end{itemize}
\end{lemma}

We need the following classical result of Kővári, Sós and Turán~\cite{kHovari1954problem}.
\begin{theorem}[\cite{kHovari1954problem}]\label{thm:kst}
There exists a function $\thmFkt{thm:kst}{g}{}: \N\to (0,\infty)$ such that if a graph $G$ has at least $\thmFkt{thm:kst}{g}{q} \cdot |V(G)|^{2-1/q}$ edges for some $q \in \mathbb{N}$, then it contains $K_{q,q}$ as a subgraph.
\end{theorem}

\newcommand{\nextarc}{\mathsf{next}}
\newcommand{\prevarc}{\mathsf{prev}}

For a trail $P$ let $<_P$ be the natural linear order of $A(P)$ defined by $a_i <_P a_j$ if and only if $i < j$. 
If $a_i$ directly precedes $a_j$ in $<_P$, we write $a_i=\prevarc_P(a_j)$ and $a_j=\nextarc_P(a_i)$.
We call an arc $e$ that belongs to $A(P) \cap A(Q)$ for some $P \in \cP$ and $Q \in \cQ$ a \emph{crossing} of $(\cP,\cQ)$, and specify \emph{$P$-crossing} (resp.~\emph{$Q$-crossing}, \emph{$(P,Q)$-crossing}) for a crossing of $(\cP,\cQ)$ contained in $P$ (resp.~$Q$, $A(P)\cap A(Q)$).
A~crossing $e$ is \emph{$(P,Q)$-safe} (or simply \emph{safe}) if there exist at least $|\cQ|/3$ trails in $\cQ \setminus \{Q\}$ whose $<_P$-minimal crossing with $P$ precedes $e$ in $<_P$.
Otherwise, it is said to be $(P,Q)$-\emph{dangerous} (or simply \emph{dangerous}).
We say that a trail $P$ is \emph{$Q$-dangerous} if it contains a dangerous $(P,Q)$-crossing.
If a trail $Q\in \cQ$ is trimmed to $\overline{Q}$, then we say that $P$ is \emph{$\overline{Q}$-dangerous} if there exists a dangerous $(P,Q)$-crossing belonging to $\overline{Q}$.

\begin{proof}[Proof of~\cref{lem:single-out}]
Define functions $g_c$ and $f_c$ by:
\begin{align*}g_c(k)&=\max\left\{8,\left\lceil\frac12\left(16\cdot \thmFkt{thm:kst}{g}{ck +k}\right)^{ck+k}\right\rceil\right\}, \\ f_c(k)&=\max\left\{3k(c+1),3\cdot\left\lceil\frac14(36 \cdot \thmFkt{thm:kst}{g}{g_c(k)})^{g_c(k)}\right\rceil\right\}.
\end{align*}
Fix $Q \in \cQ$, and let $e$ be the $<_Q$-smallest arc of $Q$ such that the number of paths in $\cP$ that have a common arc with the suffix of $Q$ that starts in $\nextarc_Q(e)$ equals $|\cP|/3$. 
Let $\overline{Q}$ be the suffix of $Q$ that starts in $e$.

\textbf{Case 1.} Suppose that $\cP(\overline{Q})=\{P \in \cP \mid A(P) \cap A(\overline{Q})\}$ contains at least two $\overline{Q}$-dangerous paths. 
Let $d$ be the $<_Q$-greatest dangerous crossing in $Q$ and let $P_2$ be the path such that $d$ is a $(P_2,Q)$-crossing. 
Then $d$ belongs to $\overline{Q}$.
Let $e'$ be the $<_Q$-greatest dangerous crossing in $A(Q)\setminus A(P_2)$ and let $P_1\in\cP(\overline{Q})\setminus\{P_2\}$ be the path such that $e'$ is a $(P_1,Q)$-crossing (with possibly $e=e'$).
Let $d'$ be the $<_{P_2}$-smallest $(P_2,Q)$-crossing belonging to the suffix of $Q$ that starts with $\nextarc_Q(e')$.
Clearly since $d$ is dangerous and belongs to $\overline{Q}$, $d'$ must be dangerous (with possibly $d=d'$). 

Let $y=\tail{d'}$.
We consider the $(x,y)$-trails: one that goes along $P_1$ to $e'$ and then continues along $Q$ to $y$, and the other that goes along $P_2$, and the $(y,x)$-trail that goes along $Q$.
All of these are arc-disjoint.
Now the minimal $(x,y)$-trails/$(y,x)$-trail contained in the above (as $y$ may occur more than one on each of them) form an onion immersion model with root $x$ and tip $y$.

Let $\cQ^\circ \subseteq \cQ$ be the set of trails which are arc-disjoint both with the prefix of $P_1$ that ends with $e'$ and the prefix of $P_2$ that ends with $d'$ (so in particular $Q\notin \cQ^\circ$).
Since $e'$ and $d'$ are both dangerous, $|\cQ^\circ|> |\cQ|-|\cQ|/3-|\cQ|/3=|\cQ|/3\geq (c+1)k$.
Let us pick an arbitrary collection $\cP^*\subseteq\cP\setminus \{P_1,P_2\}$ of $k$ trails that are arc-disjoint with the suffix of $Q$ that starts with $e'$.
Such a collection exists, as the number of trails in $\cP$ that intersect the suffix of $Q$ that starts with $\nextarc_Q(e)$ equals $|\cP|/3$, yielding $\smash{\frac23}|\cP|\geq 2(c+1)k$ candidate trails for $\cP^*$.
Since each trail in such a collection has at most $c$ arcs whose tail is $y$ (as each vertex occurs at most $c$ times on each trail), there are at most $c|\cP^*|=ck$ arcs in $A(\cQ^\circ)\cap A(\cP^*)$ whose tail is $y$. Let $\cQ^*\subseteq \cQ^\circ$ comprise any $n$ trails disjoint with these arcs.
Clearly, since $\cP^* \subseteq \cP$ and $\cQ^* \subseteq \cQ$, $(\cP^*,\cQ^*)$ is a well-crossing pair. Hence $\pi$, $\cP^*$ and $\cQ^*$ satisfy all the desired properties.

\textbf{Case 2.} Suppose that there is at most one $\overline{Q}$-dangerous path in $\cP(\overline{Q})$. 
Let $\cS$ be the set of those trails in $\cP(\overline{Q})$ whose all crossings with $\overline{Q}$ are safe, and let $\overline{\cS}$ be the set of trails from $\cS$ trimmed so that they start at $x$, have no arcs in common with $\overline{Q}$ and are maximally long.
Since $|\cP(\overline{Q})| = |\cP|/3+1$ and there is at most one dangerous trail in $\cP(\overline{Q})$, we derive $|\overline{\cS}|\geq|\cP|/3$. 

\begin{claim}[\cite{bozyk2022digraphs}]\label{cla:same-paths}
There exist subfamilies $\overline{\cS}'\subseteq \overline{\cS}$, $\cQ'\subseteq \cQ\setminus \{Q\}$ such that $|\overline{\cS}'|=g_c(k)$, $|\cQ'|=g_c(k)$ and $(\overline{\cS}',\cQ')$ is well-crossing with congestion $c$ rooted at $x$.
\end{claim}

Consider the pair $(\overline{\cS}',\cQ')$ given by \cref{cla:same-paths} and
let $\cS'$ be the set of paths in $\cS$ from which $\overline{\cS}'$ originates (i.e., the paths from $\overline{\cS}'$ before trimming).
\begin{claim}[\cite{bozyk2022digraphs}]\label{cla:paths}
There exist 
\begin{itemize}
\item distinct $P_1,P_2 \in \cS'$, 
\item a subfamily $\cP^\circ$ of $\overline{\cS}'$, and 
\item a subfamily $\cQ^\circ$ of the set $\overline{\cQ}'$ of trails from $\cQ'$ trimmed so that they end at $x$, have no arcs in common with $P_1$ or $P_2$ and are longest possible, 
\end{itemize}
such that $|\cP^\circ| = |\cQ^\circ| = (c+1)k$ and $(\cP^\circ,\cQ^\circ)$ is well-crossing with congestion $c$ rooted at~$x$.\footnotemark
\end{claim}
\footnotetext{The proof of \cref{cla:paths} in \cite{bozyk2022digraphs} uses the definition of $g_c$ for $c=1$, but the correctness of calculations for the more general case is straightforward to verify.}

Consider $P_1$, $P_2$ and $(\cP^\circ,\cQ^\circ)$ as in the claim above and denote
$e_i=\min_{<_{P_i}}\{A(\overline{Q})\cap A(P_i)\}$ for $i=1,2$. 
Assume without loss of generality that $e_1<_Q e_2$, i.e., that the $<_{P_1}$-smallest $(P_1,\overline{Q})$-crossing is further from $x$ on $\overline{Q}$ than the $<_{P_2}$-smallest $(P_2,\overline{Q})$-crossing. Let $y=\tail{e_2}$.
By the choice of $e_1$ and $e_2$, we have two $(x,y)$-trails: the one that goes along $P_1$ to $e_1$ and then continues along $Q$ to $y$, and the one that is a prefix of $P_2$, and one $(y,x)$-trail that is a suffix of $Q$.
All of these are pairwise arc-disjoint.
Now the minimal $(x,y)$-trails/$(y,x)$-trail contained in the above form an onion-immersion with root $x$ and tip $y$. 

The last step (to ensure that $y$ will not be the tail of a crossing in the found well-crossing pair) is performed similarly as in the first case: we pick an arbitrary collection $\cP^\ast$ of exactly $k$ trails among the $(c+1)n$ paths in $\cP^\circ$, and select a subfamily $\cQ^\ast$ consisting of $k$ trails which do not contain crossings of tail $y$ among the $(c+1)k$ trails in $\cQ^\circ$.
\end{proof}

Last, we note that one can adapt the proofs given in~\cite{bozyk2022digraphs} using standard tools in order to obtain the algorithmic counterparts of both~\cref{lem:onion_harvesting} and~\cref{lem:onion_star}.

\end{document}